\documentclass[11pt,reqno]{amsart}

\usepackage{amsmath, amsthm, amssymb, amscd, amsfonts}
\usepackage{fontsmpl}
\usepackage{fontsmpl,layout}
\usepackage{color}
\usepackage[all]{xy}             % diagramas
 \usepackage[active]{srcltx}   % busqueda inversa

\font \manual=manfnt at 7pt
\font \smallrm=cmr10 at 10truept
 at 9truept

\font \smallsl=cmsl10 at 9pt
\font \smallbf=cmbx10 at 9pt

\numberwithin{equation}{section}

\newcommand{\subu}[2]{{#1}_{\raise-2pt\hbox{$ \scriptstyle #2 $}}}
\newcommand{\subd}[3]{{#1}_{\raise-2pt\hbox{$ \scriptstyle #2 #3 $}}}

\newtheorem{lema}{Lemma}[subsection]
\newtheorem{theorem}[lema]{Theorem}

\newtheorem{prop}[lema]{Proposition}

\theoremstyle{definition}
\newtheorem{definition}[lema]{Definition}

\newtheorem{rmk}[lema]{Remark}
\newtheorem{rmks}[lema]{Remarks}
\newtheorem{free text}[lema]{}

\theoremstyle{remark}

\newcommand{\dbend} {$ {}^{\hbox{\manual \char127}} $}
\newcommand\id{\operatorname{id}}
\newcommand\ad{\operatorname{ad}}

\newcommand\co{\operatorname{co}}

\newcommand\cop{\operatorname{cop}}
\newcommand\End{\operatorname{End}}
\newcommand\Hom{\operatorname{Hom}}

\newcommand\Ker{\operatorname{Ker}}

\newcommand\copr{\operatorname{co}}
\newcommand \smallast {{\operatorname{\raise1,5pt\hbox{$ \scriptscriptstyle \ast $}}}}

\newcommand{\eps}{\varepsilon}
\newcommand{\ot}{\otimes}
\newcommand{\com}{\Delta}

\newcommand{\E}{{\mathcal E}}
\newcommand{\F}{{\mathcal F}}

\newcommand{\Kc}{{\mathcal K}}
\newcommand{\Lc}{{\mathcal L}}
\newcommand{\R}{{\mathcal R}}

\newcommand{\Zc}{{\mathcal Z}}
\newcommand{\bE}{\boldsymbol{E}}
\newcommand{\Ebar}{\overline{E}}
\newcommand{\Fbar}{\overline{F}}
\newcommand{\ebar}{\overline{e}}
\newcommand{\fbar}{\overline{f}}

\newcommand{\Oc}{{\mathcal O}}
\newcommand{\ydh}{{}^H_H\mathcal{YD}}

\newcommand\Lie{\operatorname{Lie}}

\def \bq {\mathbf{q}}

\def \bx {\mathbf{x}}
\def \by {\mathbf{y}}

\def\NN{\mathbb{N}}
\def\ZZ {\mathbb{Z}}
\def\QQ{\mathbb{Q}}
\def\CC{\mathbb{C}}
\def \Zppm {{\mathbb{Z}\big[\hskip1pt p,p^{-1}\big]}}
\def \Zqqm {{\mathbb{Z}\big[\hskip1pt q,q^{-1}\big]}}
\def \Zqqmsq {{\mathbb{Z}\big[\hskip1pt q^{1/2}, \hskip1pt q^{-1/2\,}\big]}}

\def \Zabqpm {{\ZZ_A\big[\bq^{\pm 1}\big]}}

\def \Qabq {\QQ_A(\bq\hskip1,7pt)}

\def \Rbq {{\mathcal{R}_{\,\bq}}}
\def \RbqB {{\mathcal{R}_{\,\bq}^{\scriptscriptstyle {B}}}}

\def \Rbqsq {{\mathcal{R}_{\,\bq}^{\scriptscriptstyle {\sqrt{\phantom{I}}}}}}
\def \Rbqsquno {{\mathcal{R}_{\,\bq\,,1}^{\scriptscriptstyle {\sqrt{\phantom{I}}}}}}
\def \RbqBsq {{\mathcal{R}_{\,\bq}^{\scriptscriptstyle {B,{\sqrt{\phantom{I}}}}}}}
\def \RbqBsquno {{\mathcal{R}_{\,\bq\,,1}^{\scriptscriptstyle {B,{\sqrt{\phantom{I}}}}}}}

\def \Rbquno {{\mathcal{R}_{\,\bq\hskip1pt,1}}}

\def \RbqBuno {{\mathcal{R}_{\,\bq\hskip1pt,1}^{\scriptscriptstyle {B}}}}

\def \Rbqeps {{\mathcal{R}_{\,\bq\hskip1pt,\hskip1pt{}\varepsilon}}}
\def \Rbqsqeps {{\mathcal{R}_{\,\bq\hskip1pt,\hskip1pt{}\varepsilon}^{\scriptscriptstyle {\sqrt{\phantom{I}}}}}}

\def \RbqBeps {{\mathcal{R}_{\,\bq\hskip1pt,\hskip1pt{}\varepsilon}^{\scriptscriptstyle {B}}}}
\def \RbqBsqeps {{{\mathcal{R}_{\,\bq\hskip1pt,\hskip1pt{}\varepsilon}^{\scriptscriptstyle {B,{\sqrt{\phantom{I}}}}}}}}

\def \Fbq {{\mathcal{F}_{\,\bq}}}
\def \FbqB {{\mathcal{F}_{\,\bq}^{\,\scriptscriptstyle {B}}}}

\def \Fbqsq {{\mathcal{F}_{\,\bq}^{\,\scriptscriptstyle {\sqrt{\phantom{I}}}}}}
\def \FbqBsq {{\mathcal{F}_{\,\bq}^{\,\scriptscriptstyle {B,{\sqrt{\phantom{I}}}}}}}

\def\II{\mathcal{I}}

\def\SS{\mathcal{S}}

\def\k{\Bbbk}

\def\J{\mathfrak{J}}

\def \Bpicc {{\scriptscriptstyle B}}
\def \Dpicc {{\scriptscriptstyle D}}

\def \erm {\mathrm{e}}
\def \etilderm {\tilde{\mathrm{e}}}
\def \frm {\mathrm{f}}
\def \ftilderm {\tilde{\mathrm{f}}}
\def \krm {\mathrm{k}}
\def \kdotrm {\dot{\mathrm{k}}}

\def \lrm {\mathrm{l}}
\def \ldotrm {\dot{\mathrm{l}}}

\def \hrm {\mathrm{h}}
\def \trm {\mathrm{t}}

\def \lieg{\mathfrak{g}}
\def \liegd{\mathfrak{g}_{\raise-2pt\hbox{$ {\scriptscriptstyle (D)} $}}}
\def \liegb{\mathfrak{g}_{\raise-2pt\hbox{$ \Bpicc $}}}
\def \liegdotb{\dot{\lieg}_{\raise-2pt\hbox{$ \Bpicc $}}}
\def \liegtildeb{\tilde{\lieg}_{\raise-2pt\hbox{$ \Bpicc $}}}
\def \lieghatb{\hat{\lieg}_{\raise-2pt\hbox{$ \Bpicc $}}}
\def \liegdqcheck{\lieg_{\raise-2pt\hbox{$ \Dpicc , \hskip-0,9pt{\scriptstyle \check{\bq}} $}}}
\def \Gtildebstar {\widetilde{G}^{\,\raise2pt\hbox{$ \scriptstyle * $}}_{\!\Bpicc}}

\def \liea{\mathfrak{a}}
\def \lieb{\mathfrak{b}}
\def \lieh{\mathfrak{h}}
\def \liehd{\mathfrak{h}_{\raise-2pt\hbox{$ \Dpicc $}}}

\def \lien{\mathfrak{n}}

\def \QEq{U_\bq(\hskip0,8pt\lieg)}
\def \QEqsq{U_\bq^{\scriptscriptstyle {\sqrt{\ }}}(\hskip0,8pt\lieg)}
\def \QEqcheck {U_{\check{\bq}}(\hskip0,8pt\lieg)}
\def \QEqchecksq {U_{\check{\bq}}^{\scriptscriptstyle {\sqrt{\ }}}(\hskip0,8pt\lieg)}
\def \Uhat{\widehat{U}}
\def \Uhatdot{\dot{\widehat{U}}}
\def \Utilde{\widetilde{U}}
\def \Uhatdotqgd {\Uhatdot_\bq(\hskip0,8pt\lieg)}
\def \Uhatdotqunogd {\Uhatdot_{\bq\hskip1pt,1}(\hskip0,8pt\lieg)}
\def \UhatdotQqunogd {\Uhatdot^{\raise-5pt\hbox{$ \, \scriptstyle \QQ $}}_{\bq\hskip1pt,1}(\hskip0,8pt\lieg)}
\def \Uhatdotqepsgd {\Uhatdot_{\bq\hskip1pt,\hskip1pt\varepsilon}(\hskip0,8pt\lieg)}
\def \UhatdotQqeps {\Uhatdot^{\raise-5pt\hbox{$ \, \scriptstyle \QQ $}}_{\bq\hskip1pt,\hskip1pt\varepsilon}}
\def \Uhatqgd {\Uhat_\bq(\hskip0,8pt\lieg)}
\def \Uhatqunogd {\Uhat_{\bq\hskip1pt,1}(\hskip0,8pt\lieg)}
\def \UhatQqunogd {\Uhat^{\raise-5pt\hbox{$ \, \scriptstyle \QQ $}}_{\bq\hskip1pt,1}(\hskip0,8pt\lieg)}
\def \Uhatqepsgd {\Uhat_{\bq\hskip1pt,\hskip1pt\varepsilon}(\hskip0,8pt\lieg)}
\def \UhatQqeps {\Uhat^{\raise-5pt\hbox{$ \, \scriptstyle \QQ $}}_{\bq\hskip1pt,\hskip1pt\varepsilon}(\hskip0,8pt\lieg)}

\def \Utildeqgd {\Utilde_\bq(\hskip0,8pt\lieg)}
\def \UtildeqQgd {\Utilde^{\raise-5pt\hbox{$ \, \scriptstyle \QQ $}}_\bq(\hskip0,8pt\lieg)}
\def \Utildequnogd {\Utilde_{\bq\hskip1pt,1}(\hskip0,8pt\lieg)}
  \def \UtildeqQunogd {\Utilde^{\raise-3pt\hbox{$ \, \scriptstyle \QQ $}}_{\bq\hskip1pt,1}(\hskip0,8pt\lieg)}
\def \Utildeqepsgd {\Utilde_{\bq\hskip1pt,\hskip1pt\varepsilon}(\hskip0,8pt\lieg)}
  \def \UtildeqQepsgd {\Utilde^{\raise-5pt\hbox{$ \, \scriptstyle \QQ $}}_{\bq\hskip1pt,\hskip1pt\varepsilon}(\hskip0,8pt\lieg)}
\def \uhatdotqeps {{\dot{\widehat{\mathfrak{u}}}}_{\,\bq\hskip1pt,\hskip1pt\varepsilon}}

  \def \uhatdotQqeps {{\dot{\widehat{\mathfrak{u}}}}^{\raise-7pt\hbox{$ \; \scriptstyle \QQ $}}_{\,\bq\hskip1pt,\hskip1pt\varepsilon}}
\def \uhatdotqepsgd {{\dot{\widehat{\mathfrak{u}}}}_{\,\bq\hskip1pt,\hskip1pt\varepsilon}(\hskip0,8pt\lieg)}
  \def \uhatdotQqepsgd {{\dot{\widehat{\mathfrak{u}}}}^{\raise-7pt\hbox{$ \; \scriptstyle \QQ $}}_{\,\bq\hskip1pt,\hskip1pt\varepsilon}(\hskip0,8pt\lieg)}
\def \uhatqeps {{\widehat{\mathfrak{u}}}_{\,\bq\hskip1pt,\hskip1pt\varepsilon}}
\def \uhatqcheckeps {{\widehat{\mathfrak{u}}}_{\,\check{\bq}\hskip1pt,\hskip1pt\varepsilon}}
\def \uhatqQeps {{\widehat{\mathfrak{u}}}^{\raise-5pt\hbox{$ \; \scriptstyle \QQ $}}_{\,\bq\hskip1pt,\hskip1pt\varepsilon}}
\def \uhatqepsgd {{\widehat{\mathfrak{u}}}_{\,\bq\hskip1pt,\hskip1pt\varepsilon}(\hskip0,8pt\lieg)}
  \def \uhatqQepsgd {{\widehat{\mathfrak{u}}}^{\raise-5pt\hbox{$ \; \scriptstyle \QQ $}}_{\,\bq\hskip1pt,\hskip1pt\varepsilon}(\hskip0,8pt\lieg)}
\def \utildeqeps {{\widetilde{\mathfrak{u}}}_{\,\bq\hskip1pt,\hskip1pt\varepsilon}}
\def \utildeqcheckeps {{\widetilde{\mathfrak{u}}}_{\,\check{\bq}\hskip1pt,\hskip1pt\varepsilon}}
  \def \utildeqQeps {{\widetilde{\mathfrak{u}}}^{\raise-5pt\hbox{$ \; \scriptstyle \QQ $}}_{\,\bq\hskip1pt,\hskip1pt\varepsilon}}
\def \utildeqepsgd {{\widetilde{\mathfrak{u}}}_{\,\bq\hskip1pt,\hskip1pt\varepsilon}(\hskip0,8pt\lieg)}
  \def \utildeqQepsgd {{\widetilde{\mathfrak{u}}}^{\raise-5pt\hbox{$ \; \scriptstyle \QQ $}}_{\,\bq\hskip1pt,\hskip1pt\varepsilon}(\hskip0,8pt\lieg)}
\def \Uhatdotquno {\Uhatdot_{\bq\hskip1pt,1}}
\def \Uhatdotqeps {\Uhatdot_{\bq\hskip1pt,\hskip1pt\varepsilon}}

\def \Uhatqeps {\Uhat_{\bq\hskip1pt,\hskip1pt\varepsilon}}

  \def \UtildeqQuno {\Utilde^{\raise-3pt\hbox{$ \scriptstyle \QQ $}}_{\bq\hskip1pt,1}}
\def \Utildeqeps {\Utilde_{\bq\hskip1pt,\hskip1pt\varepsilon}}
  \def \UtildeqQeps {\Utilde^{\raise-5pt\hbox{$ \, \scriptstyle \QQ $}}_{\bq\hskip1pt,\hskip1pt\varepsilon}}
\def\QEA{U_{\mathbf{q}}(\lieg)}
\def\QE{U_{\mathbf{q}}(\lieg)}

\def\pf{\begin{proof}}
\def\epf{\end{proof}}

\theoremstyle{plain}
\begin{document}

{\ }
 \vskip-35pt

%\rightline{}

%%%%%%%%%% Numerar ecuaciones como teoremas %%%
%\renewcommand{\theequation}{\thetheorem}
%\newcommand{\num}{\refstepcounter{theorem}}
%\renewcommand{\thesubsubsection}{\theequation}

%%%%%%%%%%%%%%%%%%%%%%%%%%%%%%%%%

%%%%%%%%% TITLE  AND OTHERS %%%%%%%%%%%%%

\title[Multiparameter quantum groups at roots of unity]
{Multiparameter quantum  \\
 groups at roots of unity}

\author[G.~A.~Garc{\'\i}a \ , \ \  F.~Gavarini]
{Gast{\'o}n Andr{\'e}s Garc{\'\i}a \ ,  \ \ Fabio Gavarini}

\address{\newline\noindent Departamento de Matem\'atica, Facultad de Ciencias Exactas
   \newline
 Universidad Nacional de La Plata   ---   CONICET
   \newline
 C. C. 172   \, --- \,   1900 La Plata, ARGENTINA
   \newline
  {\ } \quad  {\tt ggarcia@mate.unlp.edu.ar}
 \vspace*{0.5cm}
\newline \noindent Dipartimento di Matematica, Universit\`a degli Studi di Roma ``Tor Vergata''
   \newline
 Via della ricerca scientifica 1   \, --- \,   I\,-00133 Roma, ITALY
   \newline
  {\ } \quad  {\tt gavarini@mat.uniroma2.it}}

\thanks{\noindent 2020 \emph{Mathematics Subject Classification:}\,  17B37 (primary); 16T05, 16T20 (secondary).  \newline
   \emph{Keywords:} Quantum Groups, Quantum Enveloping Algebras.  \newline
%%%%%
%    $ {}^\flat $  Partially supported by CONICET, ANPCyT, Secyt (Argentina) and INdAM/GNSAGA (Italy).  \newline
%    $ {}^\sharp $  Partially supported by the MIUR  {\sl Excellence Department Project\/}  awarded to the
% Department of Mathematics of the University of Rome ``Tor Vergata'', CUP E83C18000100006.}
%%%%%
   \indent   Partially supported by CONICET, ANPCyT, Secyt (Argentina) and INdAM/GNSAGA (Italy),  and by the MIUR  {\sl Excellence Department Project\/}  awarded to the Department of Mathematics of the University of Rome ``Tor Vergata'', CUP E83C18000100006.}
%%%

%%%
% \date{\dbend \ \ --- \  {\bf \today}}
%%%

\vskip1pt

\begin{abstract}
 We address the study of multiparameter quantum groups (=MpQG's) at roots of unity,
 namely quantum universal enveloping algebras  $ \QEq $  depending on a matrix of parameters  $ \, \bq = {\big(\, q_{ij} \big)}_{i,j \in I} \, $.  This is performed via the construction of quantum root vectors and suitable ``integral forms'' of  $ \QEq \, $,  a  {\sl restricted\/}  one   --- generated by quantum divided powers and quantum binomial coefficients ---   and an  {\sl unrestricted\/}  one   --- where quantum root vectors are suitably renormalized.  The specializations at roots of unity of either form are the ``MpQG's at roots of unity'' we
%%%%%
% are investigating.
look for.
%%%%%
 In particular, we study special subalgebras and quotients of our MpQG's at roots of unity   --- namely, the multiparameter version of small quantum groups ---   and suitable associated quantum Frobenius morphisms, that link
 %%%%%
% the (specializations of) MpQG's
the MpQG's
%%%%%
 at roots of 1 with MpQG's at 1, the latter being classical Hopf algebras bearing a well precise Poisson-geometrical content.
                                                            \par
   A key point in the discussion, often at the core of our strategy, is that every MpQG is actually a  $ 2 $--cocycle  deformation of the algebra structure of (a lift of) the ``canonical'' one-parameter quantum group by Jimbo-Lusztig, so that we can often rely on already established results available for the latter.  On the other hand, depending on the chosen multiparameter  $ \bq $  our quantum groups yield (through the choice of integral forms and their specializations) different semiclassical structures, namely different Lie coalgebra structures and Poisson structures on the Lie algebra and algebraic group underlying the canonical one-parameter quantum group.
\end{abstract}

{\ } \vskip-41pt

   \centerline{ \smallrm  {\smallsl Journal of Noncommutative Geometry}  {\smallbf 16}  (2022), no.\ 3, 839--926 }
                                                 \par
   \centerline{ \smallrm  {\smallbf DOI:}  10.4171/JNCG/471  \ \ --- \ \  preprint  {\smallsl arXiv:1708.05760 [math.QA] (2020)} }
 \vskip1pt
   \centerline{\smallrm {\smallsl The original publication is available at\/} \  https://ems.press/journals/jncg/articles/7406255 }
 {\ }
\vskip15pt

\maketitle

\vskip-45pt

\tableofcontents

%%%%%%%%%%%%%%%%%%%%%%%%%%%%%%%%%%%%
%%%%%%%%% PAPER %%%%%%%%%%%%%%%%%%%%%%%

%%%%%%%%%%%% INTRODUCTION %%%%%%

\vskip-25pt

\section{Introduction}  \label{sec:intro}

\vskip5pt

   In literature, by ``quantum group'' one usually means some deformation of an algebraic object that in turn
   encodes a geometrical object describing symmetries (such as a Lie or algebraic group or a Lie algebra):
   we are interested now in the case when the geometrical object is a Lie bialgebra  $ \lieg \, $,  and the algebraic
   one its universal enveloping algebra  $ U(\lieg) \, $,  with its full structure of co-Poisson Hopf algebra.
                                                           \par
   In most cases, such a deformation depends on one single parameter, in a ``formal'' version, like with Drinfeld's
   $ U_\hbar(\lieg) \, $,  or in a ``polynomial'' one, for Jimbo-Lusztig's  $ U_q(\lieg) \, $.  But since the dawn of the theory,
   more general deformations depending on many parameters have been considered too: one then talks of
   ``multiparameter quantum groups'' (in short, MpQG's) that again exist both in formal and polynomial versions;
   see for instance  \cite{BGH},  \cite{BW1,BW2},  \cite{CM},  \cite{CV1},  \cite{Hay},  \cite{HLT},
   \cite{HP},  \cite{HPR},  \cite{Ko},  \cite{KT},  \cite{Man}, \cite{OY},  \cite{Re},  \cite{Su},  \cite{Ta}
  --- and the list might be quite longer.
                                                           \par
   In the previously mentioned papers, multiparameter quantum enveloping algebras where often introduced via
   {\it ad hoc\/}  constructions.  A very general recipe, instead, was that devised by Reshetikhin  (cf.\ \cite{Re}),
   that consists in performing a so-called  {\sl deformation by twist\/}  on a ``standard'' one-parameter quantum group.
                                                           \par
   Similarly, a dual method was developed, that starts again from a usual one-parameter
   quantum group and then performs on it a deformation by a  $ 2 $--cocycle.
   In addition, as the usual uniparameter quantum group is a quotient of the Drinfeld's quantum double
   of two Borel quantum (sub)groups, one can start by deforming (e.g., by a  $ 2 $--cocycle)
   the Borel quantum subgroups and then look at their quantum double and its quotient.
   This is the point of view adopted, for instance, in  \cite{AA2},  \cite{AAR1,AAR2},
   \cite{An1,An2,An3,An4},  \cite{AS1,AS2},  \cite{AY},  \cite{Gar},  \cite{He1,He2},  \cite{HK},
   \cite{HY}  and  \cite{Mas},  where in addition the Borel quantum (sub)groups are always thought
   of as bosonizations of Nichols algebras.
                                                           \par
   In our forthcoming papers  \cite{GG1,GG2}  we thoroughly compare deformations by twist or by  $ 2 $--cocycles
   on the standard uniparameter quantum group; up to technicalities, it turns out that the two methods yield the same results.
   Taking this into account, we adopt the point of view of deformations by  $ 2 $--cocycles,  implemented on uniparameter
   quantum groups, that are realized as (quotients of) quantum doubles of Borel quantum (sub)groups.  With this method,
   the multiparameter  $ \bq $  encoding our MpQG is used from scratch as the core datum to construct the Borel quantum
   (sub)groups and eventually remains in the description of our MpQG by generators and relations.
   In this approach, a natural constraint arises for  $ \bq \, $,  namely that it be  {\sl of Cartan type},
   to guarantee that our MpQG have finite Gelfand-Kirillov dimension.
%%%
 \vskip2pt
%%%
   In order to have meaningful specializations of a MpQG, one needs to choose a suitable
   integral form of that MpQG, and then specialize the latter: indeed, by ``specialization of a MpQG''
   one means in short the specialization of such an integral form of it.  The outcome of the specialization
   process then can strongly depend on the choice of the integral form.  For the usual case of uniparameter
   ``canonical'' quantum groups, one usually considers two types of integral forms, namely  {\sl restricted\/}
   ones (after Lusztig's) and  {\sl unrestricted\/}  ones (after De Concini and Procesi),
   whose specialization yield entirely different outcomes   --- dual to each other, in a sense.
   There also exist  {\sl mixed\/}  integral forms (due to Habiro and Thang Le) that are very interesting for
   applications in algebraic topology.
                                                         \par
   For general MpQG's, we introduce integral forms of restricted, unrestricted and mixed type,
   by directly extending the construction of the canonical setup: this is quite a natural step, yet
   (to the best of the authors' knowledge) it had not been considered so far.  Moreover, for restricted forms   ---
   for which the multiparameter has to be ``integral'', i.e.\ made of powers
   (with integral exponents) of just one single, ``basic'' parameter  $ q $  ---
   we consider two possible variants, which gives something new even in the canonical case.
   For these integral forms (of either type) we state and prove all those fundamental structure results
   (triangular decompositions, PBW Theorems, etc.) that one needs to work with them.
%%%
 \vskip2pt
%%%
   When taking specialization at  $ \, q = 1 \, $  (where  ``$ \, q \, $''  is again sort of a ``basic parameter''
   underlying the multiparameter  $ \bq \, $),  co-Poisson and Poisson Hopf structures pop up, yielding
   classical objects that bear some Poisson geometrical structure.  In detail, when specializing the restricted
   form one gets the enveloping algebra of a Lie bialgebra, and when specializing the unrestricted one the
   function algebra of a Poisson group is found: this shows some duality phenomenon, which is not surprising
   because the two integral forms are in a sense related by Hopf duality.  This feature already occurs in the
   uniparameter, canonical case: but in the present, multiparameter setup, the additional relevant fact is that the
   involved (co)Poisson structures directly depend on the multiparameter  $ \, \bq \, $.
%%%
 \vskip2pt
%%%
   Now consider instead a non-trivial root of 1, say  $ \varepsilon \, $.
   Then the specialization of a MpQG at  $ \, q = \varepsilon \, $  is tightly related with its
   specialization at  $ \, q = 1 \, $:  this link is formalized in a so-called  {\it quantum Frobenius morphism}
   --- a Hopf algebra morphism with several remarkable properties between these two specialized MpQG's ---
   moving to opposite directions in the restricted and the unrestricted case.  We complete these morphisms to
short exact sequences, whose middle objects are our MpQG's at  $ \, q = \varepsilon \, $;
the new Hopf algebras we add to complete the sequences are named {\it small MpQG's}.
                                                           \par
   Remarkably enough, we prove that the above mentioned short exact sequences
   have the additional property of being  {\it cleft\/};  thus, our specialized MpQG's at
   $ \, q = \varepsilon \, $  are  {\it cleft extensions\/}  of the corresponding small MpQG's
   and the corresponding specialized MpQG's at  $ \, q = 1 \, $   --- which are  {\sl classical\/}
   geometrical objects, see above.  Furthermore, implementing this construction in both cases
   --- with restricted and with unrestricted forms --- literally yields  {\sl two\/}  small MpQG's: nevertheless,
   we eventually prove that they do coincide indeed.
                                                           \par
   To some extent, these results (at roots of 1) are a direct generalisation of what happens
   in the uniparameter case (i.e., for the canonical multiparameter).
   However, some of our results seem to be entirely new even for the uniparameter context.
 \vskip3pt
   Finally, here is the plan of the paper.
 \vskip2pt
   In section  \ref{gen-Hopf-defs}  we set some basic facts about Hopf algebras, the bosonization process,
   cocycle deformations, etc.\   --- along with all the related notation.
                                                           \par
   Section  \ref{mpqgroups}  introduces our MpQG's: we define them by generators and relations, and we recall that we can get them as  $ 2 $--cocycle  deformations of the canonical one.
                                                           \par
   We collect in section  \ref{q-root_vects & PBW}  some fundamental results on MpQG's,
   such as the construction of
   quantum root vectors and PBW-like theorems (and related facts).  In addition, we compare the
   multiplicative structure in the canonical MpQG with that in a general MpQG,
   the latter being thought of as  $ 2 $--cocycle  deformation of the former.
                                                           \par
   In section  \ref{int-forms_mpqgs}  we introduce integral forms of our MpQG's
   --- of restricted type and of unrestricted type ---   providing all the basic results one
   needs when working with them.  We also shortly discuss mixed integral forms.
                                                           \par
   Section  \ref{Spec-roots-1}  focuses on specializations at 1, and the semiclassical structures
   arising from MpQG's by means of this process.
                                                           \par
   At last, in section  \ref{sec:spec-eps}  we finally harvest our main results.
   Namely, we deal with specializations at non-trivial roots of 1, with quantum
   Frobenius morphisms and with small MpQG's, for both the restricted version and the unrestricted one.

%
%%%%%
% \vskip15pt
%
%    \centerline{\ssmallrm ACKNOWLEDGEMENTS}
% %
%  \vskip5pt
% %
%    {\smallrm The project underlying the present paper started when the first author was visiting the
% Math.\ Department at the Universit\`a di Roma ``Tor Vergata'' under the support of the GNSAGA and
% CONICET, and continued when the second author visited the Math.\ Department of the UNLP under
% the support of the Visiting Professor program of the UNLP.  This work was supported by the
% ``National Group for the Algebraic and Geometric Structures and their Applications'' (GNSAGA - INdAM)
% of the Istituto Nazionale di Alta Matematica, and by the MIUR  {\sl Excellence Department Project\/}
% awarded to the Department of Mathematics of the University of Rome ``Tor Vergata'', CUP E83C18000100006.}
%%%%%
%

\bigskip

%%%%%%%%%%%% PRELIMINARIES %%%%%%

\section{Generalities on Hopf algebras and deformations}  \label{gen-Hopf-defs}

\smallskip

   Throughout the paper, by  $ \k $  we denote a field of characteristic zero and by  $ \k^\times $
   we denote the group of units of  $ \k \, $.
   By convention,  $ \, \NN = \{0, 1,\ldots\} \, $  and  $ \, \NN_+ := \NN \setminus \{0\} \, $.

\smallskip

 \subsection{Conventions for Hopf algebras} \label{conv-Hopf}  \
 \vskip7pt
   Our main references for the theory of Hopf algebras are  \cite{Mo},  \cite{Sw}  and  \cite{Ra},
   for Lie algebras  \cite{Hu}  and for quantum groups
   \cite{Ja}  and  \cite{BG}.  We use standard notation for Hopf algebras; the comultiplication is denoted
   $ \com $  and the antipode  $ \SS \, $.
   For the first, we use the Heyneman-Sweedler notation, namely  $ \, \com(x) = x_{(1)} \otimes x_{(2)} \, $.
                                                              \par
   Let  $ H $  be a Hopf algebra.  The  {\sl left adjoint representation\/}  of  $ H $  is the algebra morphism
%
%%%%%
%   $$  \ad_\ell : H \longrightarrow \End(H) \;\; ,  \quad  \ad_\ell(x)(y) \, := \, x_{(1)} \, y \,
%   \SS\big(x_{(2)}\big) \;\; ,   \eqno  \forall \;\, x, y \in H  $$
%%%%%
  $ \; \ad_\ell : H \longrightarrow \End(H) \; $  given by
  $ \; \ad_\ell(x)(y) := x_{(1)} \, y \, \SS\big(x_{(2)}\big) \; $  for  $ \, x, y \in H \, $;
  ,we drop the subscript  $ \ell $  unless needed; the  {\sl right\/}  adjoint action
%
%%%%%
%  given by
%   $$  \ad_r : H \longrightarrow \End(H) \;\; ,  \quad  \ad_r(x)(y) \, := \,
%   \SS\big(x_{(1)}\big) \, y \, x_{(2)} \;\; ,   \eqno  \forall \;\, x, y \in H  $$
%%%%%
  $ \; \ad_r : H \longrightarrow \End(H) $  \;is given by
  $ \; \ad_r(x)(y) := \SS\big(x_{(1)}\big) \, y \, x_{(2)} \; $  for  $ \, x, y \in H \, $.
  Any subalgebra  $ K $  of  $ H $  is said to be  {\sl normal\/}  if
  $ \, \ad_{\ell}(h)(k) \in K \, $,  $ \, \ad_{r}(h)(k) \in K \, $  for all  $ \, h \in H \, $,  $ \, k \in K \, $.
                                                              \par
   In any coalgebra  $ C \, $,  the set of group-like elements of a coalgebra is denoted by  $ G(C) \, $;
   also, we denote by  $ \, C^+ := \Ker(\epsilon) \, $  the augmentation ideal of  $ C \, $,  where
   $ \, \epsilon : C \longrightarrow \k \, $  is the counit map of  $ C \, $.  If  $ \, g, h \in G(H) \, $,
   the set of  $ (g,h) $--primitive  elements is defined to be
  $$  P_{g,h}(H)  \; := \;  \big\{\, x \in H \,\vert\, \com(x) = x \ot g + h \ot x \,\big\}  $$
In particular, we call  $ \, P(H) := P_{1,1}(H) \, $  the set of primitive elements.

\vskip9pt

   It is convenient to recall the notions of exact sequence and of  {\sl cleft extension}:

\vskip9pt

\begin{definition}  \label{def:exseq}
 {\sl (cf.\  \cite{AD})\/}  A sequence of Hopf algebras maps over a field  $ \Bbbk $
  $$  1 \relbar\joinrel\relbar\joinrel\longrightarrow B \;{\buildrel \iota \over {\relbar\joinrel\relbar\joinrel\longrightarrow}}\,
  A \;{\buildrel \pi \over {\relbar\joinrel\relbar\joinrel\longrightarrow}}\, H \relbar\joinrel\relbar\joinrel\relbar\joinrel\longrightarrow 1  $$
where  $ 1 $  denotes the Hopf algebra  $ \Bbbk \, $,  is called  {\it exact\/}  if  $ \iota $  is injective,
$ \pi $  is surjective,  $ \, \Ker(\pi) = AB^+ \, $  and
$ \, B = {}^{\copr \pi}\!A \, := \, \big\{ a \in A \,\big|\, (\pi \otimes \id)\big(\Delta(a)\big) = 1 \otimes a \big\} \, $.
%
%%%%%
%                                                          \par
%    We say that  $ A $  is a  {\it cleft extension\/}  of  $ B $  by  $ H $  if there is an $ H $--colinear
% section  $ \gamma $  of  $ \pi $  which is invertible with respect to the convolution.   \hfill  $ \diamondsuit $
%%%%%
%
 We say that  $ A $  is a  {\it cleft exten\-sion\/}  of  $ B $  by  $ H $  if there exists an
 $ H $--colinear,  convolution-invertible section  $ \gamma $  of  $ \pi \, $.   \hfill  $ \diamondsuit $
\end{definition}

\vskip5pt

%
%%%%%
%    Finally, we recall the notion of  {\it Hopf pairing\/}  and  {\it skew-Hopf pairing\/}
% between two Hopf algebras (taken from  \cite{AY},  \S 2.1,  but essentially standard):
%%%%%
%
   Finally, we recall the notions of  {\it Hopf pairing\/}  and  {\it skew-Hopf pairing\/}  of Hopf algebras:

\smallskip

\begin{definition}  \label{def_(skew-)Hopf-pairing}
 {\sl (cf.\ \cite[\S 2.1]{AY})}\,
Given two Hopf algebras  $ H $  and  $ K $  with bijective antipode over a ring  $ R $,  an
$ R $--linear  map  $ \, \eta : H \otimes_R K \longrightarrow R \, $  is called
 \vskip3pt
   --- \  {\sl Hopf pairing\/}  (between  $ H $  and  $ K \, $)  if, for all  $ \, h \in H \, $,  $ \, k \in K \, $,  one has
  $$  \displaylines{
   \eta\big( h \, , \, k_1 \, k_2 \big) \, = \, \eta\big( h_{(1)} \, , \, k_1 \big) \,
   \eta\big( h_{(2)} \, , \, k_2 \big)  \;\; ,  \qquad
 \eta\big( h_1 \, h_2 \, , \, k \big) \, = \, \eta\big( h_1 \, , \, k_{(1)} \big) \, \eta\big( h_2 \, , \, k_{(2)} \big)  \cr
   \eta\big( h \, , 1 \big) \, = \, \epsilon(h)  \;\; ,  \;\quad  \eta\big( 1 \, , k \big) \, = \, \epsilon(k)  \quad ,  \;\;\qquad
   \eta\big( S^{\pm 1}(h) \, , k \big) \, = \, \eta\big( h \, , S^{\pm 1}(k) \big)  }  $$
 \vskip3pt
   --- \  {\sl skew-Hopf pairing\/}  (between  $ H $  and  $ K \, $)  if, for all  $ \, h \in H \, $,  $ \, k \in K \, $,  one has
  $$  \displaylines{
   \eta\big( h \, , \, k_1 \, k_2 \big) \, = \, \eta\big( h_{(1)} \, , \, k_1 \big) \,
   \eta\big( h_{(2)} \, , \, k_2 \big)  \;\; ,  \qquad
 \eta\big( h_1 \, h_2 \, , \, k \big) \, = \, \eta\big( h_2 \, , \, k_{(1)} \big) \, \eta\big( h_1 \, , \, k_{(2)} \big)  \cr
 \hfill
   \eta\big( h \, , 1 \big) \, = \, \epsilon(h)  \;\; ,  \;\quad  \eta\big( 1 \, , k \big) \, =
   \, \epsilon(k)  \quad ,  \;\;\qquad  \eta\big( S^{\pm 1}(h) \, , k \big) \, =
   \, \eta\big( h \, , S^{\mp 1}(k) \big)   \hfill  \diamondsuit  }  $$
\end{definition}

\vskip3pt

   Recall that, given two Hopf  $ R $--algebras  $ H_+ $  and  $ H_- $,  and a Hopf pairing among them,
   say  $ \, \pi: H_-^{\,\cop} \otimes_{R} H_+ \longrightarrow \k \, $,  the
   {\it Drinfeld double\/}  $ D(H_-,H_+,\pi) $  is the quotient algebra  $ \, T(H_- \oplus H_+) \big/ \II \, $
   where  $ \II $  is the (two-sided) ideal generated by the relations
  $$  \displaylines{
   \qquad   1_{H_-}  = \,  1  \, = \,  1_{H_+}  \,\; ,  \qquad  a \otimes b  \, = \, a\,b  \qquad \qquad  \forall
   \;\;  a \, , b \in H_+  \text{\;\;\ or\;\;\ }  a \, , b \in H_- \;\; ,  \cr
   x_{(1)} \otimes y_{(1)} \ \pi\big(y_{(2)},x_{(2)}\big)  \, = \,  \pi\big(y_{(1)},x_{(1)}\big) \ y_{(2)} \otimes x_{(2)}
   \qquad  \forall \;\; x \in H_+ \, ,  \; y \in H_-  \;\; ;  }  $$
such a quotient  $R$--algebra  is also endowed with a standard Hopf algebra structure,
which is {\sl consistent},  in that both
$ H_+ $  and  $ H_- $  are Hopf  $ R $--subalgebras  of it.

\vskip9pt

\begin{free text}{\bf Yetter-Drinfeld modules, bosonization and Hopf algebras with a projection.}
 Let  $ H $  be a Hopf algebra with bijective antipode.  A Yetter-Drinfeld module over  $ H $
 is a left  $ H $--module  and a left  $ H $--comodule  $ V $,  with comodule structure denoted by
 $ \, \delta: V \longrightarrow H \otimes V \, $,  $ \, v \mapsto v_{(-1)} \otimes v_{(0)} \, $,  \,such that
  $$  \delta(h\cdot v)  \, = \,  h_{(1)} v_{(-1)} \SS\big(h_{(3)}\big) \otimes h_{(2)} \cdot v_{(0)}  \quad
\text{ for all }\;  v \in V \, , \, h \in H \; .  $$
 Let  $ \ydh $  be the category of Yetter-Drinfeld modules over  $ H $  with  $ H $--linear  and
$H$--colinear maps as morphisms.  The category  $ \ydh $  is monoidal and braided.
A Hopf algebra in the category  $ \ydh $  is called a  {\it braided Hopf algebra\/}  for short.
                                                         \par
   Let  $ R $  be a Hopf algebra in  $ \ydh \, $.  The procedure to obtain a usual Hopf algebra from the
   (braided) Hopf algebras  $ R $  and  $ H $  is called  {\it bosonization\/}  or  {\it Radford-Majid product},
   and it is usually denoted by  $ \, R \,\#\, H \, $.  As a vector space  $ \, R \,\#\, H := R\otimes H \, $,
   and the multiplication and comultiplication are given by the smash-product and smash-coproduct, respectively.
   That is, for all  $ \, r, s \in R \, $  and  $ \, g, h \in H \, $,  we have
\begin{align*}
(r \# g)(s \#h)  &  \; := \;  r\big(g_{(1)}\cdot s\big) \,\#\, g_{(2)} h  \\
\com(r \# g)  &  \; := \;  r^{(1)} \,\#\, {\big(r^{(2)}\big)}_{(-1)} g_{(1)} \otimes
{\big(r^{(2)}\big)}_{(0)} \,\#\, g_{(2)}  \\
\SS(r\# g)  &  \; := \;  \big(1 \,\#\, \SS_{H}\big(r_{(-1)} g\big) \big) \big( \SS_R\big(r_{(0)}\big) \,\#\, 1 \big)
\end{align*}
where  $ \, \com_R(r) = r^{(1)}\ot r^{(2)} \, $  is the comultiplication in
$ \, R \in \ydh \, $  and  $ \SS_R $  the antipode.
%
%%%%%
%  Clearly, the map  $ \, \iota : H \longrightarrow R\# H \, $  given by  $ \, \iota(h) := 1 \# h \, $
% is an injective Hopf algebra morphism, and the map  $ \, \pi : R \# H \longrightarrow H \, $
% given by  $ \, \pi(r\#h) := \epsilon_R(r)\,h \, $  is a surjective Hopf algebra morphism,
% such that  $ \, \pi \circ \iota = \id_H \, $.  Moreover, it holds that  $ \, R =
% (R\#H)^{\co\pi} \, = \big\{\, x \in R\,\#\,H \,\big|\, (\id\ot \pi)\com(x) = x\ot 1 \big\} \,$.
%%%%%
%
 The map  $ \; \iota : H \longrightarrow R\# H \, \big( h \mapsto 1 \# h \big) \, $,  \;resp.\
 $ \; \pi : R \# H \longrightarrow H \; \big( r\,\#\,h \mapsto \epsilon_R(r)\,h \big) \, $,  \,is a Hopf algebra
monomorphism, resp.\ epimorhism, and  $ \, \pi \circ \iota = \id_H \, $.  Moreover, we have
$ \, R = (R\#H)^{\co\pi} \, = \big\{\, x \in R\,\#\,H \,\big|\, (\id \otimes \pi)\com(x) = x \otimes 1 \,\big\} \, $.
%%%%%
%
                                              \par
   Conversely, let  $ A $  be a Hopf algebra with bijective antipode and  $ \, \pi : A \longrightarrow H \, $
a Hopf algebra epimorphism. If there is a Hopf algebra map  $ \, \iota : H \longrightarrow A \, $,  such that
$ \, \pi \circ \iota = \id_H \, $,  then  $ \, R := A^{\co\pi} \, $  is a braided Hopf algebra in  $ \ydh \, $, called the
{\it diagram\/}  of  $ A \, $,  and we have  $ \, A \cong R \# H \, $  as Hopf algebras.
See  \cite[11.6]{Ra}  for further details.
\end{free text}

\medskip

 \subsection{Cocycle deformations}  \label{cocyc-defs}  \
 \vskip7pt
   We recall now the standard procedure that, starting from a given Hopf algebra and a suitable
   $ 2 $--cocycle  on it, gives us
   a new Hopf algebra structure on it, with the same coproduct and a new, ``deformed'' product.
   We shall then see the special form that this construction may take when the Hopf algebra is bigraded by some
   Abelian group and the  $ 2 $--cocycle  is induced by one of that group.

\medskip

\begin{free text}{\bf First construction.}  \label{cotwist-defs_1}
 Let $ \, \big( H, m, 1, \Delta, \epsilon \,\big) \, $  be a bialgebra over a ring  $ R \, $.
%
%%%%%
% A convolution invertible linear map  $ \sigma $ in  $ \, \Hom_{\Bbbk}(H \otimes H, R\,) \, $
% is called a  {\it normalized Hopf 2-cocycle\/}  if
%   $$  \sigma(b_{(1)},c_{(1)}) \, \sigma(a,b_{(2)}c_{(2)}) \, = \,
% \sigma(a_{(1)},b_{(1)}) \, \sigma(a_{(2)}b_{(2)},c)  $$
% %
% and  $ \, \sigma (a,1) = \eps(a) = \sigma(1,a) \, $  for all  $ \, a, b, c \in H \, $,
% see  \cite[Sec.\ 7.1]{Mo}.
%%%%%
 A  {\it normalized Hopf 2-cocycle\/}  (see  \cite[Sec.\ 7.1]{Mo})  is a map  $ \sigma $  in  $ \, \Hom_{\Bbbk}(H \otimes H, R\,) \, $  which is convolution invertible and such that, for all  $ \, a, b, c \in H \, $,  we have
%%%%%
%
  $$  \sigma(b_{(1)},c_{(1)}) \, \sigma(a,b_{(2)}c_{(2)}) \, = \, \sigma(a_{(1)},b_{(1)}) \, \sigma(a_{(2)}b_{(2)},c)  $$
and  $ \, \sigma (a,1) = \eps(a) = \sigma(1,a) \, $.
%%%%%
%
 We simply call it a  $ 2 $--cocycle  if no confusion arises.
                                                                        \par
   Using a  $ 2 $--cocycle  $ \sigma $  it is possible to define a new algebra structure on  $ H $
   by deforming the multiplication: indeed, define
   $ \, m_{\sigma} = \sigma * m * \sigma^{-1} : H \otimes H \longrightarrow H \, $  by
  $$  m_{\sigma}(a,b)  \, = \,  a \cdot_{\sigma} b  \, = \,
  \sigma(a_{(1)},b_{(1)}) \, a_{(2)} \, b_{(2)} \, \sigma^{-1}(a_{(3)},b_{(3)})  \eqno \forall \;\, a, b \in H  \quad  $$
   \indent   If in addition  $ H $  is a Hopf algebra with antipode  $ \SS \, $,  then define also
$ \, \SS_\sigma : H \longrightarrow H \, $
as
$ \, \SS_{\sigma} = \sigma * \SS * \sigma^{-1} : H \longrightarrow H \, $  where
  $$  \SS_{\sigma}(a)  \, = \,  \sigma(a_{(1)},\SS(a_{(2)})) \, \SS(a_{(3)}) \, \sigma^{-1}(\SS(a_{(4)}),a_{(5)})
  \eqno \forall \;\, a \in H  \quad  $$
It is then known   --- see \cite{DT} ---   that  $ \, \big( H, m_\sigma, 1, \Delta, \epsilon \,\big) \, $  is in turn a bialgebra, and
 also that
 $ \, \big( H, m_\sigma, 1, \Delta, \epsilon, \SS_\sigma \big) \, $  is a Hopf algebra:
we shall call such a new structure on  $ H $  a  {\sl cocycle deformation\/}  of the old one, and
we shall graphically denote it by  $ H_\sigma \, $.
                                                              \par
   When dealing with a Hopf algebra  $ H $  and its deformed counterpart  $ H_\sigma $  as above, we denote by
   $ \ad_\ell $  and  $ \ad_r $  the adjoint actions in  $ H $  and by  $ \ad^\sigma_\ell $  and  $ \ad^\sigma_r $
   those in  $ H_\sigma \, $.
\end{free text}

\vskip5pt

\begin{free text}{\bf Second construction.}   \label{cotwist-defs_2}
 There is a second type of cocycle twisting   --- of algebras, bialgebras and Hopf algebras ---
 that we shall need  (cf.\ \cite{AST}  and references therein).  Let  $ \varGamma \, $  be an Abelian group,
 for which we adopt multiplicative notation, and  $ H $  an algebra over a ring  $ R $  that is  $ \varGamma $--bigraded
 (i.e., graded by  $ \, \varGamma \times \varGamma \, $):
 so  $ \, H = \bigoplus_{(\gamma,\eta)\in \varGamma\times \varGamma} H_{\gamma,\eta} \, $  with
 $ \, R \subseteq H_{1,1} \, $
 and  $ \, H_{\gamma,\eta}H_{\gamma',\eta'} \subseteq H_{\gamma\gamma',\eta\eta'} \, $.
 Given any group  $ 2 $--cocycle  $ \, c : \varGamma \times \varGamma \longrightarrow R^\times \, $  where
 $ R^\times $  is the group of units of  $ R \, $,  define a new product on  $ H \, $,  denoted by
 $ \mathop{\star}\limits_c \, $,  by
%
%%%%%
% \begin{equation}  \label{c-deformed-product}
%     h \mathop{\star}\limits_c k  \; := \;  c\big(\eta',\kappa'\big) \, {c(\eta,\kappa)}^{-1} \, h \cdot k
% %
% \end{equation}
% %
%%%%%
%
%%%
  $ \; h \mathop{\star}\limits_c k \, := \, c\big(\eta',\kappa'\big) \, {c(\eta,\kappa)}^{-1} \, h \cdot k \; $
%%%
 for all homogeneous  $ \, h \, , k \in H \, $  with degrees
 $ \, \big(\eta\,,\eta'\big), \big(\kappa\,,\kappa'\big) \in \varGamma \times \varGamma \, $.
Then  $ \, \big( H \, ; \, \mathop{\star}\limits_c \,\big) \, $  is (again) an associative algebra,
with the same unit as  $ H $  before.
                                                     \par
   As  $ \varGamma $  is free Abelian, each element of  $ H^2(\varGamma, R^\times) $  has a representative, say
   $ c \, $,  which is  {\sl bimultiplicative\/}  and such that  $ \, c\big(\eta,\eta^{-1}\big) = 1 \, $  for all
   $ \, \eta \in \varGamma \, $  (see  \cite[Proposition 1 and Lemma 4]{AST});  so we may assume that
   $ \, c: \varGamma \times \varGamma \longrightarrow R^\times \, $  is such a cocycle.  Thus
  $$  c(\gamma ,\eta^{-1})  \, = \,  c(\gamma^{-1},\eta)  \, = \,  c(\gamma ,\eta)^{-1} \;\; ,  \qquad
c(\gamma,1) \, = \, c(1,\gamma) \, = \, 1   \;\;\;   \eqno \forall \;\, \gamma, \eta \in \varGamma  \quad  $$
   \indent   Now assume  $ H $  is a bialgebra, with
   $ \, \Delta\big( H_{\alpha,\beta} \big) \subseteq \sum_{\gamma \in \varGamma} H_{\alpha,\gamma}
   \otimes_R H_{\gamma,\beta} \, $  for all  $ \, (\alpha,\beta) \in \varGamma \times \varGamma \, $  and
   $ \; \epsilon\big( H_{\alpha,\beta} \big) = 0 \; $  if  $ \, \alpha \not= \beta \, $.
   Then  $ H $  with the new product  $ \mathop{\star}\limits_c $  and the old coproduct  $ \Delta $  is a bialgebra too.
   If in addition  $ H $  is a Hopf algebra, whose antipode obeys
   $ \, \SS\big( H_{\alpha,\beta} \big) \subseteq \big( H_{\beta^{-1},\alpha^{-1}} \big) \, $   ---
   for  $ \, (\alpha,\beta) \in \varGamma \times \varGamma \, $  ---   then the new bialgebra structure on  $ H $
   (with the new product and the old coproduct) makes it again into a Hopf algebra with antipode
   $ \, \SS^{(c)} := \SS \, $  (the old one).  In all cases, we will graphically denote by  $ H^{(c)} $
   the new structure on  $ H $ obtained by this (second) cocycle twisting.

\smallskip

   In the sequel we shall compare computations in  $ H $  with computations in  $ H^{(c)} $,
   in particular regarding the adjoint action(s); in such cases, we shall denote by  $ \ad_\ell $  and  $ \ad_r $
   the adjoint actions in  $ H $  and by  $ \ad^{(c)}_\ell $  and  $ \ad^{(c)}_r $  those in  $ H^{(c)} \, $.
                                                         \par
   We shall make use of the following result (whose proof is straightforward):
\end{free text}

\vskip5pt

\begin{lema}  \label{lemma-adj}   (cf.\ \cite[Lemma 3.2]{CM})
 Let a  $ 2 $--cocycle  $ \, c : \varGamma \times \varGamma \longrightarrow R^\times \, $  as above be given,
 and assume in addition
 (with no loss of generality) that  $ c $  is bimultiplicative.  Let  $ \, e, b \in H \, $  be homogeneous with degrees
 $ (\gamma,1) $
 and  $ (\eta,1) $  respectively, and assume  $ e $  is  $ (1,h) $--primitive  with  $ \, h \in H \, $
 homogeneous of degree  $ (\gamma,\gamma) \, $.  Then
 \begin{align*}
   \ad^{(c)}_\ell(e)(b) \;  &  = \;  {c(\gamma,\eta)}^{-1} \, \ad_\ell(e)(b)  \\
   \ad^{(c)}_r(e)(b) \;  &  = \;  c(\gamma,\gamma) \big( -h^{-1} e\,b \, + \, c(\gamma,\eta) \,
   c(\eta,\gamma)^{-1} h^{-1} b\,e \big)
 \end{align*}
 In particular, if  $ \, c(\gamma,\eta) \, c(\eta,\gamma)^{-1} \, = \, 1 \, $,  then
 $ \, \ad^{(c)}_r(e)(b) \, = \, c(\gamma,\gamma) \, \ad_r(e)(b) \, $.
\end{lema}

\smallskip

\begin{free text}{\bf A relation between the two constructions}  \label{subsubsec:relation-cocycles}
 Let  $ H $  be a Hopf algebra with bijective antipode,  $ R $  a braided Hopf algebra in  $ \ydh $  and
 $ \, A = R \# H \, $  its bosonization (see  \cite{Gar}  for details).  For any  $ \, a \in R \, $,  set
 $ \, \delta(a) = a_{(-1)} \ot a_{(0)} \, $  for the left coaction of  $ H \, $.
                                                    \par
   Any Hopf  $ 2 $--cocycle  on  $ H $  gives rise to a Hopf  $ 2 $--cocycle  on  $ A $  which may deform the
$H$-module structure of  $ R $  and consequently its braided structure as well.
Specifically, let  $ \, \sigma \in \Zc^2(H,\k) \, $:  then the map $ \, \tilde{\sigma} : A \otimes A \longrightarrow \k \, $
given by
  $$  \tilde{\sigma}(r \# h , s \# k)  \, = \,  \sigma(h,k) \, \epsilon_R(r) \, \epsilon_R(s)   \eqno \forall\;\, r, s \in R \, ,
  \; h, k \in H  $$
is a normalized Hopf  $ 2 $--cocycle  such that  $ \, \tilde{\sigma}\big|_{H \otimes H} = \sigma \, $.
By \cite[Prop. 5.2]{Mas}  we have  $ \, A_{\tilde{\sigma}} = R_{\sigma} \# H_{\sigma} \, $,  where
$ \, R_{\sigma} = R \, $  as coalgebras, and the product is given by
  $$  a \cdot_{\sigma} b  \, := \,  \sigma(a_{(-1)},b_{(-1)}) \, a_{(0)} \, b_{(0)}   \eqno \text{for all}
  \quad a, b \in R \;\; .  \quad  $$
 Therefore,  $ H_\sigma $  is a Hopf subalgebra of  $ A_{\tilde{\sigma}} $  and the map
$ \, \Zc^2(H,\k) \longrightarrow \Zc^2(A,\k) \, $  given by  $ \, \sigma \mapsto \tilde{\sigma} \, $
is a section of the map  $ \, \Zc^2(A,\k) \longrightarrow \Zc^2(H,\k) \, $
induced by the restriction; in particular, it is injective.
                                                    \par
   Now assume  $ \, H = \k\varGamma \, $,  with  $ \varGamma $  a group.
   Then a normalized Hopf  $ 2 $--cocycle  on  $ H $  is equivalent to a  $ 2 $--cocycle
   $ \, \varphi \in \Zc^2(\varGamma,\k) \, $,  i.e.\ a map
   $ \, \varphi : \varGamma \times \varGamma \longrightarrow \k^\times \, $  such that
  $$  \varphi(g,h) \, \varphi(g\,h,t)  \, = \,  \varphi(h,t) \, \varphi(g,h\,t) \;\; ,  \;\quad  \varphi(g,e) \, =
  \, 1 \, = \, \varphi (e,g)  \qquad \forall \;\; g, h, t \in \varGamma \;\; .  $$
                                                    \par
   Assume  $ \, A = R \,\#\, \k\varGamma \, $  is given by a bosonization over a free Abelian group  $ \varGamma \, $.
   Then the coaction of  $ \k\varGamma $  on the elements of  $ R $  induces a
   $ \, (\varGamma \times \varGamma) $--grading  on  $ A $  with  $ \, \deg(g) := (g,g) \, $  for all
   $ \, g \in \varGamma \, $  and  $ \, \deg(a) := (g,1) \, $  if  $ \, \delta(a) = g \otimes a \, $  with  $ \, a \in R \, $
   a homogeneous element; in particular,  $ a $  is  $ (1,g) $--primitive,  since
   $ \, \com(a) = a \ot 1 + a_{(-1)} \ot a_{(0)} \, $.  If  $ \, \varphi \in \Zc^2(\varGamma,\k) \, $,  then
   $ \, A^{(\varphi^{-1})} = A_{\tilde{\varphi}} \, $,  where  $ \tilde{\varphi} $  is the Hopf  $ 2 $--cocycle  on  $ A $
   induced by  $ \varphi \, $.  Indeed, this holds true because, for  $ a $,  $ b $  homogeneous in  $ R $
   of degree  $ (g,1) $  and  $ (h,1) $  respectively, we have that
  $$  a \mathop{\star}_{\varphi^{-1}} b  \, = \, \varphi(1,1)^{-1} \, \varphi(g,h) \, a \, b \,  = \,
  \varphi(a_{(-1)},b_{(-1)}) \, a_{(0)} \, b_{(0)}  \, = \,  a \cdot_{\sigma} b.  $$
\end{free text}

\medskip

 \subsection{Basic constructions from multiparameters}  \label{basics-mprmtrs}  \
 \vskip7pt
   The definition of multiparameter quantum groups requires a whole package of related material,
   involving root data, weight lattices, etc.  This entails several different constructions,
   depending on ``multiparameters'', that we now go and present.

\medskip

\begin{free text}{\bf Root data.}  \label{root-data}
 Hereafter we fix  $ \, \theta \in \NN_+ \, $  and  $ \, I := \{1,\dots,\theta\} \, $ as before.
 Let  $ \, A := {\big(\, a_{ij} \big)}_{i, j \in I} \, $
 be a Cartan matrix of finite type; then there exists a unique diagonal matrix
 $ \, D := {\big(\hskip0,7pt d_i \, \delta_{ij} \big)}_{i, j \in I} \, $
 with positive integral, pairwise coprime entries such that  $ \, D A \, $  is symmetric.
 Let  $ \, \lieg \, $  be the finite dimensional simple Lie algebra over
 $ \CC $  associated with  $ A \, $,  let  $ \Phi $  be the (finite) root system of  $ \lieg \, $,
 with  $ \, \Pi = \big\{\, \alpha_i \,\vert\, i\in I \,\big\} \, $
 as a set of simple roots,  $ \, Q = \bigoplus_{i \in I} \ZZ \alpha_i \, $  the associated root lattice,
 $ \Phi^+ $  the set of positive roots with respect to
 $ \Pi \, $,  $ \, Q^+ = \bigoplus_{i \in I} \NN \alpha_i \, $  the positive root (semi)lattice.
 We denote by  $ P $  the associated weight lattice, with basis
 $ \, {\big\{\, \omega_i \,\big\}}_{i \in I} \, $  dual to  $ \, {\big\{\, \alpha_j \,\big\}}_{j \in I} \, $,
 namely  $ \, \omega_i(\alpha_j) = \delta_{ij} \; $  for all  $ \, i, j \in I \, $.
 Using an invariant non-degenerate bilinear form on the dual  $ \lieh^* $  of a Cartan subalgebra
 $ \lieh $  of  $ \lieg \, $,  we identify  $ Q $
 with a sublattice of  $ P \, $;  in particular, we have  $ \; \alpha_i = \sum_{j \in I} a_{ji} \, \omega_j \, $
 for all  $ \, i \in I \, $.
                                                               \par
   In this setup, we have two natural  $ \ZZ $--bilinear  pairings  $ \, P \times Q \relbar\joinrel\longrightarrow \ZZ \, $,
   that we denote by
   $ \, \langle \,\ , \ \rangle \, $  and  $ \, ( \,\ ,\,\ ) \, $,  one given by the evaluation (of weights onto roots),
   and the other one by
   $ \, (\omega_i , \alpha_ j) := d_i \delta_{ij} \, $  for all  $ \, i, j \in I \, $.  In particular, the restriction of
   $ ( \,\ ,\,\ ) $  to  $ \, Q \times Q \, $
   is a  {\sl symmetric\/}  bilinear pairing on  $ Q \, $;  moreover, both the given pairings uniquely extend to
   $ \QQ $--bilinear  pairings, still denoted by
   $ \, \langle \,\ , \ \rangle \, $  and  $ \, ( \,\ ,\,\ ) \, $,  onto
   $ \, \big( \QQ \otimes_{\raise-1pt\hbox{$ \scriptstyle \, \ZZ $}} P \, \big) \times \big( \QQ \otimes_{\raise-1pt\hbox{$ \scriptstyle \, \ZZ $}} Q \,\big) =
   \big(\, \QQ \otimes_{\raise-1pt\hbox{$ \scriptstyle \, \ZZ\! $}} P \,\big) \times \big(\, \QQ \otimes_{\raise-1pt\hbox{$ \scriptstyle \, \ZZ\! $}} P \,\big) \; $.
%%%%%
% Using these pairings, we define also the lattice  $ Q^\circ $  in  $ \, \QQ{}Q = \QQ{}P \, $  as
%%%%
 Then we define
%%%
  $$  Q^\circ  \; := \;  \big\{\, \lambda \in \QQ{}Q \;\big|\, (\lambda \, , \gamma) \in \ZZ \;\; \forall \; \gamma \in Q \,\big\}  \; =
  \;  \big\{\, \rho \in \QQ{}Q \;\big|\, (\gamma \, , \rho) \in \ZZ \;\; \forall \; \gamma \in Q \,\big\}  $$
By construction  $ \, P \subseteq Q^\circ \, $,  and equality holds true if and only if  $ \lieg $ is simply-laced.
                                                               \par
   Note that, in terms of the above symmetric pairing on  $ Q \, $,  one has
   $ \, d_i = (\alpha_i,\alpha_i) \big/ 2 \, $  for all  $ \, i \in I \, $.
   More in general, we shall use the notation  $ \, d_\alpha := (\alpha,\alpha) \big/ 2 \, $  for every
   $ \, \alpha \in \Phi^+ \, $;  in particular
   $ \, d_{\alpha_i} = d_i \, $  ($ \, i \in I \, $).  We denote by  $ W $  the Weyl group associa\-ted with the root data
   $ (\Phi,\Pi) \, $;
   it is generated by the simple reflections  $ s_i $  given by
   $ \; s_i(\beta) := \beta - \frac{2(\beta,\alpha_i)}{(\alpha_i,\alpha_i)} \, \alpha_i \; $  ($ \, i \in I \, $);
   in particular  $ \, s_i(\alpha_j) = \alpha_j - a_{ij}\,\alpha_i \, $  for  $ \, i, j \in I \, $.
\end{free text}

\smallskip

\begin{free text}{\bf Multiparameters.}  \label{multiparameters}
 Let  $ \k $  be our fixed ground field, and let  $ \, I := \{1,\dots,\theta\} \, $  be as in  \S \ref{root-data}  above.
 We fix a matrix  $ \, \bq := {\big(\, q_{ij} \big)}_{i,j \in I} \, $,  whose entries belong to  $ \k^\times $,
 that will play the role of ``parameters'' of our quantum groups.
 These can be used to construct diagonal braidings and braided spaces, see for example
 \cite{ARS}, \cite{An4}, \cite{Gar}, \cite{He2}.
                                                                  \par
%
%%%%%
% We assume that the matrix  $ \, \bq := {\big(\, q_{ij} \big)}_{i,j \in I} \, $  is
% {\sl standard\/}  and  {\sl finite},  i.e.\ its associated braiding is standard in
% the sense of  \cite{AA1}  and the associated generalized root system is finite.  We
% shall say that the matrix  $ \, \bq := {\big(\, q_{ij} \big)}_{i,j \in I} \, $  is
% {\sl of Cartan type\/}  (or  {\sl ``of Cartan type  $ A \, $''})  if there is a
% generalized Cartan matrix  $ \, A = {\big( a_{ij} \big)}_{i,j \in I} \, $  such that
% %
% \begin{equation}\label{qij-ident}
%   \qquad \qquad \qquad \qquad \quad   q_{ij} \, q_{ji}  \; = \;  q_{ii}^{\,a_{ij}}
%   \qquad \qquad   \forall \;\; i, j \in I  \quad
%       \end{equation}
% %
% indeed, as  $ \bq $  is of finite type the Cartan matrix  $ A $  is necessarily of finite type.
%%%%%
%
   We assume that  $ \, \bq := {\big(\, q_{ij} \big)}_{i,j \in I} \, $  is of  {\sl finite Cartan type  $ A \, $}
%
%%%%%
%  (hence its root system is finite),
%%%%%
%
 i.e.\ there is a Cartan matrix  $ \, A = {\big( a_{ij} \big)}_{i,j \in I} \, $  of finite type such that
\begin{equation}   \label{qij-ident}
  \qquad \qquad \quad   q_{ij} \, q_{ji}  \; = \;  q_{ii}^{\,a_{ij}}
  \qquad \qquad \quad   \forall \;\; i, j \in I  \quad
      \end{equation}
                                                         \par
   To avoid some irrelevant technicalities,  {\sl we assume that  $ A $  is indecomposable}.
                                                         \par
   For later use we fix in  $ \k $  some ``square roots'' of all the  $ q_{ii} $'s,  as follows.
   From the relations in  (\ref{qij-ident})  one finds (since the Cartan matrix  $ A $  is indecomposable) that there exists
   $ \, j_{\raise-2pt\hbox{$ \scriptstyle 0 $}} \in I \, $  such that
   $ \, q_{ii} = q_{j_{\raise-2pt\hbox{$ \scriptscriptstyle 0 $}} j_{\raise-2pt\hbox{$ \scriptscriptstyle 0 $}}}^{\,e _i} \, $
   for some  $ \, e_i \in \NN \, $,  for all  $ \, i \in I \, $.  Now  {\it we assume hereafter that  $ \k $
   contains a square root of
   $ q_{j_{\raise-2pt\hbox{$ \scriptscriptstyle 0 $}} j_{\raise-2pt\hbox{$ \scriptscriptstyle 0 $}}} \, $,
   which we fix throughout and denote by
   $ \, q_{j_{\raise-2pt\hbox{$ \scriptscriptstyle 0 $}}} := \sqrt{\,q_{j_{\raise-2pt\hbox{$ \scriptscriptstyle 0 $}} j_{\raise-2pt\hbox{$ \scriptscriptstyle 0 $}}}^{\phantom{o}}} \, $.}
   Then we set  $ \, q_i := q_{j_{\raise-2pt\hbox{$ \scriptscriptstyle 0 $}}}^{\,e _i} \, $
   (a square root of  $ q_{ii} $)  for all  $ \, i \in I \, $.
 \vskip5pt
   As recorded in  \S \ref{root-data}  above, the Cartan matrix  $ A $  is diagonalizable,
   hence we fix positive, relatively prime integers  $ d_1 $,  $ \dots $,  $ d_\theta $
   such that the diagonal matrix  $ \, D=\text{\sl diag}\,(d_1,\dots,d_\theta) \, $  symmetrizes
   $ A \, $,  i.e.\  $ \, D\,A \, $  is symmetric; in fact, each of these  $ d_i $'s  coincides
   with the corresponding exponent  $ e_i $  mentioned above.
                                               \par
   We introduce now some special cases of Cartan type multiparameter matrices.
 \vskip5pt
   {\sl  $ \underline{\hbox{Integral type}} $:}\,  We say that  $ \, \bq := {\big(\hskip0,7pt q_{ij} \big)}_{i,j \in I} \, $
   is {\it of integral type\/}  if it is of Cartan type and there exist  $ \, b_{ij} \in \ZZ \, $  such that
   $ \, q_{ij} = q^{\,b_{ij}} \, $  for  $ \, i, j \in I \, $;  then we may assume  $ \; b_{i{}i} = 2 \, d_i \; $  and
   $ \; b_{ij} + \, b_{ji} = 2 \, d_i \, a_{ij} \, $
($ \, i, j \in I \, $),  with  $ \, q = q_{j_0} \, $  and the  $ d_i $'s  as above.
  To be precise,  {\it we say also that  $ \, \bq $  is ``of integral type  $ B $''},  with
$ \, B := {\big(\hskip0,7pt b_{ij} \big)}_{i,j \in I} \in M_\theta(\ZZ) \; $.
 \vskip3pt
   {\sl  $ \underline{\hbox{Strongly integral type}} $:}\,
   We say that  $ \, \bq := {\big(\hskip0,7pt q_{ij} \big)}_{i,j \in I} \, $  is {\it of strongly integral type\/}
   if it is of integral type and in addition one has  $ \; b_{ij} \in d_i\,\ZZ \, \bigcap \, d_j\,\ZZ \; $  for all
   $ \, i, j \in I \, $.  In other words,  $ \, \bq := {\big(\hskip0,7pt q_{ij} \big)}_{i,j \in I} \, $  of Cartan type is
   {\sl strongly integral\/}  if and only if there exist integers  $ \, t^+_{ij} , t^-_{ij} \in \ZZ \, $  such that
   $ \, q_{ij} = q^{\,d_i t^+_{ij}} = q^{\,d_j t^-_{ij}} \, $  for all  $ \, i, j \in I \, $;  then we may assume
   $ \; t^\pm_{i{}i} = 2 = a_{i{}i} \; $  and  $ \; t^+_{ij} + \, t^-_{ji} = 2 \, a_{ij} \, $,  \, for  $ \, i, j \in I \, $.
 \vskip3pt
   {\sl  $ \underline{\hbox{Canonical multiparameter}} $:}\,  As a last (very) special case,
   given  $ \, q \in \k^\times \, $  consider
\begin{equation}  \label{qij-canon}
  \qquad \qquad \qquad \qquad \quad   \check{q}_{ij}  \; := \;  q^{\, d_i a_{ij}}   \qquad \qquad
  \forall \;\; i, j \in I  \quad
\end{equation}
with  $ \, d_i \; (\,i \in I\,) \, $  given as above.  These  $ \, q_{ij} = \check{q}_{ij} \, $'s
obey condition  (\ref{qij-ident}),  hence the matrix  $ \, \bq = \check{\bq} \, $  is of Cartan type  $ A \, $:
\,we shall refer to it as to  {\it the ``canonical'' case}.
 \vskip5pt
   Overall we have the following relations among different types of multiparameters:
 \vskip3pt
   \centerline{ \sl ``canonical''  $ \Longrightarrow $
 ``strongly integral''  $ \Longrightarrow $  ``integral''  $ \Longrightarrow $  ``Cartan'' }
 \vskip7pt
   By the way, when the multiparameter matrix  $ \, \bq := {\big(\, q_{ij} \big)}_{i,j \in I} \, $  is  {\sl symmetric},
   i.e.\  $ \, q_{ij} = q_{ji} \, $  (for all  $ \, i, j \in I \, $),  then the conditions  $ \, q_{ij} \, q_{ji} = q_{ii}^{a_{ij}} \, $
   read  $ \, q_{ij}^{\,2} = q^{\, 2 \, d_i a_{ij}} \, $,  hence  $ \, q_{ij} = \pm q^{\, d_i a_{ij}} \, $  (for all  $ \, i, j \in I \, $).
   This means that every symmetric multiparameter is ``almost the canonical one'', as
indeed it is the canonical one ``up to sign(s)''.
 \vskip7pt
   Finally, we assume that for each  $ \, i, j \in I \, $  there exists in the ground field  $ \k $  a square root of  $ q_{ij} \, $,
   which we fix once and for all and denote hereafter by  $ \, q_{ij}^{\,1/2} \, $;
   in addition, we require that these square roots satisfy
   the ``compatibility constraints''  $ \, q_{ii}^{\,1/2} \! = q_i \; \big(\! := q^{d_i} \big) \, $  and
   $ \; q_{ij}^{\,1/2} \, q_{ji}^{\,1/2} \, = \, {\big( q_{ii}^{\,1/2} \big)}^{\,a_{ij}} \; $  for all  $ \, i, j \in I \, $
   --- in short, we assume that ``the signs of all square roots  $ q_{ij}^{\,1/2} $  are chosen in an overall consistent way''.
                                                     \par
   Even more,  {\sl when  $ \, \bq := {\big(\hskip0,7pt q_{ij} \big)}_{i,j \in I} \, $  in particular is of integral type, say
   $ \, q_{ij} = q^{\,b_{ij}} \, $,  we fix a square root  $ \, q^{1/2} $  of  $ \, q $  in  $ \k $  and we set
   $ \, q_{i{}j}^{1/2} \! := {\big( q^{1/2} \big)}^{b_{i{}j}} \in \, \k \; $  for all  $ \, i \, , j \in I $}.
\end{free text}

\medskip

\begin{free text}{\bf Multiparameter Lie bialgebras.}  \label{MP_lie-bialgebras}
 Consider the complex Lie algebra  $ \lieg $  associated with the Cartan matrix  $ A $  as in  \S \ref{root-data},
 and let  $ \lieb_+ $  and  $ \lieb_- $  be opposite Borel subalgebras in it, containing a Cartan subalgebra  $ \lieh $
 whose associated set of roots is identified with  $ \Phi \, $.
 There is a canonical, non-degenerate pairing between  $ \lieb_+ $  and  $ \lieb_- \, $,  and using it one can construct a
 {\sl Manin double\/}  $ \, \liegd = \mathfrak{b}_+ \oplus \mathfrak{b}_- \, $,
 which is automatically endowed with a structure of Lie bialgebra.  Roughly,  $ \liegd $  is like  $ \lieg $ {\sl but\/}  with
 {\sl two copies of\/}  $ \lieh $  inside it; see  \cite{Hal}  for details (in particular Proposition 4.5 therein, with  $ \liegd $
 denoted  $ \mathfrak{c} \, $).
                                                                            \par
   Now fix in  $ \mathfrak{b}_+ $  and  $ \mathfrak{b}_- $  generators  $ \, \erm_i , \hrm_i^+ \, (\, i \in I \,) \, $
   and  $ \, \frm_i , \hrm_i^- \, (\, i \in I \,) \, $  respectively as in the usual Serre's presentation of  $ \lieg \, $.
   Then, thinking of these elements as living in  $ \liegd \, $,  the latter is just the Lie algebra over  $ \Bbbk $
   with generators  $ \, \erm_i , \hrm_i^+ , \hrm_i^- , \frm_i \, (\, i \in I \,) \, $  and relations
  $$  \begin{matrix}
     \big[ \hrm_i^+ , \erm_j \big] = +d_i\,a_{ij}\, \erm_j \; ,  \;\;
       \big[ \hrm_i^+ , \frm_j \big] = -d_i\,a_{ij}\, \frm_j \; ,  \;\;  \big[ \hrm_i^- , \erm_j \big] = +d_j\,a_{ji}\, \erm_j \; ,
       \;\;  \big[ \hrm_i^- , \frm_j \big] = -d_j\,a_{ji}\, \frm_j  \\
     \big[ \hrm_i^+ , \hrm^+_j \big] = 0 \; ,  \;\;\;  \big[ \hrm_i^- , \hrm^-_j \big] = 0 \; ,
       \;\;\;  \big[ \hrm_i^+ , \hrm^-_j \big] = 0 \; ,  \;\;\;
       \big[ \erm_i , \frm_j \big] = \delta_{ij} \, 2^{-1} \big(\, \hrm^+_i \! + \hrm^-_i \big)^{\phantom{\big|}}  \\
       {\ad(\erm_i)}^{1-a_{ij}}(\erm_j) = 0  \quad ,   \qquad   {\ad(\frm_i)}^{1-a_{ij}}(\frm_j) = 0
       \quad \qquad \big(\, i \not= j \,\big)
      \end{matrix}  $$
 Moreover  (cf.\ \cite{Hal}),  $ \liegd $  bears the unique Lie bialgebra structure given by the formulas
  $$  \begin{matrix}
   \delta(\erm_i) \, = \, \big( d_i \hrm^+_i \big) \otimes \erm_i - \erm_i \otimes \big( d_i \hrm^+_i \big)  \; ,  &
   \delta\big(\hrm^+_i\big) \, = \, 0  \\
   \delta\big(\hrm^-_i\big) \, = \, 0  \; ,  &  \delta(\frm_i) \, = \,
   \frm_i \otimes \big( d_i \hrm^-_i \big) - \big( d_i \hrm^-_i \big) \otimes \frm_i
      \end{matrix}  $$
 \vskip4pt
   Now, all this construction can be extended as follows.  Instead of the symmetric matrix  $ \, DA \, $,
   consider any square matrix  $ \, B = {\big(b_{ij}\big)}_{i, j \in I} \in M_\theta(\ZZ) \, $  such that
   $ \, B+ B^t = 2\,DA \, $.  Then one can repeat the construction in  \cite{Hal}  and then find a new Lie bialgebra
   $ \liegb $  given as follows: it is the Lie algebra over  $ \CC $  with generators
   $ \, \erm_i , \kdotrm_i , \ldotrm_i , \frm_i \, (\, i \in I \,) \, $  and relations
\begin{equation}  \label{eq: rel's_g(B)}
   \begin{matrix}
      \big[\, \kdotrm_i \, , \erm_j \big] = +b_{ij}\,\erm_j  \; ,  \;\;\;  \big[\, \kdotrm_i \, , \frm_j \big] = -b_{ij}\, \frm_j \; ,
      \;\;\;  \big[\, \ldotrm_i \, , \erm_j \big] = +b_{ji} \, \erm_j\; ,  \;\;\;  \big[\, \ldotrm_i \, , \frm_j \big] = -b_{ji} \, \frm_j  \\
      \big[\, \kdotrm_i \, , \kdotrm_j \big] = 0 \; ,  \;\;\;  \big[\, \ldotrm_i \, , \ldotrm_j \big] = 0 \; ,
      \;\;\;  \big[\, \kdotrm_i \, , \ldotrm_j \big] = 0 \; ,  \;\;\;  \big[\, \erm_i \, , \frm_j \big] =
      \delta_{ij} \, {(2\,d_i)}^{-1} \big(\, \kdotrm_i + \ldotrm_i \big)^{\phantom{\big|}}  \\
      {\ad(\erm_i)}^{1-a_{ij}}(\erm_j) = 0  \quad ,   \qquad   {\ad(\frm_i)}^{1-a_{ij}}(\frm_j) = 0
      \quad \qquad \big(\, i \not= j \,\big)
   \end{matrix}
\end{equation}
 and it bears the Lie bialgebra structure whose Lie cobracket is uniquely given by
\begin{equation}  \label{eq: form's_Lie-cobra_g(B)}
   \begin{matrix}
     \delta(\erm_i) \, = \, \kdotrm_i \otimes \erm_i - \erm_i \otimes \kdotrm_i  \; ,  &  \delta\big(\kdotrm_i\big) \, = \, 0  \\
     \delta\big(\ldotrm_i\big) \, = \, 0  \; ,  &  \delta(\frm_i) \, = \, \frm_i \otimes \ldotrm_i - \ldotrm_i \otimes \frm_i
   \end{matrix}   \qquad \qquad
\end{equation}
 Note that the Lie bialgebra  $ \liegd $  above is simply the special case of  $ \liegb $  for  $ \, B := DA \, $.
                                                                    \par
   A more detailed, thorough construction of these Lie bialgebras is presented in  \cite{GG2}.
 \vskip4pt
   Basing upon the  $ \erm_i $'s  and the  $ \frm_i $'s  we construct root vectors  $ \, \erm_\alpha \in \mathfrak{b}_+ \, $
   and  $ \, \frm_\alpha \in \mathfrak{b}_- \, $  (for all  $ \, \alpha \in \Phi^+ \, $);
   this construction takes place inside the nilpotent part of  $ \mathfrak{b}_+ $  and of  $ \mathfrak{b}_- \, $,
   hence these new elements are well-defined for each Lie bialgebra  $ \liegb $  as above.
   All these root vectors, together with the  $ \kdotrm_i $'s  and the  $ \ldotrm_i $'s,  form a
   {\sl Chevalley-type basis\/}  of  $ \liegb \, $,  \,with  $ \, \erm_{\alpha_i} = \erm_i \, $  and
   $ \, \frm_{\alpha_i} = \frm_i \, $  for all  $ \, i \in I \, $:
   indeed, up to signs this basis (hence the  $ \erm_\alpha $'s  and the  $ \frm_\alpha $'s)  is unique.
   We also recall  (cf.\ \S \ref{root-data})  the notation  $ \, d_\alpha := (\alpha,\alpha) \big/ 2 \, $  for all
   $ \, \alpha \in \Phi^+ \, $.
 \vskip4pt
   We introduce now some  {\sl $ \ZZ $--integral  forms\/} of  $ \liegb \, $.
\end{free text}

\vskip7pt

\begin{definition}  {\ }   \label{def_liegdotb-lieghatb-liegtildeb}
 Keep notation as above, in particular
%%%%%
% $ \, B = {\big(b_{ij}\big)}_{i, j \in I} \in M_\theta(\ZZ) \, $  such that
%%%%%
 $ \, B+ B^t = 2\,DA \, $.  Then:
 \vskip3pt
    {\it (a)}\,  We call  $ \liegdotb $  the Lie subalgebra over  $ \ZZ $  of  $ \liegb $  generated by the elements
    $ \; \erm_i \, $,  $ \, \frm_i \, $,  $ \, \kdotrm_i \, $,  $ \, \ldotrm_i \, $  and
    $ \, \hrm^\circ_i := {(2\,d_i)}^{-1} \big(\, \kdotrm_i +  \ldotrm_i \big) \, $  (for all  $ \, i \in I \, $);
    \,indeed, this is a  {\sl Lie bialgebra\/}  over  $ \ZZ \, $,  with
  $$  \displaylines{
   \delta(\,\erm_i) \, = \, \kdotrm_i \otimes \erm_i - \erm_i \otimes \kdotrm_i  \quad ,  \qquad
      \delta(\,\frm_i) \, = \, \frm_i \otimes \ldotrm_i - \ldotrm_i \otimes \frm_i  \cr
   \delta\big(\,\kdotrm_i\big) = \, 0  \;\;\quad ,  \;\;\qquad
      \delta\big(\,\ldotrm_i\big) = \, 0  \;\;\quad ,  \;\;\qquad
         \delta(\,\hrm^\circ_i) \, = \, 0  }  $$
 \vskip1pt
    {\it (b)}\,  We call  $ \liegtildeb $  the Lie subalgebra over  $ \ZZ $  of  $ \liegb $  generated by the elements
    $ \; \etilderm_\alpha := 2 \, d_\alpha \, \erm_\alpha \, $,  $ \, \ftilderm_\alpha := 2 \, d_\alpha \, \frm_\alpha \, $
    ($ \, \alpha \in \Phi^+ \, $),  $ \, \kdotrm_i \, $  and  $ \, \ldotrm_i \, $  ($ \, i \in I \, $);
    \,indeed, this is a  {\sl Lie bialgebra\/}  over  $ \ZZ \, $,  with
  $$  \displaylines{
   \delta(\,\etilderm_i) \, = \, \kdotrm_i \otimes \etilderm_i - \etilderm_i \otimes \kdotrm_i  \quad ,  \qquad
      \delta\big(\,\ftilderm_i\big) \, = \, \ftilderm_i \otimes \ldotrm_i - \ldotrm_i \otimes \ftilderm_i  \cr
    \delta\big(\,\kdotrm_i\big) = \, 0  \;\;\quad ,  \;\; \qquad \qquad \qquad
     \delta\big(\,\ldotrm_i\big) = \, 0    }  $$
 \vskip3pt
    {\it (c)}\,  {\it Assume in addition that  $ \, b_{ij} = d_i t^+_{ij} = d_j t^-_{ij} \, $  for some
    $ \, t^\pm_{ij} \in \ZZ \, $  ($ \, i , j \in I \, $)}.  Then we call  $ \lieghatb $  the Lie subalgebra over
    $ \ZZ $  of  $ \liegb $ generated  by the elements  $ \; \erm_i \, $,
    $ \, \frm_i \, $,  $ \, \krm_i := d^{-1}_i \kdotrm_i \, $,
    $ \, \lrm_i := d^{-1}_i \ldotrm_i \, $,  $ \, \hrm^\circ_i = 2^{-1} \big(\, \krm_i +  \lrm_i \big) \, $  (for all  $ \, i \in I \, $);
    \,indeed, this is a  {\sl Lie bialgebra\/}  over  $ \ZZ \, $,  with
  $$  \displaylines{
   \delta(\,\erm_i) \, = \, d_i (\, \krm_i \otimes \erm_i - \erm_i \otimes \krm_i \,)  \quad ,  \qquad
      \delta(\,\frm_i) \, = \, d_i (\, \frm_i \otimes \lrm_i - \lrm_i \otimes \frm_i \,)  \cr
      \hfill   \delta(\,\krm_i) \, = \, 0  \;\;\quad ,  \;\;\qquad
  \delta(\,\lrm_i) \, = \, 0  \;\;\quad ,  \;\;\qquad
         \delta\big(\,\hrm^\circ_i\big) \, = \, 0  \hfill  \diamondsuit }  $$
\end{definition}

\medskip

\begin{rmks}  \label{int-forms_of_g_B}
 \vskip3pt
   {\it (a)}\,  It is clear by definition that  $ \liegdotb \, $,  $ \liegtildeb $  and  $ \lieghatb $  are all
   $ \ZZ $--integral  forms of the Lie algebra  $ \liegb $  in  \S \ref{MP_lie-bialgebras},  i.e.\
   $ \; \CC \otimes_{\raise-1pt\hbox{$ \scriptstyle \, \ZZ\! $}} \liea \cong \liegb \, $  {\sl as Lie algebras\/}  for
   $ \, \liea \in \big\{\, \liegdotb \, , \, \liegtildeb \, , \, \lieghatb \big\} \; $.
                                                           \par
   We also remark that the elements  $ \, \etilderm_i \, $,  $ \, \ftilderm_i \, $,  $ \, \kdotrm_i \, $  and
   $ \, \ldotrm_i \, $  (with  $ \, i \in I \, $)  are enough to generate the Lie algebra
   $ \, \QQ \otimes_\ZZ \liegtildeb \, $  over  $ \QQ \, $;  therefore, the formulas given in
   Definition \ref{def_liegdotb-lieghatb-liegtildeb}{\it (b)\/}  are enough, though they do not display the values
   $ \delta\big(\,\ftilderm_\alpha\big) $  nor  $ \delta(\,\etilderm_\alpha) \, $,  to determine a unique Lie cobracket on
   $ \, \QQ \otimes_\ZZ \liegtildeb \, $,  so by restriction on  $ \, \liegtildeb \, $  too.
 \vskip3pt
   {\it (b)}\,  The fact that each of  $ \liegdotb \, $,  $ \liegtildeb $  and  $ \lieghatb $  be a Lie  {\sl sub-bialgebra\/}  of
   $ \liegb $  (hence a  $ \ZZ $--integral  form of it  {\sl as a Lie bialgebra\/}) is a direct check. It is also a consequence,
   though, of our results in  \S \ref{spec-1}  later on about specialization of suitable multiparameter quantum groups.
 \vskip3pt
   {\it (c)}\,  Definitions imply that in each Lie bialgebra  $ \liegb $   --- as well as in its  $ \ZZ $--integral  forms
   $ \liegdotb \, $,  $ \liegtildeb $  and  $ \lieghatb $  ---   the Lie  {\sl algebra\/}  structure  {\sl does\/}  depend on
   $ B \, $,  whereas the Lie  {\sl coalgebra\/}  structure  {\sl does not}.  This follows from simple observations,
   namely that the root vectors $ \erm_\alpha $  and  $ \frm_\alpha $  are independent of  $ B $,
   and that the formulas for the Lie cobracket of the  $ \kdotrm_i $'s,  the  $ \ldotrm_i $'s,
   the  $ \erm_\alpha $'s  and the  $ \frm_\alpha $'s  are independent of  $ B $  as well; this second fact
   requires a quick computation for non-simple  $ \alpha $'s,  where the condition  $ \, B + B^t = 2\,DA \, $  makes the job.
                                                           \par
   This implies that if we consider two such Lie bialgebras  $ \lieg_{\Bpicc'} $  and  $ \lieg_{\Bpicc''} \, $,
   and their corresponding basis elements (over  $ \QQ $)  $ \erm'_\alpha \, $,  $ \erm''_\alpha \, $,  etc.,
   {\it mapping  $ \; \erm'_\alpha \mapsto \erm''_\alpha \, $,  $ \; \kdotrm'_i \mapsto \kdotrm''_i \; $,
   $ \; \ldotrm'_i \mapsto \ldotrm''_i \; $  and  $ \; \frm'_\alpha \mapsto \frm''_\alpha \; $
   defines an isomorphism of Lie coalgebras}  $ \, \lieg_{\Bpicc'} \cong \lieg_{\Bpicc''} \; $,
   \,that on the other hand is  {\sl not\/}  one of Lie algebras.  The same occurs for the  $ \ZZ $--integral  forms as well.
\end{rmks}

\vskip5pt

   For later use we need yet another definition:

\vskip5pt

\begin{definition}  \label{def_Kostant-forms_U(g)}
 Given  $ \, B = {\big(b_{ij}\big)}_{i, j \in I} \in M_\theta(\ZZ) \, $  such that  $ \, B+ B^t = 2\,DA \, $,
 \,let  $ \liegb $  be the complex Lie algebra mentioned in  \S \ref{MP_lie-bialgebras}  above, and
 $ U\big(\,\liegb\big) $  its universal enveloping algebra.  We define  $ \, U_\ZZ\big(\,\liegdotb\big) \, $, resp.\
 $ \, U_\ZZ\big(\,\liegtildeb\big) \, $,  the  $ \ZZ $--subalgebra  of  $ \, U\big(\,\lieghatb\big) \, $  generated by
  $$  \displaylines{
   \phantom{resp.}  \qquad  \bigg\{ {\bigg(\, {\kdotrm_i \atop n} \bigg)} \, ,
\, {\bigg(\, {\ldotrm_i \atop n} \bigg)} \, , \, {\bigg(\, {\hrm^\circ_i \atop n} \bigg)} \, ,
\, \erm_i^{\,(n)} \, , \, \frm_i^{\,(n)} \;\bigg|\;\, i \in I , \, n \in \NN \,\bigg\} \; ,  \cr
   \text{resp.}  \,\qquad  \bigg\{ {\bigg(\, {\krm_i \atop n} \bigg)} \, , \, {\bigg(\, {\lrm_i \atop n} \bigg)} \, ,
   \, {\bigg(\, {\hrm^\circ_i \atop n} \bigg)} \, ,  \, \erm_i^{\,(n)} \, , \, \frm_i^{\,(n)} \;\bigg|\;\, i \in I , \, n \in \NN \,\bigg\} \; ,  }  $$
 where  $ \, \displaystyle{\bigg( {\trm \atop n} \bigg)} \, $  and  $ \, \text{a}^{(n)} \, $
 denote standard binomial coefficients and divided powers, and in the second case we are assuming that
 $ \, b_{ij} = d_i t^+_{ij} = d_j t^-_{ij} \, $  for some  $ \, t^\pm_{ij} \in \ZZ \, $  ($ \, i , j \in I \, $).   \hfill  $ \diamondsuit $
\end{definition}

\vskip5pt

\begin{rmks}  \label{int-forms_of_U(g_D)}
 {\it (a)}\,  By  Remarks \ref{int-forms_of_g_B}  above, it is easily seen that
 $ \, U_\ZZ\big(\,\liegdotb\big) \, $  and  $ U_\ZZ\big(\,\lieghatb\big) \, $  are  $ \ZZ $--integral forms of
 $ \, U\big(\liegb\big) \, $;  one can also find a  {\sl presentation\/}  of each of them by generators
 (the given ones) and relations.  Indeed, for both  $ \, U_\ZZ\big(\,\liegdotb\big) \, $  and  $ U_\ZZ\big(\,\lieghatb\big) \, $
 this is a simple variation of the well-known presentation of the Kostant  $ \ZZ $--integral  form of  $ U(\lieg) \, $,
 generated are binomial coefficients and divided powers of the Chevalley generators.
                                                           \par
   Moreover, as  $ \liegd $  is a Lie bialgebra,  $ U\big(\,\liegd\big) $  is in fact a  {\sl co-Poisson\/}  Hopf algebra;
   then  $ \, U_\ZZ\big(\,\liegdotb\,\big) \, $  and  $ U_\ZZ\big(\,\lieghatb\,\big) \, $  are in fact  $ \ZZ $--integral forms of
   $ \, U\big(\,\liegd\big) \, $  {\sl as co-Poisson Hopf algebras}.
 \vskip3pt
   {\it (b)}\,  By a standard fact in the ``arithmetic of binomial coefficients'',  $ \, U_\ZZ \big(\,\liegdotb\big) \, $
   contains also all ``translated'' binomial coefficients, of the form
   $ \, \displaystyle{\bigg(\, {{\kdotrm_i \! + \! z} \atop n} \bigg)} \, $,
   $ \, \displaystyle{\bigg(\, {{\ldotrm_i + z} \atop n} \bigg)} \, $  and
   $ \, \displaystyle{\bigg(\, {{\hrm_i + z} \atop n} \bigg)} \, $  for  $ \, i \in I \, $,  $ \, n \in \NN \, $,  $ \, z \in \ZZ \, $;
   then one has also a presentation of  $ U_\ZZ \big(\,\liegdotb\big) $  including these extra generators, and
   corresponding extra relations too.
%
%%%%%
% Similarly,  $ \, U_\ZZ\big(\,\lieghatb\big) \, $  contains the elements
% $ \, \displaystyle{\bigg(\, {{\krm_i + z} \atop n} \bigg)} \, $,
% $ \, \displaystyle{\bigg(\, {{\lrm_i + z} \atop n} \bigg)} \, $  and
% $ \, \displaystyle{\bigg(\, {{\hrm_i + z} \atop n} \bigg)} \, $  and
% admits a presentation including these extra generators and extra relations too.
%%%%%
%
 And similarly for  $ \, U_\ZZ\big(\,\lieghatb\big) \, $  as well.
 \vskip5pt
    {\it (c)}\,  The definition of the  $ \ZZ $--integral  forms  $ \, \liegdotb \, $,  $ \, \lieghatb \, $  and  $ \, \liegtildeb \, $
    --- of  $ \liegb $  ---   and of the forms  $ \, U_\ZZ \big(\,\liegdotb\big) \, $  and  $ \, U_\ZZ \big(\,\lieghatb\big) \, $
    --- of  $ \, U\big(\liegb\big) \, $  ---   may seem to come out of the blue, somehow.  Nevertheless, we will show in
    \S \ref{spec-1}  that they occur as direct output of a ``specialization process'' of multiparameter quantum groups
    {\sl once suitable integral forms of them are chosen}.
\end{rmks}

\vskip5pt

\begin{free text}{\bf Some  $ q $--numbers.}  \label{q-numbers}
 Throughout the paper we shall need to consider several kinds of  ``$ q $--numbers''.
 Let  $ \Zqqm $  be the ring of Laurent polynomials with integral coefficients in the indeterminate  $ q \, $.
 For every  $ \, n \in \NN \, $ we define
  $$  \displaylines{
   {(0)}_q  \, := \;  1 \;\; ,  \quad
    {(n)}_q  \, := \;  \frac{\,q^n -1\,}{\,q-1\,}  \; = \;  1 + q + \cdots + q^{n-1}  \; =
    \; {\textstyle \sum\limits_{s=0}^{n-1}} \, q^s  \qquad  \big(\, \in \, \ZZ[q] \,\big)  \cr
   {(n)}_q!  \, := \;  {(0)}_q {(1)}_q \cdots {(n)}_q  := \,  {\textstyle \prod\limits_{s=0}^n} {(s)}_q  \;\; ,  \quad
    {\binom{n}{k}}_{\!q}  \, := \;  \frac{\,{(n)}_q!\,}{\;{(k)}_q! {(n-k)}_q! \,}  \qquad  \big(\, \in \, \ZZ[q] \,\big) \cr
   {[0]}_q  := \,  1  \; ,  \;\;
    {[n]}_q  := \,  \frac{\,q^n -q^{-n}\,}{\,q-q^{-1}\,}  =  \, q^{-(n-1)} + \cdots + q^{n-1}  =
    {\textstyle \sum\limits_{s=0}^{n-1}} \, q^{2\,s - n + 1}  \!\quad  \big( \in \Zqqm \,\big)  \cr
   {[n]}_q!  \, := \;  {[0]}_q {[1]}_q \cdots {[n]}_q  = \,  {\textstyle \prod\limits_{s=0}^n} {[s]}_q  \;\; ,  \quad
    {\bigg[\, {n \atop k} \,\bigg]}_q  \, := \;  \frac{\,{[n]}_q!\,}{\;{[k]}_q! {[n-k]}_q!\,}  \qquad  \big(\, \in \, \Zqqm \,\big)  }  $$
 \vskip3pt
%
%%%%%
%    In particular, we have the identities
% %
%   $$  {(n)}_{q^2} = q^{n-1} {[n]}_q \;\; ,  \qquad  {(n)}_{q^2}! = q^{\frac{n(n-1)}{2}} {[n]}_q  \;\; ,
% \qquad  {\binom{n}{k}}_{\!\!q^{2}} = q^{k(n-k)} {\bigg[\, {n \atop k} \,\bigg]}_q  \;\; .  $$
% %
%%%%%
%
%%%
\noindent
 Moreover, we have
  $ \; \displaystyle{{(n)}_{q^2} = q^{n-1} {[n]}_q} \; $,  $ \; \displaystyle{{(n)}_{q^2}! = q^{\frac{n(n-1)}{2}} {[n]}_q} \; $,
  $ \; \displaystyle{{\binom{n}{k}}_{\!\!q^{2}} \!\! = q^{k(n-k)} {\Big[\, {n \atop k} \,\Big]}_q} \, $.
 \vskip5pt
%
%%%
   Furthermore, thinking of Laurent polynomials as functions on  $ \k^\times \, $,  for any  $ \, q \in \k^\times \, $
   we shall read every symbol above as representing the corresponding element in  $ \, \k \, $.
\end{free text}

\bigskip

\section{Multiparameter quantum groups}  \label{mpqgroups}

\smallskip

   In this section we present the notion of  {\it multiparameter quantum group},  or  {\it MpQG\/}  for short.
   We introduce it by a direct definition by generators and relations as it suits better for our purposes.
   There exists also a realization in terms of Nichols algebras of diagonal type, see for example
   \cite{ARS}, \cite{An4}. \cite{Gar}, \cite{He2}.
   Finally, we connect them with cocycle deformations of their simplest example, the ``canonical'' one.

\smallskip

 \subsection{Defining multiparameter quantum groups (=MpQG's)}  \label{def-MpQG}  \
 \vskip7pt
   In this subsection we introduce the multiparameter quantum group  $ \QE \, $,  or  {\sl MpQG\/}
   for short, associated with a matrix of parameters  $ \, \bq := {\big(\, q_{ij} \big)}_{i, j \in I} \, $
   of Cartan type (cf.\  \S \ref{multiparameters}).  We fix also scalars  $ q_i $  ($ \, i \in I \, $)
   as in  \S \ref{multiparameters},  {\it with the additional assumption that
   $ \, q_{ii}^{\,k} = q_i^{2k} \not= 1 \, $}  for all  $ \, k = 1, \dots, 1-a_{ij} \, $,  with  $ \, i, j \in I \, $  and  $ \, i \not= j \, $.

\medskip

\begin{definition}  \label{def:multiqgroup_ang}
 (cf.\ \cite{HPR})  We denote by  $ U_\bq(\hskip0,8pt\lieg) $  the unital associative  $ \k $--algebra
 generated by elements  $ \, E_i \, , \, F_i \, , \, K_i^{\pm 1} \, , \, L^{\pm 1}_i \, $  with  $ \, i \in I \, $
 obeying the following relations:
\begin{align*}
  (a)  \qquad  &  K_i^{\pm 1} L^{\pm 1}_j  \, = \,  L^{\pm 1}_j K^{\pm 1}_i \;\; ,
\qquad K_i^{\pm 1} K^{\mp 1}_i  \, = \,  1  \, = \,  L^{\pm 1}_i L^{\mp 1}_i  \\
  (b)  \qquad  &  K_i^{\pm 1} K^{\pm 1}_j  \, = \,  K^{\pm 1}_j K^{\pm 1}_i  \;\; ,
\qquad  L_i^{\pm 1} L_j^{\pm 1}  \, = \,  L_j^{\pm 1} L_i^{\pm 1}  \\
  (c)  \qquad  &  K_i \, E_j \, K_i^{-1}  \, = \,  q_{ij} \, E_j  \;\; ,  \qquad
L_i \, E_j \, L_i^{-1}  \, = \,  q_{ji}^{-1} \, E_j  \\
  (d)  \qquad  &  K_i \, F_j \, K_i^{-1}  \, = \,  q_{ij}^{-1} \, F_j  \;\; ,  \qquad
L_i \, F_j \, L_i^{-1}  \, = \,  q_{ji} \, F_j  \\
  (e)  \qquad  &  [E_i , F_j]  \, = \,  \delta_{i,j} \, q_{ii} \, \frac{\, K_i - L_i \,}{\, q_{ii} - 1 \,}  \\
%%%
  (f)  \qquad  &  \sum_{k=0}^{1-a_{ij}} (-1)^k \,
  {\bigg( {1-a_{ij} \atop k} \bigg)}_{\!\!q_{ii}} q_{ii}^{{k \choose 2}} \, q_{ij}^k \,
E_i ^{\,1-a_{ij}-k} E_j \, E_i^{\,k}  \; = \;  0   \;\; \qquad (\, i \neq j \,)  \\
%%%
  (g)  \qquad  &  \sum_{k=0}^{1-a_{ij}} (-1)^k \,
  {\bigg( {1-a_{ij} \atop k} \bigg)}_{\!\!q_{ii}} q_{ii}^{{k \choose 2}} \, q_{ij}^k \,
F_i^{\,k} F_j \, F_i^{\,1-a_{ij}-k}  \; = \;  0   \;\; \qquad (\, i \neq j \,)  \\
\end{align*}

   \indent   Moreover,  $ U_\bq(\hskip0,8pt\lieg) $  is a Hopf algebra with coproduct, counit and antipode
   determined for all  $ \, i, j \in I \, $  by
\begin{align*}
  \com(E_i) \,  &  = \,  E_i \otimes 1 + K_i \otimes E_i  \;\; ,  &
 \epsilon(E_i) \,  &  = \,  0  \;\; ,  &  \SS(E_i) \,  &  = \,  -K_i^{-1} E_i  \\
  \com(F_i) \,  &  = \,  F_i \otimes L_i + 1 \otimes F_i  \;\; ,  &
 \epsilon(F_i) \,  &  = \,  0  \;\; ,  &  \SS(F_i) \,  &  = \,  - F_i{L_i}^{-1}  \\
  \com\big(K_i^{\pm 1}\big)  &  = \,  K_i^{\pm 1} \otimes K_i^{\pm 1}  \;\;  ,  &
 \epsilon\big(K_i^{\pm 1}\big)  &  = \,  1  \;\; ,  &  \SS\big(K_i^{\pm 1}\big)  &  = \,  K_i^{\mp 1}  \\
  \com\big(L_i^{\pm 1}\big)  &  = \,  L_i^{\pm 1} \otimes L_i^{\pm 1}  \;\; ,  &
 \epsilon\big(L_i^{\pm 1}\big)  &  = \,  1  \;\; ,  &  \SS\big(L_i^{\pm 1}\big)  &  = \,  L_i^{\mp 1}
\end{align*}
 \vskip5pt
   Finally, for later use we introduce also,  for every  $ \, \lambda = \sum_{i \in I} \lambda_i \, \alpha_i \, \in \, Q \, $,
   the notation  $ \; K_\lambda := \prod_{i \in I} K_i^{\lambda_i} \; $  and
   $ \; L_\lambda := \prod_{i \in I} {L_i}^{\lambda_i} \; $.   \hfill  $ \diamondsuit $
\end{definition}

\medskip

\begin{rmk}  \label{link_QEq-QE & symm-case}
 Assume that  $ \, q \in \k^\times \, $  is not a root of unity and fix the ``canonical'' multiparameter
 $ \check{\bq} := {\big(\, \check{q}_{ij} = q^{d_i a_{ij}} \big)}_{i, j \in I} \, $  like in  (\ref{qij-canon}).
 Then we can define the corresponding MpQG, denoted  $ \QEqcheck \, $:
 the celebrated one-parameter quantum group  $ U_q(\lieg) $  by Jimbo and Lusztig is
 (up to a minimal, irrelevant change of generators) just the quotient of  $ \QEqcheck $  by the (Hopf)
 ideal generated by  $ \, \big\{ L_i - K_i^{-1} \,\big|\; i = 1, \dots, \theta \, \big\} \, $.
                                                                  \par
   As a matter of fact, that we shall deeply exploit in the present work,  {\sl most constructions usually
   carried on for  $ U_q(\lieg) $   --- like construction of (quantum) root vectors, of integral forms, etc.\ ---
   actually makes sense and apply the same to  $ \QEqcheck $  as well}.
\end{rmk}

\smallskip

   We introduce now a family of subalgebras of any MpQG, say  $ U_\bq(\hskip0,8pt\lieg) \, $,  as follows:

\medskip

\begin{definition}  \label{def:q-Bor-sbgr}
  Given  $ \, \bq := {\big(\, q_{ij} \big)}_{i,j \in I} \, $  and  $ \, U_\bq(\hskip0,8pt\lieg) \, $  as in \S \ref{def-MpQG},
  we define  $ \, U_\bq^{\,0} := U_\bq(\lieh \oplus \lieh) \, $,  $ \, U_\bq^{+,0} \, $,  $ \, U_\bq^{-,0} \, $,
  $ \; U_\bq^- := U_\bq(\lien_-) \, $,  $ \; U_\bq^+ := U_\bq(\lien_+) \, $,  $ \; U_\bq^\leq :=  U_\bq(\lieb_-) \; $
  and  $ \; U_\bq^\geq := U_\bq(\lieb_+) \; $  to be the  $ \k $--subalgebra  of  $ \QEq $  respectively generated as
  $$  \displaylines{
   U_\bq^{\,0}  \; := \;  \Big\langle\, K_i^{\pm 1} \, , \, L_i^{\pm 1} \,\Big\rangle_{i \in I}  \quad ,
\qquad   U_\bq^{+,0}  \; := \;  \Big\langle\, K_i^{\pm 1} \,\Big\rangle_{i \in I}  \quad ,
\qquad   U_\bq^{-,0}  \; := \;  \Big\langle\, L_i^{\pm 1} \,\Big\rangle_{i \in I}  \cr
%
%    U_\bq^-  \; := \;  \big\langle\, F_i \,\big\rangle_{i \in I}  \quad ,
% \qquad   U_\bq^+  \; := \;  \big\langle\, E_i \,\big\rangle_{i \in I}  \cr
%    U_\bq^\leq  \; := \;  \Big\langle\, F_i \, , \, L_i^{\pm 1} \,\Big\rangle_{i \in I}  \quad ,
% \qquad   U_\bq^\geq  \; := \;  \Big\langle\, K_i^{\pm 1} \, , \, E_i
% \,\Big\rangle_{i \in I}  }  $$
%
   U_\bq^-  := \,  \big\langle F_i \big\rangle_{\! i \in I}  \; ,
\quad   U_\bq^\leq  := \,  \Big\langle F_i \, , \, L_i^{\pm 1} \Big\rangle_{\! i \in I}  \; ,
\quad   U_\bq^\geq  := \,  \Big\langle K_i^{\pm 1} \, , \, E_i \Big\rangle_{\! i \in I}  \; ,
\quad   U_\bq^+  := \,  \big\langle E_i \big\rangle_{\! i \in I}  }  $$
   \indent   We shall refer to  $ U_\bq^\leq $  and  $ U_\bq^\geq $  as to the  {\sl positive\/}  and  {\sl negative\/}
   multiparameter quantum Borel (sub)algebras, and  $ U_\bq^{\,0} \, $,  $ U_\bq^{+,0} $  and  $ U_\bq^{-,0} $
   as to the  {\sl global},  {\sl positive\/}  and  {\sl negative\/}  multiparameter Cartan (sub)algebras.   \hfill  $ \diamondsuit $
\end{definition}

\vskip5pt

   Recall the notion of ``skew-Hopf pairing'' (cf.\  Definition \ref{def_(skew-)Hopf-pairing}).
   From \cite[Proposition 4.3]{He2}   --- see also \cite[Theorem 20]{HPR}  and  \cite[Propostion 2.4]{AY}  ---  we have:

\vskip9pt

\begin{prop}  \label{sk-H_pair}
 With the assumptions above, assume in addition that  $ \, q_{ii} \not= 1 \, $  for all indices  $ \, i \in I \, $.
 Then there exists a unique skew-Hopf pairing
 $ \; \eta : U_\bq^\geq \mathop{\otimes}\limits_\k \, {\big(U_\bq^\leq\big)}^{\cop} \relbar\joinrel\longrightarrow \k \; $
 that is non-degenerate and such that, for all  $ \, 1 \leq i, j \leq \theta \, $,  one has
 \vskip-11pt
  $$  \eta(K_i,L_j) \, = \, q_{ij} \;\; ,  \qquad  \eta(E_i,F_j) \, = \, \delta_{i,j} \, {{\, - \, q_{ii} \,} \over {\, q_{ii} - 1 \,}} \;\; ,
  \qquad  \eta(E_i,L_j) \, = \, 0 \, = \, \eta(K_i,F_j)  $$
   \indent   Moreover, for every  $ \, E \in U_\bq^+\, $,  $ \, F \in U_\bq^- \, $,
   and every Laurent monomials  $ K $  in the  $ K_i $'s  and  $ L $  in the  $ L_j $'s,  we have
  $$  \eta\big( E \, K , F \, L \big)  \; = \;  \eta(E\,,F) \, \eta(K,L)  $$
\end{prop}

\vskip7pt

   The following result states that there exist special ``tensor product factorizations'' of MpQG's
   (the last ones are usually referred to as ``triangular decompositions''):

\medskip

\begin{prop}  \label{prop:triang-decomps_U}
 {\sl (cf.\ \cite[Corollary 22]{HPR}, \cite[Corollary 2.6]{BGH})}
                                                         \par
   The multiplication in  $ U_\bq(\hskip0,8pt\lieg) $  provides  $ \, \k $--linear  isomorphisms
  $$  \displaylines{
   U_\bq^- \mathop{\otimes}_\k U_\bq^{\,0} \,\; \cong \;\, U_\bq^\leq \,\;
   \cong \;\, U_\bq^{\,0} \mathop{\otimes}_\k U_\bq^-  \quad ,
 \qquad  U_\bq^ + \mathop{\otimes}_\k U_\bq^{\,0} \,\; \cong \;\, U_\bq^\geq \,\;
 \cong \;\, U_\bq^{\,0} \mathop{\otimes}_\k U_\bq^+  \cr
   U_\bq^{\,+,0} \mathop{\otimes}_\k U_\bq^{\,-,0}  \, \cong \;  U_\bq^{\,0}  \,
   \cong \;  U_\bq^{\,-,0} \mathop{\otimes}_\k U_\bq^{\,+,0}  \quad ,
 \qquad
  U_\bq^\leq \, \mathop{\otimes}_\k \, U_\bq^\geq  \, \cong \;  U_\bq  \, \cong \;
  U_\bq^\geq \, \mathop{\otimes}_\k \, U_\bq^\leq
  \cr
   U_\bq^+ \mathop{\otimes}_\k U_\bq^{\,0} \mathop{\otimes}_\k U_\bq^-  \,
 \cong \;  U_\bq  \, \cong \;  U_\bq^- \mathop{\otimes}_\k U_\bq^{\,0} \mathop{\otimes}_\k U_\bq^+ }  $$
\end{prop}

\medskip

\begin{rmk}  \label{split-dec}
 It is clear from definitions that  $ \, U_\bq^0 = U_\bq(\lieh \oplus \lieh) \, $  has the set of monomials in the
 $ K_i^{\pm 1} $'s  and the  $ {L_i}^{\pm 1} $'s  as  $ \k $--basis.  It follows then that each triangular decompositions of
 $ \QE $  as above induces also a splitting
  $ \; U_\bq(\hskip0,8pt\lieg) = U_\bq(\lieh \oplus \lieh) \oplus {U_\bq(\hskip0,8pt\lieg)}^\oplus $ \;
 where
                                            \par
   \centerline{ $ {U_\bq(\hskip0,8pt\lieg)}^\oplus  \, := \;
   \big(\, {U_\bq(\lien_-)}^+ \!\cdot U_\bq(\liehd) \cdot U_\bq(\lien_+) \, +
   \, U_\bq(\lien_-) \cdot U_\bq(\liehd) \cdot {U_\bq(\lien_+)}^+ \, $. }
%
% \,where as usual  $ T^+ $  denotes the intersection of a subspace  $ T $  in a Hopf algebra  $ H $
% with the augmentation ideal of  $ H $  itself, usually denoted by  $ H^+ $.
%
\end{rmk}

\smallskip

 \subsection{MpQG's as cocycle deformations}  \label{deform-MpQG}  \
 \vskip7pt
   Now we want to perform on the Hopf algebras  $ U_\bq(\hskip0,8pt\lieg) $  a cocycle deformation process,
   via special types of  $ 2 $--cocycles,  like in  \S \ref{cocyc-defs},  following  \cite{AST},  \cite{DT}  and  \cite{Mo}.

\smallskip

   Let us consider  $ \, \bq :=  {\big(\, q_{ij} \big)}_{i,j \in I} \, $  and  $ U_\bq(\hskip0,8pt\lieg) $  as in  \S \ref{def-MpQG}.
   As explained in  \S \ref{multiparameters},  we fix a special element
   $ \, q_{j_{\raise-2pt\hbox{$ \scriptscriptstyle 0 $}}} \in \k^\times \, $,  also denoted by
   $ \, q := q_{j_{\raise-2pt\hbox{$ \scriptscriptstyle 0 $}}} \, $;  for this choice of  $ \, q \, $,
   we consider the canonical ``one parameter'' quantum group  $ \QEqcheck $  as in
   Remark \ref{link_QEq-QE & symm-case}.
                                                       \par
   Recall from  Definition \ref{def:multiqgroup_ang}  the notation  $ \, K_\lambda := \prod_{i \in I} K_i^{\lambda_i} \, $
   and  $ \, L_\lambda := \prod_{i \in I} {L_i}^{\lambda_i} \, $  for every
   $ \, \lambda = \sum_{i \in I} \lambda_i \alpha_i \, \in \, Q \, $.  Similarly, we shall also write
   $$  q_{\mu\,\nu}  \, :=  \! {\textstyle \prod\limits_{i,j \in I}} \, q_{ij}^{\, \mu_i \nu_j}  \; ,
  \;\;\;  q^{\;1/2}_{\mu\,\nu}  \, :=  \! {\textstyle \prod\limits_{i,j \in I}} {\big(\, q_{ij}^{\;1/2\,} \big)}^{\mu_i \nu_j}
   \qquad  \forall \;\;  \mu = {\textstyle \sum\limits_{i \in I}} \, \mu_i \, \alpha_i \, , \; \nu =
   {\textstyle \sum\limits_{j \in I}} \, \nu_j \, \alpha_j \, \in \, Q  $$
Likewise, we define also  $ \; \subu{q}{\beta} \! := q_i \, $  for every positive root  $ \, \beta \in \Phi^+ \, $
which belongs to the same orbit as the simple root  $ \alpha_i $  for the action of the Weyl group of  $ \lieg $
onto  $ Q $
(which is well-defined, by standard theory of root systems).

\medskip

\begin{definition}  \label{def-sigma}
 With the above conventions, let  $ \QEqcheck $  be the MpQG of  Remark \ref{link_QEq-QE & symm-case},
 and let  $ \, \sigma : \QEqcheck \otimes \QEqcheck \relbar\joinrel\longrightarrow \k \, $  be the unique
 $ \k $--linear  map given by
  $$  \displaylines{
   \sigma(x,y)  \; := \;  q^{\;1/2}_{\mu\,\nu}   \qquad  \text{if}  \quad
         x = K_\mu  \text{\;\ or \ }  x = L_\mu \; ,  \!\quad  y = K_\nu  \text{\;\ or \ }  y = L_\nu  \cr
   \sigma\big( \QEqcheck \, , {\QEqcheck}^\oplus \big) \; := \; 0 \; =: \; \sigma\big( {\QEqcheck}^\oplus \, , \QEqcheck \big)  }  $$
 (by  Remark \ref{split-dec}  above, this is enough to determine a unique  $ \sigma $  as requested).   \hfill  $ \diamondsuit $
\end{definition}

\medskip

   The key result that we shall rely  upon in the sequel is the following:

\medskip

\begin{theorem}  \label{thm:sigma_2-cocy}
 (cf.\ \cite[Theorem 28.]{HPR})
 Let  $ \, \bq := {\big(\, q_{i j} \,\big)}_{i, j \in I} \, $  and  $ \, q $  be as above.
 Then the map  $ \sigma $  in  Definition \ref{def-sigma}  is a normalized  $ 2 $--cocycle  of the Hopf algebra
 $ \, \QEqcheck \, $  and there exists a Hopf algebra isomorphism (with notation of\,  \S \ref{cotwist-defs_1})
  $$  U_\bq(\hskip0,8pt\lieg)  \; \cong \;  {\big(\, \QEqcheck \big)}_\sigma  $$
\end{theorem}

\medskip

\begin{rmk}
 A similar result is given in  \cite[Theorem 4.5]{HLR},  but using another  $ \sigma \, $.
\end{rmk}

\medskip

   As a last result in this section, we can show that the  $ 2 $--cocycle  deformation considered in
   Theorem \ref{thm:sigma_2-cocy}  can be also realized as a cocycle deformation in the sense of  \S \ref{cotwist-defs_2}
   as well. Indeed, let  $ \, \varGamma := \ZZ^{2\theta} \, $  be the free Abelian group generated by the
   $ K_i $'s  and  $ L_i $'s ($ \, i \in I \, $),  and  $ V_E \, $,  resp.\  $ V_F $,  be the  $ \k $--vector
   space generated by the  $ E_i $'s, resp.\ the  $ F_i $'s  ($ \, i \in I \, $).
   Then, by \cite{Gar},  we know that  $ \QEqcheck $  is a quotient of
   $ \, T\big( V_E \oplus V_F \big) \,\#\, \k\varGamma \, $  by the two-sided ideal generated by the relations
   $ (e) $,  $ (f) $  and  $ (g) \, $  in  Definition \ref{def:multiqgroup_ang}.  We have a  $ (Q \times Q) $--grading  on
   $ \, T\big( V_E \oplus V_F \big) \,\#\, \Bbbk\varGamma \, $  given by
  $$  \deg(K_i)  \, = \,  (\alpha_i,\alpha_i)  \, = \,  \deg\big(L_i\big)  \;\; ,   \qquad \deg(E_i) \, = \, (1,\alpha_i)  \;\; ,
  \qquad  \deg(F_i)  \, = \, (\alpha_i,1)  $$
 for all  $ \, i \in I \, $;  \,it coincides with the grading induced by the coaction on the Yetter-Drinfeld modules  $ V_E $
 and  $ V_F $  such that  $ \, \deg(K_i) = \deg(L_i) \, $.  As the defining relations are homogeneous with
respect to this grading, we get a  $ (Q \times Q) $--grading  on  $ \QEqcheck \, $.
                                                            \par
   Consider now the group  $ 2 $--cocycle  $ \, \varphi \in \Zc^2(\varGamma,\k) \, $  given by
   $ \, \varphi := \sigma\big|_{\varGamma \times \varGamma} \; $,  \, that is
  $$  \varphi(h,k)  \, := \,  q^{1/2}_{\mu\,\nu}  \;\qquad \text{if}  \quad  h = K_\mu  \text{\ \ or \ }
  h = L_\mu \; ,  \quad  k = K_\nu  \text{\ \ or \ }  k = L_\nu  $$
 and let  $ \tilde{\varphi} $  be the  $ 2 $--cocycle  defined on  $ \, T(V \oplus W) \,\#\, \k\varGamma \, $  as
in  \S \ref{subsubsec:relation-cocycles}.  Since  $ \varGamma $  is Abelian and
$ \; E_i \, \mathop{\cdot}\limits_{\tilde{\varphi}} \, F_j \, = \, E_i \, F_j \; $  for all  $ \, i, j \in I \, $,  we have that
  $ \;  E_i \,\mathop{\cdot}_{\tilde{\varphi}} F_j \, - \, F_j \,\mathop{\cdot}_{\tilde{\varphi}} E_i \, = \, [\,E_i \, , F_j \,] \, $,  \;
 hence  $ \tilde{\varphi} $  defines a Hopf  $ 2 $--cocycle  $ \widehat{\varphi} $  on  $ \QEqcheck \, $.
 Finally, a direct comparison shows that  $ \, \widehat{\varphi} = \sigma \, $.
 Thus, using  \S \ref{subsubsec:relation-cocycles},  we conclude that the following holds:

\vskip15pt

\begin{prop}  \label{prop:def-sigma=def-c}
 There exists a Hopf algebra identification
  $$  {\big(\, \QEqcheck \big)}_\sigma  \; = \;  {\big(\, \QEqcheck \big)}^{(\tilde{\varphi})}  $$
hence, by  Theorem \ref{thm:sigma_2-cocy},  a Hopf algebra isomorphism
$ \, \QEq \, \cong {\big(\, \QEqcheck \big)}^{(\tilde{\varphi})} \, $.
\end{prop}

\medskip

\subsection{Multiparameter quantum groups with larger torus} \label{MpQGs-larger-torus}  \
 \vskip7pt
   The MpQG's  $ \QEq $  that we considered so far have a toral part (i.e., the subalgebra  $ U_\bq^0 $
   generated by the  $ K_i^{\pm 1} $'s  and the  $ L_j^{\pm 1} $'s)  that is nothing but the group algebra of a
   double copy of the root lattice  $ Q $  of  $ \lieg \, $,  much like in the one-parameter case
   --- but for the duplication of  $ Q \, $,  say.  Now, in that (uniparameter) case, one also considers MpQG's
   with a larger toral part, namely the group algebra of any intermediate lattice between  $ Q $  and  $ P \, $;
   similarly, we can introduce MpQG's whose toral part is the group algebra of any lattice
   $ \, \varGamma_\ell \times \varGamma_r \, $  with  $ \, Q \subseteq \varGamma_\ell \, $  and
   $ \, Q \subseteq \varGamma_r \, $.

\vskip11pt

\begin{free text}  \label{tori_in_MpQG's}
 {\bf Larger tori for MpQG's.}
 Recall that the definition of the ``toral parts'' of  $ \, U_\bq(\hskip0,8pt\lieg) \, $
 --- cf.\  Definition \ref{def:q-Bor-sbgr}  ---   is independent of  $ \bq \, $:  indeed,  $ U_\bq^{+,0} $
 is the group algebra over  $ \k $  for the group  $ Q $   --- identifying  $ \, \pm\alpha_i \simeq K_i^{\pm 1} \, $
 and  $ \, \alpha \simeq K_\alpha \, $  ($ \, i \in I \, $,  $ \, \alpha \in Q \, $);  similarly,  $ U_\bq^{-,0} $
 is the group algebra (over  $ \k \, $)  of  $ Q $  again with  $ \, \alpha \simeq L_\alpha \, $,  and
 $ \, U_\bq(\lieh \oplus \lieh) := U_\bq^{\,0} \, $  is the group algebra (over  $ \k \, $)  of  $ \, Q \times Q \, $
 --- with  $ \, \big(\alpha',\alpha''\big) \simeq K_{\alpha'} \, L_{\alpha''} \, $.
                                                        \par
   Let us denote by  $ \QQ{}Q $  and  $ \QQ{}P $  the scalar extension from  $ \ZZ $  to  $ \QQ $
   of the lattices  $ \ZZ{}Q $  and  $ \ZZ{}P \, $,  respectively; note that  $ \QQ{}Q = \QQ{}P \, $.
   For any other sublattice  $ \varGamma $  in  $ \, \QQ{}Q \, \big( = \QQ{}P \big) \, $  of rank  $ \theta $
   --- the same as  $ Q $  and  $ P $  ---   we can define toral quantum groups  $ U_{\bq\,,\varGamma}^{\pm,0} $
   akin to  $ U_\bq^{\pm,0} $  but now associated with the lattice  $ \varGamma \, $,  again as group algebras;
   similarly, we have an analogue  $ U_{\bq\,,\varLambda}^0 $  of  $ U_\bq^0 $  associated with any sublattice
   $ \varLambda $  in  $ \, \QQ{}Q \times \QQ{}Q \, $  of rank  $ \, 2\,\theta \, $.
   Moreover, all these bear a natural Hopf algebra structure.  Any sublattice inclusion
   $ \, \varGamma' \leq \varGamma'' \, $  yields a unique Hopf embedding
   $ \, U_{\bq\,,\varGamma'}^{\pm,0} \subseteq U_{\bq\,,\varGamma''}^{\pm,0} \, $,
   and similar embeddings exist for the  $ U_{\bq\,,\varLambda}^0 $'s.
   We aim to use these ``larger toral MpQG's'' as toral parts of larger MpQG's; this requires some
   compatibility constraints on  $ \bq \, $, and some preliminary facts that we now settle.
                                                      \par
   Let  $ \varGamma $  be a sublattice of  $ \QQ{}Q $  of rank  $ \theta $  with  $ \, Q \leq \varGamma \, $.
   For any basis  $ \big\{ \gamma_1 , \dots \gamma_\theta \big\} $  of  $ \varGamma \, $,  let
   $ \, C := {\big( c_{ij} \big)}_{i, j \in I} \, $  be the matrix, with entries in  $ \ZZ \, $,
   that describes the change of basis (for  $ \QQ{}Q $  as a  $ \QQ $--vector  space) from
   $ {\big\{ \gamma_i \big\}}_{i \in I} $  to  $ {\big\{ \alpha_i \big\}}_{i \in I} \, $,  so
   $ \, \alpha_i = \sum_{j=1}^\theta c_{ij}\,\gamma_j \, $  for each  $ \, i \in I = \{1,\dots,\theta\} \, $.
   Let also  $ \, c := \big| \text{\sl det}(C) \big| \in \NN_+ \, $  be the absolute value of the determinant of  $ C \, $;
   this is equal to the index (as a subgroup) of  $ Q $  in  $ \varGamma \, $,
   hence it is independent of any choice of basis.  If  $ \, C^{-1} = {\big( c'_{ij} \big)}_{i, j \in I} \, $
   is the inverse matrix to  $ C \, $,  then  $ \, \gamma_i = \sum_{j=1}^\theta c'_{ij}\,\alpha_j \, $  and
   $ \, c''_{ij} := c \cdot c'_{ij} \in \ZZ \, $  for all  $ \, i, j \in I = \{1,\dots,\theta\} \, $.
                                                      \par
   Let now  $ U_{\bq\,,\varGamma}^{+,0} $  be given, as group algebra of  $ \varGamma $  over  $ \k $
   with generators  $ K_{\gamma_i}^{\pm 1} $  corresponding to the basis elements  $ \gamma_i $
   (and their opposite) in  $ \varGamma \, $  (for  $ \, i \in I \, $).  Define
   $ \, K_{\alpha_i} := \prod_{j \in I} K_{\gamma_j}^{\,c_{ij}} \, $  for all  $ \, i \in I \, $:
   then the  $ \k $--subalgebra  of  $ U_{\bq\,,\varGamma}^{+,0} $  generated by the  $ K_{\alpha_i}^{\pm 1} $'s
   is an isomorphic copy of  $ U_{\bq\,,Q}^{+,0} \, $,  which provides a realization of the Hopf algebra embedding
   $ \, U_{\bq\,,Q}^{+,0} \subseteq U_{\bq\,,\varGamma}^{+,0} \, $  corresponding to the group embedding
   $ \, Q \leq \varGamma \, $.
                                                      \par
   In the obvious symmetric way we define also the ``negative counterpart''  $ U_{\bq\,,\varGamma}^{-,0} \, $
   of  $ U_{\bq\,,\varGamma}^{+,0} \, $,  generated by elements  $ L_{\gamma_i}^{\pm 1} $  corresponding to the
   $ \pm \gamma_i $'s  in  $ \varGamma \, $,  along with a suitable embedding
   $ \, U_{\bq\,,Q}^{-,0} \subseteq U_{\bq\,,\varGamma}^{-,0} \, $  corresponding to the group embedding
   $ \, Q \leq \varGamma \, $.
                                                      \par
   Finally, given any two sublattices  $ \varGamma_\pm $  of rank  $ \theta $  in  $ \QQ{}Q $  containing
   $ Q \, $,  letting  $ \, \varGamma_\bullet := \varGamma_+ \times \varGamma_- \, $  we define
   $ \; U_{\bq\,,\varGamma_\bullet}^0 :=
   \, U_{\bq\,,\varGamma_+}^{+,0} \!\mathop{\otimes}\limits_\k \, U_{\bq\,,\varGamma_-}^{-,0} \; $;
   in this case, the basis elements for  $ \varGamma_\pm $  will be denoted by
   $ \, \gamma_i^\pm \, $  ($ \, i \in I \, $).  The previous embeddings of  $ U_{\bq\,,Q}^{\pm,0} $  into
   $ U_{\bq\,,\varGamma_\pm}^{\pm,0} $  then induces a similar embedding of
   $ \, U_\bq^0 = U_{\bq\,,Q}^{+,0} \mathop{\otimes}\limits_\k U_{\bq\,,Q}^{-,0} \, $  into
   $ U_{\bq\,,\varGamma_\bullet}^0 $  as well.
\end{free text}

\vskip5pt

\begin{free text}  \label{MpQG-larg-tor_int-case}
 {\bf Special root pairings (in the integral case).}  {\sl Let us now assume that the multiparameter
 $ \, \bq := {\big(\, q_{ij} = q^{b_{ij}} \big)}_{i,j \in I} \, $  is of  {\it integral type}\/};  we therefore use notation
 $ \, B := {\big( b_{ij} \big)}_{i,j \in I} \in M_\theta(\ZZ) \, $.  Then a  $ \ZZ $--bilinear  pairing
 $ \; {(\,\ ,\ )}_{\!{}_B} : Q \times Q \relbar\joinrel\longrightarrow \ZZ \; $  is defined via the matrix  $ B $  by
 $ \, {( \alpha_i \, , \alpha_j )}_{\!{}_B} := b_{ij} \, $  for all  $ \, i, j \in I \, $.  Moreover, by
 Proposition \ref{sk-H_pair},  we know that the pairing
 $ \, U_\bq^\geq \mathop{\otimes}\limits_\k U_\bq^\leq \relbar\joinrel\longrightarrow \k \, $
 is non-degenerate; but then (by the special properties of this pairing) its restriction to
 $ \, U_\bq^{\,0} \mathop{\otimes}\limits_\k U_\bq^{\,0} \, $  is non-degenerate too.
 Finally, from  $ \, \big\langle K_\alpha \, , L_\beta \big\rangle = q^{{(\alpha , \, \beta)}_{\!B}} \, $
 (for all  $ \, \alpha , \beta \in Q \, $)  we get that
 $ \, {(\,\ ,\ )}_{\!{}_B} \! : Q \times Q \relbar\joinrel\longrightarrow \ZZ \, $
 is non-degenerate as well, which forces  $ B $  to be invertible.

\vskip9pt

   By scalar extension,
 $ {(\,\ ,\ )}_{\!{}_B} $  yields also
 a  $ \QQ $--bilinear  pairing on  $ \, \QQ{}Q \, $,  which again is non-degenerate; we denote it also by
 $ {(\,\ ,\ )}_{\!{}_B} \, $.
 It is then meaningful to consider, for any sublattice  $ \varGamma $  in  $ \, \QQ{}Q \, $,  its  {\it left-dual}
 $ \dot{\varGamma}^{(\ell)} $
 and its  {\it right-dual}  $ \dot{\varGamma}^{(r)} $,  defined by
\begin{equation}  \label{dual-lattice}
  \begin{matrix}
    \dot{\varGamma}^{(\ell)}  \, := \,  \big\{\, \lambda \in \QQ{}Q \;\big|\, {(\lambda \, , \gamma)}_{\!{}_B} \in \ZZ \, ,
    \, \forall \; \gamma \in \varGamma \,\big\}_{\phantom{|}}  \\
    \dot{\varGamma}^{(r)}  \, := \,  \big\{\, \rho \in \QQ{}Q \;\big|\, {(\gamma , \rho)}_{\!{}_B} \in \ZZ \, ,
    \, \forall \; \gamma \in \varGamma \,\big\}^{\phantom{\big|}}
  \end{matrix}
\end{equation}
 that are sublattices in  $ \, \QQ{}Q \, $  and coincide iff  $ B $  is symmetric; then restricting the
 $ \QQ $--bilinear  pairing  $ \, {(\,\ , \ )}_{\!{}_B} : \QQ{}Q \times \QQ{}Q \relbar\joinrel\longrightarrow \QQ \, $
 to  $ \, \dot{\varGamma}^{(\ell)} \times \varGamma \, $  and  $ \, \varGamma \times \dot{\varGamma}^{(r)} \, $
 one gets  $ \ZZ $-valued  pairings  $ \, \dot{\varGamma}^{(\ell)} \times \varGamma \relbar\joinrel\longrightarrow \ZZ \, $
 and  $ \, \varGamma \times \dot{\varGamma}^{(r)} \relbar\joinrel\longrightarrow \ZZ \, $,  still denoted by
 $ \, {(\,\ , \ )}_{\!{}_B} \, $.
 \vskip7pt
   Using the matrix  $ \, B^{-1} = {\big( b'_{ij} \big)}_{i,j \in I} \, $,  we define in  $ \QQ{}Q $  the elements
\begin{equation}  \label{def-dotvarpi^l's}
  \qquad \hskip85pt   \dot{\varpi}_i^{(\ell)}  \, := \,  {\textstyle \sum}_{k \in I} \, b'_{i{}k} \, \alpha_k
  \qquad \qquad \qquad   \forall \;\; i \in I  \qquad
\end{equation}
which are characterized by the property that  $ \; {\big( \dot{\varpi}_i^{(\ell)} , \alpha_j \big)}_{\!{}_B} \, = \, \delta_{i{}j} \; $;
in short,  $ {\big\{ \dot{\varpi}_i^{(\ell)} \big\}}_{i \in I} $  is the  $ \QQ $--basis  of  $ \QQ{}Q $  which is
{\sl left}-dual  to the basis
$ {\big\{ \alpha_j \big\}}_{j \in I} $  w.r.t.\  $ \, {\big(\,\ , \ \big)}_{\!{}_B} \, $;  in particular,
$ {\big\{ \dot{\varpi}_i^{(\ell)} \big\}}_{i \in I} $
is a  $ \ZZ $--basis  of  $ \dot{Q}^{(\ell)} \, $,  the left-dual to  $ Q $  w.r.t.\  $ {\big(\,\ , \ \big)}_{\!{}_B} \, $.
Definitions give also
$ \, Q \subseteq \dot{Q}^{(\ell)} \, $  with  $ \; \alpha_i = \sum_{k \in I} b_{i{}k} \, \dot{\varpi}_k^{(\ell)} \, $
for all  $ \, i \in I \, $.
                                                  \par
   The left-right symmetrical counterpart is given once we define the elements
\begin{equation}  \label{def-dotvarpi^r's}
  \qquad \hskip85pt   \dot{\varpi}_i^{(r)}  \, := \,  {\textstyle \sum}_{k \in I} \, b'_{k{}i} \, \alpha_k
  \qquad \qquad \qquad   \forall \;\; i \in I  \qquad
\end{equation}
characterized by the property that  $ \; {\big( \alpha_j \, , \dot{\varpi}_i^{(r)} \big)}_{\!{}_B} \, = \, \delta_{j{}i} \; $;  thus
$ {\big\{ \dot{\varpi}_i^{(r)} \big\}}_{i \in I} $  is the  $ \QQ $--basis  of  $ \QQ{}Q $  which is  {\sl right}-dual  to the basis
$ {\big\{ \alpha_j \big\}}_{j \in I} $  with respect to  $ \, {\big( \,\ , \ \big)}_{\!{}_B} \, $;  in particular,
$ {\big\{ \dot{\varpi}_i^{(r)} \big\}}_{i \in I} $  is a  $ \ZZ $--basis  of  $ \dot{Q}^{(r)} \, $,  the right-dual to
$ Q $  w.r.t.\  $ {\big(\,\ , \ \big)}_{\!{}_B} \, $.  Furthermore, definitions give also  $ \, Q \subseteq \dot{Q}^{(r)} \, $
with  $ \; \alpha_i = \sum_{k \in I} b_{k{}i} \, \dot{\varpi}_k^{(r)} \, $  for all  $ \, i \in I \, $.

\vskip9pt

%%%
   {\sl  $ \underline{\text{The  {\it strongly}  integral case}} $:} \
   The previous construction has a sort of ``refinement'' when the integral type multiparameter
   $ \, \bq := {\big(\, q_{ij} = q^{b_{ij}} \big)}_{i, j \in I} \, $  is actually  {\sl  {\it strongly}  integral},
   with  $ \, b_{ij} = d_i \, t_{ij}^+ = d_j \, t_{ij}^- \, $  for all $ \, i, j \in I \, $  (cf.\  \S \ref{multiparameters}).
   In this case, consider the two  $ \ZZ $--bilinear  pairings
 $ \; {\langle\,\ ,\ \rangle}_{\!{}_{T^\pm}} : Q \times Q \relbar\joinrel\longrightarrow \ZZ \; $
defined by the matrices  $ T^+ $  and  $ T^- $   --- thus given by
$ \, {\langle \alpha_i \, , \alpha_{j\,} \rangle}_{\!{}_{T^\pm}} := t^\pm_{ij} \, $  for all  $ \, i, j \in I \, $
---   that are obviously non-degenerate (as  $ {(\,\ ,\ )}_{\!{}_B} $  is, and  $ \, D\,T^+ = B = T^-\,D \, $),
and extend them to same-name  $ \QQ $--bilinear  pairings on  $ \, \QQ{}Q \times \QQ{}Q \, $  by scalar extension.
Then define, for any sublattice  $ \varGamma $  in  $ \QQ{}Q \, $,  its  {\it left-dual\/}
and  {\it right-dual\/}  (w.r.t.\  $ T^- $  and  $ T^+ $  respectively) as
\begin{equation}  \label{dual-lattice_T+/-}
  \begin{matrix}
    \varGamma^{(\ell)}  \, := \,  \big\{\, \lambda \in \QQ{}Q \;\big|\, {\langle \lambda \, ,
    \gamma \rangle}_{\!{}_{T^-}} \!\! \in \ZZ \, , \,
    \forall \; \gamma \in \varGamma \,\big\}_{\phantom{|}}  \\
    \varGamma^{(r)}  \, := \,  \big\{\, \rho \in \QQ{}Q \;\big|\, {\langle \gamma , \rho \rangle}_{\!{}_{T^+}} \!\! \in \ZZ \, , \,
    \forall \; \gamma \in \varGamma \,\big\}^{\phantom{\big|}}
  \end{matrix}
\end{equation}
 that both are sublattices in  $ \QQ{}Q \, $;  the pairings  $ \, {\langle \,\ , \ \rangle}_{\!{}_{T^\pm}} \, $
 then restrict to  $ \ZZ $-valued  pairings
 $ \, {\langle \,\ , \ \rangle}_{\!{}_{T^-}} : \varGamma^{(\ell)} \times \varGamma \relbar\joinrel\longrightarrow \ZZ \, $
 and  $ \, {\langle \,\ , \ \rangle}_{\!{}_{T^+}} \varGamma \times \varGamma^{(r)} \relbar\joinrel\longrightarrow \ZZ \, $.
 Now consider the elements
\begin{equation}  \label{def-varpi^l/r's}
  \varpi_i^{(\ell)}  := \,  d_i\,\dot{\varpi}_i^{(\ell)}  =  {\textstyle \sum}_{k \in I} \, t^{-,\,\prime}_{i{}k} \, \alpha_k  \,\; ,
  \;\;\;  \varpi_i^{(r)} := \,  d_i\,\dot{\varpi}_i^{(r)}  =  {\textstyle \sum}_{k \in I} \, t^{+,\,\prime}_{k{}i} \, \alpha_k   \qquad   \forall \,\; i \in I  \;
\end{equation}
where  $ \, {\big( t^{\pm,\,\prime}_{ij} \big)}_{i \in I}^{j \in I} = {\big( T^\pm \big)}^{-1} \, $,  \, which are
characterized by the properties  $ \; {\big\langle \varpi_i^{(\ell)} , \alpha_j \big\rangle}_{\!{}_{T^-}} \! = \delta_{i{}j} \; $
and  $ \; {\big\langle \alpha_i \, , \varpi_j^{(r)} \big\rangle}_{\!{}_{T^+}} \! = \delta_{i{}j} \; $; in a nutshell,
$ {\big\{ \varpi_i^{(\ell)} \big\}}_{i \in I} $  is the  $ \QQ $--basis  of  $ \QQ{}Q $
which is  {\sl left}-dual  to the basis  $ {\big\{ \alpha_j \big\}}_{j \in I} $  w.r.t.\
$ \, {\big\langle \,\ , \ \big\rangle}_{\!\scriptscriptstyle T^-} \, $,  while  $ {\big\{ \varpi_i^{(r)} \big\}}_{i \in I} $
is the  {\sl right}-dual  to  $ {\big\{ \alpha_j \big\}}_{j \in I} $  w.r.t.\
$ \, {\big\langle \,\ , \ \big\rangle}_{\!\scriptscriptstyle T^+} \, $.  In particular,  $ {\big\{ \varpi_i^{(\ell)} \big\}}_{i \in I} $
is a  $ \ZZ $--basis  of  $ Q^{(\ell)} \, $,  and  $ {\big\{ \varpi_i^{(r)} \big\}}_{i \in I} $  is a  $ \ZZ $--basis  of
$ Q^{(r)} \, $  with notation of \eqref{dual-lattice_T+/-}.
Note also that definitions give  $ \, Q \subseteq Q^{(\ell)} \bigcap Q^{(r)} \, $  with
$ \; \alpha_i = \sum_{k \in I} t^-_{i{}k} \, \varpi_k^{(\ell)} \, $  and  $ \; \alpha_i = \sum_{k \in I} t^+_{k{}i} \, \varpi_k^{(r)} \, $
for all  $ \, i \in I \, $.
\end{free text}

\vskip11pt

\begin{free text}  \label{larger-MpQG's}
 {\bf MpQG's with larger tori.}  Let  $ \varGamma_+ $  and  $ \varGamma_- $  be any two lattices in
 $ \QQ{}Q $ such that  $ \, Q \leq \varGamma_\pm \, $;  then set
 $ \, \varGamma_{\!\bullet} := \varGamma_+ \times \varGamma_- \, $.  From \S \ref{tori_in_MpQG's},
 with notation fixed therein, we can consider the corresponding ``multiparameter quantum torus''
 $ U_{\bq\,,\varGamma_\bullet}^0 \, $,  that contains  $ U_\bq^0 = U_{\bq\,,Q}^0 \, $.
 For either lattice  $ \varGamma_\pm \, $,  we have a matrix  $ \, C_\pm = {\big( c_{ij}^\pm \big)}_{i,j \in I} \, $  and
 $ \, C_\pm^{\,-1} = {\big( c_{ij}^{\pm,\prime\,} \big)}_{i,j \in I} \, $,  \,with
 $ \, c_\pm := \big| \text{\sl det}(C_\pm) \big| \in \ZZ_+ \, $  and
 $ \, c_{ij}^{\pm,\prime\prime} := c_\pm \cdot c_{ij}^{\pm,\prime} \in \ZZ \, $  ($ \, i , j \in I \, $).
 \vskip3pt
   {\sl For the rest of this subsection, we make now the following assumption concerning the ground field
   $ \k $ and the multiparameter (of Cartan type)  $ \, \bq := {\big(\, q_{ij} \big)}_{i,j \in I} \, $:  for every  $ \, i, j \in I \, $,
   the field  $ \k $  contains a  $ c_\pm $--th  root of  $ \, q_{ij} \, $,  hereafter denoted by  $ q_{ij}^{\,1/c_{{}_\pm}} \, $;
   moreover, we assume that  $ \, \bq^{\,1/c_{{}_\pm}} := {\Big(\, q_{ij}^{\,1/c_{{}_\pm}} \Big)}_{i,j \in I} \, $
   {\it is of Cartan type}  too}.
 \vskip1pt
 The natural (adjoint) action of  $ U_\bq^0 $  onto  $ U_\bq $  extends, in a unique manner, to a
 $ U_{\bq\,,\varGamma_{\!\bullet}}^0 $--action
 $ \,\; \cdot :\, U_{\bq\,,\varGamma_{\!\bullet}}^0 \times U_\bq \longrightarrow U_\bq \, $,  \,given by
 \vskip-13pt
%
%%%%%
% \begin{align*}
%   K_{\gamma^+_i} \cdot\, E_j  \, = \,  q_{ij}^{\,\varGamma_{\!+}} \, E_j  \;\; ,  \qquad \qquad
%     &  \quad \qquad  L_{\gamma^-_i} \cdot\, E_j  \, = \,  {\big( q_{ji}^{\,\varGamma_{\!-}} \big)}^{-1} \, E_j  \\
%   K_{\gamma^+_i} \cdot \, F_j  \, = \,  {\big( q_{ij}^{\,\varGamma_{\!+}} \big)}^{-1} \, F_j  \;\; ,  \qquad \quad
%     &  \quad \qquad \hskip7pt  L_{\gamma^-_i} \cdot \, F_j  \, = \,  q_{ji}^{\,\varGamma_{\!-}} \, F_j  \\
%   K_{\gamma^+_i} \cdot K_{\alpha_{j}}  \, = \,   K_{\alpha_{j}}  \; ,  \quad
%       L_{\gamma^-_i} \cdot K_{\alpha_{j}}  \, = \,   K_{\alpha_{j}}  \; ,
%     &  \qquad  K_{\gamma^+_i} \cdot L_{\alpha_j}  \, = \,   L_{\alpha_j}  \; ,  \quad
%       L_{\gamma^-_i} \cdot L_{\alpha_j}  \, = \,   L_{\alpha_j}
%
% \end{align*}
%%%%%
%
  $$  \displaylines{
   K_{\gamma^+_i} \cdot E_j = q_{ij}^{\,\varGamma_{\!+}} E_j  \; ,  \quad
 L_{\gamma^-_i} \cdot E_j = {\big( q_{ji}^{\,\varGamma_{\!-}} \big)}^{-1} E_j  \; ,  \quad
  K_{\gamma^+_i} \cdot K_{\alpha_{j}} = K_{\alpha_{j}}  \; ,  \quad
    L_{\gamma^-_i} \cdot K_{\alpha_{j}} = K_{\alpha_{j}}  \cr
   K_{\gamma^+_i} \cdot L_{\alpha_j} = L_{\alpha_j}  \; ,  \quad
      L_{\gamma^-_i} \cdot L_{\alpha_j} = L_{\alpha_j}  \; ,  \quad
  K_{\gamma^+_i} \cdot F_j = {\big( q_{ij}^{\,\varGamma_{\!+}} \big)}^{-1} F_j  \; ,  \quad
    L_{\gamma^-_i} \cdot F_j = q_{ji}^{\,\varGamma_{\!-}} F_j  }  $$
 where  $ \; q_{rs}^{\,\varGamma_{\!+}} := \prod_{k \in I} {\big( q_{ks}^{\,1/c_{{}_+}} \big)}^{\!c_{rk}^{+,\prime\prime}} \; $
 and  $ \,\; q_{ae}^{\,\varGamma_{\!-}} := \prod_{k \in I} {\big( q_{ak}^{\,1/c_{{}_-}} \big)}^{\!c_{ek}^{-,\prime\prime}} \; $;
 this makes  $ U_\bq $  into a  $ U_{\bq\,,\varGamma_{\!\bullet}}^0 $--module  Hopf algebra.
 This allows us to consider the Hopf algebra
$ \, U_{\bq\,,\varGamma_{\!\bullet}}^0 \ltimes U_\bq \, $  given by the {\it smash product\/}  of
$ U_{\bq\,,\varGamma_{\!\bullet}}^0 $  and  $ U_\bq \, $:  the underlying vector space is just
$ \, U_{\bq\,,\varGamma_{\!\bullet}}^0 \otimes U_\bq \, $,  the coalgebra structure is the one
given by the tensor product of the corresponding coalgebras, and the product is given by the formula
  $$  (h \ltimes x) \, (k\ltimes y)  \; = \;  h \, k_{(1)} \ltimes \big( \SS(k_{(1)}) \cdot x \big) \, y \eqno \text{for all}
  \;\  h, k \in U_{\bq\,,\varGamma_{\!\bullet}}^0 \, , \; x, y \in U_{\bq} \;\; .  $$
Since  $ U_{\bq\,,\varGamma_{\!\bullet}}^0 $  contains  $ \, U_\bq^0 \; \big(\! =: U_{\bq\,,Q \times Q}^0 \big) \, $
as a Hopf subalgebra, it follows that  $ U_{\bq\,,\varGamma_{\!\bullet}}^0 $  itself is a right  $ U_\bq^0 $--module
Hopf algebra with respect to the adjoint action.  It is not difficult to see that, under these hypotheses, the smash
product  $ \, U_{\bq\,,\varGamma_{\!\bullet}}^0 \ltimes U_\bq \, $  maps onto a Hopf algebra structure on the
vector space  $ \, U_{\bq\,,\varGamma_{\!\bullet}}^0 \mathop{\otimes}\limits_{U_\bq^0} U_\bq \, $,
which hereafter we denote by  $ \, U_{\bq\,,\varGamma_{\!\bullet}}^0 \mathop{\ltimes}\limits_{U_\bq^0} U_\bq \, $,
see \cite[Theorem 2.8]{Le}.  We define then
\begin{equation}  \label{def-MpQG_wt_lrg-torus}
  U_{\bq\,,\varGamma_{\!\bullet}}(\hskip0,8pt\lieg)  \, \equiv \,  U_{\bq\,,\varGamma_{\!\bullet}}  \,
  := \,  U_{\bq\,,\varGamma_{\!\bullet}}^0 \mathop{\ltimes}\limits_{U_\bq^0} U_\bq  \, = \,
  U_{\bq\,,\varGamma_{\!\bullet}}^0 \mathop{\ltimes}\limits_{U_\bq^0} \QEq
\end{equation}
   \indent   It is easy to check that  $ U_{\bq\,,\varGamma_{\!\bullet}}(\hskip0,8pt\lieg) $
   and its Hopf algebra structure can be described with a presentation by generators and relations like that for
   $ \QEq $  in  Definition \ref{def:multiqgroup_ang}.
   Indeed, since the coalgebra structure is the one given by the tensor product, one has to describe only the
   algebra structure. For this, first one has to replace the generators  $ \, K_i^{\pm 1} = K_{\pm \alpha_i} \, $  and
   $ \, L_i^{\pm 1} = L_{\pm \alpha_i} \, $  with the generators  $ \, \Kc_i^{\pm 1} = K_{\pm \gamma^+_i} \, $  and
   $ \, \Lc_i^{\pm 1} = L_{\pm \gamma^-_i} \, $.  Second, replace relations  {\it (c)\/}  and  {\it (d)\/}  of
   Definition \ref{def:multiqgroup_ang}  with the following, generalized relations:
\begin{align*}
  \big(c'\big)  \qquad  &  \hskip3pt  K_{\gamma^+_i\,} E_j \, K_{\gamma^+_i}^{-1}  \, = \,  q_{ij}^{\,\varGamma_{\!+}} \, E_j  \;\; ,  \qquad
      L_{\gamma^-_i\,} E_j \, L_{\gamma^-_i}^{-1}  \, = \,  {\big( q_{ji}^{\,\varGamma_{\!-}} \big)}^{-1} \, E_j  \\
  \big(d'\big)  \qquad  &  \hskip5pt  K_{\gamma^+_i\,} F_j \, K_{\gamma^+_i}^{-1}  \, = \,
  {\big( q_{ij}^{\,\varGamma_{\!+}} \big)}^{-1} \, F_j  \;\; ,  \qquad \hskip1pt
      L_{\gamma^-_i\,} F_j \, L_{\gamma^-_i}^{-1}  \, = \,  q_{ji}^{\,\varGamma_{\!-}} \, F_j
\end{align*}
 Then, in relation  {\it (e)\/}  write each  $ \, K_i^{\pm 1} = K_{\alpha_i} \, $,  resp.\  $ \, L_i^{\pm 1} = L_{\pm \alpha_i} \, $,
 in terms of the  $ \, \Kc_j^{\pm 1} \! := K_{\pm \gamma^+_j} $'s,  resp.\  $ \, \Lc_j^{\pm 1} \! := L_{\pm \gamma^-_j} $'s;
 finally, leave relations  {\it (f)\/}  and  {\it (g)\/} unchanged.
                                 \par
   With much the same approach, one defines also the ``(multiparameter) quantum subgroups'' of
   $ U_{\bq\,,\varGamma_{\!\bullet}}(\hskip0,8pt\lieg) $  akin to those of  $ \QEq $
   (cf.\  Definition \ref{def:q-Bor-sbgr}),  that we denote by adding a subscript  $ \varGamma_{\!\bullet} \, $,
   namely  $ U_{\bq\,,\varGamma_{\!\bullet}}^+ \, $,  $ U_{\bq\,,\varGamma_\bullet}^- \, $,
   $ U_{\bq\,,\varGamma_{\!\bullet}}^\geq \, $,  $ U_{\bq\,,\varGamma_{\!\bullet}}^\leq \, $,
   $ U_{\bq\,,\varGamma_{\!\bullet}}^{+,0} $  and  $ U_{\bq\,,\varGamma_{\!\bullet}}^{-,0} \, $.

\vskip9pt

   {\sl  $ \underline{\text{The integral case}} $:}  When  $ \bq $  is of integral type, the above construction may
   have a simpler description.
   Indeed,  {\sl assume also that the lattices  $ \varGamma_+ $  and  $ \varGamma_- $
   (both containing  $ Q \, $)  are such that
   $ \, \varGamma_+ \leq \dot{Q}^{(\ell)} \, $  and  $ \, \varGamma_- \leq \dot{Q}^{(r)} \, $},  that is
   $ \; {\big( \varGamma_+ \, , Q \,\big)}_{\!{}_B} \subseteq \ZZ \; $  and
   $ \; {\big( Q \, , \varGamma_- \big)}_{\!{}_B} \subseteq \ZZ \; $
   --- notation of  \S \ref{MpQG-larg-tor_int-case}.  Then in the presentation of the MpQG
   $ U_{\bq\,,\varGamma_\bullet} $  of  \eqref{def-MpQG_wt_lrg-torus}
 the modified relations  {\it (c')\/}  and  {\it (d')\/}  mentioned above take the form
\begin{align*}
  \big(c'\big)  \qquad  &  \hskip3pt  K_{\gamma^+_i\,} E_j \, K_{\gamma^+_i}^{-1}  \, = \,  q^{+{( \gamma^+_i \! , \, \alpha_j )}_{\!B}} \, E_j  \;\; ,  \qquad
      L_{\gamma^-_i\,} E_j \, L_{\gamma^-_i}^{-1}  \, = \,  q^{-{( \alpha_j , \, \gamma^-_i )}_{\!B}} \, E_j  \\
  \big(d'\big)  \qquad  &  \hskip5pt  K_{\gamma^+_i\,} F_j \, K_{\gamma^+_i}^{-1}  \, = \,  q^{-{( \gamma^+_i \! , \, \alpha_j )}_{\!B}} \, F_j  \;\; ,  \qquad \hskip1pt
      L_{\gamma^-_i\,} F_j \, L_{\gamma^-_i}^{-1}  \, = \,  q^{+{( \alpha_j , \, \gamma^-_i )}_{\!B}} \, F_j
\end{align*}
   \indent   In particular, this means that  $ U_{\bq\,,\varGamma_\bullet} $
   {\sl is actually well-defined over the (possibly smaller) ground field generated in  $ \k $  by  $ q \, $
   --- and similarly for  $ U_{\bq\,,\varGamma_{\!\bullet}}^+ \, $,  $ U_{\bq\,,\varGamma_{\!\bullet}}^\geq \, $,  etc}.
   Therefore, the assumption that  $ \k $  contain  $ c_+ $--th  and  $ c_- $--th  roots of  $ \, q_{ij} \, $,
   that is required in the non-integral case, is not necessary in the integral one.
\end{free text}

\vskip7pt

\begin{free text}  \label{duality x larger MpQG's}
 {\bf Duality among MpQG's with larger tori.}  Let again  $ \, \varGamma_\pm \, $
 be two lattices of rank  $ \theta $  in  $ \QQ{}Q $  containing  $ Q \, $,  and set
 $ \, \varGamma_\bullet := \varGamma_+ \times \varGamma_- \, $;  then we have ``toral MpQG's''
 $ U_{\bq\,,\varGamma_\pm}^{\pm,0} $ and  $ U_{\bq\,,\varGamma_\bullet}^0 $  as in  \S \ref{tori_in_MpQG's}.
 Moreover, we have bases  $ {\big\{\, \gamma_s^\pm \,\big\}}_{s \in I} $  of  $ \varGamma_\pm \, $
 and corresponding matrices  $ \, C_\pm = {\big( c_{ij}^\pm \big)}_{i,j \in I} \, $  and
 $ \, C_\pm^{\,-1} = {\big( c_{ij}^{\pm,\prime\,} \big)}_{i,j \in I} \, $,  and the integers
 $ \, c_\pm := \big| \text{\sl det}(C_\pm) \big| \, $  and
 $ \, c_{ij}^{\pm,\prime\prime} := c_\pm \cdot c_{ij}^{\pm,\prime} \, $  ($ \, i, j \in I \, $)  as in  \S \ref{larger-MpQG's}.
 In addition, we assume that  $ \k $  contain a  $ (c_+\,c_-) $--th  root of  $ \, q_{ij} \, $,  say
 $ q_{ij}^{\,1/(c_{{}_+}c_{{}_-})} $,  and that overall the multiparameter
 $ \, \bq^{\,1/(c_{{}_+}c_{{}_-})} := {\Big(\, q_{ij}^{\,1/(c_{{}_+}c_{{}_-})} \,\Big)}_{i,j \in I}\, $  be of Cartan type.
                                                               \par
   It is straightforward to check that the skew-Hopf pairing
   $ \; \eta : U_\bq^\geq \mathop{\otimes}\limits_\k U_\bq^\leq \!\relbar\joinrel\relbar\joinrel\longrightarrow \k \; $
   in  Proposition \ref{sk-H_pair}
%
% (cf.\  Remarks \ref{rmks:mpqg's-vs-Nichols}.{\it (a)\/}  too)
%
 actually extends to a similar pairing
 $ \; U_{\bq\,,\varGamma_+}^\geq \mathop{\otimes}\limits_\k U_{\bq\,,\varGamma_-}^\leq
 \!\! \relbar\joinrel\relbar\joinrel\longrightarrow \k \; $  given by
 $ \, E_i \otimes L_{\gamma^-} \!\mapsto 0 \, $,  $ \, K_{\gamma^+} \otimes F_j \!\mapsto 0 \, $,
 $ \, E_i \otimes F_j \!\mapsto - \delta_{ij} {{\,q_{ii}\,} \over {\,q_{ii} - 1\,}} \, $,
 $ \, K_{\gamma_i^+} \otimes L_{\gamma_j^-} \mapsto \prod_{h, k \in I} \!
 {\Big( q_{hk}^{\,1/(c_{{}_+}c_{{}_-})} \Big)}^{\! c_{ih}^{+,\prime\prime} c_{jk}^{-,\prime\prime}} $
 \,for all  $ \, i, j \in I \, $,  and still denoted  $ \eta \, $.  In particular, this
 $ \; \eta : U_{\bq\,,\varGamma_+}^\geq \!\mathop{\otimes}\limits_\k
 U_{\bq\,,\varGamma_-}^\leq \!\!\! \relbar\joinrel\relbar\joinrel\longrightarrow \k \, $
 is still non-degenerate, like its restrictions
 $ \; U_{\bq\,,\varGamma_+}^+ \!\mathop{\otimes}\limits_\k U_{\bq\,,\varGamma_-}^- \!\!\!
 \relbar\joinrel\relbar\joinrel\longrightarrow \k \; $  and
 $ \; U_{\bq\,,\varGamma_+}^0 \!\mathop{\otimes}\limits_\k
 U_{\bq\,,\varGamma_-}^0 \!\!\! \relbar\joinrel\relbar\joinrel\longrightarrow \k \; $.

\vskip5pt

   {\sl When  $ \, \bq := {\big(\, q_{ij} = q^{b_{ij}} \big)}_{i \in I}^{j \in I} \, $  is of  {\it integral type}},
   and $ \, {(\,\ , \ )}_{\!{}_B} \! : Q \times Q \longrightarrow \ZZ \, $  is the associated pairing
   --- cf.\ \S \ref{MpQG-larg-tor_int-case}  ---   {\sl the previous construction may have a simpler description,
   under the additional assumption that  $ \, {\big( \varGamma_+ \, , \varGamma_- \big)}_{\!{}_B} \!\subseteq \ZZ \, $}
   --- that is equivalent to either of  $ \, \varGamma_+ \subseteq \dot{\varGamma}_-^{\,(\ell)} \, $
   and  $ \, \varGamma_- \subseteq \dot{\varGamma}_+^{\,(r)} \, $  ---   so that  $ {(\,\ , \ )}_{\!{}_B} $
   induces a pairing  $ \, {(\,\ , \ )}_{\!{}_B} \! : \varGamma_+ \times \varGamma_- \relbar\joinrel\longrightarrow \ZZ \, $.
   In the following, we shall briefly refer to such a situation by saying that
   {\it  $ \big( \varGamma_+ \, , \varGamma_- \big) $  is a pair in duality}  (w.r.t.\  $ B \, $),
   or that  {\it the lattices  $ \varGamma_+ $  and  $ \varGamma_- $  are in duality}  (w.r.t.\  $ B \, $).
   Indeed, under these assumptions we have
   $ \; \eta\big( K_{\gamma_i^+\!} \, , L_{\gamma_j^-} \big) \, = \,
   \prod_{h, k \in I} {\Big( q_{hk}^{\,1/(c_{{}_+}c_{{}_-})} \Big)}^{\! c_{ih}^{+,\prime\prime} c_{jk}^{-,\prime\prime}} \! =
   \, q^{{( \gamma_i^+ , \gamma_j^- )}_{\!B}} \; $;  in particular, requiring a  $ (c_{{}_+}c_{{}_-}) $--th  root in
   $ \k $  of every  $ q_{hk} $  is no longer necessary.
\end{free text}

\vskip4pt

\begin{rmk}  \label{MpQG-largetor=qDouble}
 It is easy to see that, using the skew-Hopf pairing  $ \eta $  between (suitably chosen) quantum Borel subgroups
 $ U_{\bq\,,\varGamma_+}^{\,\geq} $  and  $ U_{\bq\,,\varGamma_-}^{\,\leq} $  mentioned in
 \S \ref{duality x larger MpQG's} above,
%%%%%
 {\it every MpQG with larger torus, say  $ \, U_{\bq\,,\varGamma_{\!\bullet}}(\hskip0,8pt\lieg) \, $,
 can be realized as a Drinfeld double (of those quantum Borel subgroups)},
%%%%%
 so extending what happens for MpQG's with ``standard'' torus.
\end{rmk}

\bigskip

\section{Quantum root vectors and PBW theorems for MpQG's}  \label{q-root_vects & PBW}

   The first purpose of this section is to introduce root vectors for MpQG's.
   Second, we show that PBW theorems hold true for an MpQG
%%%%%
%   --- as well as for its Borel subalgebras, etc.\ ---   which involve root vectors as just mentioned.
%%%%%
 and all its relevant subalgebras.

\medskip

\subsection{Quantum root vectors in MpQG's}  \label{rvec-MpQG}  \
 \vskip7pt
   For the one-parameter quantum group  $ U_q(\lieg) $  of Lusztig several authors introduced quantum
   analogues of root vectors   --- or ``quantum root vectors'' ---   in different ways, the most common ones
   being via iterated  $ q $--brackets  or iterated adjoint action.  Lusztig gave  (cf.\ \cite{Lu})
   a general procedure, using an action on  $ U_q(\lieg) $  of the braid group associated with  $ \lieg \, $;
   later, it was extended to the multiparameter case in  \cite{He2}.

\smallskip

   To begin with, let  $ W $  the Weyl group of  $ \lieg \, $,  generated by reflections  $ \, s_i = s_{\alpha_i} \, $
   associated with the simple roots  $ \alpha_i $  of  $ \lieg $  ($ \, i \in I \, $),  and let  $ \, w_0 \in W \, $
   be the longest element in  $ W \, $.  Then the number  $ \, N := \big| \Phi^+ \big| \,  $  of positive roots
   (cf.\ \ref{root-data})  of  $ \lieg $  is also the length of any reduced expression of  $ \, w_0 \, $.
   Let us fix now one such reduced expression, say  $ \, w_0 = s_{i_1} s_{i_2} s_{i_3} \cdots s_{i_{N-1}} s_{i_N} \, $,
   so that all the following constructions will actually depend on this specific choice.
                                                           \par
   Set  $ \; \beta^k := s_{i_1} s_{i_2} \cdots s_{i_{k-1}}\big(\alpha_{i_k}\big) \; $  for all  $ \, k = 1, \dots, N \, $:
   then one has  $ \, {\big\{ \beta^k \big\}}_{k=1,2,\dots,N} \! = \Phi^+ \, $;
   in particular, all positive roots are recovered starting from the fixed reduced expression of  $ w_0 \, $,
   and in addition this also endows  $ \Phi^+ $  with a total order, namely  $ \; \beta^k \preceq \beta^h \iff k \leq h \; $.
   The same method of course can be applied to negative roots.
                                                           \par
   A similar procedure allows to construct a root vector in  $ \lieg $  for each positive root.
   First consider the braid group  $ \mathbb{B} $  associated with  $ W $,  generated by elements  $ \overline{T}_i $
   which lift the simple reflections  $ \, s_i = s_{\alpha_i} \, $  ($ \, i \in I \, $).
   There is a standard way  (cf.\ for instance  \cite{Hu})  to define a group action of  $ \mathbb{B} $  onto
   $ \lieg $  that on root space yields  $ \, \overline{T}_i(\,\lieg_\beta) = \lieg_{s_i(\beta)} \, $;
   \,using this action one can define root vectors via
  $$  x_{\beta^k}  \, := \;  \overline{T}_{i_1} \overline{T}_{i_2} \cdots \overline{T}_{i_{k-1}}\!\big(\,x_{i_k}\big)
  \; \in \;  \lieg_{\beta^k}
  \eqno  \forall \;\; k = 1, 2, \dots, N   \qquad  $$
where each  $ \, x_i \, $  is a Chevalley generator in  $ \lieg_{\alpha_i} \, $.  It is worth remarking that if
$ \beta^k $  is a simple root, say  $ \, \beta^k = \alpha_j \, $,  then the root vector  $ x_{\beta^k} $
defined above actually coincides with the generator  $ x_j $  given from scratch, so the entire
construction is overall consistent.  The same argument can be used to construct negative root vectors.

\smallskip

   This type of procedure was ``lifted'' to the one-parameter quantum case by Lusztig  (cf.\ \cite{Lu}),
   who did it introducing a suitable braid group action on  $ U_q(\lieg) \, $;  his construction was later
   extended by Heckenberger to the multiparameter case, that is to  $ \QEq $,  as we shall now shortly
   recall.  One defines   --- see  \cite{HY},  formulas (4.3--4) ---
   isomorphisms  $ T_1 $,  $ \dots $,  $ T_\theta $  which yield a  $ \mathbb{B} $--action  that lifts that on
   $ U(\lieg) \, $;  using this action one defines ``quantum root vectors''  $ E_{\beta^k} $  as given by
\begin{equation}  \label{q-root_vects-1}
  \qquad \qquad   E_{\beta^k}  \, := \;  T_{i_1} T_{i_2} \cdots T_{i_{k-1}}\!\big(E_{i_k}\big)  \;\; \in \;\;  U_\bq^+   \qquad \quad   \forall \;\; k = 1, 2, \dots, N  \quad
\end{equation}
where one finds that  $ \, E_{\beta^k} = E_j \, $  whenever  $ \, \beta^k = \alpha_j \, $;  similarly one also
constructs ``(quantum) negative root vectors''  $ F_{\beta^k} \in U_\bq^- \, $.  In the following, we shall refer to
the  $ E_{\beta^k} $'s  or the  $ F_{\beta^k} $'s  by loosely calling them ``(quantum) root vectors''.

\vskip3pt

%
%%%%%
%    {\it It is also quite remarkable that these quantum root vectors can also be realized as iterated
% braided brackets  (e.g., like in  \cite[Section 4]{HY})}.  This will be of key importance, thanks to
% the following result:
%%%%%
%
   {\it It is also remarkable that these quantum root vectors can be realized as iterated braided brackets
   (e.g., like in  \cite[Section 4]{HY})}.  This will be of key importance, by the following:

\medskip

\begin{prop}  \label{prop: proport_root-vects}
 Every quantum root vector in  $ \QEq $  is proportional to the corresponding quantum root vector in
 $ \QEqcheck $  by a coeffi\-cient that is a
%%%%%
% (Laurent)
%%%%%
 monomial in the  $ q_{ij}^{\pm 1/2} $'s.
\end{prop}

\pf
 By  Theorem \ref{thm:sigma_2-cocy}  and  Proposition  \ref{prop:def-sigma=def-c}  together we know that
  $$  \QEq  \,\; \cong \,\;  {\QEqcheck}_\sigma  \; = \;\,  {\QEqcheck}^{(\tilde{\varphi})}  $$
 for the  $ 2 $--cocycle  $ \sigma $  of  $ \QEqcheck $  and a suitable group bicharacter
 $ \varphi $  of  $ Q \, $.  Now denote by  $ E_\alpha $  a quantum root vector in  $ \QEq $
 and by  $ \check{E}_\alpha $  the corresponding (i.e., built in the same way, for the same root)
 quantum root vector in  $ \QEqcheck \, $.  Since  $ \, \check{q}_{ij} = q^{d_i a_{ij}} = q^{d_j a_{ji}} = \check{q}_{ji} \, $
 for all  $ \, i , j \in I \, $,  in  $ {\QEqcheck}_\sigma $  we have that
\begin{align*}
   &  \ad(E_j)(E_i)  \, = \,  \ad_{\sigma}(\check{E}_j)(\check{E}_i)  \; = \;  {(\check{E}_j)}_{(1)} \,\cdot_\sigma \check{E}_i \,\cdot_\sigma \SS_{\sigma}\big({(\check{E}_j)}_{(2)} \big)  \; =  \\
   &  \qquad  = \;  \check{E}_j \,\cdot_\sigma \check{E}_i \, + \, K_j \,\cdot_\sigma \, \check{E}_i \,\cdot_\sigma \big(K^{-1}_j \,\cdot_\sigma \check{E}_j \big)  \; =  \\
   &  \qquad  = \;  \sigma(K_j \,, K_i) \, \check{E}_j \, \check{E}_i \, + \, \big(\, \sigma(K_j \, , K_i) \, K_j \, \check{E}_i \big) \,\cdot_\sigma \big(\, \sigma(K_j^{-1} \, , K_j) \, K^{-1}_j \, \check{E}_j \big)  \; =  \\
   &  \qquad  = \;  \sigma(K_j \,, K_i) \, \big( \check{E}_j \, \check{E}_i \, + \, \sigma\big(K_j^{-1} \, , K_j\big) \, \sigma(K_j\,K_i \, , 1) \, K_j \, \check{E}_i \, K_j^{-1} \check{E}_j \, \sigma^{-1\!}\big( K_j \, , K_j^{-1} \big) \big) \; =  \\
   &  \qquad  = \;  q_{ji}^{1/2} \big(\check{E}_j \, \check{E}_i  \, + \, \check{q}_{ij}\, \check{E}_i \check{E}_j \,\big) \; = \;  q_{ji}^{1/2} \big( (\check{E}_j)_{(1)} \,\cdot\, \check{E}_{i} \,\cdot \, \SS\big( {(\check{E}_j)}_{(2)} \big) \big)  \; = \;  q_{ji}^{1/2} \ad(\check{E}_j) (\check{E}_i)
    \end{align*}
 Therefore, although the adjoint action is not preserved under the  $ 2 $--cocycle  deformation,
 both elements differ only by a coefficient which is a monomial in the  $ q_{ij}^{\pm 1/2} $'s.
 Since both quantum root vectors are defined by an iteration of adjoint actions (because of the very definition of the
 $ T_i $'s)  by  Lemma \ref{lemma-adj}  we infer, taking into account the explicit form of  $ \sigma \, $
 (whose values are monomials in the  $ q_{ij}^{\pm 1/2} $'s),  that the quantum root vectors
 $ E_\alpha $  and  $ \check{E}_\alpha $  associated with any root  $ \alpha $  in  $ \QEq $
 and in  $ {\QEqcheck}_\sigma \, $,  respectively, are linked by an identity
 $ \; E_\alpha \, = \, m^+_\alpha\big( \bq^{\pm 1/2} \big) \, \check{E}_\alpha \; $
 for some monomial  $ \, m^+_\alpha\big( \bq^{\pm 1/2} \big) \, $  in the  $ q_{ij}^{\pm 1/2} $'s,  as claimed.
                                                     \par
   The above accounts for all (quantum)  {\sl positive\/}  root vectors.
   A similar argument proves the claim for  {\sl negative\/}  root vectors as well.
\epf

\medskip

\subsection{Poincar{\'e}-Birkhoff-Witt (=PBW) theorems for MpQG's}  \label{PBW-MpQG} \
 \vskip7pt
   Once we have quantum root vectors, some Poincar{\'e}-Birkhoff-Witt (=PBW) theorems hold too,
   stating that suitable ordered products of quantum root vectors and/or toral generators do form a  $ \k $--basis
   of  $ \QEq $  itself.  Here is the exact claim:

\medskip

\begin{theorem}  \label{thm:PBW_MpQG}
 {\sl (``PBW Theorem'' for  $ \, \QEq $   --- cf.\ \cite[Theorem 3.6]{An4}, \cite[Theorem 4.5]{HY},  and references therein)}
   Assume quantum root vectors in  $ \QEq $  have been defined as above.  Then the set of ordered monomials
  $$  \bigg\{\; {\textstyle \prod\limits_{k=N}^1} F_{\beta^k}^{\,f_k} \,
{\textstyle \prod\limits_{j \in I}} \, L_j^{\,a_j} \, {\textstyle \prod\limits_{i \in I}}
\, K_i^{\,b_i} \, {\textstyle \prod\limits_{h=1}^N} E_{\beta^h}^{\,e_h} \;\bigg|\; f_k, a_j, b_i, e_h \in \NN \;\bigg\}  $$
is a  $ \, \k $--basis  of  $ \, \QEq \, $,  and similarly if we take the opposite order in  $ \, \Phi^+ \, $.
                                                             \par
   Similar results hold for the subalgebras  $ \, U_\bq^\geq \, $,  $ U_\bq^\leq \, $,
   $ U_\bq^+ \, $,  $ U_\bq^- $,  $ \, U_\bq^{\,+,0} \, $,  $ \, U_\bq^{\,-,0} \, $  and  $ \, U_\bq^0 \, $.
\end{theorem}

\pf
 This is proved in  \cite[Theorem 3.6]{An4} (also for  $ \bq $  not of Cartan type).
\epf

\vskip7pt

\begin{rmk}  \label{rmk:q-root_vects_lrg-MpQG's}
 It is easy to see that a suitable ``PBW Theorem'' holds as well for any generalized MpQG with larger torus
 $ U_{\bq\,,\varGamma_\bullet}(\hskip0,8pt\lieg) \, $   --- cf.\ \S \ref{MpQGs-larger-torus}.
\end{rmk}

\medskip

\subsection{Hopf duality among quantum Borel subgroups}  \label{H-duality}  \
 \vskip7pt
   Proposition \ref{sk-H_pair}  provides a skew-Hopf pairing between the two MpQG's of Borel type
   $ \, U_\bq^\geq \, $  and  $ \, U_\bq^\leq \, $,  that we denote by  $ \eta \, $.
   Again, from  \cite[Proposition 4.6]{AY},  we have a complete description of this pairing, in terms
   of PBW bases (of both sides), namely the following:

\smallskip

\begin{prop}  \label{duality_x_PBW-bases}
 Keep notation as above.  Then
  $$  \eta\bigg(\, {\textstyle \prod\limits_{k=1}^M} E_{\beta_k}^{\,e_k} \, K \, ,
  \, {\textstyle \prod\limits_{k=1}^M} F_{\beta_k}^{\,f_k} \, L \,\bigg)  \; = \;
  {\textstyle \prod\limits_{k=1}^M} \delta_{e_k,f_k} \bigg(\! {{\; {(-1)}^{h(\beta_k)} \,
  \subd{q}{\beta^k}{\beta^k} \;} \over
  {\, \subd{q}{\beta^k}{\beta^k} - 1 \,}} \!\bigg)^{\!e_k} {(e_k)}_{q_{\raise-2pt\hbox{$ \scriptscriptstyle \beta^k \beta^k $}}}{\!}!
  \cdot \eta(K,L)  $$
for all  $ \, e_k, f_k \in \NN \, $  and all  $ \, K \in U_\bq^{+,0} \, $,  $ \, L \in U_\bq^{-,0} \, $,
where  $ h(\beta_k) $  is the height of the root  $ \beta_k $  and  $ \subd{q}{\beta^k}{\beta^k} $
is defined as in \S \ref{deform-MpQG}.
\end{prop}

\vskip-9pt

\begin{rmk}  \label{rmk:duality-x-Borel-mpqsbgrp's}
 It is straightforward to see that the result above actually extends to the case when
 --- under suitable assumptions --- one considers the pairing  $ \eta $
 between two multiparameter quantum Borel subgroups
 $ U_{\bq\,,\varGamma_\bullet}^{\,\geq} $  and  $ U_{\bq\,,\varGamma_\bullet}^{\,\leq} $  like in
 \S \ref{duality x larger MpQG's}.
\end{rmk}

\medskip

\subsection{Special products in  $ \, U_\bq(\hskip0,8pt\lieg) = {\big( \QEqcheck \big)}_\sigma \, $}  \label{subsec:spec-prod_def-mult}  \
 \vskip7pt
   When performing calculations in our MpQG's, a convenient strategy is to reduce ourselves to similar
   calculations in the simpler framework of uniparameter quantum groups.
   The basic point to start from is the existence of a Hopf algebra isomorphism
  $$  U_\bq(\hskip0,8pt\lieg)  \; \cong \;  {\big(\, \QEqcheck \big)}_\sigma  $$
 (cf.\  Theorem \ref{thm:sigma_2-cocy})  where  $ \sigma $  is the  $ 2 $--cocycle  given in
 Definition \ref{def-sigma}.  Therefore, we can describe  $ U_\bq(\hskip0,8pt\lieg) $  as being the coalgebra
 $ {\big(\, \QEqcheck \big)}_\sigma $  {\sl endowed with the new, deformed product}
 $ \, \smallast := \cdot_\sigma \, $  (defined as in \S \ref{cotwist-defs_1})  and the corresponding, deformed antipode
 $ \SS_\sigma \, $.  The ``old'' product in  $ \QEqcheck $  instead will be denoted by  $ \, \check{\cdot} \; $.
 So hereafter by  $ Y^{*z} $  or  $ Y^{\,\check{\cdot}\, z} $  we shall denote the  $ z $--th  power of any
 $ \, Y \in \QEqcheck \, $  with respect to either the deformed product  $ \, \smallast \, $  or the old product
 $ \, \check{\cdot} \, $,  respectively, for any exponent  $ \, z \in \NN \, $,  or even  $ \, z \in \ZZ \, $  when  $ Y $
 is invertible.
 \vskip3pt
   For later use, we need to introduce some more notation:

\vskip11pt

\begin{definition}  \label{def:q-bin-coeff & q-div_pows}  {\ }
 \vskip3pt
   {\it (a)}\,  Let  $ \mathcal{A} $  be an algebra over a field  $ \mathbb{F} \, $,  and let
   $ \, p \in \mathbb{F} \, $ be {\sl not\/}  a root of unity.  For every
   $ \, H \in \mathcal{A} \, $,  $ \, n \in \NN \, $  and  $ \, c \in \ZZ \, $,  define the elements
\begin{equation}
  \label{q-bin_coeff}
     {\bigg( {{H\,; c} \atop n} \bigg)}_{\!\!p} := \; \prod_{s=1}^n {{\; p^{\,c+1-s} H - 1 \;} \over {\, p^s - 1 \,}}  \qquad ,
     \qquad \quad  {\bigg( {H\, \atop n} \bigg)}_{\!\!p} := {\bigg( {H\,; 0 \atop n} \bigg)}_{\!\!p}  \qquad
\end{equation}
 that are called  {\it  $ p $--binomial  coefficients (or just  ``$ \, p $--binomials'')  in  $ H $}.
 \vskip3pt
   {\it (b)}\,  For every  $ \, i \in I \, $,  $ \, \alpha \in \Phi^+ \, $,  $ \, X_i \in \big\{ E_i \, , F_i \,\big\} \, $,
   $ \, Y_\alpha \in \big\{ E_\alpha \, , F_\alpha \,\big\} \, $   --- notation as in  \S \ref{rvec-MpQG}  ---
   and all  $ \, n \in \NN \, $,  the elements in  $ \QEq $
\begin{equation}   \label{q-div_pows}
   {\ } \hskip25pt   X_i^{\,(n)} \, := \, {{\,X_i^n\,} \over {\;{(n)}_{q_{ii}\!}!\,}}  \hskip35pt ,
 \hskip43pt  Y_\alpha^{\,(n)} \, := \, {{\,Y_\alpha^n\,} \over {\;{(n)}_{q_{\raise-2pt\hbox{$ \scriptscriptstyle \alpha \alpha $}}\!}!\,}}
\end{equation}
are called  {\it quantum divided powers},  or  {\it  $ q $--divided  powers}.   \hfill  $ \diamondsuit $
\end{definition}

\smallskip

  {\sl Note that\/}  if in  $ \QEqcheck $  we consider the two products  $ \, \cdot \, $  and
  $ \, \smallast \, $  we have  {\sl two\/}  corresponding types of  $ q $--binomial  coefficients,
  hereafter denoted by  $ \, {\Big( {{X ; \, 0} \atop n} \Big)}_{\!\!p}^{\!\check{\cdot}} \, $  and
  $ \, {\Big( {{X ; \, 0} \atop n} \Big)}_{\!\!p}^{\!\!*} \, $.  Similarly, we shall consider two types of  $ q $--divided  powers,
  for which we use notation  $ Y^{\,\check{\cdot}\, (n)} \, $  and  $ Y^{\,*(n)} \, $;  indeed, the first type denotes a
  $ q $--divided  power in  $ \; \big( \QEqcheck \, , \;\check{\cdot}\; \big) \; $,  and the second one a $ q $--divided
  power in
  $ \; U_\bq(\hskip0,8pt\lieg) \, = \, {\big(\, \QEqcheck \big)}_\sigma \, = \, \big(\, \QEqcheck \, , \;\smallast\, \big) \; $.

\vskip13pt

\begin{free text}{\bf Comparison formulas.}  \label{subsubsec:comp-formulas}
 Some elementary calculations lead to explicit formulas linking same type objects in
 $ \; \big( \QEqcheck \, , \;\check{\cdot}\; \big) \; $  and in
 $ \; U_\bq(\hskip0,8pt\lieg) \, = \, \big(\, \QEqcheck \, , \;\smallast\, \big) \; $;
 we shall use them later on when studying integral forms of  $ \QEq \, $.
 \vskip3pt
   Concretely, for all  $ \, i \in \{1,\dots,\theta\} \, $,
   $ \, n, m \in \NN \, $,  $ \, z, z', z''  \in \ZZ \, $  $ \, p_s \in \big\{ q , q_s \big\} \, $,  $ \, X, Y \in \{K,L\} \, $  and
   $ \, G_i^{\pm 1} := K_i^{\pm 1} L_i^{\mp 1} \, $  we have   --- cf.\  \eqref{q-bin_coeff}  for notation ---
  $$  \displaylines{
%%%
   E_i^{\,*(n)} \, = \,  q_i^{+{n \choose 2}} E_i^{\;\check{\cdot}\, (n)}  \qquad ,
 \quad \qquad  F_i^{\,*(n)} \, = \,  q_i^{-{n \choose 2}} F_i^{\;\check{\cdot}\, (n)}  \cr
   K_i^{*z} \, = \, K_i^{\,\check{\cdot}\, z}  \quad ,  \qquad  L_i^{*z} \, = \, L_i^{\,\check{\cdot}\, z}  \quad ,  \qquad  G_i^{\,* z} \, = \, G_i^{\;\check{\cdot}\, z}  \cr
%%%
   {\bigg( {X_i \atop n} \bigg)}_{\!\!p_i}^{\!\!*}  = \,  {\bigg( {X_i \atop n} \bigg)}_{\!\!p_i}^{\!\check{\cdot}}  \quad ,
   \qquad  {\bigg( {G_i^{\pm 1} \atop n} \bigg)}_{\!\!q_{ii}}^{\!\!*} \! = \,  {\bigg( {G_i^{\pm 1} \atop n} \bigg)}_{\!\!q_{ii}}^{\!\check{\cdot}}  \cr
%%%
   E_i^{\,*(n)} \smallast E_j^{\,*(m)}  =  q_i^{+{n \choose 2}} q_{n \alpha_i , \, m \alpha_j}^{\, +1/2} q_j^{+{m \choose 2}}
   E_i^{\;\check{\cdot}\, (n)} \,\check{\cdot}\, E_j^{\;\check{\cdot}\,(m)}  =  q_i^{+{n \choose 2}} {\big( q_{ij}^{\, +1/2} \big)}^{n{}m} q_j^{+{m \choose 2}}
   E_i^{\;\check{\cdot}\, (n)} \,\check{\cdot}\, E_j^{\;\check{\cdot}\, (m)}  \cr
   E_{i_{1}}^{\,*(n_{1})} \smallast E_{i_{2}}^{\,*(n_{2})} \smallast \cdots \smallast E_{i_{s}}^{\,*(n_{s})} =
  \bigg(\, {\textstyle \prod\limits_{j=1}^s} q_{i_{j}}^{+{n_j \choose 2}} \bigg)
\bigg(\, {\textstyle \prod\limits_{j<k}} q_{n_{i_j} \alpha_{i_j} , \, n_k \alpha_{i_k}}^{+1/2} \!\bigg)
E_{i_1}^{\,\check{\cdot}\,(n_1)} \,\check{\cdot}\, E_{i_2}^{\,\check{\cdot}\,(n_2)} \,\check{\cdot} \cdots
\,\check{\cdot}\, E_{i_s}^{\,\check{\cdot}\, (n_s)} \cr
%%%
   F_i^{\,*(n)} \smallast F_j^{\,*(m)}  = \,  q_i^{-{n \choose 2}} q_{n \alpha_i , \, m \alpha_j}^{\, -1/2} q_j^{-{m \choose 2}}
   F_i^{\;\check{\cdot}\, (n)} \,\check{\cdot}\, F_j^{\;\check{\cdot}\, (m)}  = \,  q_i^{-{n \choose 2}} {\big( q_{ij}^{\, -1/2} \big)}^{n{}m} q_i^{-{m \choose 2}}
   F_i^{\;\check{\cdot}\, (n)} \,\check{\cdot}\, F_j^{\;\check{\cdot}\, (m)}  \cr
   F_{i_{1}}^{\,*(n_{1})} \smallast F_{i_{2}}^{\,*(n_{2})} \smallast \cdots \smallast F_{i_{s}}^{\,*(n_{s})} =
  \bigg(\, {\textstyle \prod\limits_{j=1}^s} q_{i_{j}}^{-{n_{j} \choose 2}} \bigg)
\bigg(\, {\textstyle \prod\limits_{j<k}} q_{n_{i_j} \alpha_{i_j} , \, n_{i_k}
\alpha_{i_k}}^{-1/2} \!\bigg) F_{i_1}^{\,\check{\cdot}\,(n_1)} \,\check{\cdot}\, F_{i_2}^{\,\check{\cdot}\,(n_2)} \,
\check{\cdot} \cdots \,\check{\cdot}\, F_{i_s}^{\,\check{\cdot}\,(n_s)}  \cr
%%%
   X_i^{*z'} \smallast\, Y_j^{*z''}  = \,  X_i^{\;\check{\cdot}\, z'} \;\check{\cdot}\; Y_j^{\;\check{\cdot}\, z''}  \; ,
 \quad  X_i^{*z'} \smallast\, G_j^{*z''}  = \,  X_i^{\;\check{\cdot}\, z'} \;\check{\cdot}\; Y_j^{\;\check{\cdot}\, z''}  \; ,
 \quad  G_i^{*z'} \smallast\, Y_j^{*z''}  = \,  G_i^{\;\check{\cdot}\, z'} \;\check{\cdot}\; Y_j^{\;\check{\cdot}\, z''}  \cr
%%%
   {\bigg( {X_i \atop n} \bigg)}_{\!\!p_i}^{\!\!*} \,\smallast\, {\bigg( {Y_j \atop m} \bigg)}_{\!\!p_j}^{\!\!*}  \, =
   \,  {\bigg( {X_i \atop n} \bigg)}_{\!\!p_i}^{\!\check{\cdot}} \;\check{\cdot}\; {\bigg( {Y_j \atop m} \bigg)}_{\!\!p_j}^{\!\check{\cdot}}  \cr
   {\bigg( {G_i^{\pm 1} \atop n} \bigg)}_{\!\!q_{i{}i}}^{\!\!*} \!\smallast\, {\bigg( {G_j^{\pm 1} \atop m} \bigg)}_{\!\!q_{j{}j}}^{\!\!*}
   = \;  {\bigg( {G_i^{\pm 1} \atop n} \bigg)}_{\!\!q_{i{}i}}^{\!\check{\cdot}} \check{\cdot}\; {\bigg( {G_j^{\pm 1} \atop m} \bigg)}_{\!\!q_{j{}j}}^{\!\check{\cdot}}  \cr
   {\bigg( {X_i \atop n} \bigg)}_{\!\!p_i}^{\!\!*} \,\smallast\, {\bigg( {G_j^{\pm 1} \atop m} \bigg)}_{\!\!q_{j{}j}}^{\!\!*}
   = \;  {\bigg( {X_i \atop n} \bigg)}_{\!\!p_i}^{\!\check{\cdot}} \;\check{\cdot}\; {\bigg( {G_j^{\pm 1} \atop m} \bigg)}_{\!\!q_{j{}j}}^{\!\check{\cdot}}  \cr
%%%
   E_i^{\,*(n)} \smallast\, F_j^{\,*(m)}  = \,  q_i^{+{n \choose 2}} q_j^{-{m \choose 2}} \, E_i^{\;\check{\cdot}\, (n)} \;\check{\cdot}\; F_j^{\;\check{\cdot}\, (m)}  \, ,
 \!\!\quad  F_j^{\,*(m)} \smallast\, E_i^{\,*(n)}  = \,  q_j^{-{m \choose 2}} q_i^{+{n \choose 2}} \, F_j^{\;\check{\cdot}\, (m)} \;\check{\cdot}\; E_i^{\;\check{\cdot}\, (n)}  \cr
   X_i^{\,*z} \,\smallast\, E_j^{\,*(n)}  \, = \;  q_{i{}j}^{\,+z\,n/2} \, X_i^{\;\check{\cdot}\, z} \;\check{\cdot}\; E_j^{\;\check{\cdot}\, (n)}  \quad ,
 \qquad  E_j^{\,*(n)} \,\smallast\, X_i^{\,*z}  \, = \;  q_{j{}i}^{\,+z\,n/2} \, E_j^{\;\check{\cdot}\, (n)} \;\check{\cdot}\; X_i^{\;\check{\cdot}\, z}  \cr
   G_i^{\,*z} \,\smallast\, E_j^{\,*(n)}  \, = \;  G_i^{\;\check{\cdot}\, z} \,\;\check{\cdot}\; E_j^{\;\check{\cdot}\, (n)}  \quad ,
 \qquad  E_j^{\,*(n)} \,\smallast\, G_i^{\,*z}  \, = \;  E_j^{\;\check{\cdot}\, (n)} \;\check{\cdot}\; G_i^{\;\check{\cdot}\, z}  \cr
   X_i^{\,*z} \,\smallast\, F_j^{\,*(m)}  \, = \;  q_{i{}j}^{\,-z\,m/2} \, X_i^{\;\check{\cdot}\, z} \;\check{\cdot}\; F_j^{\;\check{\cdot}\, (m)}  \quad ,
 \qquad  F_j^{\,*(m)} \,\smallast\, X_i^{\,*z}  \, = \;  q_{j{}i}^{\,-z\,m/2} \, F_j^{\;\check{\cdot}\, (m)} \;\check{\cdot}\; X_i^{\;\check{\cdot}\, z}  \cr
   G_i^{\,*z} \,\smallast\, F_j^{\,*(m)}  \, = \;  G_i^{\;\check{\cdot}\, z} \,\;\check{\cdot}\; F_j^{\;\check{\cdot}\, (m)}  \quad ,
 \qquad  F_j^{\,*(m)} \,\smallast\, G_i^{\,*z}  \, = \;  F_j^{\;\check{\cdot}\, (m)} \;\check{\cdot}\; G_i^{\;\check{\cdot}\, z}  \cr
%%%
   E_i^{\,*(n)} \smallast {\bigg( {X_j \atop m} \bigg)}_{\!\!p_j}^{\!\!*}  \, = \;  q_i^{+{n \choose 2}} \sum_{c=0}^m \, p_j^{-c(m-c)} \,
   \prod_{s=1}^c {{\, p_j^{1-s} \, {\big( q_{i{}j}^{+1/2} \big)}^n \! - 1 \,} \over {\, p_j^{\,s} - 1 \,}} \, E_i^{\;\check{\cdot}\, (n)} {\bigg( {X_j
   \atop {m\!-\!c}} \bigg)}_{\!\!p_j}^{\!\check{\cdot}} X_j^{\;\check{\cdot}\, c}  \cr
%
%%%%%%%
% \cr
%%%
 }  $$
  $$  \displaylines{
%%%%%%%
%%%
   {\bigg( {X_j \atop m} \bigg)}_{\!\!p_j}^{\!\!*} \! \smallast \, E_i^{\,*(n)}  \, = \;  q_i^{+{n \choose 2}} \sum_{c=0}^m \, p_j^{-c(m-c)} \,
   \prod_{s=1}^c {{\, p_j^{1-s} \, {\big( q_{j{}i}^{+1/2} \big)}^n \! - 1 \,} \over {\, p_j^{\,s} - 1 \,}} \, E_i^{\;\check{\cdot}\, (n)}
   {\bigg( {X_j \atop {m\!-\!c}} \bigg)}_{\!\!p_j}^{\!\check{\cdot}} X_j^{\;\check{\cdot}\, c}  \cr
%%%
%%%%%%%%%%%%%%%%
%    E_i^{\,*(n)} \smallast {\bigg( {G_j \atop m} \bigg)}_{\!\!q_{j{}j}}^{\!\!*}  \, = \;\,
% q_i^{+{n \choose 2}} \, E_i^{\;\check{\cdot}\, (n)} \,
%    \check{\cdot}\, {\bigg( {G_j \atop m} \bigg)}_{\!\!q_{j{}j}}^{\!\check{\cdot}}  \cr
% %
%    {\bigg( {G_j \atop m} \bigg)}_{\!\!q_{j{}j}}^{\!\!*} \smallast E_i^{\,*(n)}  \, = \;\,
% q_i^{+{n \choose 2}} \, {\bigg( {G_j \atop m} \bigg)}_{\!\!q_{j{}j}}^{\!\check{\cdot}}
% \!\check{\cdot}\, E_i^{\;\check{\cdot}\, (n)}  \cr
%%%%%%%%%%%%%%%%
%%%
   E_i^{\,*(n)} \smallast {\bigg( {G_j \atop m} \bigg)}_{\!\!q_{j{}j}}^{\!\!*}  \! = \;  q_i^{+{n \choose 2}} \, E_i^{\;\check{\cdot}\, (n)} \,
   \check{\cdot}\, {\bigg( {G_j \atop m} \bigg)}_{\!\!q_{j{}j}}^{\!\check{\cdot}}  \;\; ,  \quad
   {\bigg( {G_j \atop m} \bigg)}_{\!\!q_{j{}j}}^{\!\!*} \smallast E_i^{\,*(n)}  \! = \;  q_i^{+{n \choose 2}} \,
   {\bigg( {G_j \atop m} \bigg)}_{\!\!q_{j{}j}}^{\!\check{\cdot}} \!\check{\cdot}\, E_i^{\;\check{\cdot}\, (n)}  \cr
%%%
   F_i^{\,*(n)} \smallast {\bigg( {X_j \atop m} \bigg)}_{\!\!p_j}^{\!\!*}  \, = \;  q_i^{-{n \choose 2}} \sum_{c=0}^m \, p_j^{-c(m-c)} \,
   \prod_{s=1}^c {{\, p_j^{1-s} \, {\big( q_{i{}j}^{-1/2} \big)}^n \! - 1 \,} \over {\, p_j^{\,s} - 1 \,}} \, F_i^{\;\check{\cdot}\, (n)}
   {\bigg( {X_j \atop {m\!-\!c}} \bigg)}_{\!\!p_j}^{\!\check{\cdot}} X_j^{\;\check{\cdot}\, c}  \cr
%%%
   {\bigg( {X_j \atop m} \bigg)}_{\!\!p_j}^{\!\!*} \! \smallast \, F_i^{\,*(n)}  \, = \;  q_i^{-{n \choose 2}} \sum_{c=0}^m \, p_j^{-c(m-c)} \,
   \prod_{s=1}^c {{\, p_j^{1-s} \, {\big( q_{j{}i}^{-1/2} \big)}^n \! - 1 \,} \over {\, p_j^{\,s} - 1 \,}} \,
   {\bigg( {X_j \atop {m\!-\!c}} \bigg)}_{\!\!p_j}^{\!\check{\cdot}} X_j^{\;\check{\cdot}\, c} \, F_i^{\;\check{\cdot}\, (n)} \cr
%%%
%%%%%%%%%%%%%%%%%%
%    F_i^{\,*(n)} \smallast {\bigg( {G_j \atop m} \bigg)}_{\!\!q_{j{}j}}^{\!\!*}  \, = \;\,
% q_i^{-{n \choose 2}} \, F_i^{\;\check{\cdot}\, (n)} \,
%    \check{\cdot}\, {\bigg( {G_j \atop m} \bigg)}_{\!\!q_{j{}j}}^{\!\check{\cdot}}  \cr
%    {\bigg( {G_j \atop m} \bigg)}_{\!\!q_{j{}j}}^{\!\!*} \!\smallast\, F_i^{\,*(n)}  \, = \;\,
% q_i^{-{n \choose 2}} \, {\bigg( {G_j \atop m} \bigg)}_{\!\!q_{j{}j}}^{\!\check{\cdot}}
% \!\check{\cdot}\, F_i^{\;\check{\cdot}\, (n)}
%%%%%%%%%%%%%%%%%%
%%%
   F_i^{\,*(n)} \smallast {\bigg( {G_j \atop m} \bigg)}_{\!\!q_{j{}j}}^{\!\!*}  \! = \;  q_i^{-{n \choose 2}} \,
   F_i^{\;\check{\cdot}\, (n)} \, \check{\cdot}\, {\bigg( {G_j \atop m} \bigg)}_{\!\!q_{j{}j}}^{\!\check{\cdot}}  \;\; , \quad
   {\bigg( {G_j \atop m} \bigg)}_{\!\!q_{j{}j}}^{\!\!*} \!\smallast\, F_i^{\,*(n)}  \! = \;  q_i^{-{n \choose 2}} \,
   {\bigg( {G_j \atop m} \bigg)}_{\!\!q_{j{}j}}^{\!\check{\cdot}} \!\check{\cdot}\, F_i^{\;\check{\cdot}\, (n)}  }  $$

\vskip5pt

   In addition, more in general for root vectors we have the following.
   From  Proposition \ref{prop: proport_root-vects}  and its proof, recall that
   --- keeping notation from there ---   if we denote by  $ E_\alpha $  a quantum root vector in  $ \QEq $
   and by  $ \check{E}_\alpha $  the corresponding (i.e., for the same root) quantum root vector in
   $ \QEqcheck $  we have
  $$  E_\alpha \, = \, m^+_\alpha \, \check{E}_\alpha  \quad\qquad \text{and} \qquad\quad
 F_\alpha \, = \, m^-_\alpha \, \check{F}_\alpha  $$
for some Laurent monomials  $ \, m^+_\alpha = m^+_\alpha\big( \bq^{\pm 1/2} \big) \, $  and
$ \, m^-_\alpha = m^-_\alpha\big( \bq^{\pm 1/2} \big) \, $  in the  $ q_{ij}^{\pm 1/2} $'s.
Then an analysis like above (just a bit finer), for $ \, \alpha , \beta,
\in \Phi^+ \, $,  $ \, j \in \{1,\dots,\theta\} \, $,  yields
  $$  \displaylines{
   E_\alpha^{\,*(n)} \smallast E_\beta^{\,*(m)}  \, = \;
   q_\alpha^{+{n \choose 2}} {\big( q_{\alpha{}\beta}^{\, +1/2} \,\big)}^{n{}m} q_\beta^{+{m \choose 2}}
   {\big( m^+_\alpha \big)}^n \, {\big( m^+_\beta \big)}^m \, \check{E}_\alpha^{\;\check{\cdot}\, (n)}
   \;\check{\cdot}\; \check{E}_\beta^{\;\check{\cdot}\, (m)}  \cr
%%%%
   E_{\alpha_1}^{\,*(n_1)} \smallast \cdots \smallast E_{\alpha_s}^{\,*(n_s)}  \, =
 \bigg(\, {\textstyle \prod\limits_{j=1}^s} q_{\alpha_j}^{+{n_j \choose 2}} \!\bigg)
   \bigg(\, {\textstyle \prod\limits_{j<k}} q_{n_j \alpha_j , \, n_k \alpha_k}^{+1/2} \!\bigg)
   \bigg(\, {\textstyle \prod\limits_{j=1}^s} \big( m^+_{\alpha_j} \big)^{n_j} \!\bigg)
 \check{E}_{\alpha_1}^{\;\check{\cdot}\, (n_1)}
  \,\check{\cdot}\, \cdots \,\check{\cdot}\,
   \check{E}_{\alpha_s}^{\;\check{\cdot}\, (n_s)}  \cr
%%%%
   F_\alpha^{\,*(n)} \smallast F_\beta^{\,*(m)}  \, =
   \;  q_\alpha^{-{n \choose 2}} {\big( q_{\alpha{}\beta}^{\, -1/2} \big)}^{n{}m} q_\beta^{-{m \choose 2}}
   {\big( m^-_\alpha \big)}^n \, {\big( m^-_\beta \big)}^m \, \check{F}_\alpha^{\;\check{\cdot}\, (n)} \;\check{\cdot}\;
   \check{F}_\beta^{\;\check{\cdot}\, (m)}  \cr
   F_{\alpha_1}^{\,*(n_1)} \smallast \cdots \smallast F_{\alpha_s}^{\,*(n_s)}  \, =
 \bigg(\, {\textstyle \prod\limits_{j=1}^s} q_{\alpha_j}^{-{n_j \choose 2}} \!\bigg)
   \bigg(\, {\textstyle \prod\limits_{j<k}} q_{n_j \alpha_j , \, n_k \alpha_k}^{-1/2} \!\bigg)
   \bigg(\, {\textstyle \prod\limits_{j=1}^s} \big( m^-_{\alpha_j} \big)^{n_j} \!\bigg)
 \check{F}_{\alpha_1}^{\;\check{\cdot}\, (n_1)}
  \,\check{\cdot}\, \cdots \,\check{\cdot}\,
   \check{F}_{\alpha_s}^{\;\check{\cdot}\, (n_s)}  \cr
%%%
   E_\alpha^{\,*(n)} \smallast F_\beta^{\,*(m)}  = \,  q_\alpha^{+{n \choose 2}} q_\beta^{-{m \choose 2}} \, {\big( m^+_\alpha \big)}^n \,
   {\big( m^-_\beta \big)}^m \, \check{E}_\alpha^{\;\check{\cdot}\, (n)} \;\check{\cdot}\; \check{F}_\beta^{\;\check{\cdot}\, (m)}  \cr
   F_\beta^{\,*(m)} \smallast E_\alpha^{\,*(n)}  = \,  q_\beta^{-{m \choose 2}} q_\alpha^{+{n \choose 2}} \, {\big( m^-_\beta \big)}^m \,
   {\big( m^+_\beta \big)}^n \, \check{F}_\beta^{\;\check{\cdot}\, (m)} \;\check{\cdot}\; \check{E}_\alpha^{\;\check{\cdot}\, (n)}  \cr
%%%
   X_i^{\,*z} \smallast E_\beta^{\,*(n)}  = \,  q_{\alpha_i\,\beta}^{\,+z\,n/2} \, {\big( m_\beta^+ \big)}^n \, X_i^{\;\check{\cdot}\, z} \;\check{\cdot}\;
   \check{E}_\beta^{\;\check{\cdot}\, (n)}  \; ,
 \!\!\quad  E_\beta^{\,*(n)} \smallast X_i^{\,*z}  = \,  q_{\beta\,\alpha_i}^{\,+z\,n/2} \, {\big( m_\beta^+ \big)}^n \, \check{E}_\beta^{\;\check{\cdot}\, (n)}
 \;\check{\cdot}\; X_i^{\;\check{\cdot}\, z}  \cr
   G_i^{\,*z} \,\smallast\, E_\beta^{\,*(n)}  \, = \;  {\big( m_\beta^+ \big)}^n \, G_i^{\;\check{\cdot}\, z} \;\check{\cdot}\; \check{E}_\beta^{\;\check{\cdot}\, (n)}  \quad ,
 \qquad  E_\beta^{\,*(n)} \,\smallast\, G_i^{\,*z}  \, = \;  {\big( m_\beta^+ \big)}^n \, \check{E}_\beta^{\;\check{\cdot}\, (n)} \;\check{\cdot}\; G_i^{\;\check{\cdot}\, z}  \cr
%%%
   X_i^{\,*z} \smallast F_\beta^{\,*(m)}  = \,  q_{\alpha_i\,\beta}^{\,-z\,m/2} {\big( m_\beta^- \big)}^m X_i^{\;\check{\cdot}\, z} \;\check{\cdot}\;
   \check{F}_\beta^{\;\check{\cdot}\, (m)}  \, ,
 \!\!\!\quad  F_\beta^{\,*(m)} \smallast X_i^{\,*z}  = \,  q_{\beta\,\alpha_i}^{\,+z\,m/2} {\big( m_\beta^- \big)}^m \check{F}_\beta^{\;\check{\cdot}\, (m)}
 \;\check{\cdot}\; X_i^{\;\check{\cdot}\, z}  \cr
   G_i^{\,*z} \,\smallast\, F_\beta^{\,*(m)}  \, = \;  {\big( m_\beta^- \big)}^m \, G_i^{\;\check{\cdot}\, z} \;\check{\cdot}\; \check{F}_\beta^{\;\check{\cdot}\, (m)}  \quad ,
 \qquad  F_\beta^{\,*(m)} \,\smallast\, G_i^{\,*z}  \, = \;  {\big( m_\beta^- \big)}^m \, \check{F}_\beta^{\;\check{\cdot}\, (m)} \;\check{\cdot}\; G_i^{\;\check{\cdot}\, z}  \cr
%%%
   E_\alpha^{\,*(n)} \smallast {\bigg(\! {X_j \atop m} \bigg)}_{\!\!p_j}^{\!\!*}  \! =  q_\alpha^{+{n \choose 2}} {\big( m^+_\alpha \big)}^{\!n}
   \sum_{c=0}^m p_j^{-c(m-c)} \prod_{s=1}^c {{\, p_j^{1-s} {\big( q_{\alpha\,\alpha_j}^{+1/2} \big)}^{\!n} \! - 1 \,} \over {\, p_j^{\,s} - 1 \,}} \,
   \check{E}_\alpha^{\;\check{\cdot}\, (n)} {\bigg(\! {X_j \atop {m\!-\!c}} \bigg)}_{\!\!p_j}^{\!\check{\cdot}} \! X_j^{\;\check{\cdot}\, c}  \cr
%%%
   {\bigg(\! {X_j \atop m} \bigg)}_{\!\!p_j}^{\!\!*} \! \smallast E_\alpha^{\,*(n)}  = \,  q_\alpha^{+{n \choose 2}} {\big( m^+_\alpha \big)}^{\!n}
   \sum_{c=0}^m p_j^{-c(m-c)} \prod_{s=1}^c {{\, p_j^{1-s} \, {\big( q_{\alpha_j\,\alpha}^{+1/2} \big)}^{\!n} \! - 1 \,} \over {\, p_j^{\,s} - 1 \,}}
   {\bigg(\! {X_j \atop {m\!-\!c}} \bigg)}_{\!\!p_j}^{\!\check{\cdot}} \! X_j^{\;\check{\cdot}\, c} \check{E}_\alpha^{\;\check{\cdot}\, (n)}  \cr
%%%
   E_\alpha^{\,*(n)} \;\smallast\; {\bigg( {G_j \atop m} \bigg)}_{\!\!q_{j{}j}}^{\!\!*}  \,\; = \;\; q_\alpha^{+{n \choose 2}} {\big( m_\alpha^+ \big)}^n
   \, E_\alpha^{\;\check{\cdot}\, (n)} \;\check{\cdot}\; {\bigg( {G_j \atop m} \bigg)}_{\!\!q_{j{}j}}^{\!\check{\cdot}}  \cr
   {\bigg( {G_j \atop m} \bigg)}_{\!\!q_{j{}j}}^{\!\!*} \,\smallast\; E_\alpha^{\,*(n)}  \,\; = \;\; q_\alpha^{+{n \choose 2}}   {\big( m_\alpha^+ \big)}^n
   \, {\bigg( {G_j \atop m} \bigg)}_{\!\!q_{j{}j}}^{\!\check{\cdot}} \,\check{\cdot}\; E_\alpha^{\;\check{\cdot}\, (n)}
%%%%%%%
% \cr
%%%%%%%
 }  $$
  $$  \displaylines{
%%%%%%%
%%%
   F_\alpha^{\,*(n)} \smallast {\bigg(\! {X_j \atop m} \bigg)}_{\!\!p_j}^{\!\!*}  \! = \,  q_\alpha^{-{n \choose 2}} {\big( m^-_\alpha \big)}^{\!n}
   \sum_{c=0}^m p_j^{-c(m-c)} \prod_{s=1}^c {{\, p_j^{1-s} \, {\big( q_{\alpha\,\alpha_j}^{-1/2} \big)}^{\!n} \! - 1 \,} \over {\, p_j^{\,s} - 1 \,}}
   \check{F}_\alpha^{\;\check{\cdot}\, (n)} {\bigg(\! {X_j \atop {m\!-\!c}} \bigg)}_{\!\!p_j}^{\!\check{\cdot}} X_j^{\;\check{\cdot}\, c}  \cr
   {\bigg(\! {X_j \atop m} \bigg)}_{\!\!p_j}^{\!\!*} \! \smallast F_\alpha^{\,*(n)}  = \,  q_\alpha^{-{n \choose 2}} {\big( m^-_\alpha \big)}^{\!n}
   \sum_{c=0}^m p_j^{-c(m-c)} \prod_{s=1}^c \! {{\, p_j^{1-s} {\big( q_{\alpha_j\,\alpha}^{-1/2} \big)}^{\!n} \! - 1 \,} \over {\, p_j^{\,s} - 1 \,}} \,
   {\bigg(\! {X_j \atop {m\!-\!c}} \bigg)}_{\!\!p_j}^{\!\check{\cdot}} X_j^{\;\check{\cdot}\, c} \, \check{F}_\alpha^{\;\check{\cdot}\, (n)}  \cr
%%%
   F_\alpha^{\,*(n)} \;\smallast\; {\bigg( {G_j \atop m} \bigg)}_{\!\!q_{j{}j}}^{\!\!*}  \,\; = \;\; q_\alpha^{-{n \choose 2}}  {\big( m_\alpha^- \big)}^n
   \, F_\alpha^{\;\check{\cdot}\, (n)} \;\check{\cdot}\; {\bigg( {G_j \atop m} \bigg)}_{\!\!q_{j{}j}}^{\!\check{\cdot}}  \cr
   {\bigg( {G_j \atop m} \bigg)}_{\!\!q_{j{}j}}^{\!\!*} \,\smallast\; F_\alpha^{\,*(n)}  \,\; = \;\; q_\alpha^{-{n \choose 2}}   {\big( m_\alpha^- \big)}^n
   \, {\bigg( {G_j \atop m} \bigg)}_{\!\!q_{j{}j}}^{\!\check{\cdot}} \,\check{\cdot}\; F_\alpha^{\;\check{\cdot}\, (n)}  }  $$
\end{free text}

\bigskip

\section{Integral forms of MpQG's}  \label{int-forms_mpqgs}

\smallskip

   The main purpose of the present section is to introduce integral forms
 of our MpQG's; in particular, we shall also provide suitable PBW-like theorems for them.

\vskip15pt

\subsection{Preliminaries on integral forms}  \label{prel_int-forms}  \
 \vskip9pt
   In this subsection we fix the ground for our discussion of integral forms of MpQG's.

\vskip11pt

\begin{free text}{\bf Integral forms.}
 Let  $ S $  be any ring, and  $ M $  any  $ S $--module.  If  $ R $  is any subring of  $ S \, $,  we call
 {\sl  $ R $--integral  form\/}  (or ``integral form over  $ R \, $'')  of  $ M $
 any  $ R $--submodule  $ M_R $  of  $ M $  whose scalar extension from  $ R $  to  $ S $  is  $ M \, $,
 i.e.\  $ \,  M_R  \otimes_R \!  S= M \, $.  When  $ M $  has some richer structure (than the  $ S $--module  one)
 by  ``$ R $--integral  form'' we mean an  $ R $--integral  form that in addition respects the additional structure; in other words,
 the definition is like above but one has to replace the words ``module'' and ``submodule'' with the words referring to the
 additional, richer structure.  For instance, if  $ H $  is a  {\sl Hopf algebra\/}  over  $ S $
 by  ``$ R $--integral  form'' of it we mean any  {\sl Hopf subalgebra}  $ H_R $  over  $ R $
 such that  $ \, S \otimes_R \! H_R = H \, $.
\end{free text}

\vskip15pt

\begin{free text}{\bf The ground ring.}  \label{ground-ring}
 The integral forms of our MpQG's will be defined over a suitable ground ring.  To define it, we
 begin fixing a multiparameter matrix  $ \bq $  of Cartan type with entries in the field  $ \k \, $,
 assuming again that the Cartan matrix is indecomposable.  Starting from  $ \bq \, $,  we fix in  $ \k $
 an element  $ \, q_{j_{\raise-2pt\hbox{$ \scriptscriptstyle 0 $}}} \in \k^\times \, $,  now denoted by
 $ \, q := q_{j_{\raise-2pt\hbox{$ \scriptscriptstyle 0 $}}} \, $,  like in  \S \ref{multiparameters},  and
 square roots  $ q_{ij}^{\;1/2} $  of all the  $ q_{ij} $'s,  like in  \S \ref{multiparameters}.
 \vskip4pt
   We denote by  $ \Fbq \, $  the subfield of  $ \k $  generated by all the  $ q_{ij}^{\pm 1} $'s  ($ \, i, j \in I \, $)
   along with  $ q^{\pm 1} \, $;  moreover, we denote by  $ \Fbqsq $  the subfield of  $ \k $  generated by all the
   $ q_{ij}^{\pm 1/2} $'s  ($ \, i, j \in I \, $)  and  $ q^{\pm 1/2} \, $:  then  $ \Fbqsq $  is a field extension of  $ \Fbq \, $,
   that contains also all the square roots  $ q_i^{\pm 1/2} $'s  and  $ q_\beta^{\pm 1/2} $'s  ($ \, \beta \in \Phi^+ \, $),
   for all the  $ q_i $'s  and  $ q_\beta $'s  defined at the beginning of  \S \ref{deform-MpQG}.
   As ground ring for our integral forms, we fix the subring  $ \, \Rbq \, $  of  $ \k $  generated by all the
   $ q_{ij}^{\pm 1} $'s (for all  $ \, i, j \in I \, $)  and  $ q^{\pm 1} \, $;  moreover, we denote by  $ \Rbqsq $
   the subring of  $ \k $  generated by all the  $ q_{ij}^{\pm 1/2} $'s  ($ \, i, j \in I \, $)  and  $ q^{\pm 1/2} \, $:
   this is a ring extension of  $ \Rbq \, $,  that contains all the square roots  $ q_i^{\pm 1/2} $'s  and
   $ q_\beta^{\pm 1/2} $'s  ($ \, \beta \in \Phi^+ \, $).  The field of fractions of  $ \Rbq $  is just  $ \Fbq \, $,
   and similarly that of  $ \Rbqsq $  is just  $ \Fbqsq \, $.
%%%%%
%%%
                                                                \par
   When  $ \bq $  is of integral type we have that  $ \, \Rbq $  and  $ \Fbq $  are generated (as a ring and as a field,
   respectively) by  $ q^{\pm 1} $  alone, while  $ \Rbqsq $  and  $ \Fbqsq $  are generated by  $ q^{\pm 1/2} \, $.
                                                                \par
   Finally, if we consider MpQG's with larger tori, then we take a ground field  $ \F_{\bq_c} $
   and a ground ring $ \R_{\bq_c} $  defined like  $ \Fbq $  and  $ \Rbq $  but replacing the  $ q_{ij}^{-1} $'s
   with the $ q_{ij}^{\pm 1/c} $'s  and  $ q^{\pm 1} $  by  $ q^{\pm 1/c} \, $,  with
   $ \, c_\pm := \big| \text{\sl det}(C_\pm) \big| \, $  and  $ \; c := c_+ \, c_- \; $
   (cf.\ \S \ref{larger-MpQG's},  \S \ref{duality x larger MpQG's}).
\end{free text}

\medskip

\subsection{Integral forms of ``restricted'' type}  \label{Uhat}  \
 \vskip7pt

   Following  \S \ref{mpqgroups},  we consider the multiparameter quantum group  $ \QEq \, $
   associated with  $ \bq \, $,  defined over  $ \k \, $;  also, for the special value of  $ \, q \in \k \, $
   fixed above (depending on  $ \bq \, $),  we pick the MpQG of ``canonical type''  $ \QEqcheck $  as in
   Remark \ref{link_QEq-QE & symm-case}.  Moreover, for each  $ \, \beta \in \Phi^+ \, $
   we consider quantum root vectors  $ E_\beta $  and  $ F_\beta $   --- within  $ \QEq $
%%%%%%%
   \hbox{and within  $ \QEqcheck \, $  ---   as in  \S \ref{rvec-MpQG}.}
%%%%%%%
%%%%%%%
                                                                \par
   Lusztig's quantum groups of ``restricted type'' were introduced  (cf.\ \cite{Lu})
   as special integral forms of his uniparameter quantum group   --- which is ``almost''
   $ \, \QEqcheck \, $  ---   defined in terms of the so-called  ``$ q $--binomial  coefficients'' and
   ``$ q $--divided powers''.  We shall now perform a similar construction in the multiparameter case.

\medskip

\begin{free text}{\bf  $ q $--binomial  coefficients and their arithmetic.}  \label{q-bin_coeff_&_arithmetic}
 Let  $ p $  be any formal indeterminate,  $ \, m \in \NN \, $  and  $ \, J := \{1,2,\dots,m\} \, $.
 We consider the two algebras
  $$  \bE_m  \; := \;  \QQ(p)\big[{\big\{X_i^{\pm 1}\big\}}_{i \in J}\big]  \quad ,  \qquad
     \E_m  \; := \;  \QQ(p)\big[{\big\{\chi_i^{\pm 1}\big\}}_{i \in J}\big]  $$
of Laurent polynomials in the set of indeterminates  $ {\big\{X_i^{\pm 1}\big\}}_{i \in J} $  and
$ {\big\{\chi_i^{\pm 1}\big\}}_{i \in J} $  respectively on the field  $ \QQ(p) $  of rational functions in
$ p $ with coefficients in  $ \QQ \, $.  Both these bear unique Hopf algebra structures   --- over  $ \QQ(p) $
---   for which the  $ X_i^{\pm 1} $'s and the  $ \chi_i^{\pm 1} $'s  are group-like,
i.e.\  $ \, \Delta\big(X_i^{\pm 1}\big) = X_i^{\pm 1} \otimes X_i^{\pm 1} \, $,  $ \, \epsilon\big(X_i^{\pm 1}\big) = 1 \, $,
$ \, S\big(X_i^{\pm 1}\big) = X_i^{\mp 1} = 1 \, $  for  $ \bE_m $  and
$ \, \Delta\big(\chi_i^{\pm 1}\big) = \chi_i^{\pm 1} \otimes \chi_i^{\pm 1} \, $,  $ \, \epsilon\big(\chi_i^{\pm 1}\big) = 1 \, $,
$ \, S\big(\chi_i^{\pm 1}\big) = \chi_i^{\mp 1} = 1 \, $  for  $ \E_m \, $.
\end{free text}

\smallskip

   The following result lists some properties of  $ p $--binomial  coefficients
   (cf.\  Definition \ref{def:q-bin-coeff & q-div_pows}{\it (a)\/}),
   taken from  \cite[\S 3]{DL} (anyway, everything comes easily by induction):

\medskip

\begin{lema}  \label{commut_q-bin-coeff}
 Let  $ \mathcal{A} $  be any algebra over a field  $ \mathbb{F} \, $,  and let  $ \, p \in \mathbb{F} \, $  be {\sl not\/}
 a root of unity.  Let  $ \, X, Y, M^{\pm 1} \in \mathcal{A} \, $  with  $ \, X Y = Y X \, $.  Then for  $ \, t, s \in \NN\, $,  $ \, c \in \ZZ \, $  we have
  $$  \hfill   {\bigg( {{X \, Y \, ; \, c} \atop t} \bigg)}_{\!\!p}  \,\; = \;\,  \sum_{s=0}^t \, p^{\,(s-c_y)(s-t)} \,
  {\bigg( {{X \, ; \, c_{{}_X}} \atop {t - s}} \!\bigg)}_{\!\!p} \, Y^{t-s} \, {\bigg( {{Y \, ; \, c_{{}_Y}} \atop s} \!\bigg)}_{\!\!p}
  \qquad \hfill  \forall \;\; c_{{}_X} \! + c_{{}_Y} = c  $$
  $$  M \, M^{-1} \, = \, 1 \, = \, M^{-1} \, M  \;\; , \qquad  {\left( M \, ; c \atop 0 \right)}_{\!\!p} \, = \, 1  \;\; ,
  \qquad  (\,p - 1) {\left( M \, ; 0 \atop 1 \right)}_{\!\!p} \, = \, M - 1  $$
  $$  M^{\pm 1} {\left( M \, ; c \atop t \right)}_{\!\!p} \, = \, {\left( M \, ; c \atop t \right)}_{\!\!p} M^{\pm 1}  \; ,
  \;\quad  {\left( M \, ; c \atop t \right)}_{\!\!p} {\left( M \, ; c-t \atop s \right)}_{\!\!p} \, = \, {\bigg( {t+s \atop t} \bigg)}_{\!p} \,
  {\left( M \, ; c \atop t\!+\!s \right)}_{\!\!p}  $$
  $$  {\left( M \, ; c+1 \atop t \right)}_{\!\!p} - p^{\,t} {\left( M \, ; c \atop t \right)}_{\!\!p} \, = \, {\left( M \, ; c \atop t\!-\!1 \right)}_{\!\!p}
  \qquad   \eqno  \forall \;\; t \geq 1   \quad  $$
  $$  \phantom{OOOOOO}   {\left( M \, ; c+1 \atop t \right)}_{\!\!p} - {\left( M \, ; c \atop t \right)}_{\!\!p} \, = \, p^{\,c-t+1} \, M \,
  {\left( M \, ; c \atop t\!-\!1 \right)}_{\!\!p}   \qquad \qquad \qquad   \eqno  \forall \, t \geq 1   \quad  $$
%
%%
%%%
  $$  {\left( M \, ; c \atop t \right)}_{\!\!p} \, = \, \sum_{s \geq 0}^{s \leq c, t} p^{\,(c-s)(t-s)} {\bigg( {c \atop s} \bigg)}_{\!\!p}
  {\left( M \, ; 0 \atop t\!-\!s \right)}_{\!\! p}   \quad   \eqno  \forall \;\; c \geq 0   \quad  $$
%%%
%%
%
  $$  {\left( M \, ; -c \atop t \right)}_{\!\!p} \, = \, \sum_{s=0}^t {(-1)}^s p^{-t(c+s) + {{s+1}\choose{2}}}
  {\bigg( {s+c-1 \atop s} \bigg)}_{\!\!p} {\left( M \, ; 0 \atop t\!-\!s \right)}_{\!\!p}   \eqno  \forall \;\; c \geq 1   \;\;  $$
   If in addition  $ \mathcal{A} $  is a Hopf algebra and  $ M^{\pm 1} $  is group-like, then
  $$  \Delta \left(\! {\bigg( {{M \, ; \, c} \atop t} \bigg)}_{\!\!p} \,\right)  \; = \;  \sum_{r+s=t} \, p^{\,-r\,(s-c_2)} \, {\bigg( {{M \, ; \, c_{{}_1}} \atop r}
  \!\bigg)}_{\!\!p} \otimes \, M^r {\bigg( {{M \, ; \, c_{{}_2}} \atop s} \!\bigg)}_{\!\!p}  \;\;
  \eqno  \forall \;\; c_{{}_1} \! + c_{{}_2} = c   \quad  $$
%
%%%
  $$  \epsilon\left(\! {\bigg( {{M \, ; \, c} \atop t} \bigg)}_{\!\!p} \,\right) = \, {\bigg( {c \atop t} \bigg)}_{\!\!p}  \; ,
  \quad  \SS\left(\! {\bigg( {{M \, ; \, c} \atop t} \bigg)}_{\!\!p} \,\right) = \, {(-1)}^t \, p^{\, c\,t - {t \choose 2}} \, M^{-t} \,
  {\bigg( {{M \, ; \, t\!-\!c\!-\!1} \atop t} \bigg)}_{\!\!p}  $$
  $$  \Delta\big( M^{\pm 1} \big) \, = \, M^{\pm 1} \otimes M^{\pm 1}  \;\; ,  \qquad  \epsilon\left( M^{\pm 1} \right) \, =
  \, 1 \;\; ,  \qquad  \SS\left( M^{\pm 1} \right) \, = \, M^{\mp 1} \phantom{\bigg|^{|}}   \eqno  $$
\end{lema}

\medskip

   Inside the  $ \QQ (p) $--vector  spaces   $ \bE_m $  and  $ \E_m $  we consider the  $ \Zppm $--integral
   form of Laurent polynomials with coefficients in  $ \Zppm \, $,  namely
  $$  \bE_{m,\ZZ}  \; := \;  \Zppm\big[{\big\{X_i^{\pm 1}\big\}}_{i \in J}\big] \quad ,   \qquad
     \E_{m,\ZZ}  \; := \;  \Zppm\Big[{\big\{\chi_i^{\pm 1}\big\}}_{i \in J}\Big]  $$
which in fact are both Hopf subalgebras  (of  $ \bE_m $  and  $ \E_m \, $)  over  $ \Zppm \, $.
 \vskip5pt
   Fix some  $ \, d_i \in \ZZ \setminus \{0\} \, $  and powers  $ \, p_i := p^{d_i} \, $  for each  $ \, i \in J \, $.
   Then a unique  $ \QQ(p) $--bilinear  pairing  $ \; \big\langle\,\ ,\ \big\rangle : \bE_m \times \E_m \relbar\joinrel\longrightarrow \QQ(p) \; $
   exists, given by  $ \; \big\langle X_i^{z_i} , \chi_j^{\zeta_j} \big\rangle := p_i^{\delta_{ij} z_i \zeta_j} \; $ (for all  $ \, z_i , \zeta_j \in \ZZ \, $
   and  $ \, i, j \in J \, $).  By restriction, this clearly yields a similar  $ \Zppm $--valued  pairing between
the  $ \Zppm $--integral  forms  $ {\bE}_{m,\ZZ} $  and  $ \E_{m,\ZZ} \; $;
  indeed, this is even a Hopf pairing (cf.\  Definition \ref{def_(skew-)Hopf-pairing}).
 Finally, define
  $$  {\big( \E_{m,\ZZ} \big)}^\circ  \, := \,  \Big\{\, f \in \bE_m \,\Big|\, \big\langle f \,, \E_{m,\ZZ} \big\rangle \subseteq \Zppm \,\Big\}.  $$

\vskip7pt

%
%%%
%    It follows directly from definitions, and from  $ \; {\displaystyle {n \choose k}_{\!\!p_i}}
% \! \in \Zppm \; $  for all  $ \, n, k \in \NN \, $  (cf.\  \ref{q-numbers}),  that
%%%
%
   It follows from definitions and  $ \; {\displaystyle {n \choose k}_{\!\!p_i}} \! \in \Zppm \; $  ($ \, n, k \in \NN \, $),  cf.\  \ref{q-numbers},  that
\begin{equation}   \label{q-binom-coeff_int-val}
  \qquad \qquad   X_i^{\pm 1} \, , \; {\left({{X_i^{\pm 1} \, ; \, c} \atop n} \right)}_{\!\!p_i} \, \in \; {\big( \E_{m,\ZZ} \big)}^\circ
  \qquad \qquad \quad  \forall \,\; i \in J \, , \; c \in \ZZ \, , \; n \in \NN
\end{equation}

\vskip5pt

   Now set  $ \; {\left({{X_J \, ; \, 0} \atop \underline{n}} \right)}_{\!\underline{p}} \, X_J^{-\lfloor\, \underline{n}/2 \rfloor}
   := \prod_{i \in J} {\Big({{X_i \, ; \, 0} \atop n} \Big)}_{\!p_i} \, X_i^{-\lfloor n_i/2 \rfloor} \; $  for every
   $ \, \underline{n} := {\big( n_i \big)}_{i \in J} \in \NN^J \, $,  where  $ \, \big\lfloor n_i/2 \big\rfloor \, $
   is the greatest natural number less or equal than  $ n/2 \; $.  Then we have

\vskip13pt

\begin{prop}  \label{prop:basis_int-on-laur-mons}
 \cite[Theorem 3.1]{DL}
 \vskip4pt
   (a)\,  $ \; {\big( \E_{m,\ZZ} \big)}^\circ \, $  is a free  $ \, \Zppm $--module,
with basis
  $$  {\mathbb B}_m  \; := \;  \Bigg\{ {\left({X_J \atop \underline{n}} \right)}_{\!\!\underline{p}} \,  X_J^{-\lfloor\, \underline{n}/2 \rfloor} \;\Bigg|\;\,
  \underline{n} \in \NN^J \,\Bigg\}  $$
 \vskip4pt
   (b)\,  $ \; {\big( \E_{m,\ZZ} \big)}^\circ \, $  is the  $ \Zppm $--subalgebra  of  $ \bE_m $ generated by all the
   $ {\Big(\! {{X_i \, ; \, c} \atop n} \Big)}_{\!p_i} $'s  and the  $ X_i^{-1} $'s,  or by all the  $ {\Big(\! {{X_i^{-1} \, ; \, c} \atop n} \Big)}_{\!p_i} $'s
   and the  $ X_i $'s.  In fact, it can be presented as the Hopf\/  $ \Zppm $--algebra  with generators
   $ \; {\Big(\! {{X_i \, ; \, c} \atop n} \Big)}_{\!p_i} \, $,  $ \, X_i^{\pm 1} \, $   --- for all  $ \, i \in I \, $,  $ \, n \in \NN \, $,  $ \ c \in \ZZ \, $
   ---   and relations stating that all generators commute with each other plus all relations like in
   Lemma \ref{commut_q-bin-coeff}  but with  $ {\Big(\! {{X_i \, ; \, c} \atop n} \Big)}_{\!p_i} \, $,  $ X_i^{\pm 1} $
   and  $ p_i $  replacing  $ {\Big(\! {{M \, ; \, c} \atop n} \Big)}_{\!p} \, $,  $ M^{\pm 1} $  and  $ p $
   respectively, for all  $ \, i \in I \, $;  the Hopf structure then is given again by the same formulas as in
   Lemma \ref{commut_q-bin-coeff}  now applied to the given generators.
\end{prop}

\begin{proof}
 Due to  (\ref{q-binom-coeff_int-val}),  the  $ \Zppm $--subalgebra  of  $ \bE_m $ generated by all the
 $ {\Big({{X_i \, ; \, c} \atop n} \Big)}_{\!p_i} $'s  and the  $ X_i^{-1} $'s  is contained in  $ {\big( \E_{m,\ZZ} \big)}^\circ $
 --- and similarly if we replace each  $ X_i^{\pm 1} $  with its inverse  $ X_i^{\mp 1} \, $.
 Thus to prove the whole claim it is enough to show that  $ {\big( \E_{m,\ZZ} \big)}^\circ $  admits
 $ \mathbb{B}_m $  as  $ \Zppm $--basis:  indeed, we already have that the  $ \Zppm $--span  of  $ \mathbb{B}_m $
 is contained in  $ {\big( \E_{m,\ZZ} \big)}^\circ $,
so to prove  {\it (a)\/}
 it is enough to show that any  $ \, f \in {\big( \E_{m,\ZZ} \big)}^\circ \, $  can be written uniquely as a
 $ \Zppm $--linear  combination of elements in  $ \mathbb{B}_m \, $.
                                                        \par
   To begin with,  $ \bE_m $  over  $ \QQ(p) $  has basis the set
 $ \, \Big\{ X_J^{\underline{z}} := {\textstyle \prod\limits_{i \in J}} X_i^{z_i} \,\Big|\, \underline{z} := {\big( z_i \big)}_{i \in J} \in \ZZ^J \Big\} \, $
%
%%%
% of Laurent monomials in the  $ X_i^{\pm 1} $'s,
%%%
 and from this one easily sees that the set
 $ \, {\mathbb B}_m := \Big\{ {\textstyle {\left({X_J \atop \underline{n}} \right)}_{\!\underline{p}}
 X_J^{-\lfloor\, \underline{n}/2 \rfloor}} \,\Big|\; \underline{n} \in \NN^J \,\Big\} \, $  is a  $ \QQ(p) $--basis  too,
 which is contained in  $ {\big( \E_{m,\ZZ} \big)}^\circ $  by  (\ref{q-binom-coeff_int-val}).
 Now consider the monomials  $ \, \chi_{\underline{\nu}} := {\textstyle \prod_{j \in J}} \chi_j^{\nu_j} \, $
 (with  $ \, \underline{\nu} := {\big( \nu_j \big)}_{j \in J} \in \NN^J \, $)  in the  $ \chi_i $'s.
 By construction one has
\begin{equation}   \label{pair_q-bin_Laur-mon}
  \qquad   \left\langle {\left({X_J \atop \underline{n}} \right)}_{\!\!\underline{p}} X_J^{-\lfloor\, \underline{n}/2 \rfloor} \, ,
  \, \chi_{\underline{\nu}} \,\right\rangle  \; = \;  \prod_{i \in J} {\bigg( {\nu_i \atop n_i} \bigg)}_{\!\!p_i} \, p_i^{-\nu_i \, \lfloor n_i/2 \rfloor}
  \qquad   \forall \,\; \underline{n} \, , \, \underline{\nu} \, \in \, \NN^J
\end{equation}
   Let  $ \preceq $  be the order relation in  $ \NN^J $  given by the product of the standard order in  $ \NN \, $,
   so  $ \;\; \underline{n} \preceq \underline{\nu} \iff n_i \leq \nu_i \;\; \forall \; i \in \J \; $.  As
   $ \, {\bigg( {\displaystyle {\nu_i \atop n_i}} \bigg)}_{\!\!p_i} \! \not= 0 \iff n_i \leq \nu_i \, $,  by
   (\ref{pair_q-bin_Laur-mon})  one has
\begin{equation}   \label{pair_q-bin_non-zero}
  \qquad   \left\langle {\left({X_J \atop \underline{n}} \right)}_{\!\!\underline{p}} X_J^{-\lfloor\, \underline{n}/2 \rfloor} \, ,
  \, \chi_{\underline{\nu}} \,\right\rangle  \; \not= \;  0   \;\;\; \iff \;\;\;  \underline{n} \leq \underline{\nu}
  \qquad \quad   \big(\, \underline{n} \, , \, \underline{\nu} \, \in \, \NN^J \,\big)
\end{equation}
   \indent   Now pick an  $ \, f \in {\big( \E_{m,\ZZ} \big)}^\circ \setminus \{0\} \, $,
   and expand it (uniquely)  as a  $ \QQ(p) $--linear  combination of elements in  $ {\mathbb B}_m \, $,
   say  $ \; f = \sum_{s=1}^N c_s \, {\left( {X_J \atop {\underline{n}}^{(s)}} \!\right)}_{\!\!\underline{p}} X^{-\lfloor\, {\underline{n}}^{(s)}/2 \rfloor} \; $
   for some  $ \, c_s \in \QQ(p) \setminus \{0\} \, $.  Choose any index  $ \, \sigma \in \{1,\dots,N\} \, $
   such that  $ \, {\underline{n}}_{\,\sigma} \, $  is  {\sl minimal\/}  in  $ \, \big\{ {\underline{n}}^{(1)} , \dots , {\underline{n}}^{(N)} \big\} \, $:
   then by  (\ref{pair_q-bin_non-zero})  above and by  $ \, {\bigg( {\displaystyle {n \atop n}} \bigg)}_{\!\!p_i} \! = 1 \, $
   we get
  $$  \Big\langle f \, , \underline{\chi}_{\,{\underline{n}}^{(\sigma)}} \Big\rangle  \; = \;
  \sum_{s=1}^N c_s \, \left\langle {\left( {X_J \atop {\underline{n}}^{(s)}} \right)}_{\!\!\underline{p}}
  X_J^{-\lfloor\, {\underline{n}}^{(s)}/2 \rfloor} \, , \, \underline{\chi}_{\,{\underline{n}}^{(\sigma)}} \right\rangle  \;
  = \;  c_\sigma \, \prod_{i \in J} p_i^{-{{\underline{n}}^{(\sigma)}_i} \, \lfloor {{\underline{n}}^{(\sigma)}_i}/2 \rfloor}  $$
so that  $ \, \Big\langle f \, , \underline{\chi}_{\,{\underline{n}}^{(\sigma)}} \Big\rangle \in \Zppm \, $
--- because  $ \, f \in {\big( \E_{m,\ZZ} \big)}^\circ \setminus \{0\} \, $
---   implies at once  $ \, c_\sigma \in \Zppm \, $.  By induction on  $ N $,
one then concludes that  {\sl all\/}  coefficients  $ \, c_s \, (s=1, \dots, N) \, $
in the expansion of  $ f $  actually lie in  $  \Zppm \, $,  q.e.d.
 \vskip4pt
   Finally, the presentation mentioned in claim  {\it (b)\/}  follows from  \cite[\S 3.4]{DL}.
\end{proof}

\medskip

   As a direct consequence of the above Lemma, we have the following:

\medskip

\begin{prop}
  $ {\big( \E_{m,\ZZ} \big)}^\circ $  is a Hopf\/  $ \Zppm $--subalgebra  of  $ \bE_m \, $.
  Therefore, the former is a\/  $ \Zppm $--integral  form of the latter.
\end{prop}

\medskip

   For later use, we finish the present discussion with another result that gives the dual, somehow, of what we found for
   $ {\big( \E_{m,\ZZ} \big)}^\circ \, $:  it concerns the ``bidual'' space
  $$  {\Big(\! {\big( \E_{m,\ZZ} \big)}^\circ \Big)}^{\!\circ}  \, := \;  \Big\{\, t \in \E_m \,\Big|\,
  \big\langle \big(\E_{m,\ZZ}\big)^{\circ} , t \,\big\rangle \subseteq \Zppm \,\Big\}  $$

\medskip

\begin{prop}  \label{prop:bidual}
 The ``bidual space''  $ \, {\Big(\! {\big( \E_{m,\ZZ} \big)}^\circ \Big)}^{\!\circ} $  coincides with
$ \E_{m,\ZZ} \, $.
\end{prop}

\pf
 By definition  $ \, {\Big(\! {\big( \E_{m,\ZZ} \big)}^\circ \Big)}^{\!\circ} \!\supseteq \E_{m,\ZZ} \; $,
 we have to prove the converse inclusion.
 \vskip3pt
   Let  $ \; t \in {\Big(\! {\big( \E_{m,\ZZ} \big)}^\circ \Big)}^{\!\circ} \, $  and expand it with respect to the  $ \QQ(p) $--basis
   of  $ \E_m $  made of the Laurent monomials  $ \, \chi_{\underline{\zeta}} := {\textstyle \prod_{j \in J}} \chi_j^{\,\zeta_j} \, $
   (with  $ \, \underline{\zeta} := {\big( \zeta_j \big)}_{j \in J} \in \ZZ^J \, $)  in the  $ \chi_j $'s.  This means writing  $ t $  as
   $ \; t \, = \, \sum_{\underline{\zeta} \in \ZZ^J} c_{\,\underline{\zeta}} \, \chi_{\underline{\zeta}} \; $  for suitable
   $ \, c_{\,\underline{\zeta}} \in \QQ(p) \, $,  almost all being zero: we denote by  $ \, n(t) \in \NN \, $  the number of
   all such non-zero coefficients.  We must show that  $ \, t \in \E_{m,\ZZ} \, $,  i.e.\ all the  $ c_{\,\underline{\zeta}} $'s
   belong to  $ \Zppm \, $;  we do it by induction on  $ n(t) \, $.
                                                         \par
   As a first step, we assume that for all  $ \, \underline{\zeta} := {\big(\zeta\big)}_{j \in J} \, $  such that  $ \, c_{\,\underline{\zeta}} \neq 0 \, $
   we have  $ \, \zeta_j \geq 0 \, $  for all  $ \, j \in J \, $.  Then choose a  $ \, \underline{\zeta\,}^\uparrow \in \ZZ^J \, $
   such that  $ \, c_{\,\underline{\zeta\,}^\uparrow} \not= 0 \, $  and  $ \underline{\zeta\,}^\uparrow $  is  {\sl maximal\/}
   for that property with respect to the standard product order on  $ \ZZ^J \, $;  in other words, there exists no
   $ \, \underline{\zeta} \not= \underline{\zeta\,}^\uparrow \, $  such that  $ \, c_{\,\underline{\zeta}} \not= 0 \, $
   and  $ \, \underline{\zeta}_j \geq \underline{\zeta\,}^\uparrow_j \, $  for all  $ \, j \in J \, $.  Then we have
  $$  \Zppm  \; \ni \;  \left\langle {\left( {X_J \atop {\,\underline{\zeta\,}^\uparrow}} \right)}_{\!\!\underline{p}} \, , \, t \right\rangle  \; =
  \;  \sum_{\underline{\zeta} \in \ZZ^J} c_{\,\underline{\zeta}} \left\langle {\left( {X_J \atop {\,\underline{\zeta\,}^\uparrow}} \right)}_{\!\!\underline{p}}
  \, , \, \chi_{\underline{\zeta}}  \right\rangle  \; = \;
  \sum_{\underline{\zeta} \in \ZZ^J} c_{\,\underline{\zeta}} \prod_{j \in J} {\left( {{\zeta_j} \atop {\zeta^\uparrow_j}} \right)}_{\!\!p}  $$
by the maximality of  $ \, \underline{\zeta\,}^\uparrow \, $   --- and the properties of  $ q $--binomial  coefficients ---
we have  $ \; {\Big( {{\zeta_j} \atop {\zeta^\uparrow_j}} \Big)}_{\!\!p} = \delta_{\underline{\zeta} \, , \, \underline{\zeta\,}^\uparrow} \; $,  \,
thus the above eventually gives  $ \; c_{\underline{\zeta\,}^\uparrow} \in \Zppm \; $,  \, q.e.d.  Now look at
$ \; t' \, := \, t - c_{\underline{\zeta\,}^\uparrow} \chi_{\underline{\zeta\,}^\uparrow} \; $:  by construction we have
$ \, n\big(t'\big) = n(t) - 1 \lneqq n(t) \, $,  hence we may assume by induction that  $ \, t' \in \E_{m,\ZZ} \, $.
Then  $ \; t \, = \, t' + c_{\underline{\zeta\,}^\uparrow} \chi_{\underline{\zeta\,}^\uparrow} \, \in \, \E_{m,\ZZ} \; $  too, q.e.d.
                                                         \par
   At last, notice that  $ \, {\Big(\! {\big( \E_{m,\ZZ} \big)}^\circ \Big)}^{\!\circ} \, $  is a  $ \Zppm $--subalgebra:  in fact,
   this follows at once from the fact that  $ \, {\big( \E_{m,\ZZ} \big)}^\circ \, $  is a  $ \Zppm $--coalgebra   ---
   and we have the perfect Hopf pairing  $ \, \langle\,\ ,\ \rangle \, $  between  $ \, \E_m \, $  and  $ \, \bE_m \, $.
   As clearly all the  $ \chi_{\,\underline{\zeta}} $'s  do belong to  $ {\Big(\! {\big( \E_{m,\ZZ} \big)}^\circ \Big)}^{\!\circ} $
   and are invertible in it, it follows that for any  $ \, t \in \E_m \, $  and for any  $ \, \underline{\zeta}' \in \ZZ^J \, $  one has
   $ \, t \in {\Big(\! {\big( \E_{m,\ZZ} \big)}^\circ \Big)}^{\!\circ} \, $  if and only if  $ \; t\,\chi_{\,\underline{\zeta}'} \, \in \,
   {\Big(\! {\big( \E_{m,\ZZ} \big)}^\circ \Big)}^{\!\circ} \; $.  Now, choosing a proper  $ \, \underline{\zeta}' \in \ZZ^J \, $
   we can get  $ \, t\,\chi_{\,\underline{\zeta}'} \, $  such that in its  $ \QQ(p) $--linear  expansion in the  $ \chi_{\,\underline{\zeta}} $'s,
   say  $ \, t\,\chi_{\,\underline{\zeta}'} \, = \, \sum_{\underline{\zeta} \in \ZZ^J} c_{\,\underline{\zeta}} \, \chi_{\underline{\zeta}} \, $,
   for all the  $ \, \underline{\zeta} = {\big( \zeta_j \big)}_{j \in J} \, $'s  such that  $ \, c_{\,\underline{\zeta}} \not= 0 \, $  we have
   $ \, \zeta_j \geq 0 \, $  for all  $ \, j \in J \, $.  But then  $ \, t' := t\,\chi_{\,\underline{\zeta}'} \, $
%%%
  \hbox{is of the type we considered above, for which}
%%%
%%%
 we did prove that  $ \, t' := t\,\chi_{\,\underline{\zeta}'} \in \E_{m,\ZZ} \; $;  \,
 so the previous analysis gives  $ \, t \in \E_{m,\ZZ} \, $  too.
\epf

\smallskip

\begin{free text}{\bf The toral part of restricted MpQG's.}  \label{toral_MpQG's}
 The restricted integral form of a uniparameter quantum group  $ U_{q}(\lieg) $  was introduced by Lusztig as
 $ q $--analogue  of Chevalley's  $ \ZZ $--form  of  $ U(\lieg) \, $:  we consider here its modified version as in
 \cite{DL},  where specific changes were done in the choice of the toral part.
 The construction in  \cite{DL}  immediately extends to  $ \QEqcheck \, $,
 hence now we want to further extend it to the general case of any multiparameter  $ \QE \, $;
 nevertheless, a (mild) restriction on  $ \bq $  is necessary, in the following terms:
 \vskip7pt
   \dbend \ \ --- \  {\it From now on, all along the present section  $ \underline{\text{we assume that\/  $ \bq $
   is of  {\sl integral}}} $  {}
   $ \underline{\text{\sl type}} $  (as well as Cartan, as usual)},  say  $ \; \bq = {\big(\, q_{ij} = q^{\,b_{ij}} \big)}_{i, j \in I} \; $
   as in  \S \ref{multiparameters}.  {\sl Therefore (cf.\  \S \ref{ground-ring})  $ \Rbq \, $,  resp.\  $ \Rbqsq \, $,
   is just the subring of\/  $ \k $  generated by  $ q^{\pm 1} \, $,  resp.\  $ q^{\pm 1/2} \, $,  and  $ \Fbq \, $,
   resp.\  $ \Fbqsq \, $,  is the subfield of\/  $ \k $  generated by  $ q^{\pm 1} \, $,  resp.\  $ q^{\pm 1/2} \, $.}
 \vskip3pt
   In the following, whatever object we shall introduce that bear a structure of module over  $ \Rbq \, $,
   resp.\ over  $ \Fbq \, $,  will also have its natural counterpart defined over  $ \Rbqsq \, $,  resp.\ over
   $ \Fbqsq \, $,  that is also a scalar extension of the previous one.
 \vskip9pt
   In the following,  $ \QE $  will be the MpQG associated with  $ \bq $  as in
   Definition \ref{def:multiqgroup_ang}. Inside it   --- more precisely, inside its toral part ---
   we want to apply the construction presented in  \S \ref{q-bin_coeff_&_arithmetic},  for
   suitable choices of the  $ X_i $'s,  the  $ \chi_i $'s  and the  $ p_i $'s.

\vskip5pt

  Recall that  $ \, I := \{1,\dots,\theta\} \, $.  Define
  $ \, G_i^{\pm 1} := K_i^{\pm 1} L_i^{\mp 1} \, \big( \in U_\bq(\lieh \oplus \lieh) := U_\bq^0 \,\big) \, $  for all
  $ \, i \in I \, $,  and consider inside  $ U_\bq^0 $  the  $ \Fbq $--subalgebra  generated by the  $ K_i^{\pm 1} $'s
  and the  $ G_i^{\pm 1} $'s,  namely
  $ \; \bE_{\,2\theta} := \Fbq\big[ {\big\{ K_i^{\pm 1}, G_i^{\pm 1} \big\}}_{i,j \in I} \big] \; $;
  note also that taking the  $ L_i^{\pm 1} $'s  as generators instead of the  $ G_i^{\pm 1} $'s
  will give the same algebra.  As a matter of fact, since the  $ K_i^{\pm 1} $'s  and the  $ G_i^{\pm 1} $'s
  are group-like, this  $ \bE_{\,2\theta} $  is indeed a Hopf  $ \Fbq $--subalgebra of  $ \, U_\bq^0 \, $.
                                                     \par
   In the dual space  $ {\big(U_\bq^0\big)}^* $  we consider the  $ \Fbq $--algebra  morphisms  $ \dot{\kappa}_i^{\pm 1} $
   and  $ \gamma_i^{\pm 1} $   --- for  $ \, i \in I \, $  ---   uniquely defined by
  $$  \Big\langle K_i^{z_i} , \dot{\kappa}_j^{\zeta_j} \Big\rangle := q^{\,\delta_{ij} z_i \zeta_j} \; ,  \!\quad
      \Big\langle G_i^{z_i} , \dot{\kappa}_j^{\zeta_j} \Big\rangle := 1 \; ,  \;\quad
      \Big\langle K_i^{z_i} , \gamma_j^{\zeta_j} \Big\rangle := 1 \; ,  \!\quad
      \Big\langle G_i^{z_i} , \gamma_j^{\zeta_j} \Big\rangle := q_{ii}^{\,\delta_{ij} z_i \zeta_j}  $$
(cf.\ \S \ref{ground-ring})  for all  $ \, z_i , \zeta_j \in \ZZ \, $  and  $ \, i, j \in J \, $.  Setting also
$ \; \dot{\E}_{\,2\theta} := \Fbq\big[ {\big\{ \dot{\kappa}_i^{\pm 1}, \gamma_i^{\pm 1} \big\}}_{i,j \in I} \big] \; $
for the subalgebra in  $ {\big(U_\bq^0\big)}^* $  generated by the  $ \dot{\kappa}_i^{\pm 1} $'s  and the
$ \gamma_i^{\pm 1} $'s,  these formulas yield a non-degenerate  $ \Fbq $--pairing  between  $ \bE_{\,2\theta} $
and  $ \dot{\E}_{\,2\theta} \, $: in fact, the latter is a Hopf algebra (the  $ \dot{\kappa}_i^{\pm 1} $'s  and
$ \gamma_i^{\pm 1} $'s  being group-like), so this is actually a  {\sl Hopf\/}  pairing.
                                              \par
   Now  $ \bE_{\,2\theta} $  and  $ \dot{\E}_{\,2\theta} \, $,  paired as explained, can play the role of  $ \bE_m $
   and $ \E_m $  in  \S \ref{q-bin_coeff_&_arithmetic}  above, so we apply to them the construction presented there.
   Taking their corresponding  $ \Rbq $--integral  form of Laurent polynomials with coefficients in  $ \Rbq \, $,  namely
  $$  \bE_{\,2\theta,\Rbq} \, := \, \Rbq\big[{\big\{K_i^{\pm 1},G_i^{\pm 1}\big\}}_{i \in I}\big]  \qquad  \text{and}
  \qquad  \dot{\E}_{\,2\theta,\Rbq} := \Rbq\big[{\big\{ \dot{\kappa}_i^{\pm 1}, \gamma_i^{\pm 1} \big\}}_{i \in I}\big]  $$
 we have Hopf subalgebras over  $ \Rbq \, $,  and the previously given pairing restricts to a non-degenerate
 $ \Rbq $--valued  pairing among these two integral forms.
 \vskip5pt
   {\it Now assume in addition that the multiparameter  {\sl  $ \, \bq $  is of strongly integral type},  say
 $ \; \bq = {\big(\, q_{ij} = q^{d_i t^+_{ij}} = q^{d_j t^-_{ij}} \,\big)}_{i, j \in I} \; $}.
 Then besides the previous construction we can perform a second, parallel one, as follows.
                                                     \par
   Inside  $ \dot{\E}_{\,2\theta,\Rbq} $  we consider now  $ \, \kappa_i^{\pm 1} := \dot{\kappa}_i^{\pm d_i} \, $
   (for  $ \, i \in I \, $),  for which we have
  $$  \Big\langle K_i^{z_i} , \kappa_j^{\,\zeta_j} \Big\rangle := q^{\,\delta_{ij} d_i z_i \zeta_j}  \quad ,
  \qquad  \Big\langle G_i^{z_i} , \kappa_j^{\,\zeta_j} \Big\rangle := 1  $$
 for all  $ \, z_i , \zeta_j \in \ZZ \, $  and  $ \, i, j \in J \, $,  and set  $ \; \E_{\,2\theta} :=
\Fbq\big[ {\big\{ \kappa_i^{\pm 1}, \gamma_i^{\pm 1} \big\}}_{i,j \in I} \big] \; $  for the subalgebra in
$ \dot{\E}_{\,2\theta,\Rbq} $   --- hence in  $ {\big(U_\bq^0\big)}^* $  too ---   generated by the
$ \kappa_i^{\pm 1} $'s  and the  $ \gamma_i^{\pm 1} $'s.  Then  $ \E_{\,2\theta} $  is in fact a Hopf
$ \Fbq $--algebra  (the  $ \kappa_i^{\pm 1} $'s  being group-like, like the  $ \gamma_i^{\pm 1} $'s),
and the above formulas provides a non-degenerate Hopf pairing.  Taking now
  $$  \bE_{\,2\theta,\Rbq} \, := \, \Rbq\big[{\big\{K_i^{\pm 1},G_i^{\pm 1}\big\}}_{i \in I}\big]  \qquad  \text{and}
  \qquad  \E_{\,2\theta,\Rbq} := \Rbq\big[{\big\{\kappa_i^{\pm 1},\gamma_i^{\pm 1}\big\}}_{i \in I}\big]  $$
 we have Hopf subalgebras over  $ \Rbq \, $,  and a non-degenerate  $ \Rbq $--valued  pairing between
 them provided by restriction of the previous one.  Basing on all the above, we can now introduce the
 objects we are mainly interested into:
\end{free text}

\vskip5pt

\begin{definition}  \label{def-Uhatdotbq}  {\ }
 \vskip3pt
   \hskip-5pt   {\it (a)} \hskip3pt  $ \Uhatdot_\bq^{\,\raise-5pt\hbox{$ \scriptstyle 0 $}} \, :=
   {\big( \dot{\E}_{\,2\theta,\Rbq} \big)}^\circ
 = \,  \Big\{\, f \! \in \! \bE_{\,2\theta} \,\Big|\, \big\langle f ,
\dot{\E}_{\,2\theta,\Rbq} \big\rangle \subseteq \Rbq \Big\} \; $
 \hskip2pt  if  $ \bq $  is (just)  {\sl integral},
 \vskip2pt
   \hskip-5pt   {\it (b)} \hskip3pt  $ \Uhat_\bq^{\,\raise+1pt\hbox{$ \scriptstyle 0 $}} := {\big( \E_{\,2\theta,\Rbq} \big)}^\circ
 = \,  \Big\{\, f \! \in \! \bE_{\,2\theta} \,\Big|\, \big\langle f ,
\E_{\,2\theta,\Rbq} \big\rangle \subseteq \Rbq \Big\} \; $
%
%%%
% \,  or also  $ \; \Uhat_\bq(\lieg) := {\big( \E_{\,2\theta,\Rbq} \big)}^\circ \; $,
%%%
%
 \hskip2pt  if  $ \bq $  is  {\sl strongly integral}.   \hskip1pt  $ \diamondsuit $
\end{definition}

\vskip7pt

   By the analysis and results in  \S \ref{q-bin_coeff_&_arithmetic},  applied to the present situation, we have

\smallskip

\begin{prop}  \label{prop:struct-Uhatq0}  {\ }
 \vskip4pt
   (a)\,  $ \; \Uhatdot_\bq^{\,\raise-5pt\hbox{$ \scriptstyle 0 $}} \, $  is a Hopf  $ \, \Rbq $--subalgebra  of
   $ \bE_{\,2\theta} \, $, generated by all the  $ {\Big(\! {{K_i \, ; \, c} \atop k} \Big)}_{\!q} $'s,  the  $ K_i^{\pm 1} $'s,
   the  $ {\Big(\! {{G_i^{-1} \, ; \, c} \atop g} \Big)}_{\!q_{ii}} $'s  and the  $ G_i^{\pm 1} $'s.
 \vskip4pt
   (b)\,  $ \; \Uhatdot_\bq^{\,\raise-5pt\hbox{$ \scriptstyle 0 $}} \, $  is a free  $ \, \Rbq $--module,  with basis   --- cf.\  \eqref{q-bin_coeff}  for notation ---
  $$  \Bigg\{ \prod_{i=1}^\theta {\left({K_i \atop k_i} \right)}_{\!\!q} \,
  K_i^{-\lfloor\, k_i/2 \rfloor} \prod_{i=1}^\theta {\left({G_i \atop g_i} \right)}_{\!\!q_{ii}} \,  G_i^{-\lfloor\, g_i/2 \rfloor} \;\Bigg|\;\, k_i, g_i \in \NN \;\;
  \forall \; i=1,\dots,\theta \,\Bigg\}  $$
 \vskip4pt
   (c)\,  $ \; \Uhatdot_\bq^{\,\raise-5pt\hbox{$ \scriptstyle 0 $}} \, $  is isomorphic to the Hopf\/  $ \Rbq $--algebra  with generators
 $ \; {\Big(\! {{K_i \, ; \, c} \atop k_i} \Big)}_{\!q} \, $,  $ \, K_i^{\pm 1} \, $,
 $ \; {\Big(\! {{G_i \, ; \, c} \atop g_i} \Big)}_{\!q_{ii}} \, $,  $ \, G_i^{\pm 1} \, $
 (for all  $ \, i \in I \, $,  $ \, k_i, g_i \in \NN \, $,  $ \ c \in \ZZ \, $)
and relations stating that all these generators commute with each other, plus all relations like in
Lemma \ref{commut_q-bin-coeff}  but with
 $ {\Big(\! {{K_i \, ; \, c} \atop k_i} \Big)}_{\!q} \, $,  $ K_i^{\pm 1} $,  $ q $  and
 $ {\Big(\! {{G_i \, ; \, c} \atop g_i} \Big)}_{\!q_{ii}} \, $,  $ G_i^{\pm 1} $,  $ q_{ii} $
replacing  $ {\Big(\! {{M \, ; \, c} \atop n} \Big)}_p \, $,  $ M^{\pm 1} $  and  $ p $  respectively, for all  $ \, i \in I \, $.
Accordingly, the Hopf structure of  $ \Uhatdot_\bq^{\,\raise-5pt\hbox{$ \scriptstyle 0 $}} $  is also given in terms of
generators by formulas as in  Lemma \ref{commut_q-bin-coeff}  now applied to the given generators.
 \vskip4pt
   (d)--(e)--(f)\;  When\/  $ \bq $  is  {\sl strongly integral},  similar claims hold true for
   $ \, \Uhat_\bq^{\,\raise+1pt\hbox{$ \scriptstyle 0 $}} \, $,  up to replacing everywhere each generator
   $ {\Big(\! {{K_i \, ; \, c} \atop k_i} \Big)}_{\!q} $  and the parameter  $ q $  with the corresponding generator
   $ {\Big(\! {{K_i \, ; \, c} \atop k_i} \Big)}_{\!q_i} $  and the parameter  $ \, q_i = q^{d_i} \, $,  respectively.
 \vskip2pt
   (g)\;  $ \, \Uhatdot_\bq^{\,\raise-5pt\hbox{$ \scriptstyle 0 $}} \, $  is a Hopf  $ \Rbq $--subalgebra  of
   $ \; \Uhat_\bq^{\,\raise+1pt\hbox{$ \scriptstyle 0 $}} \, $.
\end{prop}

\vskip9pt
   One last important observation is in order:

\vskip11pt

\begin{rmk}  \label{rmk:nosqua-intform}
 Definitions imply that, beside all the generators
 $ G_i^{\pm 1} $,  $ {\Big(\! {{G_i \, ; \, c} \atop g_i} \Big)}_{\!q_{ii}} \, $,  $ K_i^{\pm 1} $  and
 $ {\Big(\! {{K_i \, ; \, c} \atop k_i} \Big)}_{\!q} $,  the algebra  $ \Uhatdot_\bq^{\,\raise-5pt\hbox{$ \scriptstyle 0 $}} $
 contains all  $ L_i^{\pm 1} $'s  and  $ {\Big(\! {{L_i \, ; \, c} \atop l_i} \Big)}_{\!q} \, $'s
 --- as they give values in  $ \Rbq $  when paired with  $ \E_{\,2\theta,\Rbq} \, $.
 This restores a perfect ``symmetry'' in the roles of the  $ K_i $'s  and the  $ L_i $'s,
 which is not apparent in the very definition of  $ \Uhatdot_\bq^{\,\raise-5pt\hbox{$ \scriptstyle 0 $}} \, $.
 Indeed, one can easily prove that  $ \Uhatdot_\bq^{\,\raise-5pt\hbox{$ \scriptstyle 0 $}} $
 can be also generated by generators built from  ``$ L $''  instead of  ``$ K $'';  and similarly for
 $ \Uhat_\bq^{\,\raise+1pt\hbox{$ \scriptstyle 0 $}} \, $.  So replacing  ``$ K $''  by  ``$ L $''
 everywhere yields a twin statement of  Proposition \ref{prop:struct-Uhatq0}.
\end{rmk}

\vskip7pt

\begin{free text}{\bf Restricted MpQG's.}  \label{restricted_MpQG's}
 We are now ready to introduce our generalization to MpQG's of the notion of restricted integral
 form introduced by Lusztig for  $ U_{q}(\lieg) $  (and later modified in  \cite{DL}).
 We keep the restriction that  $ \bq $  must be of  {\sl integral type},  say
 $ \; \bq = {\big(\, q_{ij} = q^{b_{ij}} \big)}_{i, j \in I} \; $  with  $ \, B = {\big( b_{ij} \big)}_{i,j \in I} \in M_\theta(\ZZ) \, $,
 as in  \S \ref{multiparameters}.
                                                   \par
   Again, hereafter  $ \QE $  denotes the MpQG associated with  $ \bq $  as in  Definition \ref{def:multiqgroup_ang}.
 \vskip5pt
   We recall from  Definition \ref{def:q-bin-coeff & q-div_pows}{\it (b)\/}  the notion of  {\it  $ q $--divided  powers\/}:
   given  $ \, i \in I \, $,  $ \, \alpha \in \Phi^+ \, $,  $ \, X_i \in \big\{ E_i \, , F_i \,\big\} \, $,
   $ \, Y_\alpha \in \big\{ E_\alpha \, , F_\alpha \,\big\} \, $  and  $ \, n \in \NN \, $,  we call  {\it  $ q $--divided  powers\/}
   the elements  $ \,\; X_i^{\,(n)} :=
 X_i^n \Big/ \! {(n)}_{q_{ii}\!}! \;\, $  and  $ \,\; Y_\alpha^{\,(n)} :=
 Y_\alpha^n \Big/\! {(n)}_{q_{\raise-2pt\hbox{$ \scriptscriptstyle \alpha \alpha $}}\!\!}! \;\, $  in  $ \QEq \, $.
\end{free text}

   The following result, about commutation relations between quantum binomial coefficients and
   quantum divided powers, is proved by straightforward induction:

\vskip13pt

\begin{lema}  \label{lemma:commut_q-div-pows}
 Let  $ \; \bq = {\big(\, q_{ij} = q^{\,b_{ij}} \big)}_{i, j \in I} \, $  be of integral type.
 \vskip4pt
   (a)\;  For any  $ \, i \in I \, $,  $ \, m, n, h \in \NN \, $,  $ \, c \in \ZZ \, $,  $ \, \kappa ,
   \lambda \in Q \, $,  $ \, X , Y \in \big\{ K_\kappa L_\lambda \,\big|\, \kappa , \lambda \in Q \,\big\} \, $  and
   $ \, G_i^{\pm 1} := K_i^{\pm 1} L_i^{\mp 1} $,  we have
  $$  \displaylines{
   {\bigg( {{K_\kappa \, L_\lambda \, ; \, c} \atop h} \bigg)}_{\!\!q} \, F_j^{\,(n)}  \,\; = \;\,  F_j^{\,(n)} \,
   {\bigg( {{K_\kappa \, L_\lambda \, ; \, c - n \, {\big( B^{\scriptscriptstyle \, T} \!
   \cdot \kappa - B \cdot \lambda \big)}_j} \atop h} \bigg)}_{\!\!q}  \cr
   {\bigg( {{K_\kappa \, L_\lambda \, ; \, c} \atop h} \bigg)}_{\!\!q} \, E_j^{\,(n)}  \,\; = \;\,  E_j^{\,(n)} \,
   {\bigg( {{K_\kappa \, L_\lambda \, ; \, c + n \, {\big( B^{\scriptscriptstyle \, T} \! \cdot \kappa -
   B \cdot \lambda \big)}_j} \atop h} \bigg)}_{\!\!q}  \cr
   {\bigg( {{G_i \, ; \, c} \atop h} \bigg)}_{\!\!q_{ii}} \! F_j^{\,(n)}  = \,  F_j^{\,(n)}
   {\bigg( {{G_i \, ; \, c - n a_{ij}} \atop h} \!\bigg)}_{\!\!q_{ii}}  \; , \quad
    {\bigg( {{G_i \, ; \, c} \atop h} \bigg)}_{\!\!q_{ii}} \! E_j^{\,(n)}  = \,  E_j^{\,(n)}
    {\bigg( {{G_i \, ; \, c + n a_{ij}} \atop h} \!\bigg)}_{\!\!q_{ii}}  \cr
    E_i^{\,(m)} \, F_i^{\,(n)}  \; = \;  {\textstyle\sum_{s=0}^{m \wedge n}} F_i^{\,(n-s)} \, q_{ii}^{\,s} \,
    {\bigg( {{G_i \, ; \, 2\,s - m - n} \atop s} \bigg)}_{\!\!q_{ii}} \! L_i^{\,s} \, E_i^{\,(m-s)}  }  $$
where  $ \; {\big( B^{\scriptscriptstyle \, T} \! \cdot \kappa - B \cdot \lambda \big)}_j =
\sum_{i \in I} \! \big(\, b_{ij} \kappa_i - b_{ji} \lambda_i \big) \; $  for  $ \; \kappa = \sum_{i \in I} \kappa_i \, \alpha_i \, $,
$ \, \lambda = \sum_{i \in I} \lambda_i \, \alpha_i \, $.
 \vskip4pt
   Moreover, for the Hopf structure, on  $ q $--divided  powers we have formulas
  $$  \displaylines{
   \Delta\Big(\! E_i^{(n)} \Big) = {\textstyle \sum\limits_{s=0}^n} E_i^{\,(n-s)} K_i^{\,s} \otimes E_i^{\,(s)}  \, ,  \;\;
   \epsilon\Big(\! E_i^{\,(n)} \Big) = \delta_{n,0}  \, ,  \;\;  \SS\Big(\! E_i^{\,(n)} \Big) = {(-1)}^n q_{ii}^{+{n \choose 2}} K_i^{\,-n} E_i^{\,(n)}   \cr
   \Delta\Big(\! F_i^{(n)} \Big) = {\textstyle \sum\limits_{s=0}^n} F_i^{\,(n-s)} \otimes F_i^{\,(s)} L_i^{\,n-s}  \, ,  \;\;
   \epsilon\Big(\! F_i^{\,(n)} \Big) = \delta_{n,0}  \, ,  \;\;  \SS\Big(\! F_i^{\,(n)} \Big) = {(-1)}^n q_{ii}^{-{n \choose 2}} F_i^{\,(n)} L_i^{\,-n} }  $$
while for  $ K_i^{\pm 1} $,  $ G_i^{\pm 1} $  and  $ q $--binomial  coefficients
$ {\Big( {{K_i \, ; \, c} \atop h} \Big)}_{\!\!q} \, $,  $ {\Big( {{G_i \, ; \, c} \atop h} \Big)}_{\!\!q_{ii}} $
we have formulas like in  Lemma \ref{commut_q-bin-coeff},  with  $ M $  and  $ p $  replaced by  $ K_i $
and  $ q $  or by  $ G_i $  and  $ q_{ii} \, $.
 \vskip4pt
   (b)\;  In addition,  {\sl if  $ \, \bq $  is of strongly integral type},  say
 $ \; \bq = {\big(\, q_{ij\!} = q^{d_i t^+_{ij}} \! = q^{d_j t^-_{ij}} \big)}_{i, j \in I} \, $, \,
 then besides all the above formulas we also have
  $$  \displaylines{
   {\bigg(\! {{K_i \, ; \, c} \atop h} \bigg)}_{\!\!q_i} \! F_j^{\,(n)}  \, = \,  F_j^{\,(n)} \, {\bigg(\! {{K_i \, ; \, c - n \, t^+_{ij}} \atop h} \bigg)}_{\!\!q_i}  \; ,  \quad
    {\bigg(\! {{K_i \, ; \, c} \atop h} \bigg)}_{\!\!q_i} \! E_j^{\,(n)}  \, = \,  E_j^{\,(n)} \, {\bigg(\! {{K_i \, ; \, c + n \, t^+_{ij}} \atop h} \bigg)}_{\!\!q_i}
%%%%%%%
% \cr
%%%%%%%
 }  $$
  $$  \displaylines{
%%%%%
   \, {\bigg(\! {{L_i \, ; \, c} \atop h} \bigg)}_{\!\!q_i} \! F_j^{\,(n)}  \, = \,  F_j^{\,(n)} \, {\bigg(\! {{L_i \, ; \, c + n \, t^-_{ji}} \atop h} \bigg)}_{\!\!q_i}  \; ,  \quad \;
    {\bigg(\! {{L_i \, ; \, c} \atop h} \bigg)}_{\!\!q_i} \! E_j^{\,(n)}  \, = \,  E_j^{\,(n)} \, {\bigg(\! {{L_i \, ; \, c - n \, t^-_{ji}} \atop h} \bigg)}_{\!\!q_i}  }  $$
and formulas for the Hopf structure on  $ q $--binomial  coefficients  $ {\Big( {{K_i \, ; \, c} \atop n} \Big)}_{\!\!q_i} \, $  and
$ {\Big( {{L_i \, ; \, c} \atop n} \Big)}_{\!\!q_i} \, $  like in  Lemma \ref{commut_q-bin-coeff},  with  $ M $  and  $ p $
replaced by  $ K_i $  and  $ q_i $  or by  $ L_i $  and  $ q_i \, $.   \hfill
\end{lema}

\vskip7pt

   We can now extend Lusztig's definition of  ``{\sl restricted\/}  quantum universal enveloping algebra''.
   Indeed, a straightforward extension requires that  $ \bq $  be strongly integral; nevertheless,
   we consider also a more general definition when  $ \bq $  is just integral.

\vskip11pt

\begin{definition}  \label{def:int-form_hat-MpQG}
  Let  $ \, U_\bq(\hskip0,8pt\lieg) \, $  be a MpQG over the field  $ \Fbq $  as in  Definition \ref{def:multiqgroup_ang}.
  We define a bunch of  $ \Rbq $--subalgebras  of  $ U_\bq(\hskip0,8pt\lieg) \, $,  with a specific a set of generators,
as follows:
 \vskip4pt
   {\it (a)}\,  {\sl If  $ \bq $  is of integral type},  we set
  $$  \displaylines{
   \Uhatdot_\bq^{\,\raise-5pt\hbox{$ \scriptstyle - $}}  \; := \;  \Big\langle\, F_i^{\,(n)} \,\Big\rangle_{i \in I , \, n \in \NN}  \quad ,
   \qquad  \Uhatdot_\bq^{\,\raise-5pt\hbox{$ \scriptstyle + $}}  \; := \;
\Big\langle\, E_i^{\,(n)} \,\Big\rangle_{i \in I , \; n \in \NN}  \cr
   \Uhatdot_\bq^{\raise-5pt\hbox{$ \scriptstyle \,-,0 $}}  \, := \,  \bigg\langle\, L_i^{\pm 1} \, ,
   \, {\bigg( {L_i \atop n} \bigg)}_{\!\!q} \,\bigg\rangle_{\! i \in I \, , \, n \in \NN}  \;\; ,
     \quad  \Uhatdot_\bq^{\,\raise-5pt\hbox{$ \scriptstyle \leq $}}  \, := \,  \bigg\langle\, F_i^{\,(n)} \, ,
     \, L_i^{\pm 1} \, , \, {\bigg( {L_i \atop n} \bigg)}_{\!\!q} \,\bigg\rangle_{\! i \in I \, , \, n \in \NN}  \cr
%%%%%
   \Uhatdot_\bq^{\raise-5pt\hbox{$ \scriptstyle \,+,0 $}}  \, := \,  \bigg\langle\, K_i^{\pm 1} \, ,
   \, {\bigg( {K_i \atop n} \bigg)}_{\!\!q} \,\bigg\rangle_{\! i \in I \, , \, n \in \NN}  \;\; ,
     \quad  \Uhatdot_\bq^{\,\raise-5pt\hbox{$ \scriptstyle \geq $}}  \, := \,  \bigg\langle\, K_i^{\pm 1} \, ,
     \, {\bigg( {K_i \atop n} \bigg)}_{\!\!q} \, , \, E_i^{\,(n)} \,\bigg\rangle_{\! i \in I \, , \, n \in \NN}  \cr
   \Uhatdot_\bq  \, = \,  \Uhatdot_\bq(\hskip0,8pt\lieg)  \; := \;
 \bigg\langle\, \Uhatdot_\bq^{\,\raise-5pt\hbox{$ \scriptstyle 0 $}} \,\;{\textstyle \bigcup}\;
{\Big\{ F_i^{\,(n)} , \, E_i^{\,(n)} \Big\}}_{\! i \in I \, , \, n \in \NN} \,\bigg\rangle  }  $$
 \vskip4pt
   {\it (b)}\,  {\sl If  $ \bq $  is of  {\it strongly}  integral type},  we set
  $$  \displaylines{
   \Uhat_\bq^{\,-}  \; := \;  \Big\langle\, F_i^{\,(n)} \,\Big\rangle_{i \in I , \, n \in \NN}  \;\;
   \Big( = \;  \Uhatdot_\bq^{\raise-5pt\hbox{$ \scriptstyle \,- $}} \;\Big)  \quad ,
 \qquad  \Uhat_\bq^{\,+}  \; := \;  \Big\langle\, E_i^{\,(n)} \,\Big\rangle_{i \in I , \; n \in \NN}  \;\;
 \Big( = \;  \Uhatdot_\bq^{\raise-5pt\hbox{$ \scriptstyle \,+ $}} \;\Big)  \cr
   \Uhat_\bq^{\,-,0}  \; := \;  \bigg\langle\, L_i^{\pm 1} \, ,
\, {\bigg( {L_i \atop n} \bigg)}_{\!\!q_i} \,\bigg\rangle_{\! i \in I \, , \, n \in \NN}  \;\; ,
 \qquad  \Uhat_\bq^{\,\leq}  \; := \;  \bigg\langle\, F_i^{\,(n)} \, , \, L_i^{\pm 1} \, ,
\, {\bigg( {L_i \atop n} \bigg)}_{\!\!q_i} \,\bigg\rangle_{\! i \in I \, , \, n \in \NN}  \cr
   \Uhat_\bq^{\,+,0}  \; := \;  \bigg\langle\, K_i^{\pm 1} \, ,
\, {\bigg( {K_i \atop n} \bigg)}_{\!\!q_i} \,\bigg\rangle_{\! i \in I \, , \, n \in \NN}  \;\; ,
 \qquad  \Uhat_\bq^{\,\geq}  \; := \;  \bigg\langle\, K_i^{\pm 1} \, , \,
{\bigg( {K_i \atop n} \bigg)}_{\!\!q_i} \, , \, E_i^{\,(n)}
\,\bigg\rangle_{\! i \in I \, , \, n \in \NN}  \cr
   \Uhat_\bq  \, = \,  \Uhat_\bq(\hskip0,8pt\lieg)  \; := \;
 \bigg\langle\, \Uhat_\bq^{\,\raise-5pt\hbox{$ \scriptstyle 0 $}} \,\;{\textstyle \bigcup}\;
{\Big\{ F_i^{\,(n)} , \, E_i^{\,(n)} \Big\}}_{\! i \in I \, , \, n \in \NN} \,\bigg\rangle  }  $$
 \vskip5pt
   In the sequel, we shall refer to all these objects as to  {\sl restricted\/}  MpQG's.   \hfill  $ \diamondsuit $
\end{definition}

\bigskip

   The ``restricted'' MpQG's introduced in  Definition \ref{def:int-form_hat-MpQG}
   admit a presentation by generators and relations, which generalizes the one in the canonical case  (cf.\ \cite{DL}):

\medskip

\begin{theorem}  \label{thm:pres_Uhatgdq_gens-rels}  {\ }
 \vskip5pt
   (a)\,  Let  $ \, \bq := {\big(\hskip0,7pt q_{ij} = q^{\,b_{ij}} \big)}_{i,j \in I} \, $  be of integral type.
   Then  $ \, \Uhatdot_\bq = \Uhatdot_\bq(\hskip0,8pt\lieg) \, $  is (isomorphic to) the associative, unital
   $ \Rbq $--algebra  with the following presentation by generators and relations
   .  The generators are all elements of  $ \; \Uhatdot_\bq^{\,\raise-5pt\hbox{$ \scriptstyle 0 $}} $
   as well as all elements  $ \, F_i^{\,(n)} \, $,  $ \, E_i^{\,(n)} $  (for all  $ \, i \in I $,  $ n \in \NN$),
   and the relations holding true inside  $ \Uhatdot_\bq^{\,\raise-5pt\hbox{$ \scriptstyle 0 $}} $
   as well as the following ones:
  $$  K_i^{\pm 1} \, E_j^{\,(n)} \, = \, q^{\pm n b_{ij}} E_j^{\,(n)} \, K_i^{\pm 1}  \quad ,
\qquad  K_i^{\pm 1} \, F_j^{\,(n)} \, = \, q^{\mp n b_{ij}} F_j^{\,(n)} \, K_i^{\pm 1}  $$
  $$  L_i^{\pm 1} \, E_j^{\,(n)} \, = \, q^{\mp n b_{ji}} E_j^{\,(n)} \, L_i^{\pm 1}  \quad ,
\qquad  L_i^{\pm 1} \, F_j^{\,(n)} \, = \, q^{\pm n b_{ji}} F_j^{\,(n)} \, L_i^{\pm 1}  $$
  $$  G_i^{\pm 1} \, E_j^{\,(n)} \, = \, q_{ii}^{\pm n a_{ij}} E_j^{\,(n)} \, G_i^{\pm 1}  \quad ,  \qquad
  G_i^{\pm 1} \, F_j^{\,(n)} \, = \, q_{ii}^{\mp n a_{ij}} F_j^{\,(n)} \, G_i^{\pm 1}  $$
  $$  {\bigg( {{\!K_i \, ; \, c} \atop h} \bigg)}_{\!\!q} \, E_j^{\,(n)} = E_j^{\,(n)} {\bigg( {{\!K_i \, ; \, c + n b_{ij}\!} \atop h} \bigg)}_{\!\!q}  \; ,  \;\;
   {\bigg( {{\!K_i \, ; \, c} \atop h} \bigg)}_{\!\!q} \, F_j^{\,(n)} = F_j^{\,(n)} {\bigg( {{\!K_i \, ; \, c - n b_{ij}\!} \atop h} \bigg)}_{\!\!q}  $$
  $$  {\bigg( {{\!L_i \, ; \, c} \atop h} \bigg)}_{\!\!q} \, E_j^{\,(n)} = E_j^{\,(n)} {\bigg( {{\!L_i \, ; \, c - n b_{ji}\!} \atop h} \bigg)}_{\!\!q}  \; ,  \;\;
   {\bigg( {{\!L_i \, ; \, c} \atop h} \bigg)}_{\!\!q} \, F_j^{\,(n)} = F_j^{\,(n)} {\bigg( {{\!L_i \, ; \, c + n b_{ji}\!} \atop h} \bigg)}_{\!\!q}  $$
  $$  {\bigg(\! {{G_i \, ; \, c} \atop h} \bigg)}_{\!\!q_{ii}} \! E_j^{\,(n)} = E_j^{\,(n)} {\bigg(\! {{G_i \, ; \, c + n a_{ij}} \atop h} \!\bigg)}_{\!\!q_{ii}}  \, , \;\;
   {\bigg(\! {{G_i \, ; \, c} \atop h} \bigg)}_{\!\!q_{ii}} \! F_j^{\,(n)} = F_j^{\,(n)} {\bigg(\! {{G_i \, ; \, c - n a_{ij}} \atop h} \!\bigg)}_{\!\!q_{ii}}  $$
  $$  X_i^{\,(r)} X_i^{\,(s)} = {\left( r+s \atop r \right)}_{\!\!q_{ii}} X_i^{\,(r+s)}  \quad ,  \qquad  X_i^{\,(0)} = 1
  \eqno  \forall \;\; X \in \big\{ E \, , F \big\}  \qquad  $$
  $$  \hskip-9pt   \sum_{\hskip11pt r+s=1-a_{ij}} \hskip-19pt {(-1)}^s \, q_{ii}^{s \choose 2} q_{ij}^{\,s} \, E_i^{\,(r)} E_j^{\,(1)} E_i^{\,(s)} \, = \; 0  \;\; ,   \hskip-7pt   \sum_{\hskip11pt r+s=1-a_{ij}} \hskip-19pt {(-1)}^s \, q_{ii}^{r \choose 2} q_{ij}^{\,r} \, F_i^{\,(r)} F_j^{\,(1)} F_i^{\,(s)} \, = \; 0   \;\;\quad  \big(\, i \neq j \,\big)  $$
  $$  E_i^{\,(m)} \, F_i^{\,(n)}  \; = \,\;  \sum_{s=0}^{m \wedge n} \, F_i^{\,(n-s)} \, q_{ii}^{\,s} \,
  {\bigg( {{G_i \, ; \, 2\,s - m - n} \atop s} \bigg)}_{\!\!q_{ii}} \! L_i^{\,s} \, E_i^{\,(m-s)}  $$
  Moreover, with respect to this presentation  $ \Uhatdot_\bq $  is endowed with the Hopf algebra structure (over  $ \Rbq \, $)  uniquely given by
  $$  \Delta\Big(\! E_i^{(n)} \Big) = {\textstyle \sum\limits_{s=0}^n} E_i^{\,(n-s)} K_i^{\,s} \otimes E_i^{\,(s)}  \, ,
  \;\;  \epsilon\Big(\! E_i^{\,(n)} \Big) = \delta_{n,0}  \, ,  \;\;  \SS\Big(\! E_i^{\,(n)} \Big) = {(-1)}^n q_{ii}^{+{n \choose 2}} K_i^{\,-n} E_i^{\,(n)}  $$
  $$  \Delta\Big(\! F_i^{(n)} \Big) = {\textstyle \sum\limits_{s=0}^n} F_i^{\,(n-s)} \otimes F_i^{\,(s)} L_i^{\,s}  \, ,  \;\;
  \epsilon\Big(\! F_i^{\,(n)} \Big) = \delta_{n,0}  \, ,  \;\;  \SS\Big(\! F_i^{\,(n)} \Big) = {(-1)}^n q_{ii}^{-{n \choose 2}} F_i^{\,(n)} L_i^{\,-n}  $$
and formulas for  $ \Delta \, , \epsilon \, , S $  in  Lemma \ref{commut_q-bin-coeff}  with  $ \, (M,p\,) \in \big\{ (L_i\,,q) \, , (K_i\,,q) \, , (G_i\,,q_{ii}) \big\} \, $.
%
%   $$  \text{formulas for  $ \Delta \, , \epsilon \, , S $  in  Lemma \ref{commut_q-bin-coeff}  for}
%   \;\;  (M,p\,) \in \big\{ (L_i\,,q) \, , (K_i\,,q) \, , (G_i\,,q_{ii}) \big\} \; .  $$
%
 \vskip5pt
   (b)  Let  $ \, \bq := {\big(\hskip0,7pt q_{ij} = q^{\,d_i\,t_{ij}^+} = q^{\,d_j\,t_{ij}^-} \big)}_{i,j \in I} \, $
   be of  {\sl strongly}  integral type.  Then  $ \, \Uhat_\bq = \Uhat_\bq(\hskip0,8pt\lieg) \, $  is
   (isomorphic to) the Hopf algebra over\/  $ \Rbq $  with the following presentation by generators and relations.
   The generators are all elements of  $ \; \Uhat_\bq^{\,\raise-5pt\hbox{$ \scriptstyle 0 $}} $  as well as all
   elements  $ \, F_i^{\,(n)} \, $,  $ \, E_i^{\,(n)} \, $  (for all  $ \, i \in I \, $,  $ \, n \in \NN\, $),  \, and the relations are
  $$  K_i^{\pm 1} \, E_j^{\,(n)} \, = \, q_i^{\pm n t^+_{ij}} E_j^{\,(n)} \, K_i^{\pm 1}  \quad ,
\qquad  K_i^{\pm 1} \, F_j^{\,(n)} \, = \, q_i^{\mp n t^+_{ij}} F_j^{\,(n)} \, K_i^{\pm 1}  $$
  $$  L_i^{\pm 1} \, E_j^{\,(n)} \, = \, q_i^{\mp n t^-_{ji}} E_j^{\,(n)} \, L_i^{\pm 1}  \quad ,
\qquad  L_i^{\pm 1} \, F_j^{\,(n)} \, = \, q_i^{\pm n t^-_{ji}} F_j^{\,(n)} \, L_i^{\pm 1}  $$
  $$  G_i^{\pm 1} \, E_j^{\,(n)} \, = \, q_{ii}^{\pm n a_{ij}} E_j^{\,(n)} \, G_i^{\pm 1}  \quad ,  \qquad
   G_i^{\pm 1} \, F_j^{\,(n)} \, = \, q_{ii}^{\mp n a_{ij}} F_j^{\,(n)} \, G_i^{\pm 1}  $$
  $$  {\bigg( {{\!K_i \, ; \, c} \atop h} \bigg)}_{\!\!q_i} E_j^{\,(n)} = E_j^{\,(n)} {\bigg( {{\!K_i \, ; \, c + n t^+_{ij}} \atop h} \bigg)}_{\!\!q_i}  \; ,  \;\;
   {\bigg( {{\!K_i \, ; \, c} \atop h} \bigg)}_{\!\!q_i} F_j^{\,(n)} = F_j^{\,(n)} {\bigg( {{\!K_i \, ; \, c - n t^+_{ij}} \atop h} \bigg)}_{\!\!q_i}  $$
  $$  {\bigg( {{\!L_i \, ; \, c} \atop h} \bigg)}_{\!\!q_i} E_j^{\,(n)} = E_j^{\,(n)} {\bigg( {{\!L_i \, ; \, c - n t^-_{ji}} \atop h} \bigg)}_{\!\!q_i}  \; ,  \;\;
   {\bigg( {{\!L_i \, ; \, c} \atop h} \bigg)}_{\!\!q_i} F_j^{\,(n)} = F_j^{\,(n)} {\bigg( {{\!L_i \, ; \, c + n t^-_{ji}} \atop h} \bigg)}_{\!\!q_i}  $$
  $$  {\bigg( {{G_i \, ; \, c} \atop h} \bigg)}_{\!\!q_{ii}} \! E_j^{\,(n)} = E_j^{\,(n)} {\bigg( {{G_i \, ; \, c + n a_{ij}} \atop h} \!\bigg)}_{\!\!q_{ii}}  \, , \;\;
   {\bigg( {{G_i \, ; \, c} \atop h} \bigg)}_{\!\!q_{ii}} \! F_j^{\,(n)} = F_j^{\,(n)} {\bigg( {{G_i \, ; \, c - n a_{ij}} \atop h} \!\bigg)}_{\!\!q_{ii}}  $$
  $$  X_i^{\,(r)} X_i^{\,(s)} = {\left( r+s \atop r \right)}_{\!\!q_{ii}} X_i^{\,(r+s)}  \quad ,  \qquad  X_i^{\,(0)} = 1   \eqno
  \forall \;\; X \in \big\{ E \, , F \big\}  \qquad  $$
  $$  \hskip-9pt   \sum_{\hskip11pt r+s=1-a_{ij}} \hskip-19pt {(-1)}^s \, q_{ii}^{s \choose 2} q_{ij}^{\,s} \, E_i^{\,(r)} E_j^{\,(1)} E_i^{\,(s)} \, = \; 0  \;\; ,   \hskip-7pt   \sum_{\hskip11pt r+s=1-a_{ij}} \hskip-19pt {(-1)}^s \, q_{ii}^{r \choose 2} q_{ij}^{\,r} \, F_i^{\,(r)} F_j^{\,(1)} F_i^{\,(s)} \, = \; 0   \;\;\quad  \big(\, i \neq j \,\big)  $$
  $$  E_i^{\,(m)} \, F_i^{\,(n)}  \; = \,\;  \sum_{s=0}^{m \wedge n} \, F_i^{\,(n-s)} \, q_{ii}^{\,s} \,
  {\bigg( {{G_i \, ; \, 2\,s - m - n} \atop s} \bigg)}_{\!\!q_{ii}} \! L_i^{\,s} \, E_i^{\,(m-s)}  $$
endowed with the Hopf algebra structure (over  $ \Rbq \, $)  uniquely given by
  $$  \Delta\Big(\! E_i^{(n)} \Big) = {\textstyle \sum\limits_{s=0}^n} E_i^{\,(n-s)} K_i^{\,s} \otimes E_i^{\,(s)}  \, ,  \;\;
  \epsilon\Big(\! E_i^{\,(n)} \Big) = \delta_{n,0}  \, ,  \;\;  \SS\Big(\! E_i^{\,(n)} \Big) = {(-1)}^n q_{ii}^{+{n \choose 2}} K_i^{\,-n} E_i^{\,(n)}  $$
  $$  \Delta\Big(\! F_i^{(n)} \Big) = {\textstyle \sum\limits_{s=0}^n} F_i^{\,(n-s)} \otimes F_i^{\,(s)}
  L_i^{\,n-s}  \, ,  \;\;  \epsilon\Big(\! F_i^{\,(n)} \Big)
  = \delta_{n,0}  \, ,  \;\;  \SS\Big(\! F_i^{\,(n)} \Big) = {(-1)}^n q_{ii}^{-{n \choose 2}} F_i^{\,(n)} L_i^{\,-n}  $$
  $$  \text{formulas for  $ \Delta \, , \epsilon \, , S $  in  Lemma \ref{commut_q-bin-coeff}  for}
  \;\;  (M,p\,) \in \big\{ (L_i\,,q_i) \, , (K_i\,,q_i) \, , (G_i\,,q_{ii}) \big\} \; .  $$
 \vskip1pt

   (c)\;  $ \, \Uhatdot_\bq \, $  is a Hopf  $ \, \Rbq $--subalgebra  of  $ \; \Uhat_\bq \, $.
 \vskip7pt
   (d)\,  Similar statements occur for the various {\sl restricted}  multiparameter
 quantum subgroups considered in  Definition \ref{def:int-form_hat-MpQG}.
\end{theorem}

\pf
 Everything is proved like in the canonical case  (cf.\  \cite{DL}),  taking
 Lemma \ref{lemma:commut_q-div-pows}  into account,
 but for  {\it (c)},  which follows from definitions and  Proposition \ref{prop:struct-Uhatq0}.
\epf

\medskip

   As a first, direct consequence we have the following:

\medskip

\begin{prop}  \label{prop:Uhat_subHopf}  {\ }
 \vskip3pt
   (a)  $ \; \Uhatdot_\bq = \Uhatdot_\bq(\hskip0,8pt\lieg) \, $,  resp.\  $ \Uhatdot_\bq^{\raise-5pt\hbox{$ \scriptstyle \,\leq $}} \, $,
   resp.\  $ \Uhatdot_\bq^{\raise-5pt\hbox{$ \scriptstyle \,0 $}} \, $,  resp.\  $ \, \Uhatdot_\bq^{\raise-5pt\hbox{$ \scriptstyle \,-,0 $}} \, $,
   resp.\  $ \, \Uhatdot_\bq^{\raise-5pt\hbox{$ \scriptstyle \,0 $}} \, $,  resp.\  $ \, \Uhatdot_\bq^{\raise-5pt\hbox{$ \scriptstyle \,+,0 $}} \, $,
   resp.\  $ \Uhatdot_\bq^{\raise-5pt\hbox{$ \scriptstyle \,\geq $}} \, $,  is a Hopf  $ \, \Rbq $--subalgebra (hence, it is an  $ \, \Rbq $--integral  form,
   as a Hopf algebra) of  $ \, U_\bq(\hskip0,8pt\lieg) \, $,  resp.\ of  $ \, U_\bq^\leq \, $,  resp.\ of  $ \, U_\bq^{\,-,0} \, $,  resp.\ of  $ \, U_\bq^{\,0} \, $,
   resp.\ of  $ \, U_\bq^{\,+,0} \, $,  resp.\ of  $ \, U_\bq^\geq \, $.
                                                           \par
   Similarly,  $ \; \Uhatdot_\bq^{\raise-5pt\hbox{$ \scriptstyle \,- $}} \, $,  resp.\
   $ \, \Uhatdot_\bq^{\raise-5pt\hbox{$ \scriptstyle \,+ $}} \, $,  is an  $ \, \Rbq $--subalgebra
   --- hence, it is an  $ \, \Rbq $--integral  form, as an algebra ---   of  $ \, U_\bq^{\,-} \, $,  resp.\ of  $ \, U_\bq^{\,+} \, $.
 \vskip3pt
   (b)  \; If  $ \, \bq $  is strongly integral, similar results   --- as in  (a) and (b) ---
   hold true as well when  ``$ \; \Uhatdot $''  is replaced with  ``$ \; \Uhat $''.
\end{prop}

\pf
 Indeed,  Theorem \ref{thm:pres_Uhatgdq_gens-rels}  tells us that  $ \Uhatdot_\bq $  is a Hopf subalgebra
 (over  $ \Rbq \, $)  of  $ U_\bq \, $;  moreover, the scalar extension of  $ \, \Uhatdot_\bq \, $  from  $ \Rbq $  to
$ \, \Fbq \, $ yields  $ U_\bq $  as an $ \Fbq $--module,
 just by definition: thus  $ \Uhatdot_\bq $  is an integral  $ \Rbq $--form  of  $ U_\bq \, $,  as claimed.
 The same argument applies to  $ \Uhatdot_\bq^{\raise-5pt\hbox{$ \scriptstyle \,\leq $}} \, $,
 $ \Uhatdot_\bq^{\raise-5pt\hbox{$ \scriptstyle \,0 $}} \, $,  etc., as well as to
 $ \Uhat_\bq \, $,  $ \Uhat_\bq^{\raise-5pt\hbox{$ \scriptstyle \,\leq $}} \, $,  $ \Uhat_\bq^{\raise-5pt\hbox{$ \scriptstyle \,0 $}} \, $,  etc.
\epf

\vskip13pt

   An easier result, direct consequence of  Lemma \ref{lemma:commut_q-div-pows}  above,
   is the following, about the existence of ``triangular decompositions'' for these restricted MpQG's over  $ \Rbq \, $:

\vskip15pt

\begin{prop}  \label{prop:triang-decomps_Uhat}
 {\sl (triangular decompositions for restricted MpQG's)}
 \vskip3pt
 The multiplication in  $ \, \Uhatdot_\bq $  provides  $ \, \Rbq $--module  isomorphisms
  $$  \displaylines{
   \Uhatdot_\bq^{\raise-5pt\hbox{$ \scriptstyle \,-,0 $}} \mathop{\otimes}_{\,\Rbq} \Uhatdot_\bq^{\raise-5pt\hbox{$ \scriptstyle \,0 $}} \;
   \cong \; \Uhatdot_\bq^{\raise-5pt\hbox{$ \scriptstyle \,\leq $}} \; \cong \; \Uhatdot_\bq^{\raise-5pt\hbox{$ \scriptstyle \,0 $}}
   \mathop{\otimes}_{\,\Rbq} \Uhat_\bq^{\raise-5pt\hbox{$ \scriptstyle \,-,0 $}}  \quad ,
\qquad  \Uhatdot_\bq^{\raise-5pt\hbox{$ \scriptstyle \,+,0 $}} \mathop{\otimes}_{\,\Rbq} \Uhat_\bq^{\raise-5pt\hbox{$ \scriptstyle \,0 $}}
\; \cong \; \Uhatdot_\bq^{\raise-5pt\hbox{$ \scriptstyle \,\geq $}} \; \cong \; \Uhatdot_\bq^{\raise-5pt\hbox{$ \scriptstyle \,0 $}}
\mathop{\otimes}_{\,\Rbq} \Uhatdot_\bq^{\raise-5pt\hbox{$ \scriptstyle \,+,0 $}}  \cr
   \Uhatdot_\bq^{\raise-5pt\hbox{$ \scriptstyle \,+,0 $}} \mathop{\otimes}_{\,\Rbq} \Uhatdot_\bq^{\raise-5pt\hbox{$ \scriptstyle \,-,0 $}}  \,
   \cong \;  \Uhatdot_\bq^{\raise-5pt\hbox{$ \scriptstyle \,0 $}}  \,
   \cong \;  \Uhatdot_\bq^{\raise-5pt\hbox{$ \scriptstyle \,-,0 $}} \mathop{\otimes}_{\,\Rbq} \Uhatdot_\bq^{\raise-5pt\hbox{$ \scriptstyle \,+,0 $}}  \quad ,
 \qquad
  \Uhatdot_\bq^{\raise-5pt\hbox{$ \scriptstyle \,\leq $}} \, \mathop{\otimes}_{\,\Rbq} \, \Uhatdot_\bq^{\raise-5pt\hbox{$ \scriptstyle \,\geq $}}  \,
  \cong \;  \Uhatdot_\bq  \, \cong \;  \Uhatdot_\bq^{\raise-5pt\hbox{$ \scriptstyle \,\geq $}} \, \mathop{\otimes}_{\,\Rbq} \,
  \Uhatdot_\bq^{\raise-5pt\hbox{$ \scriptstyle \,\leq $}}
\cr
   \Uhatdot_\bq^{\raise-5pt\hbox{$ \scriptstyle \,+ $}} \!\mathop{\otimes}_{\,\Rbq}\! \Uhatdot_\bq^{\raise-5pt\hbox{$ \scriptstyle \,0 $}}
   \!\mathop{\otimes}_{\,\Rbq}\! \Uhatdot_\bq^{\raise-5pt\hbox{$ \scriptstyle \,- $}}  \, \cong \;  \Uhatdot_\bq  \, \cong \;
   \Uhatdot_\bq^{\raise-5pt\hbox{$ \scriptstyle \,- $}} \!\mathop{\otimes}_{\,\Rbq}\! \Uhatdot_\bq^{\raise-5pt\hbox{$ \scriptstyle \,0 $}}
   \!\mathop{\otimes}_{\,\Rbq}\! \Uhatdot_\bq^{\raise-5pt\hbox{$ \scriptstyle \,+ $}}  }  $$
 and similarly with  ``$ \; \Uhatdot $''  replaced by  ``$ \; \Uhat $''  if  $ \, \bq $  is strongly integral.
\end{prop}

\pf
 We consider the case of  $ \Uhatdot_\bq $  and of the left-hand side isomorphism, namely the case
 $ \; \Uhatdot_\bq^{\raise-5pt\hbox{$ \scriptstyle \,+ $}} \!\mathop{\otimes}_{\,\Rbq}\!
 \Uhatdot_\bq^{\raise-5pt\hbox{$ \scriptstyle \,0 $}} \!\mathop{\otimes}_{\,\Rbq}\!
 \Uhatdot_\bq^{\raise-5pt\hbox{$ \scriptstyle \,- $}} \cong \, \Uhatdot_\bq \; $,
 \, all other cases being similar.
                                                       \par
   By definition  $ \Uhatdot_\bq $  is spanned over  $ \Rbq $  by monomials whose factors can be
   freely chosen among the elements of  $ \Uhatdot_\bq^{\raise-5pt\hbox{$ \scriptstyle \,0 $}} \, $,
   the  $ F_i^{(m)} $'s  and the  $ E_j^{(n)} $'s;  moreover, thanks to
   Proposition \ref{prop:struct-Uhatq0}{\it (b)\/}  we can replace these monomials with
   other monomials, say  $ \mathcal{M} \, $,  in the  $ {\Big(\!{K_i \atop k} \!\Big)}_{\!\!q} $'s,
   the  $ K_i^{-\kappa} $'s,  the  $ {\Big({G_i \atop g_i} \!\Big)}_{\!\!q_{ii}} $'s,  the  $ G_j^{-\gamma} $'s,
 the  $ F_i^{(m)} $'s  and the  $ E_j^{(n)} $'s.
                                                       \par
   Now, by repeated use of the commutation relations among factors of this type given in
   Lemma \ref{lemma:commut_q-div-pows}   --- plus those stating that the  $ {\Big(\!{K_i \atop k} \!\Big)}_{\!\!q} $'s,
   the  $ K_i^{-\kappa} $'s,  the  $ {\Big({G_i \atop g_i} \!\Big)}_{\!\!q_{ii}} $'s  and  the  $ G_j^{-\gamma} $'s
   all commute with each other ---   (or by the corresponding relations given in
   Theorem \ref{thm:pres_Uhatgdq_gens-rels})  one can easily see that the following holds.
   Each one of these monomials, say  $ \mathcal{M} \, $,  can be expanded into an  $ \Rbq $--linear
   combination of new monomials, say  $ \mathcal{M}_s \, $,  of the same type but having the following
   additional property: each of them has the form
   $ \, \mathcal{M}_s = \mathcal{M}_s^+ \cdot \mathcal{M}_s^0 \cdot \mathcal{M}_s^- \, $,
   where  $ \mathcal{M}_s^+ $  is a monomial in the  $ E_j^{(n)} $'s,  $ \mathcal{M}_s^0 $
   is a monomial in the  $ {\Big(\!{K_i \atop k} \!\Big)}_{\!\!q} $'s,  the  $ K_i^{-\kappa} $'s,  the
   $ {\Big({G_i \atop g_i} \!\Big)}_{\!\!q_{ii}} $'s  and the  $ G_j^{-\gamma} $'s,  and  $ \mathcal{M}_s^- $
   is a monomial in the  $ F_j^{(m)} $'s.  This means that
 \vskip1pt
   \centerline{ $ \; \mathcal{M}_s = \mathcal{M}_s^+ \cdot \mathcal{M}_s^0 \cdot \mathcal{M}_s^- \,\; \in \;\,
\Uhatdot_\bq^{\raise-5pt\hbox{$ \scriptstyle \,+ $}} \cdot\, \Uhatdot_\bq^{\raise-5pt\hbox{$ \scriptstyle \,0 $}} \cdot\,
\Uhatdot_\bq^{\raise-5pt\hbox{$ \scriptstyle \,- $}} \; $ }
 \vskip4pt
\noindent
 hence the multiplication map
 $ \; \Uhatdot_\bq^{\raise-5pt\hbox{$ \scriptstyle \,+ $}} \!\mathop{\otimes}_{\,\Rbq}\!
 \Uhatdot_\bq^{\raise-5pt\hbox{$ \scriptstyle \,0 $}} \!\mathop{\otimes}_{\,\Rbq}\! \Uhatdot_\bq^{\raise-5pt\hbox{$ \scriptstyle \,- $}}  \!\!\relbar\joinrel\longrightarrow \Uhatdot_\bq \; $
is onto.  On the other hand, the PBW  Theorem \ref{thm:PBW_MpQG}  for  $ \, U_\bq =  \QEA \, $
directly implies that this map is 1:1 as well.
\epf

\medskip

   The previous result is improved by the following ``PBW Theorem'' for our restricted MpQG's
(and their quantum subgroups as well):

\medskip

\begin{theorem}  \label{thm:PBW_hat-MpQG}
 {\sl (PBW theorem for restricted quantum groups and subgroups}
 \vskip3pt
 Let quantum root vectors in  $ \QEq $  be fixed as in  \S \ref{rvec-MpQG}.  Then the following holds:
 \vskip5pt
   (a) \,  The set of ordered monomials
  $$  \bigg\{\; {\textstyle \prod\limits_{k=N}^1} F_{\beta^k}^{\,(f_k)} \;\bigg|\; f_k \in \NN \;\bigg\}  \quad ,  \qquad  \text{resp.} \quad
      \bigg\{\; {\textstyle \prod\limits_{h=1}^N} \, E_{\beta^h}^{\,(e_h)} \;\bigg|\; e_h \in \NN \;\bigg\}  \quad ,  $$
is an  $ \, \Rbq $--basis  of  $ \, \Uhatdot_\bq^{\raise-5pt\hbox{$ \scriptstyle \,- $}} \, $,  resp.\ of
$ \, \Uhatdot_\bq^{\raise-5pt\hbox{$ \scriptstyle \,+ $}} \, $;
in particular, both  $ \, \Uhatdot_\bq^{\raise-5pt\hbox{$ \scriptstyle \,- $}} $  and  $ \Uhatdot_\bq^{\raise-5pt\hbox{$ \scriptstyle \,+ $}} $
are free  $ \, \Rbq $--modules.  The same holds true for
$ \; \Uhat_\bq^{\,\pm} \; \Big(\! = \Uhatdot_\bq^{\raise-5pt\hbox{$ \scriptstyle \,\pm $}} \,\Big) \, $
in the strongly integral case.
 \vskip5pt
   (b) \,  The set of ordered monomials
  $$  \bigg\{\, {\textstyle \prod\limits_{j \in I}} \, {\bigg( {L_j \atop l_j} \bigg)}_{\!\!q} {L_j}^{- \lfloor l_j/2 \rfloor} \;\bigg|\; l_j \in \NN \,\bigg\}  \;\; ,
  \quad  \text{resp.\ }  \;\;
   \bigg\{\, {\textstyle \prod\limits_{i \in I}} \,
   {\bigg( {K_i \atop k_i} \bigg)}_{\!\!q} {K_i}^{- \lfloor k_i/2 \rfloor} \;\bigg|\; k_i \in \NN \,\bigg\}  $$
is an  $ \, \Rbq $--basis  of  $ \, \Uhatdot_\bq^{\raise-5pt\hbox{$ \scriptstyle \,-,0 $}} \, $,  resp.\ of
$ \, \Uhatdot_\bq^{\raise-5pt\hbox{$ \scriptstyle \,+,0 $}} \, $.  Similarly, the sets
  $$  \displaylines{
   \phantom{\text{and}}  \qquad
   \bigg\{\, {\textstyle \prod\limits_{j \in I}} \, {\bigg( {L_j \atop l_j} \bigg)}_{\!\!q} {L_j}^{- \lfloor l_j / 2 \rfloor} \,
   {\textstyle \prod\limits_{i \in I}} \, {\bigg( {G_i \atop g_i} \bigg)}_{\!\!q_{ii}} \! {G_i}^{- \lfloor g_i / 2 \rfloor} \;\bigg|\; l_j , g_i \in \NN \,\bigg\}  \cr
   \text{and}  \qquad
   \bigg\{\, {\textstyle \prod\limits_{i \in I}} \, {\bigg( {G_i \atop g_i} \bigg)}_{\!\!q_{ii}} \! {G_i}^{- \lfloor g_i / 2 \rfloor} \,
   {\textstyle \prod\limits_{j \in I}} \, {\bigg( {K_j \atop k_j} \bigg)}_{\!\!q} {K_j}^{- \lfloor k_j / 2 \rfloor} \;\bigg|\; g_i , k_j \in \NN \,\bigg\}  }  $$
are  $ \, \Rbq $--bases  of  $ \, \Uhatdot_\bq^{\raise-5pt\hbox{$ \scriptstyle \,0 $}} \, $.  In particular, all
$ \, \Uhatdot_\bq^{\raise-5pt\hbox{$ \scriptstyle \,-,0 $}} \, $,  $ \, \Uhatdot_\bq^{\raise-5pt\hbox{$ \scriptstyle \,+,0 $}} \, $
and  $ \, \Uhatdot_\bq^{\raise-5pt\hbox{$ \scriptstyle \,0 $}} \, $  are free  $ \, \Rbq $--modules.
 \vskip5pt
   (c) \,  The sets of ordered monomials
  $$  \displaylines{
   \phantom{resp.\ }   \hfill
   \bigg\{\; {\textstyle \prod\limits_{k=N}^1} F_{\beta^k}^{\,(f_k)}
  \, {\textstyle \prod\limits_{j \in I}} \, {\bigg( {L_j \atop l_j} \bigg)}_{\!\!q}
  {L_j}^{- \lfloor l_j/2 \rfloor} \;\bigg|\; f_k, l_j \in \NN \;\bigg\} \;\; \phantom{,}   \hfill  \cr
   \text{and\;\;\ }   \hfill
   \bigg\{\; {\textstyle \prod\limits_{j \in I}} \, {\bigg( {L_j \atop l_j} \bigg)}_{\!\!q}
  {L_j}^{- \lfloor l_j/2 \rfloor} \, {\textstyle \prod\limits_{k=N}^1} F_{\beta^k}^{\,(f_k)}
   \;\bigg|\; f_k, l_j \in \NN \;\bigg\} \;\; ,   \hfill
%%%
% \cr
%%%
 }  $$
  $$  \displaylines{
%%%
 \text{resp.\ }   \hfill
   \bigg\{\; {\textstyle \prod\limits_{j \in I}} \,
{\bigg( {K_j \atop k_j} \bigg)}_{\!\!q} {K_j}^{- \lfloor k_j/2 \rfloor} \,
{\textstyle \prod\limits_{h=1}^N} E_{\beta^h}^{\,(e_h)} \;\bigg|\; k_j, e_h \in \NN \;\bigg\} \;\; \phantom{,}   \hfill  \cr
   \text{and\;\;\ }   \hfill
   \bigg\{\; {\textstyle \prod\limits_{h=1}^N} E_{\beta^h}^{\,(e_h)} \, {\textstyle \prod\limits_{j \in I}} \,
   {\bigg( {K_j \atop k_j} \bigg)}_{\!\!q} {K_j}^{- \lfloor k_j/2 \rfloor} \;\bigg|\; k_j, e_h \in \NN \;\bigg\} \;\; ,   \hfill  }  $$
are  $ \, \Rbq $--bases  of  $ \, \Uhatdot_\bq^{\raise-5pt\hbox{$ \scriptstyle \,\leq $}} \, $,
resp.\  $ \, \Uhatdot_\bq^{\raise-5pt\hbox{$ \scriptstyle \,\geq $}} \, $;  in particular,  $ \, \Uhatdot_\bq^{\raise-5pt\hbox{$ \scriptstyle \,\leq $}} $
and  $ \, \Uhatdot_\bq^{\raise-5pt\hbox{$ \scriptstyle \,\geq $}} $  are free  $ \, \Rbq $--modules.
 \vskip5pt
   (d) \,  The sets of ordered monomials
  $$  \displaylines{
   \bigg\{\; {\textstyle \prod\limits_{k=N}^1} F_{\beta^k}^{\,(f_k)}
  \, {\textstyle \prod\limits_{j \in I}} \, {\bigg( {L_j \atop l_j} \bigg)}_{\!\!q}
{L_j}^{- \lfloor l_j/2 \rfloor}
  \, {\textstyle \prod\limits_{i \in I}} \, {\bigg( {G_i \atop g_i} \bigg)}_{\!\!q_{ii}} \! {G_i}^{- \lfloor g_i/2 \rfloor} \, {\textstyle \prod\limits_{h=1}^N} E_{\beta^h}^{\,(e_h)}
    \;\bigg|\; f_k, l_j, g_i, e_h \in \NN \;\bigg\}  \cr
   \bigg\{\; {\textstyle \prod\limits_{h=1}^N} E_{\beta^h}^{\,(e_h)}
  \, {\textstyle \prod\limits_{j \in I}} \, {\bigg( {L_j \atop l_j} \bigg)}_{\!\!q}
{L_j}^{- \lfloor l_j/2 \rfloor}
  \, {\textstyle \prod\limits_{i \in I}} \, {\bigg( {G_i \atop g_i} \bigg)}_{\!\!q_{ii}} \! {G_i}^{- \lfloor g_i/2 \rfloor} \, {\textstyle \prod\limits_{k=N}^1} F_{\beta^k}^{\,(f_k)}
    \;\bigg|\; f_k, l_j, g_i, e_h \in \NN \;\bigg\}  \cr
   \bigg\{\; {\textstyle \prod\limits_{k=N}^1} F_{\beta^k}^{\,(f_k)}
  \, {\textstyle \prod\limits_{i \in I}} \, {\bigg( {G_i \atop g_i} \bigg)}_{\!\!q_{ii}} \! {G_i}^{- \lfloor g_i/2 \rfloor}
  \, {\textstyle \prod\limits_{j \in I}} \, {\bigg( {K_j \atop k_j} \bigg)}_{\!\!q}
{K_j}^{- \lfloor k_j/2 \rfloor} \, {\textstyle \prod\limits_{h=1}^N}
  E_{\beta^h}^{\,(e_h)} \;\bigg|\; f_k, g_i, k_j, e_h \in \NN \;\bigg\}  \cr
   \bigg\{\; {\textstyle \prod\limits_{h=1}^N} E_{\beta^h}^{\,(e_h)}
  \, {\textstyle \prod\limits_{i \in I}} \, {\bigg( {G_i \atop g_i} \bigg)}_{\!\!q_{ii}} \! {G_i}^{- \lfloor g_i/2 \rfloor}
  \, {\textstyle \prod\limits_{j \in I}} \, {\bigg( {K_j \atop k_j} \bigg)}_{\!\!q}
{K_j}^{- \lfloor k_j/2 \rfloor} \, {\textstyle \prod\limits_{k=N}^1} F_{\beta^k}^{\,(f_k)}
    \;\bigg|\; f_k, g_i, k_j, e_h \in \NN \;\bigg\}  }  $$
are  $ \, \Rbq $--bases  of  $ \, \Uhatdot_\bq = \Uhatdot_\bq(\hskip0,8pt\lieg) \, $;  in particular,
$ \, \Uhatdot_\bq = \Uhatdot_\bq(\hskip0,8pt\lieg) \, $  is a free  $ \, \Rbq $--module.
 \vskip5pt
   (e) \,  In addition, when  $ \bq $  is strongly integral similar results   --- akin to (b), (c) and (d) ---   hold true if\,  ``$ \; \Uhatdot $''  is replaced
   with  ``$ \; \Uhat \, $''  and every  $ \, {\Big( {L_j \atop l_j} \!\Big)}_{\!q} \, $,  resp.\  $ \, {\Big( {K_j \atop k_j} \!\Big)}_{\!q} \, $,  is replaced
   with  $ \, {\Big( {L_j \atop l_j} \!\Big)}_{\!q_j} \, $,  resp.\  $ \, {\Big( {K_j \atop k_j} \!\Big)}_{\!q_j} \, $.
\end{theorem}

\pf
   {\it (a)}\,  It is a classical result, due to Lusztig, that the claim holds true for
   $ \Uhatdot_{\check{\bq}}^{\raise-5pt\hbox{$ \scriptstyle \,- $}} \, $,  i.e.\ the latter is free as an
   $ \Rbq $--module  with PBW-type basis given by the ordered monomials in the  $ {\check{F}}_\beta^{(f)} $'s,
   taken with respect to the product  ``$ \, \cdot \, $''  in $ \, \Uhatdot_{\check{\bq}}^{\raise-5pt\hbox{$ \scriptstyle \,- $}} \, $;
   the same monomials then form a PBW-like  $ \Rbqsq $--basis  of
   $ \; {\Big( \Uhatdot_{\check{\bq}}^{\raise-5pt\hbox{$ \scriptstyle \,- $}} \Big)}^{\!\scriptscriptstyle {\sqrt{\ }}} \!\! :=
   \Rbqsq \!\otimes_\Rbq \! \Uhatdot_{\check{\bq}}^{\raise-5pt\hbox{$ \scriptstyle \,- $}} \, $  as well.
                                                            \par
   Now, the formulas in  \S \ref{subsubsec:comp-formulas}  show that the above mentioned ``restricted'' PBW monomials in
   $ \, \Big( {\Big( \Uhatdot_{\check{\bq}}^{\raise-5pt\hbox{$ \scriptstyle \,- $}} \Big)}^{\!\scriptscriptstyle {\sqrt{\ }}} , \, \cdot \, \Big) \, $
   are proportional, by a coefficient which is a power of  $ q^{\pm 1/2} \, $  (hence invertible in the ground ring  $ \Rbqsq \, $)
   to their ``counterparts'' (with the same exponents for each root vector) in
   $ \; {\Big( \Uhatdot_\bq^{\raise-5pt\hbox{$ \scriptstyle \,- $}} \Big)}^{\!\scriptscriptstyle {\sqrt{\ }}} \! \subseteq
   {\big( \QEq \big)}^{\scriptscriptstyle {\sqrt{\ }}} = \, \Big(\! {\big( \QEqcheck \big)}^{\!\scriptscriptstyle {\sqrt{\ }}} \, , \,
   \smallast \,\Big) \, $,  \, i.e.\ with respect to the ``deformed'' product  ``$ \, \smallast \, $''  in
   $ \, {\big( \QEqcheck \big)}^{\!\scriptscriptstyle {\sqrt{\ }}} \!\! := \Rbqsq \!\otimes_\Rbq \! \QEqcheck  \, $.
   In other words, using notation of  \S \ref{subsubsec:comp-formulas}  we can write in short
 \vskip-10pt
  $$  {\textstyle \prod\limits_{k=N}^1} \check{F}_{\!\beta^k}^{\,\check{\cdot}\,(f_k)}  \; = \;
q^{\,z/2} {\textstyle \prod\limits_{k=N}^1} F_{\!\beta^k}^{\,*(f_k)}  $$
 \vskip-5pt
\noindent
 for some  $ \, z \in \ZZ \, $.  Therefore, as the PBW monomials
 $ \, {\textstyle \prod\limits_{k=N}^1} \check{F}_{\!\beta^k}^{\,\check{\cdot}\,(f_k)} \, $  form an
 $ \Rbqsq $--basis  of  $ \, {\Big( \Uhatdot_{\check{\bq}}^{\raise-5pt\hbox{$ \scriptstyle \,- $}} \Big)}^{\!\scriptscriptstyle {\sqrt{\ }}} \, $
 we can argue that the  $ \, {\textstyle \prod\limits_{k=N}^1} F_{\!\beta^k}^{\,*(f_k)} \, $'s  form an  $ \Rbqsq $--basis  of
 $ \, {\Big( \Uhatdot_\bq^{\raise-5pt\hbox{$ \scriptstyle \,- $}} \Big)}^{\!\scriptscriptstyle {\sqrt{\ }}} \, $  too.
                                          \par
   On the other hand, as direct consequence of  Theorem \ref{thm:PBW_MpQG}  we have that the same
   $ \, {\textstyle \prod\limits_{k=N}^1} F_{\!\beta^k}^{\,*(f_k)} \, $'s  also form a  $ \Fbq $--basis  of  $ \, U_\bq^{\raise-5pt\hbox{$ \scriptstyle \,- $}} \, $.
   Thus any  $ \, u_- \in {\Big( \Uhatdot_\bq^{\raise-5pt\hbox{$ \scriptstyle \,- $}} \Big)}^{\!\scriptscriptstyle {\sqrt{\ }}} \, $
   will have a unique expansion as  $ \Rbqsq $--linear  combination of the
   $ \, {\textstyle \prod\limits_{k=N}^1} F_{\!\beta^k}^{\,*(f_k)} \, $'s,  but also a unique expansion as  $ \Fbq $--linear
   combination of the same ``restricted'' PBW-like monomials.  Then the coefficients in both expansions must coincide,
   and since $ \, \Rbqsq \cap \Fbq = \Rbq \, $  they must belong to  $ \Rbq \, $;  so the
   $ \, {\textstyle \prod\limits_{k=N}^1} F_{\!\beta^k}^{\,*(f_k)} \, $'s  form an  $ \Rbq $--basis  of
   $ \, U_\bq^{\raise-5pt\hbox{$ \scriptstyle \,- $}} \, $,  as claimed.
                                          \par
   The same argument applies for the part of the claim concerning  $ \Uhatdot_\bq^{\raise-5pt\hbox{$ \scriptstyle \,+ $}} \, $.
 \vskip5pt
   {\it (b)}\,  This follows by construction together with  Proposition \ref{prop:basis_int-on-laur-mons}.
 \vskip5pt
   {\it (c)--(d)}\,  These follow at once from claims  {\it (a)--(b)\/}  along with the existence of triangular
   decompositions as given in  Proposition \ref{prop:triang-decomps_Uhat}.
 \vskip5pt
   {\it (e)}\,  This is proved by the same arguments used for claims  {\it (a)\/}  through  {\it (d)}.
\epf

\vskip9pt

\begin{rmk}
 It is worth stressing that the construction of restricted MpQG's does not ``match well'' with the
 process of cocycle deformation, even if one extends scalars from  $ \Rbq $  to  $ \Rbqsq $
 --- and from  $ \Fbq $  to  $ \Fbqsq $  accordingly.  In fact, if we label every MpQG's over  $ \Rbqsq $  or
 $ \Fbqsq $  by a superscript  ``$ \, {}^{\sqrt{\,\ }} \, $'',  what happens is that, although
 $ \; U_\bq^{\,\scriptscriptstyle{\sqrt{\,\ }}}\!(\hskip0,8pt\lieg) = \!
 {\big( U_{\check{\bq}}^{\,\scriptscriptstyle{\sqrt{\,\ }}}\!(\hskip0,8pt\lieg) \big)}_{\!\sigma} \! =
 U_{\check{\bq}}^{\,\scriptscriptstyle{\sqrt{\,\ }}}\!(\hskip0,8pt\lieg) \; $  as  $ \Rbqsq $--modules,
 for integral forms one has in general
  $ \; \Uhatdot_\bq^{\,\raise-3pt\hbox{$ \scriptscriptstyle{\sqrt{\,\ }} $}}\!(\hskip0,8pt\lieg) \! \not=
\Uhatdot_{\check{\bq}}^{\,\raise-3pt\hbox{$ \scriptscriptstyle{\sqrt{\,\ }} $}}\!(\hskip0,8pt\lieg) \; $
  as  $ \Rbqsq $--modules;  and  {\it a fortiori}  $ \; \Uhatdot_\bq(\hskip0,8pt\lieg) \not=
  \Uhatdot_{\check{\bq}}(\hskip0,8pt\lieg) \; $.
  Similarly holds for all ``quantum subgroups''.
                                                     \par
   In order to see that, let us consider an element of
   $ \, \Uhatdot_\bq^{\,\raise-3pt\hbox{$ \scriptscriptstyle{\sqrt{\,\ }} $}}\!(\hskip0,8pt\lieg) =
   \Big(\, \Uhatdot_{\check{\bq}}^{\,\raise-3pt\hbox{$ \scriptscriptstyle{\sqrt{\,\ }} $}}\!(\hskip0,8pt\lieg) \, ,
   \, \smallast \,\Big) \, $
  of the form  $ \; E_\alpha^{\,*(n)} \smallast \, {\Big(\! {K_j \atop m} \!\Big)}_{\!\!q}^{\!\!*} \; $:  \,
   from  \S \ref{subsubsec:comp-formulas}  we have the formula
  $$  E_\alpha^{\,*(n)} \smallast {\bigg(\! {K_j \atop m} \bigg)}_{\!\!q}^{\!\!*}  \! = \,  q_\alpha^{+{n \choose 2}} {\big( m^+_\alpha \big)}^{\!n}
  \sum_{c=0}^m q^{-c(m-c)} \prod_{s=1}^c {{\, q^{1-s} {\big( q_{\alpha\,\alpha_j}^{+1/2} \big)}^{\!n} \! - 1 \,} \over {\, q^{\,s} - 1 \,}} \,
  \check{E}_\alpha^{\;\check{\cdot}\, (n)} {\bigg(\! {K_j \atop {m\!-\!c}} \bigg)}_{\!\!q}^{\!\check{\cdot}} \! K_j^{\;\check{\cdot}\, c}  $$
here, the right-hand side term is the expansion of  $ \, E_\alpha^{\,*(n)} \smallast {\Big(\! {X_j \atop m} \Big)}_{\!\!p_j}^{\!\!*} \, $
into an  $ \Fbq $--linear combination of the elements of
  $$  \bigg\{\; {\textstyle \prod\limits_{h=1}^N} \check{E}_{\beta^h}^{\,\check{\cdot}(e_h)}
  \, {\textstyle \prod\limits_{i \in I}} \, {\textstyle {\Big( {G_i \atop g_i} \!\Big)}_{\!\!q_{ii}}} \! {G_i}^{- \lfloor g_i/2 \rfloor}
  \, {\textstyle \prod\limits_{j \in I}} \, {\textstyle {\Big( {K_j \atop k_j} \!\Big)}_{\!\!q}}
{K_j}^{- \lfloor k_j/2 \rfloor} \, {\textstyle \prod\limits_{k=N}^1} \check{F}_{\beta^k}^{\,\check{\cdot}\,(f_k)}
    \;\bigg|\; f_k \, , g_i \, , k_j \, , e_h \in \NN \;\bigg\}  $$
which, being an  $ \Rbqsq $--basis  of  $ \, \Uhatdot_\bq^{\,\raise-3pt\hbox{$ \scriptscriptstyle{\sqrt{\,\ }} $}}\!(\hskip0,8pt\lieg) \, $
--- by  Theorem \ref{thm:PBW_hat-MpQG}{\it (d)\/}  above ---   is also an  $ \Fbqsq $--basis  of
$ \, \Fbqsq \otimes_\Rbqsq \Uhatdot_\bq^{\,\raise-3pt\hbox{$ \scriptscriptstyle{\sqrt{\,\ }} $}}\!(\hskip0,8pt\lieg) =
U_\bq^{\,\raise3pt\hbox{$ \scriptscriptstyle{\sqrt{\,\ }} $}}\!(\hskip0,8pt\lieg) \, $.
Now, in the above expansion the coefficients
  $$  q_\alpha^{+{n \choose 2}} {\big( m^+_\alpha \big)}^{\!n} \, q^{-c(m-c)} \prod_{s=1}^c {{\, q^{1-s}
  {\big( q_{\alpha\,\alpha_j}^{+1/2} \big)}^{\!n} \! - 1 \,} \over {\, q^{\,s} - 1 \,}}  $$
in general  {\sl do not belong to  $ \Rbqsq \, $}: therefore, we have
 $ \; E_\alpha^{\,*(n)} \smallast \, {\Big(\! {K_j \atop m} \!\Big)}_{\!\!q}^{\!\!*} \not\in
\Uhatdot_{\check{\bq}}^{\,\raise-3pt\hbox{$ \scriptscriptstyle{\sqrt{\,\ }} $}}\!(\hskip0,8pt\lieg) \;  $
 whereas  $ \; E_\alpha^{\,*(n)} \smallast \, {\Big(\! {K_j \atop m} \!\Big)}_{\!\!q}^{\!\!*} \in
\Uhatdot_\bq^{\,\raise-3pt\hbox{$ \scriptscriptstyle{\sqrt{\,\ }} $}}\!(\hskip0,8pt\lieg) \; $  by definition.
This shows that  $ \; \Uhatdot_\bq^{\,\raise-3pt\hbox{$ \scriptscriptstyle{\sqrt{\,\ }} $}}\!(\hskip0,8pt\lieg) \not=
\Uhatdot_{\check{\bq}}^{\,\raise-3pt\hbox{$ \scriptscriptstyle{\sqrt{\,\ }} $}}\!(\hskip0,8pt\lieg) \; $  inside
$ \, \QEqsq = \QEqchecksq \, $  (as  $ \Fbq $--modules); in fact, it even proves that
$ \, {\Big( \Uhatdot_\bq^{\raise-5pt\hbox{$ \scriptstyle \,\geq $}} \Big)}^{\!\raise1pt\hbox{$ \scriptscriptstyle{\sqrt{\,\ }} $}} \!\! \not=
{\Big( \Uhatdot_{\check{\bq}}^{\raise-5pt\hbox{$ \scriptstyle \,\geq $}} \Big)}^{\!\raise1pt\hbox{$ \scriptscriptstyle{\sqrt{\,\ }} $}} \, $,
and similarly one shows that  $ \; {\Big( \Uhatdot_\bq^{\raise-5pt\hbox{$ \scriptstyle \,\leq $}} \Big)}^{\!\raise1pt\hbox{$ \scriptscriptstyle{\sqrt{\,\ }} $}}
\!\! \not= {\Big( \Uhatdot_{\check{\bq}}^{\raise-5pt\hbox{$ \scriptstyle \,\leq $}} \Big)}^{\!\raise1pt\hbox{$ \scriptscriptstyle{\sqrt{\,\ }} $}} \; $
 too.
\end{rmk}

\medskip

\subsection{Integral forms of ``unrestricted'' type}  \label{Utilde}  \
 \vskip7pt
   Beside Lusztig's ``restricted'' integral form, a second integral form of  $ U_q(\lieg) $
   was introduced by De Concini, Kac and Procesi: the ground ring in that case was  $ \k\big[\,q,q^{-1}\big] \, $,
   but one can easily prove   --- using the analogue (in that context) of  Proposition \ref{duality_x_PBW-bases}  ---
   that their definition does work the same over  $ \Zqqm $  too, so it yields  {\sl an integral form over  $ \, \Zqqm \, $},
   with suitable PBW-like basis, etc. Their construction can be easily extended to  $ \QEqcheck $;
   hereafter, we extend this (obvious) generalization to any MpQG such as  $ \QEq \, $.
                                                                           \par
   Let us fix a multiparameter matrix  $ \, \bq := {\big(\, q_{ij} \big)}_{i,j \in I} \, $  and the corresponding MpQG
   $ \QEq $  as in  \S \ref{def-MpQG};  then fix the special parameter  $ q $  (depending on $ \bq \, $)  and the
   ``canonical'' multiparameter  $ \, \check{\bq} := {\big(\, \check{q}_{ij} := q^{\, d_i a_{ij}} \big)}_{i,j \in I} \, $  as in
   \S \ref{deform-MpQG}.  Finally, assume quantum root vectors  $ E_\alpha \, $,  $ F_\alpha \, $
   (for all  $ \, \alpha \in \Phi^+ \, $)  have been fixed, as in  \S \ref{rvec-MpQG},
   and consider for them the following ``renormalizations'' (where  $ q_{\alpha\,\alpha} $  is defined as in
   \S \ref{deform-MpQG})
\begin{equation}  \qquad
  \quad   \Ebar_\alpha  \, := \,  \big(\, q_{\alpha\,\alpha} \! - 1 \big) \, E_\alpha  \quad ,
  \qquad   \Fbar_\alpha  \, := \,  \big(\, q_{\alpha\,\alpha} \! - 1 \big) \, F_\alpha
\quad \qquad  \forall \;\; \alpha \in \Phi^+   \qquad
\end{equation}

\smallskip

   Mimicking the construction in  \cite{DP},  we introduce the following definition:

\medskip

\begin{definition}  \label{def:int-form_tilde-Uqqm}
 For any multiparameter  $ \, \bq := {\big(\, q_{ij} \big)}_{i,j \in I} \, $  as in  \S \ref{multiparameters},
 fix modified quantum root vectors  $ \Ebar_\alpha $  and  $ \Fbar_\alpha $
 (for all  $ \, \alpha \in \Phi^+ \, $)  of  $ \QEq $  as above.  Then define in  $ \QEq $  the following  $ \Rbq $--subalgebras:
  $$  \displaylines{
   \Utilde_\bq^{\,-}  \; := \;  \big\langle\, \Fbar_\alpha \,\big\rangle_{\alpha \in \Phi^+}
\quad ,  \qquad  \Utilde_\bq^{\,0}  \; := \;  \big\langle\, {L_i}^{\pm 1} , \, K_i^{\,\pm 1} \,\big\rangle_{i \in I}  \quad ,
\qquad  \Utilde_\bq^{\,+}  \; := \;  \big\langle\, \Ebar_\alpha \,\big\rangle_{\alpha \in \Phi^+}  \cr
   \Utilde_\bq^{\,\leq}  \; := \;  \big\langle\, \Fbar_\alpha \, , \, {L_i}^{\pm 1} \,\big\rangle_{\alpha \in \Phi^+ , \, i \in I}  \;\; \quad ,  \;\;
   \qquad  \Utilde_\bq^{\,\geq}  \; := \;  \big\langle\, K_i^{\,\pm 1} , \, \Ebar_\alpha \,\big\rangle_{i \in I , \, \alpha \in \Phi^+}  \cr
%
%%%%%
%    \Utilde_\bq^{\,-,0}  \, := \;  \big\langle {L_i}^{\pm 1} \big\rangle_{i \in I}  \,\; ,  \;\;
% \Utilde_\bq  \, := \;  \big\langle\, \Fbar_\alpha \, , \, {L_i}^{\pm 1} , \, K_i^{\,\pm 1} ,
% \, \Ebar_\alpha \,\big\rangle_{i \in I , \, \alpha \in \Phi^+}  \,\; ,  \;\;
% \Utilde_\bq^{\,+,0}  \, := \;  \big\langle K_i^{\,\pm 1} \big\rangle_{i \in I}  \cr
%%%%%
%
   \Utilde_\bq^{\,-,0}  \! :=  \big\langle {L_i}^{\pm 1} \big\rangle_{i \in I}  \, ,  \,\;
 \Utilde_\bq(\lieg) = \Utilde_\bq  := \,  \big\langle\, \Fbar_\alpha \, , \, {L_i}^{\pm 1} , \, K_i^{\,\pm 1} ,
 \, \Ebar_\alpha \,\big\rangle_{i \in I , \, \alpha \in \Phi^+}  \, ,  \,\;
 \Utilde_\bq^{\,+,0}  \! :=  \big\langle K_i^{\,\pm 1} \big\rangle_{i \in I}  }  $$
   \indent   In the following, we shall refer to this kind of MpQG as  {\sl unrestricted}.   \hfill  $ \diamondsuit $
\end{definition}

\vskip9pt

   Contrary to the case of restricted integral forms, if we extend scalars to  $ \Rbqsq $
   then all unrestricted ones are indeed  $ 2 $--cocycle  deformations of their canonical counterparts,
   just like it happens with MpQG's over  $ \Fbq \, $.  This follows from direct analysis through the formulas in
   \S \ref{subsubsec:comp-formulas},  as the following shows:

\vskip11pt

\begin{prop}  \label{prop:Utilde-cocy-def}
 The Hopf algebra  $ \; \Utilde_\bq^{\,\raise-3pt\hbox{$ \scriptscriptstyle{\sqrt{\,\ }} $}} \! =
 \Utilde_\bq^{\,\raise-3pt\hbox{$ \scriptscriptstyle{\sqrt{\,\ }} $}}\!(\hskip0,8pt\lieg) := \Rbqsq \! \otimes_\Rbq \! \Utilde_\bq(\hskip0,8pt\lieg) \; $
 is a\/  $ 2 $--cocycle  deformation of its canonical counterpart, namely
  $$  \Utilde_\bq^{\,\raise-3pt\hbox{$ \scriptscriptstyle{\sqrt{\,\ }} $}}\!(\hskip0,8pt\lieg)  \;
  = \;  {\big(\, \Utilde_{\check{\bq}}^{\,\raise-3pt\hbox{$ \scriptscriptstyle{\sqrt{\,\ }} $}}\!(\hskip0,8pt\lieg) \big)}_\sigma  \; =
  \;  {\big(\, \Utilde_{\check{\bq}}^{\,\raise-3pt\hbox{$ \scriptscriptstyle{\sqrt{\,\ }} $}}\!(\hskip0,8pt\lieg) \big)}^{\!(\tilde{\varphi})}  $$
(see  Theorem \ref{thm:sigma_2-cocy}  and  Proposition \ref{prop:def-sigma=def-c}  for notation).
 Similarly   --- using a superscript  ``$ \, {}^{\sqrt{\,\ }} \, $''  to denote scalar extension to  $ \Rbqsq \, $  ---
 $ \, {\big( \Utilde_\bq^\leq \big)}^{\!\scriptscriptstyle{\sqrt{\,\ }}} $,  resp.\
$ \, {\big( \Utilde_\bq^{\,0} \big)}^{\!\scriptscriptstyle{\sqrt{\,\ }}} $,  resp.\
$ \, {\big( \Utilde_\bq^\geq \big)}^{\!\scriptscriptstyle{\sqrt{\,\ }}} $,  is a  $ 2 $--cocycle  deformation of
$ \; {\big( \Utilde_{\check{\bq}}^\geq \big)}^{\!\scriptscriptstyle{\sqrt{\,\ }}} $,  resp.\ of
$ \; {\big( \Utilde_{\check{\bq}}^{\,0} \big)}^{\!\scriptscriptstyle{\sqrt{\,\ }}} $,  resp.\ of  $ \; {\big( \Utilde_{\check{\bq}}^\leq \big)}^{\!\scriptscriptstyle{\sqrt{\,\ }}} $.
%
%%%%%
%  In particular, we have  $ \, \Utilde_\bq^{\,
% \raise-3pt\hbox{$ \scriptscriptstyle{\sqrt{\,\ }} $}}\!(\hskip0,8pt\lieg) =
% \Utilde_{\check{\bq}}^{\,\raise-3pt\hbox{$ \scriptscriptstyle{\sqrt{\,\ }} $}}\!
% (\hskip0,8pt\lieg) \, $,  $ \, {\big( \Utilde_\bq^{\,\leq} \big)}^{\!
% \scriptscriptstyle{\sqrt{\,\ }}} \! = {\big( \Utilde_{\check{\bq}}^{\,\leq}
% \big)}^{\!\scriptscriptstyle{\sqrt{\,\ }}} $,
% $ \, {\big( \Utilde_\bq^{\,0} \big)}^{\!\scriptscriptstyle{\sqrt{\,\ }}} \! =
% {\big( \Utilde_{\check{\bq}}^{\,0} \big)}^{\!\scriptscriptstyle{\sqrt{\,\ }}} $
%  and  $ \, {\big( \Utilde_\bq^{\,\geq} \big)}^{\!\scriptscriptstyle{\sqrt{\,\ }}}
% \! = {\big( \Uhat_{\check{\bq}}^{\,\geq} \big)}^{\!\scriptscriptstyle{\sqrt{\,\ }}} $
% as  $ \, \Rbqsq $--coalgebras.  Furthermore,  $ \; {\big( \Utilde_\bq^{\,-}
% \big)}^{\!\scriptscriptstyle{\sqrt{\,\ }}} \! = {\big( \Utilde_{\check{\bq}}^{\,-}
% \big)}^{\!\scriptscriptstyle{\sqrt{\,\ }}} \, $  and  $ \; {\big( \Utilde_\bq^{\,+}
% \big)}^{\!\scriptscriptstyle{\sqrt{\,\ }}} \! = {\big( \Utilde_{\check{\bq}}^{\,+}
% \big)}^{\!\scriptscriptstyle{\sqrt{\,\ }}} \, $  as  $ \Rbqsq $--modules.
%%%%%
%
 In particular,
 $ \, \Utilde_\bq^{\,\raise-3pt\hbox{$ \scriptscriptstyle{\sqrt{\,\ }} $}}\!(\hskip0,8pt\lieg) = \Utilde_{\check{\bq}}^{\,\raise-3pt\hbox{$ \scriptscriptstyle{\sqrt{\,\ }} $}}\!(\hskip0,8pt\lieg) \, $,
 $ \, {\big( \Utilde_\bq^{\,\leq} \big)}^{\!\scriptscriptstyle{\sqrt{\,\ }}} \! = {\big( \Utilde_{\check{\bq}}^{\,\leq} \big)}^{\!\scriptscriptstyle{\sqrt{\,\ }}} $,
 $ \, {\big( \Utilde_\bq^{\,0} \big)}^{\!\scriptscriptstyle{\sqrt{\,\ }}} \! = {\big( \Utilde_{\check{\bq}}^{\,0} \big)}^{\!\scriptscriptstyle{\sqrt{\,\ }}} $
 and  $ \, {\big( \Utilde_\bq^{\,\geq} \big)}^{\!\scriptscriptstyle{\sqrt{\,\ }}} \! =
 {\big( \Uhat_{\check{\bq}}^{\,\geq} \big)}^{\!\scriptscriptstyle{\sqrt{\,\ }}} $  as  $ \, \Rbqsq $--coalgebras,
 and  $ \; {\big( \Utilde_\bq^{\,\pm} \big)}^{\!\scriptscriptstyle{\sqrt{\,\ }}} \! =
{\big( \Utilde_{\check{\bq}}^{\,\pm} \big)}^{\!\scriptscriptstyle{\sqrt{\,\ }}} \, $  as  $ \Rbqsq $--modules.
                                                            \par
   It follows that all of  $ \, \Utilde_\bq^{\,\scriptscriptstyle{\sqrt{\,\ }}} $,
   $ \, {\big( \Utilde_\bq^{\,\leq} \big)}^{\!\scriptscriptstyle{\sqrt{\,\ }}} $,
   $ {\big( \Utilde_\bq^{\,0} \big)}^{\!\scriptscriptstyle{\sqrt{\,\ }}} $,
   $ {\big( \Utilde_\bq^{\,\geq} \big)}^{\!\scriptscriptstyle{\sqrt{\,\ }}} $,
   $ {\big( \Utilde_\bq^{\,-} \big)}^{\!\scriptscriptstyle{\sqrt{\,\ }}} $  and
   $ \, {\big( \Utilde_\bq^{\,+} \big)}^{\!\scriptscriptstyle{\sqrt{\,\ }}} $  are independent of the choice of quantum root vectors  $ E_\beta $  and  $ F_\beta $  (for  $ \, \beta \in \Phi^+ $).
%%%%%
% occurring in their definition.
%%%%%
%
%
\end{prop}

\pf
 The same analysis as in the proof of  Theorem \ref{thm:PBW_hat-MpQG}{\it (a)\/}  shows
 --- looking at the proper formulas from  \S \ref{subsubsec:comp-formulas}  ---   that
the identities
  $ \; {\big( \Utilde_\bq^{\,-} \big)}^{\!\scriptscriptstyle{\sqrt{\,\ }}} \! =
  {\big( \Utilde_{\check{\bq}}^{\,-} \big)}^{\!\scriptscriptstyle{\sqrt{\,\ }}} $,
  $ \; {\big( \Utilde_\bq^{\,+} \big)}^{\!\scriptscriptstyle{\sqrt{\,\ }}} \! =
  {\big( \Utilde_{\check{\bq}}^{\,+} \big)}^{\!\scriptscriptstyle{\sqrt{\,\ }}} $  hold as
  $ \Rbqsq $--modules,  and  $ \, {\big( \Utilde_\bq^{\,0} \big)}^{\!\scriptscriptstyle{\sqrt{\,\ }}} \! =
  {\big( \Utilde_{\check{\bq}}^{\,0} \big)}^{\!\scriptscriptstyle{\sqrt{\,\ }}} $  as  $ \, \Rbqsq $--coalgebras;
  more precisely, the latter identity can be read as
  $ \, {\big( \Utilde_\bq^{\,0} \big)}^{\!\scriptscriptstyle{\sqrt{\,\ }}} \! =
  {\Big(\! {\big(\, \Utilde_{\check{\bq}}^{\,0} \big)}^{\!\scriptscriptstyle{\sqrt{\,\ }}} \,\Big)}_{\!\sigma} =
  {\Big(\! {\big(\, \Utilde_{\check{\bq}}^{\,0} \big)}^{\!\scriptscriptstyle{\sqrt{\,\ }}} \,\Big)}^{\!(\tilde{\varphi})} \, $,
  by the very definitions and thanks to  Theorem \ref{thm:sigma_2-cocy}  and  Proposition \ref{prop:def-sigma=def-c}.
                                                            \par
   The same argument proves also
   $ \, \Utilde_\bq^{\,\scriptscriptstyle{\sqrt{\,\ }}}\!(\hskip0,8pt\lieg) = \Utilde_{\check{\bq}}^{\,\scriptscriptstyle{\sqrt{\,\ }}}\!(\hskip0,8pt\lieg) \, $,
   $ \, {\big( \Utilde_\bq^{\,\leq} \big)}^{\!\scriptscriptstyle{\sqrt{\,\ }}} \! = {\big( \Utilde_{\check{\bq}}^{\,\leq} \big)}^{\!\scriptscriptstyle{\sqrt{\,\ }}} \, $
   and  $ \, {\big( \Utilde_\bq^{\,\geq} \big)}^{\!\scriptscriptstyle{\sqrt{\,\ }}} \! = {\big( \Uhat_{\check{\bq}}^{\,\geq} \big)}^{\!\scriptscriptstyle{\sqrt{\,\ }}} $
   as  $ \, \Rbqsq $--coalgebras;  more precisely, one has
   $ \, \Utilde_\bq^{\,\scriptscriptstyle{\sqrt{\,\ }}}\!(\hskip0,8pt\lieg) =
   {\big(\, \Utilde_{\check{\bq}}^{\,\scriptscriptstyle{\sqrt{\,\ }}}\!(\hskip0,8pt\lieg) \big)}_\sigma =
   {\big(\, \Utilde_{\check{\bq}}(\hskip0,8pt\lieg) \big)}^{(\tilde{\varphi})} \, $,
   $ \; {\big( \Utilde_\bq^{\,\leq} \big)}^{\!\scriptscriptstyle{\sqrt{\,\ }}} \! =
   {\Big(\! {\big(\, \Utilde_{\check{\bq}}^{\,\leq} \big)}^{\!\scriptscriptstyle{\sqrt{\,\ }}} \,\Big)}_{\!\sigma} \! =
   {\Big(\! {\big(\, \Utilde_{\check{\bq}}^{\,\leq} \big)}^{\!\scriptscriptstyle{\sqrt{\,\ }}} \,\Big)}^{\!(\tilde{\varphi})} \, $
   and  $ \; {\big( \Utilde_\bq^{\,\geq} \big)}^{\!\scriptscriptstyle{\sqrt{\,\ }}} =
   {\Big(\! {\big(\, \Utilde_{\check{\bq}}^{\,\geq} \big)}^{\!\scriptscriptstyle{\sqrt{\,\ }}} \,\Big)}_{\!\sigma} =
   {\Big(\! {\big(\, \Utilde_{\check{\bq}}^{\,\geq} \big)}^{\!\scriptscriptstyle{\sqrt{\,\ }}} \,\Big)}^{\!(\tilde{\varphi})} \, $.
\epf

\vskip9pt

   As for restricted MpQG's, we have a PBW theorem for unrestricted ones too:

\vskip13pt

\begin{theorem}  \label{thm:PBW_tilde-MpQG}
 {\sl (PBW theorem for unrestricted MQG's   --- and subgroups)}
 \vskip4pt
   (a) \,  The set of ordered monomials
  $$  \bigg\{\; {\textstyle \prod\limits_{k=N}^1} \Fbar_{\beta^k}^{\;f_k} \;\bigg|\; f_k \in \NN \;\bigg\}  \quad ,  \qquad  \text{resp.} \quad
      \bigg\{\; {\textstyle \prod\limits_{h=1}^N} \, \Ebar_{\beta^h}^{\;e_h} \;\bigg|\; e_h \in \NN \;\bigg\}  \quad ,  $$
is an  $ \, \Rbq $--basis  of  $ \, \Utilde_\bq^- \, $,  resp.\ of  $ \, \Utilde_\bq^+ \, $;  in particular, both these are free  $ \, \Rbq $--modules.
 \vskip4pt
   (b) \,  The set of ordered monomials
  $$  \bigg\{\, {\textstyle \prod\limits_{j \in I}} {L_j}^{a_j} \;\bigg|\; a_j \in \ZZ \,\bigg\}
\, ,  \;\;  \text{resp.\ }  \;
   \bigg\{\, {\textstyle \prod\limits_{i \in I}} {K_i}^{b_i} \;\bigg|\; b_i \in \ZZ \,\bigg\}
\, ,  \;\;  \text{resp.\ } \;
   \bigg\{\; {\textstyle \prod\limits_{j \in I}} {L_j}^{a_j} {K_i}^{b_i} \;\bigg|\; a_j, b_i \in \ZZ \,\bigg\}  $$
is an  $ \, \Rbq $--basis  of  $ \, \Utilde_\bq^{\,-,0} \, $,  resp.\ of  $ \, \Utilde_\bq^{\,+,0} \, $,  resp.\ of  $ \, \Utilde_\bq^{\,0} \, $,  hence all these are free  $ \, \Rbq $--modules.
 \vskip4pt
   (c) \,  The sets of ordered monomials
  $$  \displaylines{
   \phantom{\text{resp.\ }}  \quad  {\bigg\{\; {\textstyle \prod\limits_{k=N}^1} \Fbar_{\beta^k}^{\;f_k}
  \, {\textstyle \prod\limits_{j \in I}} {L_j}^{a_j} \bigg\}}_{\! f_k \in \NN \, , \, a_j \in \ZZ}  \quad  \text{and\ }
  \qquad  {\bigg\{\; {\textstyle \prod\limits_{j \in I}} {L_j}^{a_j}
  {\textstyle \prod\limits_{k=N}^1} \Fbar_{\beta^k}^{\;f_k} \bigg\}}_{\! f_k \in \NN \, , \, a_j \in \ZZ} \;\; ,  \cr
   \text{resp.\ }  \quad  {\bigg\{\; {\textstyle \prod\limits_{i \in I}} \, {K_i}^{b_i} \, {\textstyle \prod\limits_{h=1}^N}
   \Ebar_{\beta^h}^{\;e_h} \bigg\}}_{\! b_i \in \ZZ \, , \, e_h \in \NN}  \quad  \text{and\ }  \qquad  {\bigg\{\; {\textstyle \prod\limits_{h=1}^N}
   \Ebar_{\beta^h}^{\;e_h} \, {\textstyle \prod\limits_{i \in I}} \, {K_i}^{b_i} \bigg\}}_{\! b_i \in \ZZ \, , \, e_h \in \NN} \;\; ,  }  $$
are  $ \, \Rbq $--bases  of  $ \, \Utilde_\bq^\leq \, $,  resp.\ of  $ \, \Utilde_\bq^\geq \, $;  in particular, $ \, \Utilde_\bq^\leq $
and  $ \, \Utilde_\bq^\geq $  are free  $ \, \Rbq $--modules.
 \vskip4pt
   (d) \,  The sets of ordered monomials
  $$  \displaylines{
   \phantom{and}  \qquad
   \bigg\{\; {\textstyle \prod\limits_{k=N}^1} \Fbar_{\beta^k}^{\;f_k} \,
 {\textstyle \prod\limits_{j \in I}} \, {L_j}^{a_j} \, {\textstyle \prod\limits_{i \in I}}
 \, {K_i}^{b_i} \, {\textstyle \prod\limits_{h=1}^N} \, \Ebar_{\beta^h}^{\;e_h}
 \;\bigg|\; f_k, e_h \in \NN \, , \, a_j, b_i \in \ZZ \;\bigg\}  \cr
   \text{and}  \qquad\,
   \bigg\{\; {\textstyle \prod\limits_{h=1}^N} \, \Ebar_{\beta^h}^{\;e_h} \,
 {\textstyle \prod\limits_{j \in I}} \, {L_j}^{a_j} \, {\textstyle \prod\limits_{i \in I}}
 \, {K_i}^{b_i} \, {\textstyle \prod\limits_{k=N}^1} \Fbar_{\beta^k}^{\;f_k}
 \;\bigg|\; f_k, e_h \in \NN \, , \, a_j, b_i \in \ZZ \;\bigg\}  }  $$
 are  $ \, \Rbq $--bases  of  $ \, \Utilde_\bq \, $;  in particular,  $ \, \Utilde_\bq = \Utilde_\bq(\lieg) \, $
itself is a free  $ \, \Rbq $--module.
\end{theorem}

\pf
 {\it (a)}\,  Entirely similar to the proof of  Theorem \ref{thm:PBW_hat-MpQG}.
 \vskip7pt
   {\it (b)}\,  This is obvious from definitions.
 \vskip7pt
   {\it (c)}\,  We can apply once more the same ideas as for  Theorem \ref{thm:PBW_hat-MpQG},  thus finding that
 $ \; \mathbb{B} := {\bigg\{\; {\textstyle \prod\limits_{k=N}^1} \Fbar_{\beta^k}^{\;f_k} \,
 {\textstyle \prod\limits_{j \in I}} {L_j}^{a_j} \bigg\}}_{\!f_k \in \NN \, , \, a_j \in \ZZ} \; $
 is an  $ \Rbq $--basis  of  $ \, \Utilde_\bq^\leq \, $,  the case for  $ \, \Utilde_\bq^\geq \, $  being entirely similar.
 The claim is true when  $ \, \bq = \check{\bq} \, $,  by the results in  \cite{DP};  moreover, by
 Proposition \ref{prop:Utilde-cocy-def}  we have
 $ \, {\big( \Utilde_\bq^{\,\leq} \big)}^{\!\scriptscriptstyle{\sqrt{\,\ }}} \! =
 {\big( \Utilde_{\check{\bq}}^{\,\leq} \big)}^{\!\scriptscriptstyle{\sqrt{\,\ }}} $
 as  $ \, \Rbqsq $--coalgebras,  so  $ \mathbb{B} $  is also an  $ \Rbqsq $--basis  of
 $ {\big( \Utilde_\bq^{\,\leq} \big)}^{\!\scriptscriptstyle{\sqrt{\,\ }}} $.
 On the other hand, it follows from  Theorem \ref{thm:PBW_MpQG}  that  $ \mathbb{B} $  is also an
 $ \Fbq $--basis  of  $ {\big( \Utilde_\bq^{\,\leq} \big)}^{\!\scriptscriptstyle{\sqrt{\,\ }}} $.
 Thus any
 $ \, u \in \Utilde_\bq^{\,\leq} \; \Big(\! \subseteq {\big( \Utilde_\bq^{\,\leq} \big)}^{\!\scriptscriptstyle{\sqrt{\,\ }}} \! \cap
 \Utilde_\bq^{\,\leq} \Big) \; $  uniquely expands as an  $ \Rbqsq $--linear  combination of elements in  $ \mathbb{B} $
 but also uniquely expands as an  $ \Fbq $--linear combination of such elements: we conclude that the coefficients in
 these expansions belong to  $ \, \Rbqsq \cap \Fbq = \Rbq \, $,  \,q.e.d.
 \vskip7pt
   {\it (d)}\,  This is proved by the same arguments as  {\it (c)\/}  above.
\epf

\vskip9pt

   A direct fallout of the previous result is the following:

\vskip13pt

\begin{prop}  \label{prop:triang-decomps_Utilde}
 {\sl (triangular decompositions for unrestricted MpQG's)}
 \vskip1pt
 The multiplication in  $ \, \Utilde_\bq $  provides  $ \, \Rbq $--module  isomorphisms
  $$  \displaylines{
   \Utilde_\bq^- \mathop{\otimes}_{\,\Rbq} \Utilde_\bq^{\,0} \;
     \cong \; \Utilde_\bq^{\,\leq} \; \cong \; \Utilde_\bq^{\,0} \mathop{\otimes}_{\,\Rbq} \Utilde_\bq^-  \quad ,
 \qquad  \Utilde_\bq^+ \mathop{\otimes}_{\,\Rbq} \Utilde_\bq^{\,0} \; \cong \; \Utilde_\bq^{\,\geq} \;
     \cong \; \Utilde_\bq^{\,0} \mathop{\otimes}_{\,\Rbq} \Utilde_\bq^+  \cr
   \Utilde_\bq^{\,+,0} \!\mathop{\otimes}_{\,\Rbq}\! \Utilde_\bq^{\,-,0}  \, \cong \;  \Utilde_\bq^{\,0}  \,
     \cong \;  \Utilde_\bq^{\,-,0} \!\mathop{\otimes}_{\,\Rbq}\! \Utilde_\bq^{\,+,0}  \; ,
 \quad  \Utilde_\bq^+ \!\mathop{\otimes}_{\,\Rbq}\! \Utilde_\bq^{\,0} \!\mathop{\otimes}_{\,\Rbq}\! \Utilde_\bq^-  \,
     \cong \;  \Utilde_\bq  \, \cong \;  \Utilde_\bq^- \!\mathop{\otimes}_{\,\Rbq}\! \Utilde_\bq^{\,0} \!\mathop{\otimes}_{\,\Rbq}\! \Utilde_\bq^+  }  $$
\end{prop}

\pf
 Direct from  Theorem \ref{thm:PBW_tilde-MpQG}  above.
\epf

\vskip9pt

   Here is a second consequence:

\vskip13pt

\begin{prop}  \label{prop:Utilde_subHopf}  {\ }
 \vskip3pt
   (a)  $ \; \Utilde_\bq = \Utilde_\bq(\hskip0,8pt\lieg) \, $,  resp.\  $ \Utilde_\bq^\leq \, $,  resp.\
   $ \Utilde_\bq^{\,0} \, $,  resp.\  $ \Utilde_\bq^\geq \, $,  is a Hopf  $ \, \Rbq $--subalgebra  (hence is an
   $ \, \Rbq $--integral  form as a Hopf algebra) of  $ \, U_\bq(\hskip0,8pt\lieg) \, $,  resp.\ of  $ \, U_\bq^\leq \, $,
   resp.\ of  $ \, U_\bq^{\,0} \, $,  resp.\ of  $ \, U_\bq^\geq \, $.
 \vskip3pt
   (b)  $ \; \Utilde_\bq^\pm $  is an  $ \, \Rbq $--subalgebra  (hence an  $ \, \Rbq $--integral  form, as an algebra) of
   $ \, U_\bq^\pm \, $.
\end{prop}

\pf
 Claim  {\it (b)\/}  is obvious, by construction, and similarly also claim  {\it (a)\/}  for
 $ \Utilde_\bq^0 \, $;  the other cases are similar, so we restrict ourselves to one of them,
 say that of  $ \Utilde_\bq \, $.
                                                  \par
   Once again, the canonical case (i.e.\  $ \, \bq = \check{\bq} \, $)  follows from the results
   in  \cite{DP},  suitably adapted to the present context; then by  Proposition \ref{prop:Utilde-cocy-def}
   above the same result also holds true for  $ \Utilde_\bq^{\scriptscriptstyle \sqrt{\,\ }} $  with  {\sl any\/}
   possible  $ \bq $   --- that is,  $ \Utilde_\bq^{\scriptscriptstyle \sqrt{\,\ }} $  is a Hopf
   $ \, \Rbqsq $--subalgebra  of
   $ \, U_\bq^{\scriptscriptstyle \sqrt{\,\ }} \! := U_\bq^{\scriptscriptstyle \sqrt{\,\ }}\!(\hskip0,8pt\lieg) \, $,
   for any possible  $ \bq \, $.  In particular  $ \Utilde_\bq^{\scriptscriptstyle \sqrt{\,\ }} $  is an
   $ \Rbqsq $--subcoalgebra  of  $ U_\bq^{\scriptscriptstyle \sqrt{\,\ }} \, $,  hence given any
   $ \, u \in \Utilde_\bq \; \big( \subseteq \Utilde_\bq^{\scriptscriptstyle \sqrt{\,\ }} \,\big) \, $  we have
   $ \, \Delta(u) \in \Utilde_\bq^{\scriptscriptstyle \sqrt{\,\ }} \otimes_\Rbqsq
   \Utilde_\bq^{\scriptscriptstyle \sqrt{\,\ }} \, $.  By  Theorem \ref{thm:PBW_tilde-MpQG}  the
   $ \Rbqsq $--module  $ \, \Utilde_\bq^{\scriptscriptstyle \sqrt{\,\ }} \otimes_\Rbqsq
   \Utilde_\bq^{\scriptscriptstyle \sqrt{\,\ }} \, $  is free with a basis made of homogeneous tensors
   $ \, v' \otimes v'' \, $  in which both  $ v' $  and  $ v'' $  are PBW monomials as given in
   Theorem \ref{thm:PBW_tilde-MpQG}{\it (d)\/}:  thus  $ \Delta(u) $  has a unique expansion of the form
   $ \; \Delta(u) = \sum_s c_s \, v'_s \otimes v''_s \; $  for some  $ \, c_s \in \Rbqsq \, $.
   On the other hand, the same set of ``PBW homogeneous tensors'' of the form  $ \, v' \otimes v'' \, $
   as above is also an  $ \Fbq $--basis  of  $ \, U_\bq \otimes_\Fbq U_\bq \, $:  hence, since  $ U_\bq $
   is an  $ \Fbq $--coalgebra  and  $ \, u \in \Utilde_\bq \subseteq U_\bq \, $,  we have also a unique
   $ \Fbq $--linear expansion of  $ \Delta(u) $  into  $ \; \Delta(u) = \sum_s a_s \, v'_s \otimes v''_s \; $.
   Comparing both expansion inside  $ \, U_\bq^{\scriptscriptstyle \sqrt{\,\ }} \otimes_\Fbqsq
   U_\bq^{\scriptscriptstyle \sqrt{\,\ }} \, $   --- which also has the set of all ``PBW homogeneous tensors''
   $ \, v' \otimes v'' \, $  as  $ \Fbqsq $--monomials  ---  we find
   $ \; c_s = a_s \, \in \, \Rbqsq \cap \Fbq \, = \, \Rbq \; $  for all  $ s \, $,  which means that
   $ \; \Delta(u) \, \in \, \Utilde_\bq \otimes_\Rbq \Utilde_\bq \; $.  So  $ \Utilde_\bq $  is an
   $ \Rbq $--subcoalgebra  of  $ U_\bq \, $,  and similar arguments prove it is stable by the antipode,
   hence is a Hopf  $ \Rbq $--subalgebra.
\epf

\medskip

\subsection{Integral forms for MpQG's with larger torus}  \label{int-form-&-larger-torus}  \
 \vskip7pt
   In  \S \ref{MpQGs-larger-torus}  we introduced generalized MpQG's, denoted
   $ \, U_{\bq,\varGamma_{\!\bullet}} \equiv U_{\bq,\varGamma_{\!\bullet}}(\hskip0,8pt\lieg) \, $,
   whose toral part is the group algebra of any lattice  $ \, \varGamma_{\!\bullet} = \varGamma_+ \times \varGamma_- \, $
   with  $ \varGamma_\pm $  being rank  $ \theta $  lattices such that  $ \, Q \leq \varGamma_\pm \leq \QQ{}Q \, $;
%
%%%%%
% (using notation as explained there);
%%%%%
%
 in particular,  {\sl this required additional assumptions on the ground field  $ \k \, $,  namely that  $ \k $
 contain suitable roots of the  $ q_{ij} $'s,  see \S \ref{larger-MpQG's}}.  We shall now consider one such
 generalized MpQG, say  $ U_{\bq,\varGamma_{\!\bullet}} \, $,  making assumptions on  $ \k $  as mentioned above,
 and introduce integral forms for it, quickly explaining the few changes one needs in the previously described
 treatment of integral forms for  $ \, U_\bq \, $,  that is the case  $ \, \varGamma_{\!\bullet} = Q \times Q \, $.

\vskip11pt

\begin{free text}  \label{restr-int-forms_x_MpQG's_lrg-torus}
 {\bf Restricted integral forms for MpQG's with larger torus.}\,   {\sl Assume that\/  $ \bq $  is of  {\it integral}  type}.
 Then a  $ \ZZ $--bilinear  form  $ \, {(\,\ ,\ )}_{\!{}_B} \, $  is defined on  $ \QQ{}Q \, $,
 and we have well-defined sublattices  $ {\dot{Q}}^{(\ell)} $  and  $ {\dot{Q}}^{(r)} $  in  $ \QQ{}Q $
 (notation of  \S \ref{MpQG-larg-tor_int-case}).
%
%%%
%                                                              \par
%    Let  $ \varGamma_\pm $  be sublattices in  $ \QQ{}Q $  of rank  $ \theta $  containing
% $ Q \, $,  and  $ \, \varGamma_{\!\bullet} := \varGamma_+ \times \varGamma_- \, $;  {\sl
% we assume that  $ \, \varGamma_+ \subseteq {\dot{Q}}^{(\ell)} \, $  and  $ \, \varGamma_-
% \subseteq {\dot{Q}}^{(r)} \, $}.
% %                                                              \par
%    Let  $ U_{\bq\,,\varGamma_{\!\bullet}} $  be the MpQG with torus patterned on
% $ \varGamma_{\!\bullet} $  (as in  \S \ref{larger-MpQG's},  with the assumptions
% on  $ \k $  as required therein); then
%
%%%
%
 {\sl We assume in addition that  $ \, \varGamma_+ \subseteq {\dot{Q}}^{(\ell)} \, $  and
 $ \, \varGamma_- \subseteq {\dot{Q}}^{(r)} \, $}.
 Then
 we can define a ``restricted integral form''  $ \Uhatdot_{\bq\,,\varGamma_{\!\bullet}} $  for
 $ U_{\bq\,,\varGamma_{\!\bullet}} \, $,  akin to  $ \Uhatdot_\bq $   ---
 so that  $ \, \varGamma_\pm = Q \, $  yields  $ \, \Uhatdot_{\bq\,,\varGamma_{\!\bullet}} = \Uhatdot_\bq \, $  ---
 as follows.
                                                             \par
   Let  $ \, {\big\{ \gamma^\pm_i \big\}}_{i \in I} \, $  be a basis of  $ \varGamma_\pm \, $:
   then  {\sl in  Definition \ref{def:int-form_hat-MpQG}{\it (a)},  replace every occurrence of
   ``$ \, K_i^{\pm 1} \, $''  with  ``$ \, {K_{\gamma_i^+}}^{\hskip-9pt\pm 1} \, $''
   and every occurrence of  ``$ \, L_i^{\pm 1} \, $''  with  ``$ \, {L_{\gamma_i^-}}^{\hskip-9pt\pm 1} \, $''
   --- so each  $ q $--binomial  coefficient  $ \, {\Big(\! {{K_i\,;\,c} \atop n} \Big)}_{\!\!q} \, $  is replaced by
   $ \, {\Big(\! {{K_{\gamma_i^+}\,;\,c} \atop {\;\;\;n}} \Big)}_{\!\!q} \, $,  etc.; this yields the very
   {\bf definition}  of  $ \, \Uhatdot_{\bq\,,\varGamma_{\!\bullet}} \, $}.
 \vskip3pt
   Basing on this definition, one easily finds that  {\it all results presented in  \S \ref{Uhat}
   above for  $ \, \Uhatdot_\bq $  have their direct counterpart for  $ \Uhatdot_{\bq\,,\varGamma_{\!\bullet}} $  as well}.
   Moreover, the natural embedding  $ \, U_\bq \subseteq U_{\bq\,,\varGamma_{\!\bullet}} \, $
   between MpQG's   --- induced by the inclusion  $ \, Q \times Q \subseteq \varGamma_{\!\bullet} \, $  ---
   clearly restricts to a similar embedding  $ \, \Uhatdot_\bq \subseteq \Uhatdot_{\bq\,,\varGamma_{\!\bullet}} \, $
   of integral forms.  Similar comments apply to the various subalgebras of  $ \Uhatdot_\bq $
   for their natural counterparts in  $ \Uhatdot_{\bq\,,\varGamma_{\!\bullet}} \, $.
 \vskip5pt
   Similarly,  {\sl assume now that\/  $ \bq $  is of\/  {\it strictly integral}  type},  so that the sublattices
   $ Q^{(\ell)} $  and  $ Q^{(r)} $  are defined in  $ \QQ{}Q $  (cf.\ \S \ref{MpQG-larg-tor_int-case});
   concerning  $ \varGamma_\pm \, $,  this time  {\sl we assume in addition that
   $ \, \varGamma_+ \subseteq Q^{(\ell)} \, $  and  $ \, \varGamma_- \subseteq Q^{(r)} \, $}.
   Then we can define a second ``restricted integral form''  $ \Uhat_{\bq\,,\varGamma_{\!\bullet}} $
   for  $ U_{\bq\,,\varGamma_{\!\bullet}} \, $,  direct analogue to  $ \Uhat_\bq \, $,  as follows.
                                                             \par
   Given bases  $ \, {\big\{ \gamma^\pm_i \big\}}_{i \in I} \, $  of  $ \varGamma_\pm \, $ as above,
   {\sl in  Definition \ref{def:int-form_hat-MpQG}{\it (b)\/}  replace every occurrence of  ``$ \, K_i^{\pm 1} \, $''
   with  ``$ \, {K_{\gamma_i^+}}^{\hskip-9pt\pm 1} \, $''  and every occurrence of  ``$ \, L_i^{\pm 1} \, $''  with
   ``$ \, {L_{\gamma_i^-}}^{\hskip-9pt\pm 1} \, $'':  in particular, every  $ q_i $--divided  binomial coefficient
   $ \, {\Big(\! {{K_i\,;\,c} \atop n} \Big)}_{\!\!q_i} \, $  is replaced by
   $ \, {\Big(\! {{K_{\gamma_i^+}\,;\,c} \atop {\;\;\;n}} \Big)}_{\!\!q_i} \, $,  etc.;
   then read the outcome, by assumption, as the very  {\bf definition}  of  $ \, \Uhat_{\bq\,,\varGamma_{\!\bullet}} \, $}.
                                                             \par
   In force of this definition, one can easily find that  {\it all results presented  in  \S \ref{Uhat}
   about  $ \, \Uhat_\bq $  have their direct counterpart for  $ \Uhat_{\bq\,,\varGamma_{\!\bullet}} $  as well}.
   In addition, the embedding  $ \, U_\bq \subseteq U_{\bq\,,\varGamma_{\!\bullet}} \, $  restricts to an embedding
   $ \, \Uhat_\bq \subseteq \Uhat_{\bq\,,\varGamma_{\!\bullet}} \, $  between integral forms.  All this applies also to the
   natural counterparts in  $ \Uhat_{\bq\,,\varGamma_{\!\bullet}} $  of the different subalgebras of  $ \Uhat_\bq \, $.
\end{free text}

\vskip4pt

\begin{free text}  \label{unrestr-int-forms_x_MpQG's_lrg-torus}
 {\bf Unrestricted integral forms for MpQG's with larger torus.}\,   {\sl Let now\/  $ \bq $  be of general
 (though Cartan) type, and make no special assumptions on  $ \varGamma_\pm \, $}.
 Then we can define for  $ U_{\bq\,,\varGamma_{\!\bullet}} $  an ``unrestricted integral form''
 $ \Utilde_{\bq\,,\varGamma_{\!\bullet}} \, $,  akin to  $ \Utilde_\bq $   --- in that
 $ \, \Utilde_{\bq\,,\varGamma_{\!\bullet}} = \Utilde_\bq \, $  when  $ \, \varGamma_\pm = Q \, $  ---
 in the following, very simple way.
                                                             \par
   Let  $ \, {\big\{ \gamma^\pm_i \big\}}_{i \in I} \, $  be bases of  $ \varGamma_\pm \, $,  as before: now,
   {\sl in  Definition \ref{def:int-form_tilde-Uqqm},  replace every occurrence of  ``$ \, K_i^{\pm 1} \, $''  with
   ``$ \, {K_{\gamma_i^+}}^{\hskip-9pt\pm 1} \, $''  and every occurrence of  ``$ \, L_i^{\pm 1} \, $''  with
   ``$ \, {L_{\gamma_i^-}}^{\hskip-9pt\pm 1} \, $'';  then take the final outcome as the very  {\bf definition}  of
   $ \, \Utilde_{\bq\,,\varGamma_{\!\bullet}} \, $}.
                                                             \par
   Starting from this definition, one easily checks that  {\it all results presented in  \S \ref{Utilde}  for
   $ \, \Utilde_\bq $  have a direct counterpart for  $ \Utilde_{\bq\,,\varGamma_{\!\bullet}} $  too}.
   Also, the natural embedding  $ \, U_\bq \subseteq U_{\bq\,,\varGamma_{\!\bullet}} \, $
   between MpQG's implies by restriction a similar embedding
   $ \, \Utilde_\bq \subseteq \Utilde_{\bq\,,\varGamma_{\!\bullet}} \, $
   between the corresponding unrestricted integral forms.  Finally, similar comments apply to
   the natural counterparts in  $ \Utilde_{\bq\,,\varGamma_{\!\bullet}} $
   of the various subalgebras considered in  $ \Utilde_\bq \, $.
\end{free text}

\medskip

\subsection{Duality among integral forms}  \label{duality-int-forms}  \
 \vskip7pt
   If we take two quantum Borel subgroups  $ U_\bq^{\,\geq} $  and  $ U_\bq^{\,\leq} \, $,
   we know that they are in duality via a non-degenerate skew-Hopf pairing as in  \S \ref{H-duality}.
   Now,  {\sl assuming that  $ \bq $  is of integral type},  if we take on either side integral forms of
   opposite nature, say  $ \Uhatdot_\bq^{\raise-5pt\hbox{$ \scriptstyle \,\geq $}} \, $,  or
   $ \Uhat_\bq^{\,\geq} \, $,  and  $ \Utilde_\bq^{\,\leq} $   --- or  $ \Utilde_\bq^{\,\geq} $  and
   $ \Uhatdot_\bq^{\raise-5pt\hbox{$ \scriptstyle \,\leq $}} \, $,  or  $ \Uhat_\bq^{\,\leq} \, $  ---
   we find that they are ``dual to each other'' with respect to that pairing. To state this properly, we
   need to work with MpQG's with (suitably paired) larger tori.  The correct statement is the following:

\smallskip

\begin{prop}  \label{q-Borel_mut-dual}
 Let  $ \varGamma_\pm $  be rank  $ \theta $  sublattices of\/  $ \QQ{}Q $  containing  $ Q \, $,
 let  $ \, U_{\bq\,,\varGamma_+}^\geq $  and  $ \, U_{\bq\,,\varGamma_-}^\leq $
 be the associated Borel MpQG's, and let
%
%%%
% $ \, \eta_0 \!\! : U_{\bq\,,\varGamma_+}^{+,0} \!\otimes U_{\bq\,,\varGamma_-}^{-,0}
% \!\!\relbar\joinrel\relbar\joinrel\longrightarrow \k \, $  be the (skew-)Hopf pairing
% induced by the pairing  $ \, \eta : U_{\bq\,,\varGamma_+}^\geq \!\otimes U_{\bq\,,\varGamma_-}^\leq
% \!\!\relbar\joinrel\relbar\joinrel\longrightarrow \k \, $  of\/  \S \ref{duality x larger MpQG's}.
%%%
%
 $ \, \eta : U_{\bq\,,\varGamma_+}^\geq \!\otimes U_{\bq\,,\varGamma_-}^\leq \!\!\relbar\joinrel\relbar\joinrel\longrightarrow \k \, $
 be the skew-Hopf pairing of\/  \S \ref{duality x larger MpQG's}.
 \vskip3pt
   (a)\,  Assume that\/  $ \, \bq = {\big(\, q^{b_{ij}} \big)}_{i,j \in I} \, $  is of  {\sl integral}  type,
   and (with notation of\/  \S \ref{MpQG-larg-tor_int-case})  that
%
%   $ \, {\big( \varGamma_+ \, , \varGamma_- \big)}_{\!{}_B} \!\subseteq \ZZ \, $
%
 $ \, \varGamma_+ = \dot{\varGamma}_-^{(\ell)} \, $  and  $ \, \varGamma_- = \dot{\varGamma}_+^{(r)} \, $.
 Then
  $$  \displaylines{
   \Uhatdot_{\bq\,,\varGamma_+}^{\raise-5pt\hbox{$ \scriptstyle \,+,0 $}}  =
   \,  \Big\{\, u \in U_{\bq\,,\varGamma_+}^{+,0} \;\Big|\; \eta\Big( u \, , \Utilde_{\bq\,,\varGamma_-}^{-,0} \Big) \subseteq \Rbq \,\Big\}
%
%   \; =: \;  {\Big( \Utilde_{\bq\,,\varGamma_-}^{-,0} \Big)}^{\!\circ}
%
  \cr
   \Utilde_{\bq\,,\varGamma_-}^{-,0}  = \,  \Big\{\, v \in U_{\bq\,,\varGamma_-}^{-,0} \;\Big|\;
   \eta\Big( \Uhatdot_{\bq\,,\varGamma_+}^{\raise-5pt\hbox{$ \scriptstyle \,+,0 $}} , v \Big) \subseteq \Rbq \,\Big\}
  \cr
   \Uhatdot_\bq^{\raise-5pt\hbox{$ \scriptstyle \,+ $}}  \,
   = \;  \Uhatdot_{\bq\,,\varGamma_+}^{\raise-5pt\hbox{$ \scriptstyle \,+ $}}
   = \,  \Big\{\, u \in U_{\bq\,,\varGamma_+}^{\,+} \;\Big|\; \eta\Big( u \, , \Utilde_{\bq\,,\varGamma_-}^{\,-} \Big) \subseteq \Rbq \,\Big\}
  \cr
   \; \Utilde_\bq^{\,-}  \, = \;
\Utilde_{\bq\,,\varGamma_-}^{\,-}  = \,  \Big\{\, v \in U_{\bq\,,\varGamma_-}^{\,-} \;\Big|\;
\eta\Big( \Uhatdot_{\bq\,,\varGamma_+}^{\raise-5pt\hbox{$ \scriptstyle \,+ $}} , v \Big) \subseteq \Rbq \,\Big\}
  \cr
   \Uhatdot_{\bq\,,\varGamma_+}^{\raise-5pt\hbox{$ \scriptstyle \,\geq $}}  =
   \,  \Big\{\, u \in U_{\bq\,,\varGamma_+}^{\,\geq} \;\Big|\; \eta\Big( u \, , \Utilde_{\bq\,,\varGamma_-}^{\,\leq} \Big) \subseteq \Rbq \,\Big\}  \cr
   \; \Utilde_{\bq\,,\varGamma_-}^{\,\leq}  = \,  \Big\{\, v \in U_{\bq\,,\varGamma_-}^{\,\leq} \;\Big|\;
   \eta\Big( \Uhatdot_{\bq\,,\varGamma_+}^{\raise-5pt\hbox{$ \scriptstyle \,\geq $}} , v \Big) \subseteq \Rbq \,\Big\}
  }  $$
 and similarly reversing the roles of  ``\,+''  and  ``\,-''  and of  ``$ \,\geq $''  and  ``$ \,\leq $''.
 \vskip3pt
   (b)\,  Assume that\/  $ \, \bq = {\big(\, q^{\,d_i{}t^+_{ij}} = q^{\,d_j{}t^-_{ij}} \big)}_{i,j \in I} \, $  is of
   {\sl strongly integral}  type  (cf. \S \ref{MpQG-larg-tor_int-case}  for notation).  If
   $ \, \varGamma_- = \varGamma_+^{(r)} \, $  --- cf.\  \eqref{dual-lattice_T+/-}  ---
 then
  $$  \displaylines{
   \Uhat_{\bq\,,\varGamma_+}^{+,0}  = \,  \Big\{\, u \in U_{\bq\,,\varGamma_+}^{+,0} \;\Big|\;
   \eta\Big( u \, , \Utilde_{\bq\,,\varGamma_-}^{-,0} \Big) \subseteq \Rbq \,\Big\}  \cr
   \Utilde_{\bq\,,\varGamma_-}^{-,0}  = \,  \Big\{\, v \in U_{\bq\,,\varGamma_-}^{-,0} \;\Big|\;
   \eta\Big( \Uhat_{\bq\,,\varGamma_+}^{+,0} , v \Big) \subseteq \Rbq \,\Big\}
%%%
%%%%%%%
% \cr
%%%%%%%
 }  $$
  $$  \displaylines{
%%%%%
   \; \Uhat_\bq^{\,+}  \, = \;  \Uhat_{\bq\,,\varGamma_+}^{\,+}  = \,  \Big\{\, u \in U_{\bq\,,\varGamma_+}^{\,+}
\;\Big|\; \eta\Big( u \, , \Utilde_{\bq\,,\varGamma_-}^{\,-} \Big) \subseteq \Rbq \,\Big\}  \cr
   \; \Utilde_\bq^{\,-}  \, = \;  \Utilde_{\bq\,,\varGamma_-}^{\,-}  = \,  \Big\{\, v \in U_{\bq\,,\varGamma_-}^{\,-}
\;\Big|\; \eta\Big( \Uhat_{\bq\,,\varGamma_+}^{\,+} , v \Big) \subseteq \Rbq \,\Big\}  \cr
   \Uhat_{\bq\,,\varGamma_+}^{\,\geq}  = \,  \Big\{\, u \in U_{\bq\,,\varGamma_+}^{\,\geq} \;\Big|\;
   \eta\Big( u \, , \Utilde_{\bq\,,\varGamma_-}^{\,\leq} \Big) \subseteq \Rbq \,\Big\}  \cr
   \; \Utilde_{\bq\,,\varGamma_-}^{\,\leq}  = \,  \Big\{\, v \in U_{\bq\,,\varGamma_-}^{\,\leq} \;\Big|\;
   \eta\Big( \Uhat_{\bq\,,\varGamma_+}^{\,\geq} , v \Big) \subseteq \Rbq \,\Big\}  }  $$
   \indent   If instead  $ \, \varGamma_+ = \varGamma_-^{(\ell)} \, $  --- cf.\  \eqref{dual-lattice_T+/-}  again --- then
  $$  \displaylines{
   \Uhat_{\bq\,,\varGamma_-}^{-,0}  = \,  \Big\{\, v \in U_{\bq\,,\varGamma_-}^{-,0} \;\Big|\;
   \eta\Big( \Utilde_{\bq\,,\varGamma_+}^{+,0} , v \Big) \subseteq \Rbq \,\Big\}  \cr
   \, \Utilde_{\bq\,,\varGamma_+}^{+,0}  = \,  \Big\{\, u \in U_{\bq\,,\varGamma_+}^{+,0} \;\Big|\;
   \eta\Big( u \, , \Uhat_{\bq\,,\varGamma_-}^{-,0} \Big) \subseteq \Rbq \,\Big\}  \cr
   \Uhat_\bq^{\,-}  \, = \;  \Uhat_{\bq\,,\varGamma_-}^{\,-}  = \,  \Big\{\, v \in U_{\bq\,,\varGamma_-}^{\,-} \;\Big|\;
   \eta\Big( \Utilde_{\bq\,,\varGamma_+}^{\,+} , v \Big) \subseteq \Rbq \,\Big\}  \cr
   \Utilde_\bq^{\,+}  \, = \;  \Utilde_{\bq\,,\varGamma_+}^{\,+}  = \,  \Big\{\, u \in U_{\bq\,,\varGamma_+}^{\,+}
\;\Big|\; \eta\Big( u \, , \Uhat_{\bq\,,\varGamma_-}^{\,-} \Big) \subseteq \Rbq \,\Big\}  \cr
%%%
   \Uhat_{\bq\,,\varGamma_-}^{\,\leq}  = \,  \Big\{\, v \in U_{\bq\,,\varGamma_-}^{\,\leq} \;\Big|\;
\eta\Big( \Utilde_{\bq\,,\varGamma_+}^{\,\geq} , v \Big) \subseteq \Rbq \,\Big\}  \cr
   \; \Utilde_{\bq\,,\varGamma_+}^{\,\geq}  = \,  \Big\{\, u \in U_{\bq\,,\varGamma_+}^{\,\geq} \;\Big|\; \eta\Big( u \, , \Uhat_{\bq\,,\varGamma_-}^{\,\leq} \Big) \subseteq \Rbq \,\Big\}  }  $$
\end{prop}

\pf
 {\it (a)}\,  The assumptions imply  $ \, {\big( \varGamma_+ \, , \varGamma_- \big)}_{\!{}_B} \!\subseteq \ZZ \, $,
 hence  $ \; \eta\big( K_{\gamma_i^+\!} \, , L_{\gamma_j^-} \big) \, = \, q^{{( \gamma_i^+ , \gamma_j^- )}_{\!{}_B}} \; $
 --- cf.\  \S \ref{duality x larger MpQG's}  ---   that in turn implies
 $ \,\; \displaystyle{ \eta\Bigg(\! {\bigg( {{K_{\gamma_i^+\!} \, ; 0} \atop n} \bigg)}_{\!\!q} \, ,
 \, L_{\gamma_j^-} \Bigg) = {\bigg( {{( \gamma_i^+ , \gamma_j^- )}_{\!{}_B} \atop n} \!\bigg)}_{\!\!q} } \, \in \, \Rbq \;\, $.
 Taking PBW bases on both sides, this is enough to prove  $ \; \eta\Big( \Uhatdot_{\bq\,,\varGamma_+}^{\raise-5pt\hbox{$ \scriptstyle \,+,0 $}} \, ,
 \Utilde_{\bq\,,\varGamma_-}^{-,0} \Big) \, \subseteq \, \Rbq \; $;  therefore we get
 $ \; \Uhatdot_{\bq\,,\varGamma_+}^{\raise-5pt\hbox{$ \scriptstyle \,+,0 $}} \subseteq
 \Big\{\, u \in U_{\bq\,,\varGamma_+}^{+,0} \;\Big|\; \eta\Big( u \, , \Utilde_{\bq\,,\varGamma_-}^{-,0} \Big) \subseteq \Rbq \,\Big\} \; $
 and on the other hand also
 $ \; \Utilde_{\bq\,,\varGamma_-}^{-,0} \subseteq \Big\{\, v \in U_{\bq\,,\varGamma_-}^{-,0}
 \;\Big|\; \eta\Big( \Uhatdot_{\bq\,,\varGamma_+}^{\raise-5pt\hbox{$ \scriptstyle \,+,0 $}} , v \Big) \subseteq \Rbq \,\Big\} \; $.
 This proves ``half'' the result we claimed true, thus we still need some additional work to do.
 \vskip1pt
   Since  $ \, \varGamma_+ = \dot{\varGamma}_-^{(\ell)} \, $  and  $ \, \varGamma_- = \dot{\varGamma}_+^{(r)} \, $,
   we can fix bases  $ {\big\{ \gamma_i^\pm \big\}}_{i \in I} $  of  $ \varGamma_\pm $
   that are dual to each other, namely  $ \; {\big( \gamma_h^+ \, , \gamma_k^-  \big)}_{\!{}_B} = \, \delta_{h,k} \; $
   for all  $ \, h, k \in I \, $;  \,so we get
   $ \; \eta\,\Big( K_{\gamma_h^+}^{\,z_h^+} \, , L_{\gamma_k^-}^{z_k^-} \Big) = q^{\delta_{h,k}\,z_h^+\,z_k^-} \; $.
   As a consequence, the arguments used for  Proposition \ref{prop:basis_int-on-laur-mons} and
   Proposition \ref{prop:bidual} apply again (with  $ \, \eta \, $  replacing the pairing  $ \, \big\langle\,\ ,\ \big\rangle \, $
   and the  $ K_{\gamma_h^+} $'s,  resp.\ the  $ L_{\gamma_k^-} $'s,  playing the role of the  $ X_i $'s,
   resp.\ of the  $ \chi_j $'s) now proving claim  {\it (a)}.  Indeed, the analysis developed for those results now shows
%%%%%
% (using PBW bases on either side)
%%%%%
 that  $ \Uhatdot_{\bq\,,\varGamma_+}^{\raise-5pt\hbox{$ \scriptstyle \,+,0 $}} $  and
 $ \Utilde_{\bq\,,\varGamma_-}^{-,0} $  contain bases that, up to invertible coefficients (powers of  $ q \, $),
 are dual to each other, and that is enough to conclude.
 \vskip1pt
%
%%%%%
%    The statement about  $ \, \Uhatdot_\bq^{\raise-5pt\hbox{$ \scriptstyle \,+ $}} =
% \Uhatdot_{\bq\,,\varGamma_+}^{\raise-5pt\hbox{$ \scriptstyle \,+ $}} \, $  and  $ \, \Utilde_\bq^{\,-}
% = \Utilde_{\bq\,,\varGamma_-}^{\,-} \, $  (that are actually independent of  $ \varGamma_\pm \, $,  by
% definition) is a direct consequence of PBW theorems for both sides and of  Proposition \ref{duality_x_PBW-bases}.
% Then from this result, that for  $ \Uhatdot_{\bq\,,\varGamma_+}^{\raise-5pt\hbox{$ \scriptstyle \,+,0 $}} $
% and  $ \Utilde_{\bq\,,\varGamma_-}^{-,0} $  and the triangular decompositions in  Proposition
% \ref{prop:triang-decomps_Uhat}  and  Proposition \ref{prop:triang-decomps_Utilde},  we finally get
% the statement concerning  $ \Uhatdot_{\bq\,,\varGamma_+}^{\raise-5pt\hbox{$ \scriptstyle \,\geq $}} $
% and  $ \Utilde_{\bq\,,\varGamma_-}^{\,\leq} $  proved as well.
%%%%%
%
   The claim about  $ \, \Uhatdot_\bq^{\raise-5pt\hbox{$ \scriptstyle \,+ $}} \! =
   \Uhatdot_{\bq\,,\varGamma_+}^{\raise-5pt\hbox{$ \scriptstyle \,+ $}} \, $  and  $ \, \Utilde_\bq^{\,-} \! =
\Utilde_{\bq\,,\varGamma_-}^{\,-} $  (both independent of  $ \varGamma_\pm \, $)
is a consequence of PBW theorems for both sides and of  Proposition \ref{duality_x_PBW-bases}.
Then from this result, the one for  $ \Uhatdot_{\bq\,,\varGamma_+}^{\raise-5pt\hbox{$ \scriptstyle \,+,0 $}} $  and
$ \Utilde_{\bq\,,\varGamma_-}^{-,0} $  and the triangular decompositions in
Proposition \ref{prop:triang-decomps_Uhat}  and  Proposition \ref{prop:triang-decomps_Utilde},
we finally get the statement concerning  $ \Uhatdot_{\bq\,,\varGamma_+}^{\raise-5pt\hbox{$ \scriptstyle \,\geq $}} $
and  $ \Utilde_{\bq\,,\varGamma_-}^{\,\leq} $  as well.
 \vskip1pt
   The statement with switched  ``\,+''  and  ``\,-''  or  ``$ \,\geq\, $''  and  ``$ \,\leq\, $''  goes the same way.
 \vskip4pt
   {\it (b)}\,   Up to minimal changes, this is proved much like claim  {\it (a)}.
\epf

\vskip9pt

\begin{rmk}  \label{rmk:applic_q-Borel_duality}
 One can use the previous result to deduce properties of a (Hopf) algebra on either side
 --- e.g.\  $ \Utilde_{\bq\,,\varGamma_-}^{\,\leq} $,  say ---   out of properties on the other side   ---
 $ \Uhatdot_{\bq\,,\varGamma_+}^{\raise-5pt\hbox{$ \scriptstyle \,\geq $}} $  or  $ \Uhat_{\bq\,,\varGamma_+}^{\,\geq} $
 in the example.  For instance,  $ \Utilde_{\bq\,,\varGamma_-}^{\,\leq} $  is an  $ \Rbq $--algebra
 (hard to prove directly!) because  $ \Uhat_{\bq\,,\varGamma_+}^{\,\geq} $  is an  $ \Rbq $--coalgebra
 (that follows from its definition).  Similarly, we deduce that  $ \Utilde_\bq^{\,+} $
 {\sl is independent of any choice of quantum root vectors\/}  (that do enter in the definition!) because
 it is ``the dual'' of  $ \Uhat_\bq^{\,-} $  and  {\sl the latter  {\it is}  independent},
%%%%%%
   \hbox{by definition, of any such choice.}
%%%%%%
%%%%%%
%
\end{rmk}

\medskip

\subsection{ Integral forms of ``mixed'' type}  \label{subsec:mixed-int-forms}  \
 \vskip7pt
   Let us consider two quantum Borel subgroups  $ U_{\bq\,,\varGamma_+}^{\,\geq} $  and
   $ U_{\bq\,,\varGamma_-}^{\,\leq} $  as in  \S \ref{duality-int-forms}  above, with  $ \bq $  {\sl integral},
   linked by the skew-Hopf pairing  $ \eta $  of  \S \ref{duality x larger MpQG's}.  Assuming in addition that the lattices
   $ \varGamma_\pm $  fit the conditions required in  Theorem \ref{q-Borel_mut-dual}  (according to whether $ \bq $
   is strongly integral or not), that theorem tells us that the pairing  $ \eta $  yields by restriction  $ \Rbq $--valued
   skew-Hopf pairings, still denoted  $ \eta \, $,  for the pairs of  $ \Rbq $--Hopf  algebras
   $ \, \Big(\, \Uhatdot_{\bq\,,\varGamma_+}^{\raise-5pt\hbox{$ \scriptstyle \,\geq $}} \, ,
   \, \Utilde_{\bq\,,\varGamma_-}^{\,\leq} \Big) \, $  and  $ \, \Big(\, \Utilde_{\bq\,,\varGamma_+}^{\,\geq} \, ,
   \, \Uhatdot_{\bq\,,\varGamma_-}^{\raise-5pt\hbox{$ \scriptstyle \,\leq $}} \Big) \, $,  or the pairs
   $ \, \Big(\, \Uhat_{\bq\,,\varGamma_+}^{\raise-5pt\hbox{$ \scriptstyle \,\geq $}} \, ,
   \, \Utilde_{\bq\,,\varGamma_-}^{\,\leq} \Big) \, $  and
   $ \, \Big(\, \Utilde_{\bq\,,\varGamma_+}^{\,\geq} \, , \,
   \Uhat_{\bq\,,\varGamma_-}^{\raise-5pt\hbox{$ \scriptstyle \,\leq $}} \Big) \, $  when
   $ \bq $  is strongly integral.  Moreover, as the original pairing  $ \eta $  is non-degenerate,
   the same holds true for its restrictions to  $ \Rbq $--integral  forms of the original quantum Borel subgroups.
   Therefore, much like each MpQG  $ \QEq $  can be realized as Drinfeld double via the original pairing  $ \eta $
   (cf.\  Remark \S \ref{MpQG-largetor=qDouble}),  the restrictions of the latter lead us to define the following:

\vskip11pt

\begin{definition}  \label{def:mixed-int-forms}
 With assumption as above   --- thus  $ \bq $  {\sl is of integral type}  ---   we define the following Hopf algebras over
 $ \, \Rbq $  as  {\sl Drinfeld doubles\/}  (cf.\ \S \ref{conv-Hopf})
  $$  \displaylines{
   \dot{\overrightarrow{U}}_{\!\!\bq\,,\varGamma_{\!\bullet}} \, := \,
   \dot{\overrightarrow{U}}_{\!\bq\,,\varGamma_{\!\bullet}}(\hskip0,8pt\lieg)  \; = \;
   D\Big( \Uhatdot_{\bq\,,\varGamma_+}^{\raise-5pt\hbox{$ \scriptstyle \,\geq $}} \, , \,
   \Utilde_{\bq\,,\varGamma_-}^{\,\leq} \, , \, \eta \Big)  \cr
   \dot{\overleftarrow{U}}_{\!\bq\,,\varGamma_{\!\bullet}} \, := \,
   \dot{\overleftarrow{U}}_{\!\bq\,,\varGamma_{\!\bullet}}(\hskip0,8pt\lieg)  \; = \;
   D\Big( \Utilde_{\bq\,,\varGamma_+}^{\,\geq} \, , \, \Uhatdot_{\bq\,,\varGamma_-}^{\raise-5pt\hbox{$ \scriptstyle \,\leq $}} \, , \, \eta \Big)  \cr
  }  $$
 where  $ \, \varGamma_{\!\bullet} := \varGamma_+ \times \varGamma_- \; $.  If in addition the multiparameter
 $ \, \bq $  is also  {\it strongly integral},  then we define similarly also the Hopf  $ \Rbq $--algebras
 (again as Drinfeld doubles)
  $$  \displaylines{
   \overrightarrow{U}_{\!\!\bq\,,\varGamma_{\!\bullet}} \, := \, \overrightarrow{U}_{\!\!\bq\,,
   \varGamma_{\!\bullet}}(\hskip0,8pt\lieg)  \; = \;  D\Big( \Uhat_{\bq\,,\varGamma_+}^{\raise-5pt\hbox{$ \scriptstyle \,\geq $}} \, , \, \Utilde_{\bq\,,\varGamma_-}^{\,\leq} \, , \, \eta \Big)  \cr
   \hskip9pt \hfill   \overleftarrow{U}_{\!\!\bq\,,\varGamma_{\!\bullet}} \, := \, \overleftarrow{U}_{\!\!\bq\,,\varGamma_{\!\bullet}}(\hskip0,8pt\lieg)  \;
   = \;  D\Big( \Utilde_{\bq\,,\varGamma_+}^{\,\geq} \, , \, \Uhat_{\bq\,,\varGamma_-}^{\raise-5pt\hbox{$ \scriptstyle \,\leq $}} \, , \, \eta \Big)   \hfill  \diamondsuit  }  $$
\end{definition}

\vskip7pt

   The following claim points out the main properties of these new objects:

\vskip11pt

\begin{theorem}
 Keep assumptions and notations as above.  Then  $ \, \dot{\overrightarrow{U}}_{\!\!\bq\,,\varGamma_{\!\bullet}} \, $,
 resp.\  $ \dot{\overleftarrow{U}}_{\!\!\bq\,,\varGamma_{\!\bullet}} \, $,  is an  $ \Rbq $--integral  form (as Hopf algebra) of
 $ \, U_{\bq\,,\varGamma_{\!\bullet}} \, $,  with PBW-type basis
\begin{align*}
   &  \bigg\{\; {\textstyle \prod\limits_{h=1}^N} E_{\beta^h}^{\,(e_h)} \, {\textstyle \prod\limits_{j \in I}} \,
   {\bigg( {L_{\gamma^+_j} \atop l_j} \!\bigg)}_{\!\!q} L_{\gamma^+_j}^{-\lfloor l_j/2 \rfloor}
 \, {\textstyle \prod\limits_{i \in I}} \, K_{\!\gamma^-_i}^{\,k_i} \, {\textstyle \prod\limits_{t=N}^1}
 \Fbar_{\!\beta^t}^{\,f_t} \;\bigg|\; e_h, l_j, k_i, f_t \in \NN \;\bigg\}  \;\; ,   \quad  \cr
   \text{resp.} \qquad  &  \hfill   \bigg\{\; {\textstyle \prod\limits_{h=1}^N} \Ebar_{\beta^h}^{\,e_h} \,
   {\textstyle \prod\limits_{j \in I}} \, L_{\gamma^+_j}^{\,l_j} \, {\textstyle \prod\limits_{i \in I}} \,
   {\bigg( {K_{\!\gamma^-_i} \atop k_i} \!\bigg)}_{\!\!q} K_{\!\gamma^-_i}^{\,-\lfloor k_i/2 \rfloor} \,
   {\textstyle \prod\limits_{t=N}^1} F_{\beta^t}^{\,(f_t)} \;\bigg|\; e_h, l_j, k_i, f_t \in \NN \;\bigg\}   \quad
\end{align*}
(notation of  \S \ref{restr-int-forms_x_MpQG's_lrg-torus})  as well as variations of these, changing the
order of factors in the PBW monomials.  Similarly, if  $ \, \bq $  is  {\sl strongly integral}  then
$ \overrightarrow{U}_{\!\!\bq\,,\varGamma_{\!\bullet}} \, $,  resp.\  $ \overleftarrow{U}_{\!\!\bq\,,\varGamma_{\!\bullet}} \, $,
is an  $ \Rbq $--integral  form (as Hopf algebra) of  $ \, U_{\bq\,,\varGamma_{\!\bullet}} \, $,  with PBW-type basis
\begin{align*}
   \bigg\{\;  &  {\textstyle \prod\limits_{h=1}^N} E_{\beta^h}^{\,(e_h)} \, {\textstyle \prod\limits_{j \in I}} \,
   {\bigg( {L_{\gamma^+_j} \atop l_j} \!\bigg)}_{\!\!q_i} L_{\gamma^+_j}^{-\lfloor l_j/2 \rfloor}
 \, {\textstyle \prod\limits_{i \in I}} \, K_{\!\gamma^-_i}^{\,k_i} \, {\textstyle \prod\limits_{t=N}^1}
 \Fbar_{\!\beta^t}^{\,f_t} \;\bigg|\; e_h, l_j, k_i, f_t \in \NN \;\bigg\}  \;\; ,   \quad  \cr
   \text{resp.} \qquad  \bigg\{\;  &  {\textstyle \prod\limits_{h=1}^N} \Ebar_{\beta^h}^{\,e_h} \,
   {\textstyle \prod\limits_{j \in I}} \, L_{\gamma^+_j}^{\,l_j} \, {\textstyle \prod\limits_{i \in I}} \,
   {\bigg( {K_{\!\gamma^-_i} \atop k_i} \!\bigg)}_{\!\!q_i} K_{\!\gamma^-_i}^{\,-\lfloor k_i/2 \rfloor} \,
   {\textstyle \prod\limits_{t=N}^1} F_{\beta^t}^{\,(f_t)} \;\bigg|\; e_h, l_j, k_i, f_t \in \NN \;\bigg\}   \quad
\end{align*}
(as well as variations of these, changing the order of factors in the PBW monomials).
                                                              \par
   In addition,  $ \, \dot{\overrightarrow{U}}_{\!\!\bq\,,\varGamma_{\!\bullet}} \, $,  resp.\
   $ \, \dot{\overleftarrow{U}}_{\!\!\bq\,,\varGamma_{\!\bullet}} \, $,  resp.\
   $ \, \overrightarrow{U}_{\!\!\bq\,,\varGamma_{\!\bullet}} \, $,  resp.\  $ \, \overleftarrow{U}_{\!\!\bq\,,\varGamma_{\!\bullet}} \, $,
   coincides with the  $ \, \Rbq $--subalgebra  of  $ \, U_{\bq\,,\varGamma_{\!\bullet}} $  generated by
   $ \, \Uhatdot_{\bq\,,\varGamma_+}^{\raise-5pt\hbox{$ \scriptstyle \,\geq $}} \, $  and
   $ \, \Utilde_{\bq\,,\varGamma_-}^{\,\leq} \, $,  resp.\ by  $ \, \Utilde_{\bq\,,\varGamma_+}^{\,\geq} \, $
   and  $ \, \Uhatdot_{\bq\,,\varGamma_-}^{\raise-5pt\hbox{$ \scriptstyle \,\leq $}} \, $,  resp.\ by
   $ \, \Uhat_{\bq\,,\varGamma_+}^{\,\geq} \, $  and  $ \, \Utilde_{\bq\,,\varGamma_-}^{\,\leq} \, $,
   resp.\ by  $ \, \Utilde_{\bq\,,\varGamma_+}^{\,\geq} \, $  and  $ \, \Uhat_{\bq\,,\varGamma_-}^{\,\leq} \, $.
\end{theorem}

\pf
 Indeed, the result follows at once by construction, together with the fact that
 $ \, \Uhatdot_{\bq\,,\varGamma_+}^{\raise-5pt\hbox{$ \scriptstyle \,\geq $}} \, $,  $ \, \Utilde_{\bq\,,\varGamma_-}^{\,\leq} \, $,
 etc., actually are integral forms of the corresponding quantum Borel subgroups defined over  $ \k \, $,  and with the PBW Theorems for them.
\epf

\vskip11pt

   We are also interested in yet other mixed integral forms, defined as follows.  Inside
   $ \, \Uhatdot_{\bq\,,\varGamma_+}^{\raise-5pt\hbox{$ \scriptstyle \,\geq $}} \, $  (or inside
   $ \, \Uhat_{\bq\,,\varGamma_+}^{\raise-5pt\hbox{$ \scriptstyle \,\geq $}} \, $,  it is the same),
   denote by  $ \, \Uhat_{\bq\,,\varGamma_+}^{\,\geq\,,\,\uparrow} \, $  the  $ \Rbq $--subalgebra
   generated by  $ \Uhat_\bq^+ $  and  $ \Utilde_{\bq\,,\varGamma_+}^{\,0} \, $
 (recall that by Remark \ref{rmk:nosqua-intform}, one has
 $\Utilde_{\bq\,,\varGamma_+}^{\,0} \subseteq \Uhatdot_{\bq\,,\varGamma_+}^{\raise-5pt\hbox{$ \scriptstyle \, 0$}}
 \subseteq \Uhat_{\bq\,,\varGamma_+}^{\raise-3pt\hbox{$ \scriptstyle \, 0$}} $);
 this is indeed an  $ \Rbq $--integral  form of  $ U_{\bq\,,\varGamma_+}^{\,\geq} $  (as a Hopf subalgebra),
 and the non-degenerate skew-Hopf pairing  $ \; \eta : \Uhatdot_{\bq\,,\varGamma_+}^{\raise-5pt\hbox{$ \scriptstyle \,\geq $}}
 \! \otimes_\Rbq \! \Utilde_{\bq\,,\varGamma_-}^{\raise-5pt\hbox{$ \scriptstyle \,\leq $}}
 \hskip-5pt \relbar\joinrel\relbar\joinrel\longrightarrow \Rbq \; $  restricts to a similar pairing
 $ \; \eta : \Uhat_{\bq\,,\varGamma_+}^{\,\geq\,,\,\uparrow} \! \otimes_\Rbq \! \Utilde_{\bq\,,\varGamma_-}^{\raise-5pt\hbox{$ \scriptstyle \,\leq $}}
 \relbar\joinrel \relbar\joinrel\relbar\joinrel\longrightarrow \Rbq \; $.  Similarly, we consider the  $ \Rbq $--subalgebra
 $ \, \Uhat_{\bq\,,\varGamma_-}^{\,\leq\,,\,\uparrow} \, $  of
 $ \, \Uhatdot_{\bq\,,\varGamma_-}^{\raise-5pt\hbox{$ \scriptstyle \,\leq $}} \, $  (or of
 $ \, \Uhat_{\bq\,,\varGamma_-}^{\raise-5pt\hbox{$ \scriptstyle \,\leq $}} \, $)  generated by
 $ \Uhat_\bq^- $  and  $ \Utilde_{\bq\,,\varGamma_-}^{\,0} \, $,  which again is an  $ \Rbq $--integral  form of
 $ U_{\bq\,,\varGamma_-}^{\,\leq} \, $  for which we have a non-degenerate skew-Hopf pairing
 $ \; \eta : \Utilde_{\bq\,,\varGamma_+}^{\raise-5pt\hbox{$ \scriptstyle \,\geq $}} \!
 \otimes_\Rbq \! \Uhat_{\bq\,,\varGamma_-}^{\,\leq\,,\,\uparrow} \relbar\joinrel \relbar\joinrel\relbar\joinrel\longrightarrow \Rbq \; $
 induced by the original skew-Hopf pairing between our multiparameter quantum Borel subgroups over  $ \k \, $.
 {\sl In addition, we do not assume that the multiparameter  $ \, \bq $  be of integral type, nor we assume
 $ \varGamma_+ $  and  $ \varGamma_- $  to be in duality (as in  \S \ref{larger-MpQG's})}.  All this allows the following

\vskip13pt

\begin{definition}
 For  {\sl any\/}  multiparameter  $ \bq $  (of Cartan type) and  $ \; \varGamma_{\!\bullet} := \varGamma_+ \times \varGamma_- \; $,  we define the following Hopf algebras over  $ \, \Rbq $  as  {\sl Drinfeld doubles\/}  with respect to the above mentioned skew-Hopf pairings:
  $$  \displaylines{
   \check{U}_{\bq\,,\varGamma_{\!\bullet}}^{\,\scriptscriptstyle (+)} \, := \,
   \check{U}_{\bq\,,\varGamma_{\!\bullet}}^{\,\scriptscriptstyle (+)}\!(\hskip0,8pt\lieg)  \; = \;
   D\Big( \Uhat_{\bq\,,\varGamma_+}^{\,\geq\,,\,\uparrow} \, , \, \Utilde_{\bq\,,\varGamma_-}^{\raise-5pt\hbox{$ \scriptstyle \,\leq $}} \, , \, \eta \Big)  \cr
   \hfill   \check{U}_{\bq\,,\varGamma_{\!\bullet}}^{\,\scriptscriptstyle (-)} \, := \,
   \check{U}_{\bq\,,\varGamma_{\!\bullet}}^{\,\scriptscriptstyle (-)}\!(\hskip0,8pt\lieg)  \; = \;
   D\Big( \Utilde_{\bq\,,\varGamma_+}^{\raise-5pt\hbox{$ \scriptstyle \,\geq $}} \, , \, \Uhat_{\bq\,,\varGamma_-}^{\,\leq\,,\,\uparrow} \, , \, \eta \Big)   \hfill  \diamondsuit  }  $$
\end{definition}

\vskip5pt

   The main properties of these more Hopf algebras are summarized as follows:

\vskip11pt

\begin{theorem}
 Keep notation as above.
 \vskip5pt
   {\it (a)}\;  The Hopf algebras  $ \, \check{U}_{\bq\,,\varGamma_{\!\bullet}}^{\,\scriptscriptstyle (+)} \, $  and
   $ \, \check{U}_{\bq\,,\varGamma_{\!\bullet}}^{\,\scriptscriptstyle (-)} \, $  are both  $ \Rbq $--integral  forms
   (as Hopf algebras) of  $ \, U_{\bq\,,\varGamma_{\!\bullet}} \, $,  with PBW-type basis
\begin{align*}
   &  \bigg\{\; {\textstyle \prod\limits_{h=1}^N} E_{\beta^h}^{\,(e_h)} \, {\textstyle \prod\limits_{j \in I}} \,
   L_{\gamma^+_j}^{\;l_j} \, {\textstyle \prod\limits_{i \in I}} \, K_{\!\gamma^-_i}^{\,k_i} \,
   {\textstyle \prod\limits_{t=N}^1} \Fbar_{\!\beta^t}^{\,f_t} \;\bigg|\; e_h, l_j, k_i, f_t \in \NN \;\bigg\}  \;\; ,  \cr
   \text{and} \qquad  &  \hfill   \bigg\{\; {\textstyle \prod\limits_{h=1}^N} \Ebar_{\beta^h}^{\,e_h} \,
   {\textstyle \prod\limits_{j \in I}} \, L_{\gamma^+_j}^{\,l_j} \, {\textstyle \prod\limits_{i \in I}} \,
   K_{\!\gamma^-_i}^{\,k_i} \, {\textstyle \prod\limits_{t=N}^1} F_{\beta^t}^{\,(f_t)} \;\bigg|\; e_h, l_j, k_i, f_t \in \NN \;\bigg\}
\end{align*}
(notation of  \S \ref{restr-int-forms_x_MpQG's_lrg-torus})  for
$ \, \check{U}_{\bq\,,\varGamma_{\!\bullet}}^{\,\scriptscriptstyle (+)} \, $  and
$ \, \check{U}_{\bq\,,\varGamma_{\!\bullet}}^{\,\scriptscriptstyle (-)} \, $  respectively, as well as variations of
these (changing the order of factors in the PBW monomials).
 \vskip5pt
   {\it (b)}\;  $ \, \check{U}_{\bq\,,\varGamma_{\!\bullet}}^{\,\scriptscriptstyle (+)} \, $,  resp.\
   $ \, \check{U}_{\bq\,,\varGamma_{\!\bullet}}^{\,\scriptscriptstyle (-)} \, $,  coincides with the
   $ \, \Rbq $--subalgebra  of  $ \, U_{\bq\,,\varGamma_{\!\bullet}} $  generated by
   $ \, \Uhat_{\bq\,,\varGamma_{\!\bullet}}^{\,+} \, $,  $ \, \Utilde_{\bq\,,\varGamma_{\!\bullet}}^{\,0} \, $
   and  $ \, \Utilde_{\bq\,,\varGamma_{\!\bullet}}^{\,-} \, $,  resp.\ by
   $ \, \Utilde_{\bq\,,\varGamma_{\!\bullet}}^{\,+} \, $,  $ \, \Utilde_{\bq\,,\varGamma_{\!\bullet}}^{\,0} \, $  and
   $ \, \Uhat_{\bq\,,\varGamma_{\!\bullet}}^{\,-} \, $.
 \vskip5pt
   {\it (c)}\;  Both  $ \, \check{U}_{\bq\,,\varGamma_{\!\bullet}}^{\,\scriptscriptstyle (+)} \, $  and
   $ \, \check{U}_{\bq\,,\varGamma_{\!\bullet}}^{\,\scriptscriptstyle (-)} \, $  have obvious triangular decompositions
   analogous to those in  Propositions \ref{prop:triang-decomps_Uhat}  and  \ref {prop:triang-decomps_Utilde}.
\end{theorem}

\pf
 Here again, everything follows easily by construction, through our previous results on integral forms
 of multiparameter quantum Borel subgroups.
\epf

\vskip9pt

\begin{rmk}
 Defining  $ \, \check{U}_{\bq\,,\varGamma_{\!\bullet}}^{\,\scriptscriptstyle (+)} \, $  and  $ \, \check{U}_{\bq\,,\varGamma_{\!\bullet}}^{\,\scriptscriptstyle (-)} \, $,  as well as   $ \, \dot{\overrightarrow{U}}_{\!\!\bq\,,\varGamma_{\!\bullet}} \, $,  $ \, \dot{\overleftarrow{U}}_{\!\!\bq\,,\varGamma_{\!\bullet}} \, $,
 $ \, \overrightarrow{U}_{\!\!\bq\,,\varGamma_{\!\bullet}} \, $  and  $ \, \overleftarrow{U}_{\!\!\bq\,,\varGamma_{\!\bullet}} \, $,
 as  {\sl Drinfeld doubles\/}  provides great advantages, namely we get for free that\;
                                                      \par
   {\it (1)}\, they are  {\sl Hopf algebras},
                                                      \par
   {\it (2)}\, they have nice PBW bases (and triangular decompositions),
                                                      \par
   {\it (3)}\, they are  $ \Rbq $--integral  forms of  $ \, U_{\bq\,,\varGamma_{\!\bullet}} \, $
   --- since they are tensor products (as Drinfeld doubles!) of integral forms of multiparameter quantum Borel subgroups.
 \vskip5pt
   In fact, we already saw that these algebras coincide with suitable
   $ \Rbq $--subalgebras  in  $ \, U_{\bq\,,\varGamma_{\!\bullet}} \, $;  yet, proving properties
   {\it (1)--(3)\/}  by direct approach would  {\sl not\/}  be trivial.
\end{rmk}

\vskip11pt

\begin{free text}{\bf The link with the uniparameter case.}
 For the uniparameter quantum group  $ U_q(\lieg) $  of Drinfeld and Jimbo, possibly with larger torus,
 one can define  $ \Rbq $--integral  forms  $ \Uhat_\bq(\lieg) \, $,
%%%%%
%  --- as well as  $ \Uhatdot_\bq(\lieg) $  ---
%%%%%
 $ \Uhatdot_\bq(\lieg) $
%%%
 and  $ \Utilde_\bq(\lieg) $  much like we did with our MpQG's, constructing them as generated by
 quantum divided powers and binomial coefficients
%%%%%
% (the {\sl restricted\/}  type, after Lusztig)
%%%%%
 or by renormalized quantum root vectors;
%%%%%
% (the {\sl unrestricted\/}  type, after De Concini and Procesi);
%%%%%
 note that now  $ \, \Rbq = \Zqqm \, $.  Similarly, one can define another  $ \Rbq $--subalgebra  of
 $ U_q(\lieg) \, $,  denoted by  $ \check{U}_q^{\scriptscriptstyle {(-)}}(\lieg) \, $,  generated by
 $ \Uhat_q^- \, $,  $ \Utilde_q^0 $  and  $ \Utilde_q^+ \, $,  first introduced in \cite{HL}.
%%%%%
% In fact, this ``mixed'' object
%%%%%
  This is again an  $ \Rbq $--integral  form of  $ \, U_q(\lieg) \, $,  for which triangular decomposition and PBW
  Theorems hold true, deduced from the similar results for  $ \, \Uhat_q(\lieg) $  and  $ \, \Utilde_q(\lieg) \, $.
  One also has its ``symmetric counterpart'', say  $ \check{U}_q^{\scriptscriptstyle {(+)}}(\lieg) \, $,  generated by
  $ \Utilde_q^- \, $,  $ \Utilde_q^0 $  and  $ \Uhat_q^+ \, $.
 \vskip3pt
   The construction of  $ \, \Rbq $--integral  forms (again with  $ \, \Rbq = \Zqqm \, $)  of restricted or
   unrestricted type also extends to the context of  {\sl twisted quantum groups
   $ U_{q,M}^\varphi(\lieg) $  \`a la Costantini-Varagnolo\/}  (see  \cite{CV1,CV2}),  still denoted
   $ \, \Uhat_{q,M}^{\,\varphi}(\lieg) \, $  and  $ \, \Utilde_{q,M}^{\,\varphi}(\lieg) \, $  in the restricted and the
   unrestricted case respectively.  Then one has also corresponding integral forms for the
%
%%%%%
% positive or negative nilpotent part and for the toral part of  $ U_{q,M}^\varphi(\lieg) \, $,
%%%%%
%
 various relevant (Hopf) subalgebras (Borel, nilpotent, etc.),
%%%%%
 triangular decompositions, PBW bases, etc.   --- see  \cite{Gav} for details.
 Moreover, one can define also in this context  {\sl mixed integral forms}
 $ \, \check{U}_{q,M}^{{\scriptscriptstyle (-)} , \, \varphi}(\lieg) \, $  and
 $ \, \check{U}_{q,M}^{{\scriptscriptstyle (+)} , \, \varphi}(\lieg) \, $,  namely
 \vskip3pt
   {\it (a)}\;  $ \, \check{U}_{q,M}^{{\scriptscriptstyle (-)} , \, \varphi}(\lieg) \, $  is the
   $ \Rbq $--subalgebra  of  $ U_{q,M}^{\,\varphi}(\lieg) $
   generated by  $ \Uhat_q^- $,  $ \Utilde_{q,M}^{\,0} \, $,  $ \Utilde_q^+ $,
 \vskip2pt
   {\it (b)}\;  $ \, \check{U}_{q,M}^{{\scriptscriptstyle (+)} , \, \varphi}(\lieg) \, $  is the
   $ \Rbq $--subalgebra  of  $ U_{q,M}^{\,\varphi}(\lieg) $
   generated by  $ \Utilde_q^- $,  $ \Utilde_{q,M}^{\,0} \, $,  $ \Uhat_q^+ $,
 \vskip4pt
\noindent
 (note that the occurrence of  $ \varphi $  is irrelevant at the algebra level, so that
 $ \, \check{U}_{q,M}^{{\scriptscriptstyle (\pm)} , \, \varphi}(\lieg) = \check{U}_{q,M}^{\scriptscriptstyle (\pm)}(\lieg) \, $
 as  $ \Rbq $--algebras;  the coalgebra structure, on the contrary, is affected).
                                                         \par
   Using the properties of  $ \, \Uhat_q^\mp \, $,  $ \, \Utilde_{q,M}^{\,0} \, $  and  $ \, \Utilde_q^\pm \, $
   presented in  \cite{Gav},  one can prove that  $ \, \check{U}_{q,M}^{{\scriptscriptstyle (-)} , \, \varphi}(\lieg) \, $
   and  $ \, \check{U}_{q,M}^{{\scriptscriptstyle (+)} , \, \varphi}(\lieg) \, $  are again  $ \Rbq $--integral  forms
   --- as Hopf algebras --- of $ U_{q,M}^{\,\varphi}(\lieg) \, $.  In fact, the trivial case  $ \, \varphi = 0 \, $  gives
   $ \, U_{q,M}^{\varphi=0}(\lieg) = U_{q,M}(\lieg) \, $,  \,the standard uniparameter quantum group associated with  $ M $,
   so the case of  $ U_{q,M}^{\,\varphi}(\lieg) $  and its  $ \Rbq $--integral  forms (restricted, or unrestricted, or mixed)
   is a direct generalization of what occurs with  $ U_{q,M}(\lieg) \, $.
                                                         \par
   On the other hand, it is proved in  \cite{GG1}  that Costantini-Varagnolo's twisted quantum groups
   $ U_{q,M}^{\,\varphi}(\lieg) $
   are just quotients of MpQG's  $ \, U_{\bq , M \times Q}(\lieg) \, $  with  $ \bq $  ranging in a special
   subset of strongly integral type multiparameters.
   It follows that the same link exists among their integral forms of either type, including the mixed one.
\end{free text}

\vskip7pt

\begin{free text}{\bf Applications to topological invariants.}
 As mentioned above, the mixed form  $  \check{U}_q^{\scriptscriptstyle {(-)}}(\lieg) $  was introduced in  \cite{HL}.
 In that paper, the authors provide a construction of a ``universal quantum invariant'' of integral homology spheres,
 call it  $ J_M \, $:  this ``lifts'' the well-known Witt-Reshetikhin-Turaev (=WRT) knot invariant
 $ \tau_{\scriptscriptstyle M}^\lieg(\varepsilon) $  of  $ M $  associated with  $ \lieg $  and any root of unity
 $ \varepsilon \, $,  in that  $ \tau_{\scriptscriptstyle M}^\lieg(\varepsilon) $  is obtained by evaluation of
 $ J_M $  at  $ \varepsilon \, $.  Unlike the definition of the WRT invariant, the construction of this ``universal''
 invariant  $ J_M $  does not involve representations, so it provides a unified, representation-free definition of
 quantum invariants of integral homology spheres, performed in terms of the form
 $ \check{U}_q^{\scriptscriptstyle (-)}(\lieg) \, $.
 \vskip3pt
   Now, having introduced ``multiparameter mixed integral forms''
   $\, \check{U}_{q,M}^{{\scriptscriptstyle (\pm)} , \, \varphi}(\lieg) \, $  and even
   $ \, \check{U}_{\bq\,, \, M \times Q}^{\,\scriptscriptstyle (\pm)}(\hskip0,8pt\lieg) \, $,
%%%
 {\it we might expect that the construction of  $ J_M $  could be extended, starting from
 $ \, \check{U}_{q,M}^{{\scriptscriptstyle (\pm)} , \, \varphi}(\lieg) \, $  or even
 $ \, \check{U}_{\bq\,, \, M \times Q}^{\,\scriptscriptstyle (\pm)}(\hskip0,8pt\lieg) \, $
 instead of  $ \, \check{U}_q^{\scriptscriptstyle (-)}(\lieg) \, $,  thus providing entirely
 new topological invariants for knots (and links) and integral homology spheres}.
\end{free text}

\bigskip

\section{Specialization of MpQG's at 1}  \label{Spec-roots-1}

\vskip7pt

%
%%%%%
%    In this section we undertake the study of MpQG's for which the parameters  $ q_{ii} $
% (for all  $ i \, $)  are just 1; in fact, as every  $ q_{ii} $  is just a power of a
% special element  $ \, q \in \k^\times \, $,  requiring all the  $ q_{ii} $'s  to be 1
% amounts to requiring  $ q $  itself to be 1.
%%%%%
%
   In this section we study those MpQG's for which all the  $ q_{ii} $'s  are 1; in fact, as every  $ q_{ii} $
   is a power of a single  $ \, q \in \k^\times $,  requiring $ \, q_{ii} = 1 \, $  for all  $ i $  amounts to requiring  $ \, q = 1 \, $.
                                                                         \par
   Note that if  $ \, q_{ii} = 1 \, $  for some  $ i $,  the very definition of  $ \QEq $  makes no sense,
   so we have to be more subtle.  First we take  $ \QEq $  as defined over a ``generic'' multiparameter
   $ \, \bq := {\big(\, q_{ij} \big)}_{i, j \in I} \, $  of Cartan type; then we consider its  $ \ZZ $--forms
   $ \Uhatdotqgd \, $,  $ \Uhatqgd $  and  $ \Utildeqgd \, $,  defined over  $ \Rbq $  (under suitable
   ``integrality'' assumptions on  $ \bq $  for the first two cases); finally, for either form we specialize  $ q $
   --- hence all  the  $ q_{ii} $'s  ---   to a root of unity (or to 1, in particular), which will make sense just
   because our ground ring will be set to be  $ \Rbq \, $.

\medskip

\subsection{The ``generic'' ground rings}  \label{gen-ground-rings}  \
 \vskip7pt
   As a first step in the process sketched above, we formalize the loose ideas of
   ``generic parameter of Cartan type'' and of ``generic parameter of (a specific) integral type''. Indeed,
%
%%%%%
% this means that we fix a special ground ring generated by ``indeterminate parameters''
% that is ``universal'', in a sense, in that by specialization it maps onto any ring  $ \Rbq $
% (inside a field  $ \k $)  as in  \S \ref{ground-ring}  for any possible multiparameter  $ \bq $
% of Cartan type, or of fixed integral type, chosen inside  $ \k \, $.  In a nutshell,
%%%%%
 this ``universal ring of multiparameters'' will be the ring of functions on the  $ \ZZ $--scheme  of all
%
%%%%%
% multiparametric matrices
%%%%%
 multiparameter
%%%
 $ \bq $  of Cartan type, or of (fixed) integral type.
                                                       \par
   Similarly, we introduce also the (universal) rings generated by
   ``square roots of indeterminate parameters'', both for the Cartan type and for the integral type case.

\vskip9pt

\begin{free text}  \label{univ-ring-params-Cartan_type}
 {\bf The universal ring of multiparameters (of Cartan type).}
 Let hereafter  $ \, \ZZ\big[\,\bx^{\pm 1}\big] := \ZZ\big[ {\big\{ x_{ij}^{\pm 1} \big\}}_{i, j \in I} \,\big] \, $
 be the ring of Laurent polynomials with coefficients in  $ \ZZ $  in the indeterminates  $ x_{ij} $  ($ \, i, j \in I \, $),  and
 let  $ \, A := {\big( a_{ij} \big)}_{i, j \in I} \, $  be an indecomposable Cartan matrix of finite type.
 Consider the quotient ring
  $$  \Zabqpm  \,\; := \;\,  \ZZ\big[\,\bx^{\pm 1}\big] \bigg/ \Big( {\big\{ x_{ij} \, x_{ji} - x_{ii}^{\,a_{ij}} \big\}}_{i, j \in I} \Big)  $$
in which we denote by  $ q_{ij} $  the image of every  $ x_{ij} $  (for  $ \, i, j \in I \, $).
This is the ring of global sections of an affine scheme over  $ \ZZ \, $,  call it  $ \mathfrak{C}_A \, $:
by definition, the set of  $ \k $--points  of this scheme (for any field  $ \k \, $)  is just the set of all matrices
$ \, \bq = {\big(\, q_{ij} \big)}_{i, j \in I} \, $  of parameters of Cartan type  $ A $  with entries in  $ \k $
as in \S \ref{multiparameters}.
                                                         \par
   From all the identities  $ \; q_{ij} \, q_{ji} = q_{ii}^{\,a_{ij}} \; $  in  $ \Zabqpm \, $,  one finds
   --- by direct inspection of different cases of possible Cartan matrices  $ \, A = {\big( a_{ij} \big)}_{i, j \in I} \, $
   ---   that there exists  $ \, j_{\raise-2pt\hbox{$ \scriptstyle 0 $}} \in I \, $  such that
   $ \, q_{ii} = q_{j_{\raise-2pt\hbox{$ \scriptscriptstyle 0 $}} \, j_{\raise-2pt\hbox{$ \scriptscriptstyle 0 $}}}^{\,n_i} \, $
   for some  $ \, n_i \in \NN \, $,  for all  $ \, i \in I \, $;  indeed, we can take  $ \, n_i = d_i $  ($ \, i \in I \, $)  as in
   \S \ref{multiparameters}.  From this and the relations between the  $ q_{ij} $'s, it is easy to argue that
   $ \mathfrak{C}_A $  is a torus, of dimension  $ \, {{\,\theta\,} \choose {\,2\,}} + 1 \, $:  in particular, it is irreducible.
   Then  $ \Zabqpm $  is a domain, so we can take its field of fractions, denoted by  $ \Qabq \, $;
   in the following, we denote again by  $ q_{ij} $  ($ \, i, j \in I \, $)  the image of  $ x_{ij} $  in  $ \Qabq $  too.
                                                         \par
   By construction, the matrix  $ \, \bq := {\big(\, q_{ij} \big)}_{i, j \in I} \, $  is a Cartan type matrix of parameters in
   $ \, \k := \Qabq \, $  in the sense of \S \ref{multiparameters};  in addition,  {\sl none of the  $ q_{ii} $'s  is a root of unity}.
                                                         \par
   Now consider the ring extension  $ \Rbq $  of  $ \, \Zabqpm \, $  generated by a (formal) square root of
   $ \, q_{j_{\raise-2pt\hbox{$ \scriptscriptstyle 0 $}} \, j_{\raise-2pt\hbox{$ \scriptscriptstyle 0 $}}} \, $
   --- hereafter denoted by
   $ \, q := q_{j_{\raise-2pt\hbox{$ \scriptscriptstyle 0 $}} \, j_{\raise-2pt\hbox{$ \scriptscriptstyle 0 $}}}^{\,1/2} \, $
   ---   namely
  $$  \Rbq  \,\; := \;\,  \big(\Zabqpm\big)[x] \bigg/ \Big(\, x^2 - q_{j_{\raise-2pt\hbox{$ \scriptscriptstyle 0 $}}
  \, j_{\raise-2pt\hbox{$ \scriptscriptstyle 0 $}}} \Big)   \quad \qquad  \text{so that}  \qquad  q_{j_{\raise-2pt\hbox{$ \scriptscriptstyle 0 $}}
  \, j_{\raise-2pt\hbox{$ \scriptscriptstyle 0 $}}}^{\,1/2} := \, \overline{x} \, \in \, \Rbq  $$
 and then let  $ \Fbq $  be its field of fractions, such that
 $ \; \Fbq \, \cong \big(\Qabq\big)[x] \bigg/ \! \Big(\, x^2 - \, q_{j_{\raise-2pt\hbox{$ \scriptscriptstyle 0 $}}
 \, j_{\raise-2pt\hbox{$ \scriptscriptstyle 0 $}}} \Big) \; $.
 We still denote by  $ q_{ij} $  the images in  $ \Rbq $  and in  $ \Fbq $  of the ``old'' elements with same name in
 $ \Zabqpm $  and  $ \Qabq \, $.  We shall also write  $ \, q_i^{\pm 1} := q^{\pm d_i} \, $  for all  $ \, i \in I \, $,
 so to be consistent with  \S \ref{multiparameters};  in particular,
 $ \; q_{j_{\raise-2pt\hbox{$ \scriptscriptstyle 0 $}} \, j_{\raise-2pt\hbox{$ \scriptscriptstyle 0 $}}}^{\,1/2} := q
 = q_{j_{\raise-2pt\hbox{$ \scriptscriptstyle 0 $}}} \; $.  Note in addition that we also have
  $$  \Rbq  \,\; \cong \;\,  \big( \ZZ\big[\, {\big\{ x_{ij}^{\pm 1} \big\}}_{i, j \in I} \,\big] \big)[x]
  \bigg/ \Big(\, {\big\{ x_{ij} \, x_{ji} - x_{ii}^{\,a_{ij}} \big\}}_{i, j \in I} \, \cup
  \big\{ x^2 - x_{j_{\raise-2pt\hbox{$ \scriptscriptstyle 0 $}} \, j_{\raise-2pt\hbox{$ \scriptscriptstyle 0 $}}} \big\} \Big)  $$
                                                                       \par
   In turn, we define also
  $$  \Rbqsq  \; := \;  \ZZ\big[ {\big\{ \xi_{ij}^{\pm 1/2} \big\}}_{i, j \in I} \,\big]\big[ \xi^{\pm 1/2} \big] \bigg/
   \Big( {\Big\{ \xi_{ij}^{1/2} \, \xi_{ji}^{1/2} \! - {\big( \xi_{i{}i}^{1/2\,} \big)}^{a_{ij}} \Big\}}_{i, j \in I} \! \cup
   \Big\{\! {\big(\xi^{1/2\,}\big)}^2 \! - \xi_{j_{\raise-2pt\hbox{$ \scriptscriptstyle 0 $}} \, j_{\raise-2pt\hbox{$ \scriptscriptstyle 0 $}}}^{\,1/2} \Big\} \Big)  $$
 \vskip0pt
\noindent
 which is again a domain, and  $ \, \Fbqsq \, $  as being the field of fractions of the former.
 In both cases, we denote by  $ q_{ij}^{\pm 1/2} $  and by  $ q^{\pm 1/2} $  the image of
 $ \xi_{ij}^{\pm 1/2} $  and  $ \xi^{\pm 1/2} $  respectively (in short, one reads these symbols
 as  ``$ \, \xi := \sqrt{x\,} \, $'').  Note that  $ \Rbqsq $  and  $ \Fbqsq $  are naturally isomorphic with
 $ \Rbq $  and  $ \Fbq $  respectively, but we rather see the formers as ring or field extensions of the
 latters via the natural embeddings
%
%%%
%   $$  \Rbqnsq \,\lhook\joinrel\relbar\joinrel\relbar\joinrel\longrightarrow\, \Rbq  \quad ,
%  \qquad  \Fbqnsq \,\lhook\joinrel\relbar\joinrel\relbar\joinrel\relbar\joinrel\longrightarrow\, \Fbq  $$
%%%
%
 $ \; \Rbq \lhook\joinrel\relbar\joinrel\longrightarrow \Rbqsq \; $  and
 $ \; \Fbq \lhook\joinrel\relbar\joinrel\longrightarrow \Fbqsq \; $
 given   --- in both cases ---   by  $ \, q_{ij}^{\pm 1} \mapsto {\big(\, q_{ij}^{\pm 1/2\,} \big)}^2 \, $  and
 $ \, q^{\pm 1} \mapsto {\big(\, q^{\pm 1/2\,} \big)}^2 \, $.
 \vskip3pt
   Finally, observe also that the ring  $ \Zqqm \, $,  resp.\  $ \Zqqmsq \, $,  of Laurent polynomial
%%%%%
% (with coefficients in $ \ZZ \, $)
%%%%%
 in the indeterminate  $ q \, $,  resp.\  $ q^{1/2} $,  naturally embeds in  $ \Rbq \, $,  resp.\  in  $ \Rbqsq \, $,
 and the same occurs with their corresponding fields of fractions.  Then each module over  $ \Zqqm \, $,
 resp.\  $ \Zqqmsq \, $,  turns into a module over  $ \Rbq \, $,  resp.\  $ \Rbqsq \, $,  by scalar extension.
 \vskip3pt
  The very reason for introducing the above definitions, which explains the ``universality'' of both  $ \Rbq $  and
  $ \Rbqsq \, $,  is the following.  If  $ \k $  is any field and  $ \, \bar{\bq} := {\big( {\bar{q}}_{ij} \big)}_{i, j \in I} \, $
  is any multiparameter of Cartan type  $ A $  chosen in  $ \k $
  --- i.e., all the  $ {\bar{q}}_{ij} $'s  belong to  $ \k $  ---   then there exist unique ring morphisms
  $ \, \Rbq \!\longrightarrow \k \, $  and  $ \, \Rbqsq \!\longrightarrow \k \, $  given by
  $ \, q_{ij}^{\,\pm 1} \mapsto {\bar{q}}_{ij}^{\,\pm 1} \, $  and  $ \, q_{ij}^{\,\pm 1/2} \mapsto {\bar{q}}_{ij}^{\,\pm 1/2} \, $
  (for all  $ \, i, j \in I \, $)  respectively   --- and similarly if  $ \Rbq $  and  $ \Rbqsq $  are replaced by the fields
  $ \Fbq $  and  $ \Fbqsq \, $.  The images of these morphisms are the subrings  $ \mathcal{R}_{\bar\bq} $  and
  $ \mathcal{R}_{\bar\bq}^{\scriptscriptstyle {\sqrt{\phantom{I}}}} $  of  $ \k $  respectively generated by the
  $ {\bar{q}}_{ij} $'s  and by the  $ {\bar{q}}_{ij}^{\,1/2} $'s,  like in  \S \ref{ground-ring}   --- or the corresponding fields,
  if one starts with  $ \Fbq $  and  $ \Fbqsq \, $.
\end{free text}

\vskip7pt

\begin{free text}  \label{univ-ring-params-integr_type}
 {\bf The universal ring of multiparameters of integral type.}
 Let  $ \, A := {\big( a_{ij} \big)}_{i, j \in I} \, $  be a fixed indecomposable Cartan matrix of finite type as in
 \S \ref{univ-ring-params-Cartan_type},  and let  $ \, B := {\big( b_{ij} \big)}_{i, j \in I} \, $  be a matrix with
 entries in  $ \ZZ $  like in  \S \ref{multiparameters}. We consider the rings
  $$  \RbqB  \, := \,  \Zqqm  \quad  ,   \qquad  \RbqBsq  \, := \,  \Zqqmsq  $$
and the corresponding fields of fractions
 $ \, \FbqB := \QQ(q) \, $  and  $ \, \FbqBsq := \QQ\big(q^{1/2}\big) \, $,
together with the natural ring embeddings
 $ \; \RbqB \lhook\joinrel\relbar\joinrel\relbar\joinrel\longrightarrow \RbqBsq \; $
and
 $ \; \FbqB \lhook\joinrel\relbar\joinrel\relbar\joinrel\longrightarrow \FbqBsq \; $
given (in both cases) by  $ \, q^{\pm 1} \mapsto {\big( q^{\pm 1/2\,} \big)}^2 \, $.  In all these rings,
we consider the elements  $ \; q_{ij} := q^{\,b_{ij}} \in \RbqB \subseteq \FbqB \; $  and
$ \; q_{ij}^{\pm 1/2} := {\big( q^{\pm 1/2\,} \big)}^{b_{ij}} \in \RbqBsq \subseteq \FbqBsq \; $
for all  $ \, i, j \in I \, $.
 \vskip5pt
  Much like for the previous case of  $ \Rbq $  and  $ \Rbqsq \, $,  the rings  $ \RbqB $  and  $ \RbqBsq $  are
  ``universal'' among all those generated by multiparameters of type  $ B $  in any field  $ \k \, $,  in the following sense.
  If  $ \k $  is any field and  $ \, \bar{\bq} := {\big( {\bar{q}}_{ij} \big)}_{i, j \in I} \, $  is any multiparameter of integral type
  $ B $  in  $ \k \, $,  so that  $ \, {\bar{q}}_{ij} = {\bar{q}}^{\,b_{ij}} \, $  for some
  $ \, \bar{q} \in \k \, $  (for all  $ \, i, j \in I \, $),  then there exist unique ring morphisms  $ \, \RbqB \longrightarrow \k \, $
  and  $ \, \RbqBsq \longrightarrow \k \, $  given by  $ \, q^{\pm 1} \mapsto {\bar{q}}^{\,\pm 1} \, $  and
  $ \, q^{\pm 1/2} \mapsto {\bar{q}}^{\,\pm 1/2} \, $,  \,so that
  $ \, q_{ij} := q^{\,b_{ij}} \mapsto {\bar{q}}^{\,b_{ij}} = {\bar{q}}_{ij} \, $  and
  $ \, q_{ij}^{\,\pm 1/2} \mapsto {\big( {\bar{q}}^{\,\pm 1/2\,} \big)}^{b_{ij}} = {\bar{q}}_{ij}^{\,\pm 1/2} \, $  ($ \, i, j \in I \, $);
  similarly with the fields  $ \FbqB $  and  $ \FbqBsq $  replacing  $ \RbqB $  and  $ \RbqBsq \, $.
  The images of these morphisms are the subrings  $ \mathcal{R}_{\bar\bq} $  and
  $ \mathcal{R}_{\bar\bq}^{\scriptscriptstyle \sqrt{\phantom{I}}} $  (independent of  $ B \, $)  of  $ \k $
  re\-spective\-ly generated by the  $ {\bar{q}}_{ij}^{\,\pm 1} $'s  and by the  $ {\bar{q}}_{ij}^{\,\pm 1/2} $'s,
  i.e.\ by  $ \bar{q}^{\,\pm 1} $  and by  $ \bar{q}^{\,\pm 1/2} $  respectively, like in  \S \ref{ground-ring}
  (or the corresponding fields, if we deal with  $ \FbqB $  and  $ \FbqBsq \, $).
 \vskip7pt
   Finally, notice that we have a natural, ``hierarchical'' link between our universal rings (or fields)
   of Cartan or integral type: namely, there exist unique epimorphisms
 \vskip5pt
   \centerline{ $ \; \Rbq \relbar\joinrel\twoheadrightarrow \RbqB \;\; \big(\, q^{\pm 1} \! \mapsto q^{\pm 1} \,\big) $
\qquad  and  \qquad
                $ \; \Rbqsq \relbar\joinrel\twoheadrightarrow \RbqBsq \;\; \big(\, q^{\pm 1/2} \! \mapsto q^{\pm 1/2} \,\big) $ }
 \vskip3pt
\noindent
 so
 $ \; \Rbq \!\bigg/\! \Big( {\big\{ q_{ij} - q^{\,b_{ij}} \big\}}_{i, j \in I} \Big) \,\cong\, \RbqB \; $  and
 $ \; \Rbqsq \!\bigg/\! \Big( \Big\{ q_{ij}^{\,1/2} - {\big( q^{\pm 1/2} \big)}^{b_{ij}} \Big\}_{i, j \in I\,} \Big) \,\cong\, \RbqBsq \; $.
\end{free text}

\medskip

\subsection{Specialization at 1}  \label{spec-1}  \
 \vskip7pt
   Let  $ \bq \, $,  $ \Rbq $  and  $ \Rbqsq $  be fixed as in  \S \ref{gen-ground-rings}  above
   together with all related notation.
                                                          \par
   We consider the quotient ring  $ \; \Rbquno := \Rbq \Big/ (\,q-\!1) \, \Rbq \; $;  by construction,
   the latter is generated by invertible elements  $ \, y^{\pm 1}_{ij} := q^{\pm 1}_{ij} \!\mod (q-1)\,\Rbq \, $
   which obey only the relations  $ \, y_{j_0{}j_0}^{\pm 1} = 1 \, $  and  $ \, y_{ij}^{\pm 1} \, y_{ji}^{\pm 1} = 1 \, $,
   so that  $ \, y_{ji}^{\pm 1} = y_{ij}^{\mp 1} \, $.  It follows that  $ \Rbquno $  is just the ring of Laurent polynomials in the
   $ \big( {\theta \atop 2} \big) $  indeterminates  $ \, y_{ij}^{1/2} \, $,  $ \, i < j \, $.  In the sequel, we write
   $ \, y_{\alpha\,\gamma} \, $  for an element in  $ \Rbquno $  defined like in  \S \ref{deform-MpQG}
   but for using the  $ y_{ij} $'s  instead of the  $ q_{ij} $'s.
                                                          \par
   It is clear that   $ \Rbquno $  is also an  $ \Rbq $--algebra by scalar restriction through the canonical ring
   epimorphism  $ \, \Rbq \relbar\joinrel\relbar\joinrel\twoheadrightarrow \Rbq \Big/ (\,q-\!1) \, \Rbq \, =: \, \Rbquno \, $.
 \vskip3pt
   For every matrix  $ \, B := {\big( b_{ij} \big)}_{i, j \in I} \, $  with entries in  $ \ZZ $  as in  \S \ref{multiparameters},
   we define the ring  $ \; \RbqBuno := \RbqB \Big/ (\,q-\!1) \, \RbqB \; $.  Note that $ \; \RbqBuno \cong \ZZ \; $
   since  $ \, \RbqB \cong \Zqqm \, $,  \,and the epimorphism  $ \, \Rbq \!\relbar\joinrel\twoheadrightarrow \RbqB \, $
   induces a similar epimorphism  $ \, \Rbquno \!\relbar\joinrel\twoheadrightarrow \RbqBuno \, $  at  $ \, q = 1 \, $.
 \vskip5pt
   Similarly, the  {\it ``specialization at  $ \, q^{1/2} = 1 \, $''}  of both  $ \Rbqsq $  and  $ \RbqBsq \, $  will be
 $ \; \Rbqsquno := \Rbqsq \Big/ \big(\, q^{\,1/2} \! - \! 1 \big) \, \Rbqsq \; $  and
 $ \; \RbqBsquno := \RbqBsq \Big/ \big(\, q^{1/2} \! - \! 1 \big) \, \RbqBsq \; $;
 we write  $ \, y_{i{}j}^{\pm 1/2} \, $  for the image of  $ \, q_{i{}j}^{\pm 1/2} \, $  in  $ \Rbqsquno $  and
 $ \RbqBsquno \, $,  and overall
 $ \, \by \! := \! {\big( y_{i{}j} \big)}_{i, \, j \in I} \, $,  $ \, \by^{1/2} \! := \! {\big( y_{i{}j}^{1/2} \,\big)}_{i, \, j \in I} \, $.
 Again, we have an epimorphism  $ \, \Rbqsquno \!\relbar\joinrel\relbar\joinrel\twoheadrightarrow \RbqBsquno \, $
 induced by $ \, \Rbqsq \!\relbar\joinrel\relbar\joinrel\twoheadrightarrow \RbqBsq \, $.
%                                                              \par
   Finally, the ring extensions  $ \, \Rbq \!\lhook\joinrel\longrightarrow\! \Rbqsq \, $  and
   $ \, \RbqB \!\lhook\joinrel\longrightarrow\! \RbqBsq \, $  yield extensions
   $ \; \Rbquno \!\lhook\joinrel\longrightarrow\! \Rbqsquno \; $  and
   $ \; \RbqBuno \!\lhook\joinrel\longrightarrow\! \RbqBsquno \; $;  \, in fact,  {\it the latter is actually an isomorphism}
%%%%%
% %
%  \vskip3pt
% %
%    \centerline{ $ \RbqBuno \;{\buildrel \cong \over
% {\lhook\joinrel\relbar\joinrel\relbar\joinrel\relbar\joinrel\twoheadrightarrow}}\;
% \RbqBsquno  \qquad  \big( \cong \ZZ \big) $ }
%%%%%
  $ \; \RbqBuno \,{\buildrel \cong \over {\lhook\joinrel\relbar\joinrel\relbar\joinrel\twoheadrightarrow}}\, \RbqBsquno \;\, \big( \cong \ZZ \big) \, $.

\vskip7pt

   Before going on and studying specializations of our objects at  $ \, q = 1 \, $,
   we recall some well-known facts of quantization theory:

\vskip11pt

\begin{free text}  \label{Pois/coPois-structs}
{ \bf (Co-)Poisson structures on semiclassical limits.}\,
Let  $ A $  be any (commutative unital) ring, let $ \, p \in A \, $  be  {\sl non-invertible\/}  in  $ A \, $,
and  $ \, A_{p=0} := A \Big/\! (p) = A \Big/ p\,A \, $.  Whatever follows applies to  $ \, A \in \big\{ \Rbq \, , \RbqB \big\} \, $
and  $ \, p := q - 1 \, $  or  $ \, A \in \big\{ \Rbqsq \, , \RbqBsq \,\big\} \, $  and  $ \, p := q^{1/2} - 1 \, $.
                                                                      \par
   Consider an  $ A $--module  $ H \, $,  and let  $ \; H_{p=0} \, := \, H \Big/ p \, H \; $  be its  {\sl specialization\/}
   at  $ \, p = 0 \, $;  clearly the latter is automatically an  $ A_{p=0\,} $--module.
%
%%%%%
%                                                                       \par
%    If  $ H $  has a structure of (unital, associative)  $ A $--algebra,  then  $ H_{p=0} $
% bears a canonical quotient structure of (unital, associative)  $ A_{p=0\,} $--algebra;  and
% if  $ H $  instead is a (counital, coassociative)  $ A $--coalgebra,  then  $ H_{p=0} $  has
% a canonical quotient structure of (counital, coassociative)  $ A_{p=0\,} $--coalgebra.
% Moreover, if  $ H $  is a bialgebra or a Hopf algebra over  $ A \, $,  then the above
% mentioned quotient structures onto  $ H_{p=0} $  make it into a (quotient) bialgebra
% or Hopf algebra over  $ A_{p=0} \, $.
%%%%%
%
 If in addition  $ H $  has a structure of  $ A $--algebra,  or of a  $ A $--coalgebra,
 or of a bialgebra or Hopf algebra over  $ A \, $,  then the  $ H_{p=0} $  also inherits the same kind of (quotient)
 structure over  $ A_{p=0} \, $.
                                                                      \par
   Furthermore, the following holds (see, e.g.,  \cite[Chapter 6]{CP}):
 \vskip4pt
   {\it (a)}\,  If  $ H $  has a structure of (unital, associative)  $ A $--algebra  such that  $ H_{p=0} $  is
   {\sl commutative},  then  $ H_{p=0} $  bears a canonically structure of (unital, associative)  {\sl Poisson algebra\/}
   over  $ A_{p=0} \, $,  whose Poisson bracket is uniquely given by
%
%%%%%
% \begin{equation}
%%%%%
   $$  \{x,y\} \, := \, {{\;x'\,y' - y'\,x'\;} \over {\;p\;}}  \mod p\,H   \qquad \qquad  \forall \quad  x \, , \; y \, \in \, H_{p=0}  $$
%
%%%%%
% \end{equation}
%%%%%
%
 for any  $ \, x', y' \in H \, $  such that  $ \; x := \, x' \mod p H \, $,  $ \; y := \, y' \mod p\,H \; $.
                                                                      \par
   If in addition  $ H $  is a  {\sl bialgebra\/}  or  {\sl Hopf algebra\/}  over  $ A \, $,  then the above Poisson
   bracket together with the quotient structure of bialgebra or Hopf algebra (over  $ A_{p=0} \, $)  make
   $ H_{p=0} $  into a  {\sl Poisson bialgebra\/}  or  {\sl Poisson Hopf algebra\/}  over  $ A_{p=0} \, $.
 \vskip4pt
   {\it (b)}\,  If  $ H $  has a structure of (counital, coassociative)  $ A $--coalgebra  such that  $ H_{p=0} $  is
   {\sl cocommutative},  then  $ H_{p=0} $  bears a canonically structure of (counital, coassociative)
   {\sl co-Poisson algebra\/}  over  $ A_{p=0} \, $,  whose Poisson cobracket is uniquely given by
%
%%%%%
% \begin{equation}
%%%%%
  $$  \nabla(x) \, := \, {{\;\Delta(x') - \Delta^{\text{op}}(x')\;} \over {\;p\;}}  \mod p\,H
  \qquad \qquad  \forall \quad  x \, \in \, H_{p=0}  $$
%
%%%%%
% \end{equation}
%%%%%
%
 for any  $ \, x' \in H \, $  such that  $ \; x := \, x' \mod p\,H \; $.
                                                                      \par
   If in addition  $ H $  is a  {\sl bialgebra\/}  or  {\sl Hopf algebra\/}  over  $ A \, $,
   then the above Poisson cobracket together with the quotient structure of bialgebra or Hopf algebra
   (over  $ A_{p=0} \, $)  make  $ H_{p=0} $  into a  {\sl co-Poisson bialgebra\/}  or  {\sl co-Poisson Hopf algebra\/}
   over  $ A_{p=0} \, $.
 \vskip5pt
   As a last remark, we recall that if  $ \mathfrak{l} $  is a Lie algebra and the Hopf algebra  $ U(\mathfrak{l}) $
   is actually a  {\sl co-Poisson\/}  Hopf algebra, then  $ \mathfrak{l} $  canonically inherits a structure of Lie bialgebra,
   with the original Lie bracket and the Lie cobracket given by restriction of the Poisson cobracket in
   $ U(\mathfrak{l}) \, $.  As a consequence, if  $ \, H_{p=0} \cong U(\mathfrak{l}) \, $  as Hopf algebras
   ($ H $  as above) for some Lie algebra  $ \mathfrak{l} \, $,  then the latter is bears a Lie  {\sl bialgebra\/}  structure,
   induced by  $ H $  as explained.
\end{free text}

\vskip7pt

   Now we fix  $ \Fbq \, $,  $ \Rbq \, $,  and  $ \FbqB \, $,  $ \RbqB \, $,  as in \S \ref{gen-ground-rings}.
   Note that their generators  $ \, q_{ij} \, $  (for all $ i $  and  $ j \, $)  form inside either field  $ \Fbq $  and  $ \FbqB $
   a multiparameter matrix  $ \, \bq := {\big( q_{ij} \big)}_{i,j \in I} \, $  of Cartan type, and even of  {\sl integral\/}
   type (namely, type  $ B \, $)  in the case of  $ \FbqB \, $.  We consider then the associated MpQG's defined over
   $ \Fbq $  and over  $ \FbqB \, $,  both denoted by  $ \QEq \, $;  nevertheless,
   we shall loosely distinguish the two cases
%%%%%
% by saying that we are ``in the  {\sl general\/}  case'' when the ground ring is
% $ \Fbq $  or  $ \Rbq \, $,  and that we are instead ``in the  {\sl integral\/}
% case'' when the ground ring is  $ \FbqB $  or  $ \RbqB \, $.
%%%%%
 by saying that we are ``in the  {\sl general},  resp.\  {\sl integral},  case'' when the ground ring is
 $ \Fbq $  or  $ \Rbq \, $,  resp.\  $ \FbqB $  or  $ \RbqB \, $.
 \vskip3pt
   In the general case, we consider in  $ \QEq $  the unrestricted integral form  $ \Utilde_\bq(\hskip0,8pt\lieg) \, $,
   defined over the ring  $ \Rbq $  as in  \S \ref{Utilde}.  In the integral case instead, we pick in  $ \QEq $
   the restricted integral forms  $ \Uhatdot_\bq(\hskip0,8pt\lieg) $  and   --- in the strictly integral case ---
   $ \Uhat_\bq(\hskip0,8pt\lieg) \, $,  defined over $ \RbqB $ as in  \S \ref{Uhat},  and the unrestricted form
   $ \Utilde_\bq(\hskip0,8pt\lieg) $  too   ---  over  $ \RbqB $  again.

\vskip7pt

   We can now introduce the first type of specialization we are interested into:

\vskip13pt

\begin{definition}  \label{def_spec-1}  \
 \vskip4pt
   {\it (a)}\,  {\sl Let  $ \, \bq $  be of integral type}.  We call  {\it specialization of  $ \, \Uhatdotqgd \, $
   at  $ \, q = 1 \, $}  the quotient
 \vskip-9pt
  $$  \Uhatdotqunogd  \; := \;  \Uhatdotqgd \Big/ (q-\!1) \, \Uhatdotqgd  $$
endowed with its natural (quotient) structure of Hopf algebra over  $ \, \RbqBuno \; \big(\! = \ZZ \,\big) \, $.
                                                            \par
   As a matter of notation, setting  $ \, \Uhatdot_\bq := \Uhatdotqgd \, $  we shall denote
  $$  \displaylines{
   \bigg(\, {{\kdotrm_i \, ; \, c} \atop n} \bigg) := {\bigg( {{K_i \, ; \, c} \atop n} \bigg)}_{\!\!q}  \hskip-3pt \mod (q-\!1) \, \Uhatdot_\bq  \;\; ,   \quad
 \bigg(\, {\kdotrm_i \atop n} \bigg) := {\bigg( {K_i \atop n} \bigg)}_{\!\!q}  \hskip-3pt \mod (q-\!1) \, \Uhatdot_\bq  \cr
   \bigg(\, {{\ldotrm_i \, ; \, c} \atop n} \bigg) := {\bigg( {{L_i \, ; \, c} \atop n} \bigg)}_{\!\!q}  \hskip-3pt \mod (q-\!1) \, \Uhatdot_\bq  \;\; ,   \quad
 \bigg(\, {\ldotrm_i \atop n} \bigg) := {\bigg( {L_i \atop n} \bigg)}_{\!\!q}  \hskip-3pt \mod (q-\!1) \, \Uhatdot_\bq  \cr
   \bigg(\, {{\hrm_i \, ; \, c} \atop n} \bigg) := {\bigg( {{G_i \, ; \, c} \atop n} \bigg)}_{\!\!q_{ii}}  \hskip-8pt \mod (q-\!1) \, \Uhatdot_\bq  \;\; ,
 \quad  \bigg(\, {\hrm_i \atop n} \bigg) := {\bigg( {G_i \atop n} \bigg)}_{\!\!q_{ii}}  \hskip-8pt \mod (q-\!1) \, \Uhatdot_\bq  \cr
   \kdotrm_i := \bigg(\, {\kdotrm_i \atop 1} \bigg)  \quad ,   \;\qquad
 \ldotrm_i := \bigg(\, {\ldotrm_i \atop 1} \bigg)  \quad ,   \quad \qquad
 \hrm_i := \bigg(\, {\hrm_i \atop 1} \bigg)  \cr
   \text{e}_\alpha^{\,(n)} := E_\alpha^{\,(n)}  \mod (q-\!1) \, \Uhatdot_\bq \quad ,  \quad \qquad
 \text{f}_\alpha^{\;(n)} := F_\alpha^{\,(n)}  \mod (q-\!1) \, \Uhatdot_\bq  }  $$
 for all  $ \, i \in I \, $,  $ \, c \in \ZZ \, $,  $ \, n \in \NN\, $,  $ \, \alpha \in \Phi^+ \, $.
 \vskip3pt
   {\sl If  $ \, \bq $  is of  {\it strongly integral}  type},  then we call  {\it specialization of  $ \, \Uhatqgd \, $
   at  $ \, q = 1 \, $}  the quotient
 \vskip-19pt
  $$  \Uhatqunogd  \; := \;  \Uhatqgd \Big/ (q-\!1) \, \Uhatqgd  $$
endowed with its natural (quotient) structure of Hopf algebra over  $ \, \RbqBuno \; \big(\! = \ZZ \,\big) \, $.
Like above, setting  $ \, \Uhat_\bq := \Uhatqgd \, $  we shall write
(for  $ \, i \in I \, $,  $ \, c \in \ZZ \, $,  $ \, n \in \NN\, $,  $ \, \alpha \in \Phi^+ \, $)
  $$  \displaylines{
    \bigg(\, {{\krm_i \, ; \, c} \atop n} \bigg) := {\bigg( {{K_i \, ; \, c} \atop n} \bigg)}_{\!\!q_i}  \hskip-3pt \mod (q-\!1) \, \Uhat_\bq \quad ,   \;\quad
 \bigg(\, {\krm_i \atop n} \bigg) := {\bigg( {K_i \atop n} \bigg)}_{\!\!q_i}  \hskip-3pt \mod (q-\!1) \, \Uhat_\bq  \cr
    \bigg(\, {{\lrm_i \, ; \, c} \atop n} \bigg) := {\bigg( {{L_i \, ; \, c} \atop n} \bigg)}_{\!\!q_i}  \hskip-3pt \mod (q-\!1) \, \Uhat_\bq \quad ,   \;\quad
 \bigg(\, {\lrm_i \atop n} \bigg) := {\bigg( {L_i \atop n} \bigg)}_{\!\!q_i}  \hskip-3pt \mod (q-\!1) \, \Uhat_\bq  \cr
    \bigg(\, {{\hrm_i \, ; \, c} \atop n} \bigg) := {\bigg( {{G_i \, ; \, c} \atop n} \bigg)}_{\!\!q_{ii}}  \hskip-8pt \mod (q-\!1) \, \Uhat_\bq  \;\; ,   \quad
 \bigg(\, {\hrm_i \atop n} \bigg) := {\bigg( {G_i \atop n} \bigg)}_{\!\!q_{ii}}  \hskip-8pt \mod (q-\!1) \, \Uhat_\bq  \cr
    \krm_i := \bigg(\, {\krm_i \atop 1} \bigg)  \quad ,   \;\qquad
 \lrm_i := \bigg(\, {\lrm_i \atop 1} \bigg)  \quad ,   \quad \qquad
 \hrm_i := \bigg(\, {\hrm_i \atop 1} \bigg)  \cr
   \text{e}_\alpha^{\,(n)} := E_\alpha^{\,(n)}  \mod (q-\!1) \, \Uhat_\bq \quad ,  \quad \qquad
 \text{f}_\alpha^{\;(n)} := F_\alpha^{\,(n)}  \mod (q-\!1) \, \Uhat_\bq  }  $$
 \vskip4pt
   {\it (b)}\,  Let  $ \, \bq $  be arbitrary (of Cartan type).  We call  {\it specialization of  $ \, \Utildeqgd $
   --- defined over either  $ \, \Rbq $  or  $ \, \RbqB $  ---   at  $ \, q = 1 \, $}  the quotient
 \vskip-5pt
  $$  \Utildequnogd  \; := \;  \Utildeqgd \Big/ (q-\!1) \, \Utildeqgd  $$
endowed with its natural (quotient) structure of Hopf algebra   --- over  $ \, \Rbquno $  or  $ \RbqBuno $  respectively.
As a matter of notation, we shall denote (for all  $ \, \alpha \in \Phi^+ \, $,  $ \, i \in I \, $)
   $$  \displaylines{
    f_\alpha := \Fbar_\alpha  \mod (q-\!1) \, \Utilde_\bq(\hskip0,8pt\lieg)  \quad ,   \qquad
 e_\alpha := \Ebar_\alpha  \mod (q-\!1) \, \Utildeqgd  \cr
    \hskip9pt \hfill   l_i^{\,\pm 1} := L_i^{\pm 1}  \mod (q-\!1) \, \Utildeqgd  \;\; ,   \;\quad
 k_i^{\,\pm 1} := K_i^{\pm 1}  \mod (q-\!1) \, \Utildeqgd   \hfill  \diamondsuit  }  $$
\end{definition}

\smallskip

\begin{rmk}  \label{spec-1_ext}
 Note that the specializations introduced above can be also realized, alternatively, as scalar extensions, namely
%%%
  $ \,\; \Uhatdotqunogd := \RbqBuno \mathop{\otimes}\limits_{\;\RbqB} \Uhatdotqgd \; $,  $ \,\;
\Uhatqunogd \, := \, \RbqBuno \mathop{\otimes}\limits_{\;\RbqB} \Uhatqgd \;\, $  and
  $ \,\; \Utildequnogd \, := \, \RbqBuno \mathop{\otimes}\limits_{\;\RbqB} \Utildeqgd \;\, $  or
  --- according to what is the chosen ground ring for  $ \Utildeqgd $  ---   also
  $ \,\; \Utildequnogd := \Rbquno \mathop{\otimes}\limits_{\;\Rbq} \Utildeqgd \;\, $.
\end{rmk}

\medskip

   Our first, key result about specialization at  $ \, q = 1 \, $  is the following:

\medskip

\begin{theorem}  \label{thm:specUqhat_q=1}
 Let  $ \, \bq := {\big(\, q_{ij} = q^{b_{ij}} \,\big)}_{i, j \in I} \, $  be as above, with
 $ \, B := {\big( b_{ij} \big)}_{i, j \in I} \in M_\theta(\ZZ) \, $  such that  $ \, B + B^t = DA \, $.  Then the following holds:
 \vskip3pt
   {\it (a)}\,  $ \Uhatdotqunogd $  is a (cocommutative) co-Poisson Hopf algebra, which is isomorphic to
   $ U_\ZZ\big(\hskip0,8pt\liegdotb\big) $   --- cf.\ Definition \ref{def_Kostant-forms_U(g)}  ---
   the latter being endowed with the Poisson co-bracket uniquely induced by the Lie cobracket of
   $ \, \liegdotb $ ---  cf.\ Definition \ref{def_liegdotb-lieghatb-liegtildeb}{\it (a)}.  Indeed, an explicit isomorphism
   $ \; \Uhatdotqunogd \;{\buildrel \cong \over {\lhook\joinrel\relbar\joinrel\relbar\joinrel\relbar\joinrel\twoheadrightarrow}}\; U_\ZZ\big(\hskip0,8pt\liegdotb\big) \; $  is given by
  $$  \bigg(\, {\kdotrm_i \atop n} \bigg) \mapsto \bigg(\, {\kdotrm_i \atop n} \bigg)  \; ,
   \;\;  \bigg(\, {\ldotrm_i \atop n} \bigg) \mapsto \bigg(\, {\ldotrm_i \atop n} \bigg)  \; ,
   \;\;\bigg(\, {\hrm_i \atop n} \bigg) \mapsto \bigg(\, {\hrm_i \atop n} \bigg)  \; ,  \;\;
   \text{\rm e}_\alpha^{\,(n)} \mapsto \, \text{\rm e}_\alpha^{\,(n)}  \; ,  \;\;  \text{\rm f}_\alpha^{\,(n)} \mapsto \, \text{\rm f}_\alpha^{\,(n)}  $$
   \indent   Similar statements hold true for the specialization at  $ \, q = 1 $  of
   $ \, \Uhatdot{}_\bq^{\,\geq} \, $,  $ \Uhatdot{}_\bq^{\,\leq} \, $,  $ \Uhatdot{}_\bq^{\,0} \, $,  etc.
 \vskip3pt
   {\it (b)}\,  $ \Uhatqunogd $  is a (cocommutative) co-Poisson Hopf algebra, which is isomorphic to
   $ U_\ZZ\big(\hskip0,8pt\lieghatb\big) $   --- cf.\ Definition \ref{def_Kostant-forms_U(g)} ---
   the latter being endowed with the Poisson co-bracket uniquely induced by the Lie cobracket of  $ \, \lieghatb $   --- cf.\ Definition \ref{def_liegdotb-lieghatb-liegtildeb}{\it (c)\/}.  Indeed, an explicit isomorphism  $ \; \Uhatqunogd \;{\buildrel \cong \over {\lhook\joinrel\relbar\joinrel\relbar\joinrel\relbar\joinrel\twoheadrightarrow}}\;
   U_\ZZ\big(\hskip0,8pt\lieghatb\big) \; $  is given by
  $$  \bigg(\, {\krm_i \atop n} \bigg) \mapsto \bigg(\, {\krm_i \atop n} \bigg)  \; ,  \;\;  \bigg(\, {\lrm_i \atop n} \bigg) \mapsto
  \bigg(\, {\lrm_i \atop n} \bigg)  \; ,  \;\;\bigg(\, {\hrm_i \atop n} \bigg) \mapsto \bigg(\, {\hrm_i \atop n} \bigg)  \; ,  \;\;
  \text{\rm e}_\alpha^{\,(n)} \mapsto \, \text{\rm e}_\alpha^{\,(n)}  \; ,  \;\;  \text{\rm f}_\alpha^{\,(n)} \mapsto \, \text{\rm f}_\alpha^{\,(n)}  $$
   \indent   Similar statements hold true for the specialization at  $ \, q = 1 $  of  $ \, \Uhat_\bq^{\,\geq} \, $,  $ \Uhat_\bq^{\,\leq} \, $,  $ \Uhat_\bq^{\,0} \, $,  etc.
\end{theorem}

\begin{proof}
 By the definitions and the structure results for  $ \Uhatdotqunogd $  and  $ \Uhatqunogd $   in  \S \ref{Uhat}  (in particular,  Theorem \ref{thm:pres_Uhatgdq_gens-rels})  the proof is a straightforward check.  Indeed, from the presentation of  $ \Uhatdotqgd $  and  $ \Uhatqgd $  in Theorem \ref{thm:pres_Uhatgdq_gens-rels}  we get similar presentations of  $ \Uhatdotqunogd $  and  $ \Uhatqunogd \, $:  comparing these presentations with those mentioned in  Remark \ref{int-forms_of_U(g_D)}{\it (a)\/}  for  $ \, U_\ZZ\big(\hskip0,8pt\liegdotb\big) \, $  and  $ U_\ZZ\big(\hskip0,8pt\lieghatb\big) \, $,  sheer calculations show that the formulas in the above statement provide well-defined isomorphisms, as claimed.
                                               \par
   Hereafter we give a sample of these ``sheer calculations''.  Out of the commutation formulas among generators of  $ \, \Uhatdotqgd \, $   --- cf.\  Theorem \ref{thm:pres_Uhatgdq_gens-rels}{\it (a)}  ---   we get
  $$  \displaylines{
   {\bigg( {K_i \atop 1} \bigg)}_{\!\!q} \, E_j^{\,(n)}  \,\; = \;\,  E_j^{\,(n)} \,
   {\bigg( {{K_i \; ; \, n\,b_{ij}} \atop 1} \bigg)}_{\!\!q}  \,\; = \;\,  E_j^{\,(n)} \,
   \bigg( {\bigg( {K_i \atop 1} \bigg)}_{\!\!q} \, + \, {\big( n\,b_{ij} \big)}_q \, K_i \bigg)  \,\; =   \hfill  \cr
   \hfill   = \;\,  E_j^{\,(n)} \; {\bigg( {K_i \atop 1} \bigg)}_{\!\!q} \, + \, {\big( n \, b_{ij} \big)}_q \, E_j^{\,(n)} \, K_i  }  $$
 Then, when we specialize this formula at  $ \, q = 1 \, $   --- that is, we take it modulo  $ \, (q\!-1) \, \Uhatdotqgd \, $  ---   the left-hand side and right-hand side become, respectively,
%
%%%%%
%   $$  \displaylines{
%    {\bigg( {K_i \atop 1} \bigg)}_{\!\!q} \, E_j^{\,(n)}  \,\; \equiv \;\,  \bigg(\, {{\kdotrm_i} \atop 1} \bigg)
% \, \erm_j^{\,(n)}   \qquad \quad  \Big(\hskip-7pt \mod (q - \! 1) \, \Uhatdotqgd \,\Big)  \cr
% %
%   E_j^{\,(n)} \, {\bigg(\! {K_i \atop 1} \bigg)}_{\!\!q} \, + \, {\big( n b_{ij} \big)}_q \, E_j^{\,(n)} K_j
% \,\; \equiv \;\,  \erm_j^{\,(n)}  \bigg(\, {{\kdotrm_i} \atop 1} \bigg) \, + \, n \, b_{ij} \, \erm_j^{\,(n)}
% \;\quad  \Big(\hskip-9pt \mod (q - \! 1) \, \Uhatdotqgd \Big)  }  $$
%%%%%
%
  $$  \displaylines{
   {\bigg( {K_i \atop 1} \bigg)}_{\!\!q} \, E_j^{\,(n)}  \,\; \equiv \;\,  \bigg(\, {{\kdotrm_i} \atop 1} \bigg) \, \erm_j^{\,(n)}   \qquad \quad  \Big(\hskip-7pt \mod (q - \! 1) \, \Uhatdotqgd \,\Big)  \cr
%%%
   E_j^{\,(n)} \, {\bigg(\! {K_i \atop 1} \bigg)}_{\!\!q} \, + \, {\big( n\,b_{ij} \big)}_q \, E_j^{\,(n)} K_j  \,\; \equiv \;\,  \erm_j^{\,(n)}  \bigg(\, {{\kdotrm_i} \atop 1} \bigg) \, + \, n \, b_{ij} \, \erm_j^{\,(n)}   \;\quad  \Big(\hskip-9pt \mod (q - \! 1) \, \Uhatdotqgd \Big)  }  $$
%%%%%
%
 because  $ \; K_i \, = \, 1 \, + \, (q - \! 1) \, {\displaystyle {\bigg(\! {K_i \atop 1} \bigg)}_{\!\!q}} \, \equiv \, 1  \mod (q - \! 1) \, \Uhatdotqgd \; $.  This shows that the relation
  $$  {\bigg( {K_i \atop 1} \bigg)}_{\!\!q} \, E_j^{\,(n)}  \,\; = \;\,  E_j^{\,(n)} \;
  {\bigg( {K_i \atop 1} \bigg)}_{\!\!q} \, + \, {\big( n\,b_{ij} \big)}_q \, E_j^{\,(n)} \, K_i  $$
involving some generators of  $ \Uhatdotqgd \, $,  through the specialization process turns into
  $$  \bigg(\, {{\kdotrm_i} \atop 1} \bigg) \, \erm_j^{\,(n)}  \,\; = \;\,  \erm_j^{\,(n)}  \bigg(\, {{\kdotrm_i} \atop 1} \bigg) \, + \, n \, b_{ij} \, \erm_j^{\,(n)}  $$
among the corresponding elements in  $ U_\ZZ\big(\hskip0,8pt\liegdotb\big) \, $;
but this relation is indeed one of those occurring in the presentation of  $ U_\ZZ\big(\hskip0,8pt\liegdotb\big) $  itself by generators and relations.
                                    \par
   With a similar analysis, one sees that the generators in  $ \Uhatdotqunogd $  do respect all relations that hold true among the same name generators of  $ \, U_\ZZ\big(\hskip0,8pt\liegdotb\big) \, $.  In addition,  {\sl there are no extra relations\/}  because we have PBW bases for  $ \Uhatdotqgd $  which specialize to similar bases for  $ \Uhatdotqunogd \, $,  and the latter correspond to PBW bases of  $ U_\ZZ\big(\hskip0,8pt\liegdotb\big) \, $.
                                        \par
   Finally, since  $ \,\; K_i \, = \, 1 \, + \, (q - \! 1) \, {\displaystyle {\bigg( {K_i \atop 1} \bigg)}_{\!\!q}} \, \equiv \, 1  \mod (q - \! 1) \, \Uhatdotqgd \;\, $  and also  $ \,\; L_i \, = \, 1 \, + \, (q - \! 1) \, {\displaystyle {\bigg( {L_i \atop 1} \bigg)}_{\!\!q}} \, \equiv \, 1 \; \mod (q - \! 1) \, \Uhatdotqgd \;\, $,  it follows from Theorem \ref{thm:pres_Uhatgdq_gens-rels}  that
   $ \; \com\Big(\! E_i^{(n)} \Big) \,\equiv\, {\textstyle \sum\limits_{s=0}^n} E_i^{\,(n-s)} \otimes E_i^{\,(s)} \; $  and
   $ \; \com\Big(\! F_i^{(n)} \Big) \,\equiv\, {\textstyle \sum\limits_{s=0}^n} F_i^{\,(n-s)} \otimes F_i^{\,(s)} \; $  modulo
   $ \; \bigg( (q - \! 1) \, \Uhatdotqgd \, \otimes \Uhatdotqgd \,+\, \Uhatdotqgd\ot (q - \! 1) \, \Uhatdotqgd \bigg) \;\, $.
 This implies that  $ \, \Uhatdotqunogd \, $  is a cocommutative Hopf algebra.
\end{proof}

\vskip9pt

\begin{rmk}
 In sight of  Theorem \ref{thm:specUqhat_q=1}  above, the fact that the  $ \liegdotb $'s,  resp.\  the
 $ \lieghatb $'s,  for different  $ \bq $'s  be all isomorphic
 {\sl as Lie coalgebras}   --- cf.\ Remarks \ref{int-forms_of_g_B}{\it (c)}  ---   is a direct consequence of the fact that all the
 $ \Uhatdotqgd $'s,  resp.\  the  $ \Uhatqgd $'s,  for different  $ \bq $  are isomorphic  {\sl as coalgebras},
 as this happens for the  $ U_\bq $'s.
\end{rmk}

\vskip7pt

   Next, we study the structure of  $ \; \Utildequnogd \, := \, \Utilde_\bq(\hskip0,8pt\lieg) \Big/ (q - \! 1) \, \Utilde_\bq(\hskip0,8pt\lieg) \; $.
%%%%%
% ,  \,the specialization at  $ \, q = 1 \, $  of the (unrestricted)
% $ \Rbq $--integral  form  $ \Utilde_\bq(\hskip0,8pt\lieg) \, $.
%%%%%
 For the first results, the multiparameter  $ \, \bq $  is assumed to be generic, i.e.\ just of Cartan type.
 \vskip5pt
   Let  $ \; \Utilde_\bq^{\scriptscriptstyle {\sqrt{\phantom{I}}}} \! := \Rbqsq \otimes_{\raise-3pt \hbox{$ \scriptstyle \Rbq $}} \hskip-5pt \Utilde_\bq \; $,  and let  $ \, \Utilde_{\bq\hskip1pt,1}^{\scriptscriptstyle {\sqrt{\phantom{I}}}}(\hskip0,8pt\lieg) := \Utilde_\bq^{\scriptscriptstyle {\sqrt{\phantom{I}}}} \Big/ \big(\, q^{1/2} \! - \! 1 \big) \, \Utilde_\bq^{\scriptscriptstyle {\sqrt{\phantom{I}}}} \; $
   be the specia\-lization  of $ \, \Utilde_\bq^{\scriptscriptstyle {\sqrt{\phantom{I}}}} \, $  at  $ \; q^{1/2} = 1 \, $.  For any affine Poisson group-scheme  $ \widetilde{G}^{\,*}_{\!\scriptscriptstyle DA} $  over
 $ \, \mathcal{R}_{\,\check{\bq}\hskip1pt,1}^{\,\scriptscriptstyle {\sqrt{\phantom{I}}}} $  dual to
 $ \, \tilde{\lieg}_{\raise-2pt\hbox{$ \scriptscriptstyle DA $}} \, $,  i.e.\
 $ \, \text{\sl Lie}\big( \widetilde{G}^{\,*}_{\!\scriptscriptstyle DA} \,\big)
 \cong \tilde{\lieg}_{\raise-2pt\hbox{$ \!\scriptscriptstyle DA $}}^{\,*} \; $,
%
%%%%%
% as Lie bialgebras, we denote by  $ \, \Oc^{\raise1pt \hbox{$ \scriptscriptstyle
% {\sqrt{\phantom{I}}} $}}\!\big(\widetilde{G}^{\,*}_{\!\scriptscriptstyle DA}\big) \, $
% its representing Hopf algebra.
%%%%%
%
 \,we let  $ \, \Oc^{\raise1pt \hbox{$ \scriptscriptstyle {\sqrt{\phantom{I}}} $}}\!\big(\widetilde{G}^{\,*}_{\!\scriptscriptstyle DA}\big) \, $  be its representing Hopf algebra.

\vskip13pt

\begin{prop}  \label{Utildesquno_2-coc-def}  {\ }
 \vskip3pt
 $ \, \Utilde_{\bq\hskip1pt,1}^{\scriptscriptstyle{\sqrt{\phantom{I}}}}(\hskip0,8pt\lieg) \, $  is a  $ 2 $--cocycle  deformation of  $ \, \Oc^{\raise1pt \hbox{$ \scriptscriptstyle {\sqrt{\phantom{I}}} $}}\!\big(\widetilde{G}^{\,*}_{\!\scriptscriptstyle DA}\big) \, $  for some (uniquely defined) connected, simply connected affine Poisson group-scheme  $ \widetilde{G}^{\,*}_{\!\scriptscriptstyle DA} $  over  $ \, \mathcal{R}_{\,\check{\bq}\hskip1pt,1}^{\,\scriptscriptstyle {\sqrt{\phantom{I}}}} $  dual to  $ \, \tilde{\lieg}_{\raise-2pt\hbox{$ \scriptscriptstyle DA $}} \, $  (as above).
\end{prop}

\pf
 Having taken the largest ground ring  $ \Rbqsq $  instead of  $ \Rbq \, $,  Proposition \ref{prop:Utilde-cocy-def}  applies, giving us
 $ \; \Utilde_\bq^{\scriptscriptstyle {\sqrt{\phantom{I}}}}(\hskip0,8pt\lieg) =
 {\Big( \Utilde_{\check{\bq}}^{\scriptscriptstyle {\sqrt{\phantom{I}}}}(\hskip0,8pt\lieg) \!\Big)}_{\!\sigma} \; $
 for a specific  $ 2 $--cocycle  $ \sigma $  as in Definition \ref{def-sigma}    --- in particular, depending only on the  $ q_{i{}j}^{\pm 1/2} \, $'s.  By its very construction this  $ \sigma $  induces, modding out  $ \, \big(\, q^{1/2} \! - 1 \big) \, $,  a similar  $ 2 $--cocycle,  denoted  $ \sigma_{\raise-2pt\hbox{$ \scriptstyle 1 $}} \, $,  of the specialized Hopf algebra
 $ \; \Utilde_{\check{\bq}\hskip1pt,1}^{\scriptscriptstyle {\sqrt{\phantom{I}}}}(\hskip0,8pt\lieg) := \Utilde_{\check{\bq}}(\hskip0,8pt\lieg) \Big/ \big(\, q^{1/2} \! - \! 1 \big) \, \Utilde_{\check{\bq}}(\hskip0,8pt\lieg) \; $;
 \,therefore we get
\begin{equation}
   \Utilde_\bq^{\scriptscriptstyle {\sqrt{\phantom{I}}}}(\hskip0,8pt\lieg) \Big/ \big(\, q^{1/2} \! - \! 1 \big) \,
   \Utilde_\bq^{\scriptscriptstyle {\sqrt{\phantom{I}}}}(\hskip0,8pt\lieg)  \;\; = \;\;
   {\Big(\; \Utilde_{\check{\bq}}^{\scriptscriptstyle {\sqrt{\phantom{I}}}}(\hskip0,8pt\lieg) \Big/ \big(\, q^{1/2} \! - \! 1 \big) \,
   \Utilde_{\check{\bq}}^{\scriptscriptstyle {\sqrt{\phantom{I}}}}(\hskip0,8pt\lieg) \Big)}_{\sigma_{\raise-2pt\hbox{$ \scriptscriptstyle 1 $}}}
   \label{Utqsq1 = Utqsqch1-sigma}
\end{equation}
   \indent   Now, for the usual one-parameter quantum group  $ U_q(\lieg) $  in  \cite{DP}   --- cf.\ also  \cite{Gav}  ---   one has a similar construction for  $ \Utilde_q(\lieg) $   --- which is nothing but the quotient of  $ \Utilde_{\check{\bq},1}(\hskip0,8pt\lieg) $  modulo  $ \, \big( L_i \! - \! K_i^{-1} \big) \, $  for all  $ i \, $  ---   for which one has
\begin{equation}  \label{spec_Utilde-DP}
  \Utilde_1(\lieg)  \,\; := \;\,  \Utilde_q(\lieg) \Big/ \big(\, q^{1/2} \! - 1 \big) \, \Utilde_q(\lieg)  \,\; = \;\,  \Oc\big(\widetilde{G}^*\big)
\end{equation}
for some (uniquely defined) connected, simply connected affine Poisson group-scheme
 $ \widetilde{G}^* $  whose cotangent Lie bialgebra is such that
 $ \; \Lie\big(\widetilde{G}^*\big) \, \cong \, {\Big(\; \tilde{\lieg}_{\raise-2pt\hbox{$ \scriptscriptstyle DA $}} \,
\Big/ {\big(\, \krm_i \! + \lrm_i \big)}_{i \in I} \Big)}^* \; $  as Lie bialgebras.  Once more, this result can be
easily ``lifted'' to the level of the quantum double of  $ U_q(\lieg) $,  which is nothing but  $ \QEqcheck $:
the resulting construction is exactly that of the integral form  $ \Utilde_{\check{\bq}}(\hskip0,8pt\lieg) $  within  $ \QEqcheck $,  and the results in  \cite{DP}  then turn into (sort of) a ``quantum double version'' of  (\ref{spec_Utilde-DP}),  namely
  $$  \Utilde_{\check{\bq},1}(\hskip0,8pt\lieg)  \,\; := \;\,  \Utilde_{\check{\bq}}(\hskip0,8pt\lieg) \Big/ (q - \! 1) \, \Utilde_{\check{\bq}}(\hskip0,8pt\lieg)  \,\; = \;\,  \Oc\big(\widetilde{G}^{\,*}_{\!\scriptscriptstyle DA}\big)  $$
 with  $ \widetilde{G}^{\,*}_{\!\scriptscriptstyle DA} $  a connected Poisson group-scheme whose cotangent Lie bialgebra is  $ \tilde{\lieg}_{\raise-2pt\hbox{$ \scriptscriptstyle DA $}} \, $.  Extending scalars from  $ \Rbquno $  to  $ \Rbqsquno $  this yields
\begin{equation}
  \Utilde_{\check{\bq},1}^{\scriptscriptstyle {\sqrt{\phantom{I}}}}(\hskip0,8pt\lieg)  \,\; := \;\,
  \Utilde_{\check{\bq}}^{\scriptscriptstyle {\sqrt{\phantom{I}}}}(\hskip0,8pt\lieg) \Big/ \big(\, q^{1/2} \! - \! 1 \big) \,
  \Utilde_{\check{\bq}}^{\scriptscriptstyle {\sqrt{\phantom{I}}}}(\hskip0,8pt\lieg)  \,\; = \;\,
  \Oc^{\raise1pt \hbox{$ \scriptscriptstyle {\sqrt{\phantom{I}}} $}}\!\big(\widetilde{G}^{\,*}_{\!\scriptscriptstyle DA}\big)   \label{spec_Utildesq-prim}
\end{equation}
 Finally, matching  (\ref{Utqsq1 = Utqsqch1-sigma})  and  (\ref{spec_Utildesq-prim})  the claim is proved.
\epf

\medskip

   The previous result can be reformulated as follows:  {\sl up to scalar extension   --- from  $ \Rbquno $  to  $ \Rbqsquno $  ---   the Hopf algebra  $ \, \Utildequnogd $  is a  $ 2 $--cocycle  deformation of the Hopf algebra  $ \, \Oc\big(\widetilde{G}^{\,*}_{\!\scriptscriptstyle DA}\big) \, $}.  Actually, we can provide the following, more precise statement:

\medskip

\begin{theorem}  \label{struct-Utildequno}
 \  $ \Utilde_{\bq\hskip1pt,1}(\hskip0,8pt\lieg) $  is a  $ \, \by $--polynomial  and Laurent
 $ \, \by $--polynomial  algebra over  $ \, \Rbquno $  (with  $ \, \by $  as in  \S \ref{spec-1}),  namely
  $$  \Utildequnogd  \;\; \cong \;\,  \Rbquno \Big[{\big\{ f_\alpha , l_i^{\pm 1} , k_i^{\pm 1} , e_\alpha \big\}}_{\alpha \in \Phi^+}^{i \in I} \Big]  $$
where the indeterminates  $ \mathbf{y} $--commute  among them in the following sense:
  $$  \displaylines{
   f_{\alpha'} \, f_{\alpha''} \, = \, y_{\alpha'' \alpha'} \, f_{\alpha''} \, f_{\alpha'}  \;\; ,  \qquad
e_{\alpha'} \, f_{\alpha''} \, = \, f_{\alpha''} \, e_{\alpha'}  \;\; ,  \qquad
e_{\alpha'} \, e_{\alpha''} \, = \, y_{\alpha' \alpha''} \, e_{\alpha''} \, e_{\alpha'}  \cr
   k_i^{\pm 1} \, e_\alpha \, = \, y_{\alpha_i \, \alpha}^{\,\pm 1} \, e_\alpha \, k_i^{\pm 1}  \quad ,
   \quad \qquad  l_i^{\pm 1} \, e_\alpha \, = \, y_{\alpha \, \alpha_i}^{\,\mp 1} \, e_\alpha \, l_i^{\pm 1}  \cr
   k_i^{\pm 1} \, f_\alpha \, = \, y_{\alpha_i \, \alpha}^{\,\mp 1} \, f_\alpha \, k_i^{\pm 1}  \quad ,
   \quad \qquad  l_i^{\pm 1} \, f_\alpha \, = \, y_{\alpha \, \alpha_i}^{\,\pm 1} \, f_\alpha \, l_i^{\pm 1}  \cr
   k_i^{\pm 1} \, k_j^{\pm 1} \, = \, k_j^{\pm 1} \, k_i^{\pm 1}  \;\; ,  \qquad  k_i^{\pm 1} \, l_j^{\pm 1} \,
   = \, l_j^{\pm 1} \, k_i^{\pm 1}  \;\; ,  \qquad  k_i^{\pm 1} \, k_j^{\pm 1} \, = \, k_j^{\pm 1} \, k_i^{\pm 1}  }  $$
\end{theorem}

\pf
 All formulas but those in the first line are direct consequence of definitions, so we can dispose of them, and we are left with proving the first three.
                                                                \par
   We begin with the mid formula  $ \; e_{\alpha'} \, f_{\alpha''} \, = \, f_{\alpha''} \, e_{\alpha'} \, $,  \,
   for which we have to compare  $ \; \Ebar_{\alpha'} \, \Fbar_{\alpha''} \; $  with
   $ \; \Fbar_{\alpha''} \, \Ebar_{\alpha'} \; $  within  $ \Utilde_\bq(\hskip0,8pt\lieg) \, $;
   and in order to do that, we shall compare these products with the similar product taken inside
   $ \Utilde_{\check{\bq}}(\hskip0,8pt\lieg) \, $,  where the multiplication is deformed by a  $ 2 $--cocycle
   $ \sigma $  as in  Proposition \ref{prop:Utilde-cocy-def}.
                                                                \par
   Indeed, in the rest of the proof we extend scalars from  $ \Rbq $  to  $ \, \Rbqsq $  and thus work with
   $ \, \Utilde_\bq^{\scriptscriptstyle \sqrt{\ }}(\hskip0,8pt\lieg) \, $  and
   $ \, \Utilde_{\check{\bq}}^{\scriptscriptstyle \sqrt{\ }}(\hskip0,8pt\lieg) \, $;  for the former we identify
   $ \, \Utilde_\bq^{\scriptscriptstyle \sqrt{\ }}(\hskip0,8pt\lieg) = {\big(\, \Utilde_{\check{\bq}}^{\scriptscriptstyle \sqrt{\ }}(\hskip0,8pt\lieg) \big)}_\sigma \, $
   as  $ \Rbqsq $--modules,  which is correct by  Proposition \ref{Utildesquno_2-coc-def}.  In particular, inside
   $ \, \Utilde_{\check{\bq}}^{\scriptscriptstyle \sqrt{\ }}(\hskip0,8pt\lieg) \, $  we shall consider the original product of
   $ \Utilde_{\check{\bq}}^{\scriptscriptstyle \sqrt{\ }}(\hskip0,8pt\lieg) \, $,  hereafter denoted
   by  ``$ \, \check{\cdot} \; $'',  and the  $ \sigma $--deformed  product --- yielding the product in
   $ \Utilde_{\check{\bq}}^{\scriptscriptstyle \sqrt{\ }}(\hskip0,8pt\lieg) $  ---   denoted by  ``$ \, \smallast \, $''.
 \vskip3pt
   By the results in  \cite{DP}   --- cf.\ also  \cite{Gav}  ---   suitably adapted to the present ``quantum double setup'', we know that  $ \Utilde_{\check{\bq}}(\hskip0,8pt\lieg) $  is commutative modulo  $ (q-\!1) \, $:  this implies that  $ \; \check{\!\Ebar}_{\alpha'} \; \check{\cdot} \;\,\check{\!\Fbar}_{\alpha''} \; $ in  $ \Utilde_{\check{\bq}}(\hskip0,8pt\lieg) $  can be written as
\begin{equation}  \label{Ea-comm-Fa_Utbqcheck}
  \check{\!\Ebar}_{\alpha'} \; \check{\cdot} \;\,\check{\!\Fbar}_{\alpha''}  \,\; = \;\,  \check{\!\Fbar}_{\alpha''} \; \check{\cdot} \;\,
  \check{\!\Ebar}_{\alpha'} \, + \, (q-\!1) \, {\textstyle \sum_s} \, c_s \; {\overline{\mathcal{M}}}^{\;\check{\cdot}}_s
\end{equation}
for some  $ \, c_s \in \Rbq \, $,  where  $ \,\check{\!\Ebar}_{\alpha'} $  and  $ \,\check{\!\Fbar}_{\alpha''} $  are (renormalised) quantum root vectors in  $ \Utilde_{\check{\bq}}(\hskip0,8pt\lieg) $  and the
$ \, {\overline{\mathcal{M}}}^{\;\check{\cdot}}_s \, $'s  are PBW monomials in a PBW basis of
$ \Utilde_{\check{\bq}}(\hskip0,8pt\lieg) $  like in  Theorem \ref{thm:PBW_tilde-MpQG}{\it (d)}.
                                                                \par
   Let us look now for the counterpart of formula  (\ref{Ea-comm-Fa_Utbqcheck})  in  $ \Utilde_\bq(\hskip0,8pt\lieg) \, $
   --- thought of as embedded into  $ \, \Utilde_\bq^{\scriptscriptstyle \sqrt{\ }}(\hskip0,8pt\lieg) =
   {\big(\, \Utilde_{\check{\bq}}^{\scriptscriptstyle \sqrt{\ }}(\hskip0,8pt\lieg) \big)}_\sigma \; $.
   Thanks to  Proposition \ref{prop: proport_root-vects}  we have
  $$  \Ebar_{\alpha'} \, = \, m^+_{\alpha'}\big( \bq^{\pm 1/2} \big) \, \check{\!\Ebar}_{\alpha'}  \quad ,  \qquad  \Fbar_{\alpha''} \,
  = \, m^-_{\alpha''}\big( \bq^{\pm 1/2} \big) \, \check{\!\Fbar}_{\alpha''}  $$
for suitable Laurent monomials  $ \, m^+_{\alpha'}\big( \bq^{\pm 1/2} \big) \, $  and  $ \, m^-_{\alpha''}\big( \bq^{\pm 1/2} \big) \, $
in the  $ q_{ij}^{1/2} \, $'s  (each of which is trivial if the corresponding root is simple);
 Now, the formulas in  \S \ref{subsubsec:comp-formulas}  give
  $$  \check{\!\Ebar}_{\alpha'} \,\smallast\; \check{\!\Fbar}_{\alpha''} \, = \, \check{\!\Ebar}_{\alpha'} \;\check{\cdot}\;\,
  \check{\!\Fbar}_{\alpha''}  \quad ,  \qquad  \check{\!\Fbar}_{\alpha''} \,\smallast\; \check{\!\Ebar}_{\alpha'} \,
  = \, \check{\!\Fbar}_{\alpha''} \;\check{\cdot}\;\, \check{\!\Ebar}_{\alpha'}  $$
(by the same analysis as that before  Proposition \ref{prop:def-sigma=def-c});  on the other hand, again by
Proposition \ref{prop: proport_root-vects}  and by  \S \ref{cotwist-defs_2}
we have that every PBW monomial in  $ \Utilde_{\check{\bq}}(\hskip0,8pt\lieg) \, $,  say
$ {\overline{\mathcal{M}}}^{\;\check{\cdot}} \, $,  has the form
%
%%%%%
%   $$  {\overline{\mathcal{M}}}^{\;\check{\cdot}}  \; = \;
% m_{\,\raise-1pt\hbox{$ \scriptscriptstyle \overline{\mathcal{M}} $}}
% \big( \bq^{\pm 1/2} \big) \, \overline{\mathcal{M}}^{\,\smallast}  $$
%%%%%
%
 $ \; {\overline{\mathcal{M}}}^{\;\check{\cdot}} = m_{\,\raise-1pt\hbox{$ \scriptscriptstyle \overline{\mathcal{M}} $}} \big( \bq^{\pm 1/2} \big) \, \overline{\mathcal{M}}^{\,\smallast} \; $
 where  $ m_{\,\raise-1pt\hbox{$ \scriptscriptstyle \overline{\mathcal{M}} $}}\big( \bq^{\pm 1/2} \big) $
is a suitable Laurent monomial in the  $ q_{ij}^{\pm 1/2} $'s.  Tidying everything up, from
(\ref{Ea-comm-Fa_Utbqcheck})  and the identities here above   --- writing  $ \, m^\pm_\alpha :=
m^\pm_\alpha\big( \bq^{\pm 1/2} \big) \, $  and
$ \, m_{\,\raise-1pt\hbox{$ \scriptscriptstyle \overline{\mathcal{M}} $}} := m_{\,\raise-1pt\hbox{$ \scriptscriptstyle \overline{\mathcal{M}} $}}
\big( \bq^{\pm 1/2} \big) \, $  ---   we find
  $$  \displaylines{
   \Ebar_{\alpha'} \,\smallast\; \Fbar_{\alpha''}  \,\; = \;\,
m^+_{\alpha'} \, m^-_{\alpha''} \;\,
 \check{\!\Ebar}_{\alpha'} \,\smallast\; \check{\!\Fbar}_{\alpha''}  \; = \;\,
m^+_{\alpha'} \, m^-_{\alpha''} \;\,
 \check{\!\Ebar}_{\alpha'} \;\check{\cdot}\;\, \check{\!\Fbar}_{\alpha''}  \; =   \hfill  \cr
   = \;\,  m^+_{\alpha'} \, m^-_{\alpha''} \;
\Big(\;\, \check{\!\Fbar}_{\alpha''} \; \check{\cdot} \;\, \check{\!\Ebar}_{\alpha'} \, + \, (q-\!1) \, {\textstyle \sum_s} \, c_s \;
{\overline{\mathcal{M}}}^{\;\check{\cdot}}_s \,\Big)  \; =   \qquad \qquad \qquad  \cr
   \hfill   = \;\,  m^+_{\alpha'} \, m^-_{\alpha''} \;
\Big( {\big( m^+_{\alpha'} \big)}^{-1} \, {\big( m^-_{\alpha''} \big)}^{-1} \, \Fbar_{\alpha''} \;\smallast\;\, \Ebar_{\alpha'} \,
+ \, (q-\!1) \, {\textstyle \sum_s} \, c_s \; m_{\,\raise-1pt\hbox{$ \scriptscriptstyle {\overline{\mathcal{M}}}_s $}} \,
{\overline{\mathcal{M}}}^{\,\smallast}_s \,\Big)  \; =   \quad  \cr
   \hfill   = \;\,  \Fbar_{\alpha''} \;\smallast\; \Ebar_{\alpha'} \, + \, (q-\!1) \, {\textstyle \sum_s} \, c_s \; m^+_{\alpha'}
   \, m^-_{\alpha''} \, m_{\,\raise-1pt\hbox{$ \scriptscriptstyle {\overline{\mathcal{M}}}_s $}} \,
   {\overline{\mathcal{M}}}^{\,\smallast}_s  }  $$
that is, in the end,
\begin{equation}  \label{Ea'-Fa"_with sqrt}
  \Ebar_{\alpha'} \;\smallast\; \Fbar_{\alpha''}  \,\; = \;\,  \Fbar_{\alpha''} \;\smallast\; \Ebar_{\alpha'} \, + \, (q - \! 1) \,
  {\textstyle \sum_s} \, c_s \; m^+_{\alpha'} \, m^-_{\alpha''} \, m_{\,\raise-1pt\hbox{$ \scriptscriptstyle {\overline{\mathcal{M}}}_s $}}
  \, {\overline{\mathcal{M}}}^{\,\smallast}_s
\end{equation}
which is almost what we need, as the right-hand side belongs to
$ \, \Utilde_\bq^{\scriptscriptstyle \sqrt{\ }}(\hskip0,8pt\lieg) \, $
but possibly not to  $ \, \Utilde_\bq(\hskip0,8pt\lieg) \, $.  To fix this detail, we take the expansion of
$ \; \Ebar_{\alpha'} \;\smallast\; \Fbar_{\alpha''} \; $  as an  $ \, \Rbq $--linear  combination of the PBW basis of the  $ \,  {\overline{\mathcal{M}}}^{\,\smallast}_r \, $'s  (which includes  $ \, \Fbar_{\alpha''} \;\smallast\; \Ebar_{\alpha'} \, $  too), namely  $ \,\; \Ebar_{\alpha'} \;\smallast\; \Fbar_{\alpha''} \, = \, {\textstyle \sum_r} \, \kappa_r \; {\overline{\mathcal{M}}}^{\,\smallast}_r \; $  for some  $ \, \kappa_r \in \Rbq \, $;  comparing the latter with  \eqref{Ea'-Fa"_with sqrt}  we get  $ \; c_s \; m^+_{\alpha'} \, m^-_{\alpha''} \, m_{\,\raise-1pt\hbox{$ \scriptscriptstyle {\overline{\mathcal{M}}}_s $}} \! \in \Rbq \; $  for every  $ s \, $.  Then  \eqref{Ea'-Fa"_with sqrt}  is an identity in  $ \Utildeqgd \, $,  which implies
  $$  \Ebar_{\alpha'} \;\smallast\; \Fbar_{\alpha''}  \; = \; \Fbar_{\alpha''} \;\smallast\; \Ebar_{\alpha'}   \qquad  \mod (q - \! 1) \, \Utildeqgd  $$
whence eventually  $ \; e_{\alpha'} \, f_{\alpha''} \, = \, f_{\alpha''} \, e_{\alpha'} \; $,  \, q.e.d.
 \vskip5pt
   We turn now to proving the identity  $ \; e_{\alpha'} \, e_{\alpha''} \, = \, y_{\alpha' \alpha''} \, e_{\alpha''} \, e_{\alpha'} \, $,  \,
   for which we need to compare  $ \; \Ebar_{\alpha'} \, \Ebar_{\alpha''} \; $  with  $ \; \Ebar_{\alpha''} \, \Ebar_{\alpha'} \; $
   within  $ \Utilde_\bq(\hskip0,8pt\lieg) \, $.  To begin with,
 from the results  \cite[\S\S 9, 10 and 12]{DP}   --- suitably adapted, as usual, to the present, ``quantum double framework'' ---
 in the standard case of  $ \Utilde_{\check{\bq}}(\hskip0,8pt\lieg) $  we have
\begin{equation}  \label{Etildea'-comm-Etildea''_Utbq-prel}
  \check{\!\Ebar}_{\alpha'} \;\check{\cdot}\;\, \check{\!\Ebar}_{\alpha''}  \,\; = \;\,  \check{q}_{\alpha',\alpha''} \;
  \check{\!\Ebar}_{\alpha''} \;\check{\cdot}\;\, \check{\!\Ebar}_{\alpha'} \, + \, (q - \! 1) \, {\textstyle \sum_{\underline{\alpha}}} \,
  \check{c}_{\underline{\alpha}} \; \check{\!\Ebar}_{\underline{\alpha}}^{\;\raise-3pt\hbox{$ \scriptstyle e_{\underline{\alpha}} $}}
\end{equation}
for all  $ \, \alpha', \alpha'' \in \Phi^+ \, $,  where  $ \, \check{q}_{\alpha',\alpha''} = q^{(\alpha',\alpha'')/2} \, $
by definition,  $ \, \check{c}_{\underline{\alpha}} \in \Zqqm \; \big(\! \subseteq \Rbq \big) \, $  for all  $ \underline{\alpha} $
and the  $ \, \check{\!\Ebar}_{\underline{\alpha}}^{\;\raise-3pt\hbox{$ \scriptstyle e_{\underline{\alpha}} $}} \,$'s  are
PBW monomials in the  $ \, \check{\!\Ebar}_\alpha $'s  alone.
                                                       \par
   Now from  \eqref{Etildea'-comm-Etildea''_Utbq-prel}  we deduce a parallel identity in  $ \Utildeqgd \, $.
   Namely, acting like in the first part of the proof   --- basing again on the formulas in  \S \ref{subsubsec:comp-formulas}  ---   we find
  $$  \displaylines{
   \Ebar_{\alpha'} \;\smallast\; \Ebar_{\alpha''}  \; = \;
m^+_{\alpha'} \, m^+_{\alpha''} \;\, \check{\!\Ebar}_{\alpha'} \,\smallast\; \check{\!\Ebar}_{\alpha''}  \; =
\;  m^+_{\alpha'} \, m^-_{\alpha''} \, q_{\alpha',\alpha''}^{\,+1/2} \;\; \check{\!\Ebar}_{\alpha'} \;\check{\cdot}\;\, \check{\!\Ebar}_{\alpha''}  \; =   \hfill  \cr
   \hfill   = \;\,  m^+_{\alpha'} \, m^-_{\alpha''} \, q_{\alpha',\alpha''}^{\,+1/2} \;\, \Big(\; \check{q}_{\alpha',\alpha''} \;
   \check{\!\Ebar}_{\alpha''} \;\check{\cdot}\;\, \check{\!\Ebar}_{\alpha'} \, + \, (q - \! 1) \, {\textstyle \sum_{\underline{\alpha}}} \,
   \check{c}_{\underline{\alpha}} \; \check{\!\Ebar}_{\underline{\alpha}}^{\;\raise-3pt\hbox{$ \scriptstyle e_{\underline{\alpha}} $}} \,\Big)  \; =   \quad  \cr
   \qquad   = \;\,  m^+_{\alpha'} \, m^+_{\alpha''} \, q_{\alpha',\alpha''}^{\,+1/2} \; \check{q}_{\alpha',\alpha''} \,
   {\big( m^+_{\alpha''} \big)}^{-1} \, {\big( m^+_{\alpha'} \big)}^{-1} \, q_{\alpha'',\alpha'}^{\,-1/2} \; \Ebar_{\alpha''} \;\smallast\; \Ebar_{\alpha'} \,\; +   \hfill  \cr
   \hfill   + \;\, (q - \! 1) \, {\textstyle \sum_{\underline{\alpha}}} \, \check{c}_{\underline{\alpha}} \; m^+_{\alpha'} \, m^+_{\alpha''} \,
   q_{\alpha',\alpha''}^{\,+1/2} \; \mu_{\underline{\alpha}} \; \Ebar_{\underline\alpha}^{\;\raise-3pt\hbox{$ \scriptstyle e_{\underline{\alpha}} $}}  \,\;
   =   \quad  \cr
   \hfill   = \;\,  q_{\alpha',\alpha''}^{\,+1/2} \; \check{q}_{\alpha',\alpha''} \, q_{\alpha'',\alpha'}^{\,-1/2} \; \Ebar_{\alpha''}
   \;\smallast\; \Ebar_{\alpha'} \; + \, (q-\!1) \, {\textstyle \sum_{\underline{\alpha}}} \, \check{c}_{\underline{\alpha}} \;
   m^+_{\alpha'} \, m^+_{\alpha''} \, q_{\alpha',\alpha''}^{\,+1/2} \; \mu_{\underline{\alpha}} \;
   \Ebar_{\underline\alpha}^{\;\raise-3pt\hbox{$ \scriptstyle e_{\underline{\alpha}} $}}  }  $$
where  $ \mu_{\underline{\alpha}} $  is yet another Laurent monomial in the  $ q_{ij}^{\pm 1/2} $'s  and each
$ \Ebar_{\underline\alpha}^{\;\raise-3pt\hbox{$ \scriptstyle e_{\underline{\alpha}} $}} $  is the unique PBW monomial in the
$ \Ebar_\alpha $'s  that corresponds (in an obvious sense) to
$ \check{\!\Ebar}_{\underline{\alpha}}^{\;\raise-3pt\hbox{$ \scriptstyle e_{\underline{\alpha}} $}} \, $. Thus
\begin{equation}  \label{Ea'-Ea"_with sqrt}
  \Ebar_{\alpha'} \;\smallast\; \Ebar_{\alpha''}  \,\; = \;\,  q_{\alpha',\alpha''}^{\,+1/2} \; \check{q}_{\alpha',\alpha''} \, q_{\alpha'',\alpha'}^{\,-1/2} \;
  \Ebar_{\alpha''} \;\smallast\; \Ebar_{\alpha'} \; + \, (q - \! 1) \, {\textstyle \sum_{\underline{\alpha}}} \, c_{\underline{\alpha}} \;
  \Ebar_{\underline\alpha}^{\;\raise-3pt\hbox{$ \scriptstyle e_{\underline{\alpha}} $}}
\end{equation}
where  $ \; c_{\underline{\alpha}} \, := \, \check{c}_{\underline{\alpha}} \; m^+_{\alpha'} \, m^+_{\alpha''} \, q_{\alpha',\alpha''}^{\,+1/2} \;
\mu_{\underline{\alpha}} \, \in \, \Rbqsq \; $.  But we also know that the  $ \Ebar_{\underline\alpha} $'s  form a
PBW basis over  $ \Rbq $  for  $ \Utildeqgd \, $,  hence  $ \; \Ebar_{\alpha'} \,\smallast\; \Ebar_{\alpha''} \; $  uniquely expands
into an  $ \Rbq $--linear  combination of these monomials: comparing such an expansion with  \eqref{Ea'-Ea"_with sqrt}
we find that all coefficients  $ c_{\underline{\alpha}} $  therein necessarily belong to  $ \Rbq \, $:  then  \eqref{Ea'-Ea"_with sqrt}
itself is an identity in  $ \Utildeqgd $   --- i.e., not only in  $ \Utilde_\bq^{\scriptscriptstyle \sqrt{\ }}\!(\hskip0,8pt\lieg) \, $.
Therefore, from  \eqref{Ea'-Ea"_with sqrt}  we deduce
\begin{equation}  \label{Ea'-Ea"_modulo_(q-1)}
  \Ebar_{\alpha'} \;\smallast\; \Ebar_{\alpha''}  \,\; \cong \;\,  q_{\alpha',\alpha''}^{\,+1/2} \; \check{q}_{\alpha',\alpha''} \,
  q_{\alpha'',\alpha'}^{\,-1/2} \; \Ebar_{\alpha''} \;\smallast\; \Ebar_{\alpha'}   \qquad  \mod (q-\!1) \, \Utildeqgd
\end{equation}
   \indent   Finally, since  $ \, \bq := {\big( q_{ij} \big)}_{i, j \in I} \, $  and  $ \, \check{q}_{ij} := q^{d_i a_{ij}} = q^{d_j a_{ji}} \, $  for all  $ \, i, j \in I \, $,
   we just compute that  $ \,\; q_{\alpha',\alpha''}^{\,+1/2} \; \check{q}_{\alpha',\alpha''} \, q_{\alpha'',\alpha'}^{\,-1/2} \, = \, q_{\alpha',\alpha''} \; $,  \,
   whose coset in  $ \; \Rbquno := \Rbq \Big/ (q - \! 1)\,\Rbq \; $  is just  $ y_{\alpha'{}\alpha''} \, $;  therefore  \eqref{Ea'-Ea"_modulo_(q-1)}
   yields  $ \; e_{\alpha'} \, e_{\alpha''} \, = \, y_{\alpha'{}\alpha''} \, e_{\alpha''} \, e_{\alpha'} \; $  as claimed.
 \vskip5pt
   A similar procedure shows that  $ \; f_{\alpha'} \, f_{\alpha''}  = \, y_{\alpha''{}\alpha'} \, f_{\alpha''} \, f_{\alpha'} \, $,
  \, which ends the proof.
\epf

\vskip9pt

\begin{rmks}
 In  \cite[\S 3]{An4}, a different construction eventually leads to a result comparable with  Theorem \ref{struct-Utildequno}  above, although slightly weaker.  In general, we prefer to follow a different approach, because it exploits an independent argument and is more consistent with our global approach in the present work, mostly based on the fact that  $ \; U_\bq(\hskip0,8pt\lieg) = {\big( U_{\check{\bq}}(\hskip0,8pt\lieg) \big)}_\sigma \; $.  In addition, some results of  \cite{An4}  cannot be directly applied to our context of integral forms and specializations, so we must resort to an alternative strategy.
\end{rmks}

\vskip9pt

   When the multiparameter  $ \bq $  is of  {\sl integral\/}  type the last two previous results get a stronger importance from a geometrical point of view.  In fact, the following is a refinement of  Proposition \ref{Utildesquno_2-coc-def}  but {\sl we provide for it an independent proof}.

\vskip9pt

\begin{theorem}  \label{thm: Utildequno_q-int-type}  {\ }
 \vskip3pt
   Let  $ \, \bq $  be of  {\sl integral type},  and  $ \, \Utildequnogd $  defined over  $ \, \RbqB = \Zqqm \, $. Then  $ \, \Utildequnogd $  is (isomorphic to) the representing Hopf algebra  $ \, \Oc\big( \Gtildebstar \big) $  of a connected affine Poisson group-scheme over  $ \, \ZZ $  whose cotangent Lie bialgebra is  $ \, \liegtildeb $  as described in  Definition \ref{def_liegdotb-lieghatb-liegtildeb}.
                                                                     \par
   Similar statements hold true for the specialization at  $ \, q = 1 $  of  $ \, \Utilde_\bq^{\,\geq} \, $,  $ \Utilde_\bq^{\,\leq} \, $,  $ \Utilde_\bq^{\,0} \, $,  etc.
\end{theorem}

\pf
 First of all, when  $ \bq $  is of integral type, so  $ \, q_{ij} = q^{\,b_{ij}} \, $  (for all  $ i, j \, $),  we have
  $$  y_{ij}  \; := \;  q_{ij} \!\mod (q - \! 1)  \; = \;  q^{\,b_{ij}} \!\!\mod (q - \! 1)
 \; = \;  1^{\,b_{ij}}  \; = \;  1   \eqno  \forall \;\; i , j \in I  \qquad  $$
 therefore  Theorem \ref{struct-Utildequno}  tell us that  $ \Utildequnogd $  is a  {\sl commutative\/}  Hopf algebra
(of Laurent polynomials); it follows then that  $ \; \Utildequnogd = \Oc(\mathcal{G}\,) \; $  for some affine group-scheme, say  $ \mathcal{G} \, $.  Moreover, from  Proposition \ref{Utildesquno_2-coc-def}  (with notation as in its proof) we know that  $ \; \Oc(\mathcal{G}\,) = \Utildequnogd = {\Oc\big( {\widetilde{G}^{\,\raise2pt\hbox{$ \scriptstyle * $}}_{\!\scriptscriptstyle DA}} \,\big)}_{\sigma_{\raise-2pt\hbox{$ \scriptscriptstyle 1 $}}} \, $  where the group-scheme  $ {\widetilde{G}^{\,\raise2pt\hbox{$ \scriptstyle * $}}_{\!\scriptscriptstyle DA}} $  is connected   --- in other words,  $ \, \Oc\big( {\widetilde{G}^{\,\raise2pt\hbox{$ \scriptstyle * $}}_{\!\scriptscriptstyle DA}} \,\big) \, $  has no non-trivial idempotents.  Now, as  $ \, q_{ij} = q^{b_{ij}} \, $  for  $ \, q = 1 \, $  the ``specialized'' cocycle  $ \sigma_1 $ is trivial   --- namely,  $ \, \sigma_1 = \epsilon \otimes \epsilon \, $  ---   which implies that  $ \; \Oc(\mathcal{G}\,) = {\Oc\big( {\widetilde{G}^{\,\raise2pt\hbox{$ \scriptstyle * $}}_{\!\scriptscriptstyle DA}} \,\big)}_{\sigma_{\raise-2pt\hbox{$ \scriptscriptstyle 1 $}}} \! = {\Oc\big( {\widetilde{G}^{\,\raise2pt\hbox{$ \scriptstyle * $}}_{\!\scriptscriptstyle DA}} \,\big)} \, $,  \,hence  $ \, \mathcal{G} = {\widetilde{G}^{\,\raise2pt\hbox{$ \scriptstyle * $}}_{\!\scriptscriptstyle DA}} \, $  as group-schemes.  In addition, by  Remark \ref{Pois/coPois-structs}{\it (a)\/}  the Hopf algebra  $ \, \Oc(\mathcal{G}\hskip0,5pt) = \Utildequnogd \, $,  being commutative, inherits from  $ \Utildeqgd $  a  {\sl Poisson\/}  structure, hence it is a Poisson Hopf algebra: thus  $ \mathcal{G} $  itself is in fact a  {\sl Poisson\/}  group-scheme.
                                                              \par
   We point out that the Poisson structure on  $ \, \Oc(\mathcal{G}\hskip0,5pt) = \Utildequnogd \, $  is induced by the multiplication in  $ \, \Utildeqgd = {\big(\, \Utilde_{\check{\bq}}(\hskip0,8pt\lieg) \big)}_{\sigma_{\raise-1pt\hbox{$ \scriptscriptstyle \bq $}}} \, $,  \,which in turn depends on  $ \bq \, $.  Thus  $ \, \mathcal{G} \, $  and  $ \, {\widetilde{G}^{\,\raise2pt\hbox{$ \scriptstyle * $}}_{\!\scriptscriptstyle DA}} \, $,  although coinciding as group-schemes, do not share, in general, the same Poisson structure.
                                                                         \par
   What is still missing for having  $ \, \mathcal{G} = \Gtildebstar \, $  is proving that the cotangent Lie bialgebra of  $ \mathcal{G} $  is isomorphic to  $ \liegtildeb \, $,  defined as in  Definition \ref{def_liegdotb-lieghatb-liegtildeb}.
                                                                          \par
   First we recall the definition of the cotangent Lie bialgebra of  $ \mathcal{G} \, $.  If  $ \, \mathfrak{m}_e := \Ker\big( \epsilon_{\Oc(\mathcal{G}\hskip0,5pt)} \big) \, $  is the augmentation ideal of  $ \Oc(\mathcal{G}\hskip0,5pt) \, $,  the quotient  $ \, \mathfrak{m}_e \Big/ \mathfrak{m}_e^{\,2} \, $  has a canonical structure of Lie coalgebra, such that its linear dual is the tangent Lie algebra of  $ \, \mathcal{G} \, $.  In addition, the properties of the Poisson bracket in  $ \Oc(\mathcal{G}\hskip0,5pt) $  imply that  $ \mathfrak{m}_e $  is a Lie subalgebra (even a Lie ideal, indeed) of the Lie algebra  $ \, \big( \Oc(\mathcal{G}\hskip0,5pt) \, , \{\,\ ,\ \} \big) \, $,  and  $ \mathfrak{m}_e^{\,2} $  is a Lie ideal in  $ \, \big(  \mathfrak{m}_e \, , \{\,\ ,\ \} \big) \, $,  whence  $ \, \mathfrak{m}_e \Big/ \mathfrak{m}_e^{\,2} \, $  has a quotient Lie algebra structure; together with the Lie coalgebra structure, the latter makes  $ \, \mathfrak{m}_e \Big/ \mathfrak{m}_e^{\,2} \, $  into a Lie bialgebra.  As a matter of notation, we set  $ \, \overline{x} \in \mathfrak{m}_e \Big/ \mathfrak{m}_e^{\,2} \, $  to denote the coset in  $ \, \mathfrak{m}_e \Big/ \mathfrak{m}_e^{\,2} \, $  of any  $ \, x \in \mathfrak{m}_e \; $.
                                                                          \par
   As a consequence of the PBW Theorem for  $ \Utildeqgd $   --- i.e.\
   Theorem \ref{thm:PBW_tilde-MpQG},  or directly of  Theorem \ref{struct-Utildequno}  ---
   taking into account that the  $ e_\alpha $'s,  the  $ f_\alpha $'s,  the  $ (k_i \! - \! 1) $'s  and the
   $ (\,l_i \! - \! 1) $'s,  with  $ \, \alpha \in \Phi^+ $,  $ \, i \in I $,  all lie in  $ \mathfrak{m}_e \, $,
   one has that a basis of  $ \, \mathfrak{m}_e \Big/ \mathfrak{m}_e^{\,2} \, $  is given by the
   $ \overline{e_\alpha} \, $'s,  the  $ \overline{f_\alpha} \, $'s,  the  $ \overline{(k_i \! - \! 1)} \, $'s
   and the  $ \overline{(\,l_i \! - \! 1)} \, $'s  altogether.  Our aim now is to prove the following
 \vskip3pt
   {\it  $ \underline{\text{Claim}} $:}\,  there exists a Lie bialgebra isomorphism
   $ \; \phi : \mathfrak{m}_e \Big/ \mathfrak{m}_e^{\,2} \;{\buildrel \cong \over
   {\lhook\joinrel\relbar\joinrel\relbar\joinrel\relbar\joinrel\twoheadrightarrow}}\; \liegtildeb \; $
   given by  $ \; \phi : \, \overline{e_\alpha} \mapsto \etilderm_\alpha \, $,
   $ \; \overline{f_\alpha} \mapsto \ftilderm_\alpha \, $,  $ \; \overline{(k_i - \! 1)} \mapsto \kdotrm_i \; $
   and  $ \; \overline{(l_i - \! 1)} \mapsto \ldotrm_i \, $,  \, for all  $ \, \alpha \in \Phi^+ $,  $ \, i \in I \, $.
 \vskip7pt
   To begin with, given  $ \, \alpha , \beta \in \Phi^+ \, $,  we show that
   $ \; \phi\big( \big[ \overline{e_\alpha} \, , \overline{e_\beta} \big] \big) = \big[ \phi(\overline{e_\alpha}) ,
   \phi(\overline{e_\beta}) \big] \; $.  First observe that our root vectors  $ \, \erm_\gamma $  in  $ \lieg $
   come from the simple ones via a construction \`a la Chevalley (see  \cite[Chapter\ II, \S 25.2]{Hu}),  so that
   $ \; [\,\erm_\alpha \, , \erm_\beta] = c_{\alpha,\beta} \, \erm_{\alpha+\beta} \; $  for suitable
   $ \, c_{\alpha,\beta} \in \ZZ \, $.  Moreover, since (under our assumption that  $ \lieg $  be simple)
   there are only two possible root lengths, we have  $ \, d_{\alpha+\beta} \in \{ d_\alpha \, , d_\beta \} \, $;
   so if  $ \, d_\alpha = d_\beta \, $ we write  $ \, d_\delta := d_\alpha \, (= d_\beta) \, $  and if
   $ \, d_\alpha \not= d_\beta \, $  we call  $ \, d_\delta \, $  the unique element of
   $ \, \{ d_\alpha \, , d_\beta \} \setminus \{ d_{\alpha+\beta} \} \, $.  Then recall that
   (for all  $ \, \gamma \in \Phi^+ \, $)
  $$  \displaylines{
   e_\gamma \, := \;\, \Ebar_\gamma \mod (q-\!1) \, \Utilde_\bq   \qquad ,  \qquad \quad   \erm_\gamma \,
   := \;\, E_\gamma \mod (q-\!1) \, \Uhatdot_\bq  \cr
   \big[ \overline{e_\alpha} \, , \overline{e_\beta} \big]  \, :=  \;
   \big\{ e_\alpha \, , e_\beta \big\} \hskip-3pt  \mod \mathfrak{m}_e^{\,2}   \,\; ,
   \quad   \big\{ e_\alpha \, , e_\beta \big\}  \, := \; {(q-\!1)}^{-1} \big[\, \Ebar_\alpha \, ,
   \Ebar_\beta \,\big]  \hskip-3pt  \mod \! (q-\!1) \, \Utilde_\bq  }  $$
\noindent
 Second, since  $ \Utilde_\bq $  is commutative modulo  $ (q-\!1) \, $,  we have
 $ \; \big[\, \Ebar_\alpha \, , \Ebar_\beta \,\big] \, = \, (q-\!1) \, \E \; $  for some
 $ \, \E \in \Utilde_\bq \cap U_\bq^+ = \Utilde_\bq^+ \, $   ---
 so that  $ \; \big\{ e_\alpha \, , e_\beta \big\} := \, \E \hskip-3pt \mod \! (q-\!1) \, \Utilde_\bq \; $.
 On the other hand, from $ \; [\,\erm_\alpha \, , \erm_\beta] = c_{\alpha,\beta} \, \erm_{\alpha+\beta} \; $
 (see above) and $ \; \erm_\gamma := \, E_\gamma \hskip-4pt \mod (q-\!1) \, \Uhatdot_\bq \; $
 together we get  $ \; \big[\, E_\alpha \, , E_\beta \,\big] \, = \, c_{\alpha,\beta} \, E_{\alpha+\beta} + (q-\!1) \,
 \mathfrak{E} \; $  for some  $ \, \mathfrak{E} \in \Uhat_\bq \cap U_\bq^+ = \Uhat_\bq^+ \, $.
 The latter implies
  $$  \displaylines{
   \big[\, \Ebar_\alpha \, , \Ebar_\beta \,\big]  \; = \;
\big( q_{\alpha\alpha} - 1 \big) \big( q_{\beta\beta} - 1 \big) \, \big[\, E_\alpha \, , E_\beta \,\big]  \; =   \hfill  \cr
   = \;  c_{\alpha,\beta} \,
\big( q_{\alpha\alpha} - 1 \big) \big( q_{\beta\beta} - 1 \big)
 \, E_{\alpha+\beta} + (q-\!1) \,
\big( q_{\alpha\alpha} - 1 \big) \big( q_{\beta\beta} - 1 \big)
 \, \mathfrak{E}  \; =  \cr
   \hfill   = \;  c_{\alpha,\beta} \,
\big( q_{\delta\delta} - 1 \big)
 \, \Ebar_{\alpha+\beta} + (q-\!1) \,
\big( q_{\alpha\alpha} - 1 \big) \big( q_{\beta\beta} - 1 \big)
 \, \mathfrak{E}  }  $$
 and comparing the last term with the previous identity  $ \; \big[\, \Ebar_\alpha \, , \Ebar_\beta \,\big] \, = \, (q-\!1) \, \E \; $
%%%%%
% --- with  $ \, \E \in \Utilde_\bq \cap U_\bq^+ = \Utilde_\bq^+ \, $  ---
%%%%%
 yields
  $$  \E  \; = \;  c_{\alpha,\beta} \, {\big(2\,d_\delta\big)}_{\!q} \, \Ebar_{\alpha+\beta} + \big(\, q_{\alpha\alpha} - 1 \big)
  \big(\, q_{\beta\beta} - 1 \big) \, \mathfrak{E}  $$
                                                   \par
   Then expanding  $ \, \mathfrak{E} \, $  w.r.t.\ the  $ \Rbq $--PBW basis  of  $ \Uhat_\bq^+ $
   (made of ordered products of  $ q $--divided  powers  $ E_\gamma^{(n_\gamma)} $'s)
   and comparing with the expansion of
 $ \,
\big( q_{\alpha\alpha} - 1 \big) \big( q_{\beta\beta} - 1 \big) \,
\mathfrak{E} \, $   --- which must necessarily belong to  $ \Utilde_\bq^+ $  ---   in terms of the  $ \Rbq $--PBW basis
of  $ \Utilde_\bq^+ $
(made of ordered monomials in the  $ \Ebar_\gamma $'s)  we eventually find that
  $$
%
% \big( q_\alpha - q_\alpha^{-1} \big) \big( q_\beta - q_\beta^{-1} \big)
%
\big( q_{\alpha\alpha} - 1 \big) \big( q_{\beta\beta} - 1 \big)
 \, \mathfrak{E}  \,\; =
  \; {\textstyle \sum\limits_{k \geq 2}} \,\; {\textstyle \sum\limits_{\gamma_1 , \dots , \gamma_k \in \Phi^+}}
  \hskip-5pt c'_{\,\gamma_1 , \dots , \gamma_k} \, \Ebar_{\gamma_1} \! \cdots \Ebar_{\gamma_k}  \, + \,  (q-\!1) \, \E'  $$
for some  $ \, c'_{\,\gamma_1 , \dots , \gamma_k} \in \Rbq \, $  and some  $ \, \E' \in \Utilde_\bq^+ \, $.  Therefore
  $$  \E  \,\; = \;\,  c_{\alpha,\beta} \,
%
% {[d_\delta]}_q \big( 1 + q^{-1} \big)
%
{\big(2\,d_\delta\big)}_{\!q} \, \Ebar_{\alpha+\beta} \; + \,
  {\textstyle \sum\limits_{k \geq 2}} \,\; {\textstyle \sum\limits_{\gamma_1 , \dots , \gamma_k \in \Phi^+}}
  \hskip-5pt c'_{\gamma_1 , \dots , \gamma_k} \, \Ebar_{\gamma_1} \! \cdots \Ebar_{\gamma_k}  \, + \,  (q-\!1) \, \E'  $$
which in turn implies
  $$  \displaylines{
   \E  \hskip-4pt  \mod \! (q-\!1) \, \Utilde_\bq  \,\; = \;\,  \Big(\, c_{\alpha,\beta} \,
%
% {[d_\delta]}_q \big( 1 + q^{-1} \big)
%
{\big(2\,d_\delta\big)}_{\!q} \, \Ebar_{\alpha+\beta}  \hskip-4pt  \mod \! (q-\!1) \, \Utilde_\bq \;\Big)  \; +   \qquad \qquad \qquad  \cr
   \qquad   + \; \Big(\; {\textstyle \sum\limits_{k \geq 2}} \,\; {\textstyle \sum\limits_{\gamma_1 , \dots , \gamma_k \in \Phi^+}}
   \hskip-5pt c'_{\gamma_1 , \dots , \gamma_k} \, \Ebar_{\gamma_1} \! \cdots \Ebar_{\gamma_k}  \hskip-4pt  \mod \! (q-\!1) \, \Utilde_\bq \;\Big)  \; =  \cr
   \hfill   = \;\,  c_{\alpha,\beta} \; 2 \; d_\delta \; e_{\alpha+\beta}  \; + \;  {\textstyle \sum\limits_{k \geq 2}} \,\;
   {\textstyle \sum\limits_{\gamma_1 , \dots , \gamma_k \in \Phi^+}} \hskip-5pt c_{\,\gamma_1 , \dots , \gamma_k}
   \; e_{\gamma_1} \cdots e_{\gamma_k}  }  $$
%
%%%%%
% where  $ \; c_{\,\gamma_1 , \dots , \gamma_k} \, := \, \big(\, c'_{\,\gamma_1 , \dots , \gamma_k}
% \hskip-1pt  \mod \! (q-\!1) \, \Rbq \,\big) \;  $  belongs to  $ \, \Rbq \Big/ (q-\!1) \, \Rbq \,
% = \Rbquno \; $.  Eventually, all this yields
%%%%%
%
 with  $ \; c_{\,\gamma_1 , \dots , \gamma_k} \, := \, \big(\, c'_{\,\gamma_1 , \dots , \gamma_k} \hskip-1pt  \mod \! (q-\!1) \, \Rbq \,\big) \, \in \, \Rbq \Big/ (q-\!1) \, \Rbq \, = \Rbquno \; $.  This yields
  $$  \displaylines{
   \big[\, \overline{e_\alpha} \, , \overline{e_\beta} \,\big]  \, =  \Big( \big\{ e_\alpha \, , e_\beta \big\}  \hskip-4pt  \mod \mathfrak{m}_e^{\,2} \,\Big)  \, =   \hfill  \cr
   = \,  \Big( \Big( {(q-\!1)}^{-1} \big[\, \Ebar_\alpha \, , \Ebar_\beta \,\big]  \hskip-4pt  \mod \! (q-\!1) \, \Utilde_\bq \,\Big)  \hskip-4pt \mod \mathfrak{m}_e^{\,2} \,\Big)  \, =  \cr
   \hfill   = \,  \Big( \Big(\, \E  \hskip-4pt  \mod \! (q-\!1) \, \Utilde_\bq \,\Big)  \hskip-4pt \mod \mathfrak{m}_e^{\,2} \,\Big)  \, =  \cr
   = \;\,  \Big( \Big(\; c_{\alpha,\beta} \; 2 \; d_\delta \; e_{\alpha+\beta}  \; + \;  {\textstyle \sum\limits_{k \geq 2}} \,\;
   {\textstyle \sum\limits_{\gamma_1 , \dots , \gamma_k \in \Phi^+}} \hskip-5pt c_{\,\gamma_1 , \dots , \gamma_k} \;
   e_{\gamma_1} \cdots e_{\gamma_k} \Big)  \hskip-4pt \mod \mathfrak{m}_e^{\,2} \,\Big)  \; = \;  c_{\alpha,\beta} \; 2 \; d_\delta \; \overline{e_{\alpha+\beta}}  }  $$
that is in short  $ \; \big[\, \overline{e_\alpha} \, , \overline{e_\beta} \,\big] \, = \, 2 \; d_\delta \; c_{\alpha,\beta} \; \overline{e_{\alpha+\beta}} \;\, $.
   Now, from the last identity we compute
\begin{equation}  \label{phi(bracket)}
   \phi\big( \big[\, \overline{e_\alpha} \, , \overline{e_\beta} \,\big] \big)  \; = \;  \phi\big(\, 2 \; d_\delta \; \overline{c}_{\alpha,\beta}
   \; \overline{e_{\alpha+\beta}} \,\big)  \; = \;  2 \; d_\delta \; \overline{c}_{\alpha,\beta} \; \phi\big(\, \overline{e_{\alpha+\beta}} \,\big)  \;
   = \;  2 \; d_\delta \; \overline{c}_{\alpha,\beta} \; {\check{\erm}}_{\alpha+\beta}
\end{equation}
by definition of  $ \phi \, $.  On the other hand, we have also
  $$  \big[ \phi(\overline{e_\alpha}) \, , \phi(\overline{e_\beta}) \big]  \, = \,
  [ {\check{\erm}}_\alpha \, , {\check{\erm}}_\beta ]  \, = \,
  2 \, d_\alpha \, 2 \, d_\beta \, [\, \erm_\alpha \, , \erm_\beta \,]  \, = \,
  2 \, d_\alpha \, 2 \, d_\beta \, c_{\alpha,\beta} \, \erm_{\alpha+\beta}  \, = \,
  2 \, d_\delta \, c_{\alpha,\beta} \, {\check{\erm}}_{\alpha+\beta}  $$
comparing this with  \eqref{phi(bracket)}  eventually gives  $ \; \phi\big( \big[\, \overline{e_\alpha} \, , \overline{e_\beta} \,\big] \big) \,
= \, \big[ \phi(\overline{e_\alpha}) \, , \phi(\overline{e_\beta}) \big] \; $,  q.e.d.

 \vskip5pt
   Acting in the same way, one finds also
  $$  \displaylines{
   \big[\, \overline{k_i - \! 1} \, , \overline{e_\alpha} \,\big]  \; = \;  \Big(\, \big\{ k_i - \! 1 \, , e_\alpha \big\} \!\! \mod \mathfrak{m}_e^{\,2} \,\Big)  \; =   \hfill  \cr
   = \,  \Big( \Big( {(q - \! 1)}^{-1} \big( (K_i - \! 1) \Ebar_\alpha - \Ebar_\alpha (K_i - \! 1) \big) \!\!\! \mod \! (q - \! 1) \,
   \Utildeqgd \Big) \!\!\! \mod \mathfrak{m}_e^{\,2} \,\Big)  \, =  \cr
   \hfill   = \;  \Big(\, \Big(\, {\big( d^+_{i,\alpha} \big)}_{\!q} \, \Ebar_\alpha K_i \! \mod (q - \! 1) \, \Utildeqgd \,\Big) \! \mod \mathfrak{m}_e^{\,2} \,\Big)  \; = \;  \cr
   \hfill   = \;  \Big(\, d^+_{i,\alpha} \, e_\alpha k_i \! \mod \mathfrak{m}_e^{\,2} \,\Big)  \; = \;  d^+_{i,\alpha} \, \overline{e_\alpha}  }  $$
where  $ \, d^+_{i,\alpha} := +\sum_{j \in I} b_{ij} c_j \, $  with  $ \, \alpha = \sum_{j \in I} c_j \alpha_j \, $,  so in the end
\begin{equation}  \label{comm-rel_k-e_alpha-Utilde_q=1}
  \big[\, \overline{k_i - \! 1} \, , \overline{e_\alpha} \,\big]  \; = \;  d^+_{i,\alpha} \, \overline{e_\alpha}   \qquad \qquad  \forall \;\; i \in I \, , \; \alpha \in \Phi^+
\end{equation}
Similarly, one finds also
\begin{equation}  \label{comm-rel_l-e_alpha-Utilde_q=1}
  \big[\, \overline{l_i - \! 1} \, , \overline{e_\alpha} \,\big]  \; = \;  d^-_{i,\alpha} \, \overline{e_\alpha}   \qquad \qquad  \forall \;\; i \in I \, , \; \alpha \in \Phi^+
\end{equation}
with  $ \, d^-_{i,\alpha} := -\sum_{j \in I} b_{ji} c_j \, $  for  $ \, \alpha = \sum_{j \in I} c_j \alpha_j \, $.
Likewise, parallel formulas to  (\ref{comm-rel_k-e_alpha-Utilde_q=1})  and  (\ref{comm-rel_l-e_alpha-Utilde_q=1})
hold true when the  $ \overline{e_\alpha} $'s  are replaced by the  $ \overline{f_\alpha} $'s.
 \vskip3pt
   Finally, comparing the Lie brackets (inside  $ \, \mathfrak{m}_e \Big/ \mathfrak{m}_e^{\,2} \, $)  given explicitly in
   (\ref{comm-rel_k-e_alpha-Utilde_q=1})  and  (\ref{comm-rel_l-e_alpha-Utilde_q=1}),  and the similar ones where the  $ \overline{f_\gamma} $'s
   are replaced by the  $ \overline{e_\gamma} $'s, with the analogue brackets inside  $ \liegtildeb $  of the corresponding elements through the map
   $ \phi $  as given in the  {\it Claim},  one easily sees that the latter map is indeed a Lie algebra morphism.  In addition, it is invertible because it
   maps a basis to a basis.  Moreover, this is also an isomorphism of Lie  {\sl bi\/}algebras  because the formulas for the Lie cobracket do correspond
   on either side on all elements of the form  $ \; \overline{e_i} \, $,  $ \overline{f_i} \, $,  $ \overline{k_i - \! 1} \, $  and  $ \overline{l_i - \! 1} \, $  (with  $ \, i \in I \, $),
   which is enough to conclude   --- cf.\  Remarks \ref{int-forms_of_g_B}.  In fact, this is again a matter of
%
% sheer
%
 bookkeeping: for instance,
%
%%%%%
% using notation  $ \, \mathfrak{m}_\otimes^{[2]} :=  \mathfrak{m}_e \otimes \mathfrak{m}_e^{\,2}
% + \mathfrak{m}_e^{\,2} \otimes  \mathfrak{m}_e \, $,  the Lie cobracket on  $ \overline{e_i} $
% is given (canonically) by
%%%%%
%
%%%%%
 writing  $ \, \mathfrak{m}_\otimes^{[2]} :=  \mathfrak{m}_e \otimes \mathfrak{m}_e^{\,2} +  \mathfrak{m}_e^{\,2} \otimes  \mathfrak{m}_e \, $,  \,one has
%%%%%
  $$  \displaylines{
   \delta(\,\overline{e_i}\,)  \; = \;  \Big(\, \big( \Delta(e_i) - \Delta^{\text{op}}(e_i) \big) \!\! \mod \mathfrak{m}_\otimes^{[2]} \,\Big)  \; =   \hfill  \cr
   \hfill   = \;  \Big(\, \Big(\, \big( \Delta\big(\Ebar_i\big) - \Delta^{\text{op}}\big(\Ebar_i\big) \big) \! \mod (q-\!1) \, \Utildeqgd^{\otimes 2} \,\Big) \! \mod
   \mathfrak{m}_\otimes^{[2]} \,\Big)  \; =  \cr
   = \;  \Big( \Big( \big( (K_i - 1) \otimes \Ebar_i - \Ebar_i \otimes (K_i - 1) \big) \! \mod (q-\!1) \, \Utildeqgd^{\otimes 2} \,\Big) \! \mod \mathfrak{m}_\otimes^{[2]} \,\Big)  \; =  \cr
   \qquad \qquad   = \;  \Big(\, \big( (k_i - \! 1) \otimes e_i - e_i \otimes (k_i - \! 1) \big) \! \mod \mathfrak{m}_\otimes^{[2]} \,\Big)  \; =  \cr
   \hfill   = \;  \overline{(k_i - \! 1)} \otimes \overline{e_i} \, - \, \overline{e_i} \otimes \overline{(k_i - \! 1)}  }  $$
which means  $ \; \delta(\,\overline{e_i}\,) \, = \, \overline{(k_i - \! 1)} \otimes \overline{e_i} \, - \, \overline{e_i} \otimes \overline{(k_i - \! 1)} \; $.
Through the formulas given in the  {\it Claim},  this last identity corresponds to
$ \; \delta(\etilderm_i) \, = \, \kdotrm_i \otimes \etilderm_i \, - \, \etilderm_i \otimes \kdotrm_i \; $  given in
Definition \ref{def_liegdotb-lieghatb-liegtildeb}{\it (b)\/}  for  $ \liegtildeb \, $.  Likewise it holds for the other cases.
\epf

\bigskip

\section{Specialization of MpQG's at roots of unity}  \label{sec:spec-eps}
 \vskip7pt
   In this section we study MpQG's for which all parameters  $ q_{ii} $  are roots of unity.
%%%%%
% As each  $ q_{ii} $  is just a power of a special element  $ \, q \in \k^\times \, $,
% requiring all the  $ q_{ii} $'s  to be roots of unity,
%%%%%
 Once again, this
%%%
 amounts to requiring  $ q $  itself to be a root of unity, or just 1.  As we already considered the case  $ \, q = 1 \, $,
 we assume this root to be different from 1 itself.

\medskip

\subsection{Specialization at roots of unity}  \label{spec-eps}  {\ }
 \vskip7pt

   Let again  $ \Rbq $  and  $ \RbqB $  be fixed as in  \S \ref{gen-ground-rings};  fix also a positive, odd integer  $ \ell $
   {\sl which is coprime with all the  $ d_i $'s  ($ \, i \in I \, $)  given in  \S \ref{gen-ground-rings}},
   and let  $ p_\ell(x) $  be the  $ \ell $--th  cyclotomic poynomial in  $ \ZZ[x] \, $.  We consider the special element
   $ \, q \in \Rbq \, $  and the quotient ring  $ \; \Rbqeps := \, \Rbq \Big/ p_\ell(q) \, \Rbq \; $,  and we call  $ \varepsilon $
   the image of  $ q $  in  $ \Rbqeps \; $.
                                                                    \par
   By construction, the ring  $ \Rbqeps $  is generated by invertible elements  $ \, \varepsilon_{ij}^{\,\pm 1} \, $  each of
   whom is the image in  $ \Rbqeps $  of the corresponding generator  $ q_{ij}^{\pm 1} $  of  $ \, \Rbq \, $;  since
   $ \, \varepsilon_{i{}i}^{\,\ell} = 1 \, $  for all  $ i \, $,  all these generators only obey the relations
   $ \, {\big( \varepsilon_{ij}^{\,\pm 1} \, \varepsilon_{ji}^{\,\pm 1} \big)}^\ell = 1 \, $.  We denote by
   $ \, \varepsilon_{\alpha\,\gamma} \, $  the element in  $ \Rbqeps $  defined like in  \S \ref{deform-MpQG}
   but for using the
   $ \varepsilon_{ij} $'s  instead of the  $ q_{ij}^{\pm 1} $'s,  so that  $ \, \varepsilon_{\alpha\,\gamma}^{\pm 1} \, $
   is nothing but the image in  $ \Rbqeps $  of  $ \, q_{\alpha\,\gamma}^{\pm 1} \in \Rbq \, $.  Finally,  $ \Rbqeps $  is an
   $ \Rbq $--algebra  by scalar restriction via the canonical epimorphism
   $ \, \Rbq \relbar\joinrel\relbar\joinrel\twoheadrightarrow \Rbqeps \, $.
                                                          \par
   Replacing  $ \Rbq $  with  $ \RbqB $  everywhere, we set  $ \; \RbqBeps := \, \RbqB \Big/ p_\ell(q) \, \RbqB \; $,  for
   which we use again such notation as  $ \, \varepsilon \, $,  $ \, \varepsilon_{ij} \, $,  etc., noting in addition that now
   $ \, \varepsilon_{ij} = \varepsilon^{\,b_{ij}} \, $.
   Then the natural epimorphisms  $ \, \Rbq \relbar\joinrel\relbar\joinrel\twoheadrightarrow \RbqB \, $
   yields a similar one  $ \, \Rbqeps \relbar\joinrel\relbar\joinrel\twoheadrightarrow \RbqBeps \; $.
                                                                   \par
   Furthermore, it is worth stressing that the isomorphism  $ \, \RbqB \cong \Zqqm \, $
   induces in turn  $ \; \RbqBeps \, \cong \, \Zqqm \Big/ p_\ell(q)\,\Zqqm \, =: \, \ZZ[\varepsilon] \, $,
   the latter being the ring extension of  $ \ZZ $  by any (formal) primitive  $ \ell $--th  root of unity  $ \varepsilon \, $.
 \vskip5pt
   Similarly, we define  $ \; \Rbqsqeps := \, \Rbqsq \Big/ p_\ell\big(q^{\,1/2}\big) \, \Rbqsq \; $
   and denote by  $ \varepsilon^{1/2} $,  $ \varepsilon_{i{}j}^{\,1/2} $,  etc., the image of
   $ q^{1/2} $,  $ q_{i{}j}^{\,1/2} $,  etc.,  in  $ \Rbqsqeps \; $;  and likewise for
   $ \; \RbqBsqeps := \, \RbqBsq \Big/ p_\ell\big(q^{\,1/2}\big) \, \RbqBsq \; $,  for which we have in addition
   $ \, \RbqBsqeps \cong \ZZ\big[\varepsilon^{\,1/2}\big] \, $  where  $ \varepsilon^{\,1/2} $  is again a primitive
   $ \ell $--th  root of unity.  The projection  $ \, \Rbqsq \!\relbar\joinrel\relbar\joinrel\twoheadrightarrow \! \RbqBsq \, $
   induces an epimorphism  $ \, \Rbqsqeps \!\relbar\joinrel\relbar\joinrel\twoheadrightarrow \! \RbqBsqeps \, $,
   while the embeddings  $ \, \Rbq \!\lhook\joinrel\relbar\joinrel\relbar\joinrel\rightarrow \! \Rbqsq \, $  and
   $ \, \Rbqeps \!\lhook\joinrel\relbar\joinrel\relbar\joinrel\rightarrow \! \Rbqsqeps \, $  induce embeddings
   $ \, \RbqB \!\lhook\joinrel\relbar\joinrel\relbar\joinrel\rightarrow \! \RbqBsq \, $  and
   $ \, \RbqBeps \!\lhook\joinrel\relbar\joinrel\relbar\joinrel\rightarrow \! \RbqBsqeps \, $  respectively.  {\sl In addition},
   for the last map the following holds:

\vskip13pt

\begin{lema}  \label{lemma: RbqBeps = RbqBsqeps}
 The morphism  $ \; \RbqBeps \! \lhook\joinrel\relbar\joinrel\relbar\joinrel\rightarrow \! \RbqBsqeps \, $,  given by
 $ \, \varepsilon^{\pm 1} \!\mapsto\! {\big( \varepsilon^{\pm 1/2} \big)}^2 \, $,  is an isomorphism, whose inverse
 $ \; \RbqBsqeps \! \lhook\joinrel\relbar\joinrel\relbar\joinrel\twoheadrightarrow \RbqBeps \, $  is given by
 $ \, \varepsilon^{\,\pm 1/2} \mapsto \varepsilon^{\pm (\ell+1)/2} \, $.
\end{lema}

\vskip7pt

We introduce now the ``specialization at  $ \, q = \varepsilon \, $''  of the integral forms
   $ \Uhatdot_\bq(\hskip0,8pt\lieg) \, $,  $ \Uhat_\bq(\hskip0,8pt\lieg) $  and  $ \Utilde_\bq(\hskip0,8pt\lieg) \, $
   --- over the ring  $ \RbqB $  or  $ \Rbq $  ---   of our MpQG's  $ U_\bq(\hskip0,8pt\lieg) \, $.

\vskip15pt

\begin{definition}  \label{def_spec-eps}  \
 \vskip5pt
   {\it (a)}\,  Let  $ \bq $  be a multiparameter matrix of Cartan type: given  $ \Utilde_\bq(\hskip0,8pt\lieg) $
   over the ground ring
   $ \Rbq \, $,  we call  {\it specialization of  $ \, \Utilde_\bq(\hskip0,8pt\lieg) $  at  $ \, q = \varepsilon \, $}  the quotient
  $$  \Utildeqepsgd  \; := \;  \Utilde_\bq(\hskip0,8pt\lieg) \Big/ p_\ell(q) \, \Utilde_\bq(\hskip0,8pt\lieg)  \,\; \cong \;\,
  \Rbqeps \otimes_{\raise-3pt\hbox{$ \scriptstyle \Rbq $}} \! \Utilde_\bq(\hskip0,8pt\lieg)  $$
endowed with its natural (quotient) structure of Hopf algebra over  $ \Rbqeps \; $.
 \vskip4pt
   {\it (b)}\,  Let  $ \bq $  in addition be of  {\sl integral type}   --- hence  $ \, \RbqB = \Zqqm \, $.  Then:
 \vskip2pt
   \quad ---  {\it (b.1)}\,  given  $ \Utilde_\bq(\hskip0,8pt\lieg) $  over the ground ring  $ \RbqB \, $,  we call  {\it specialization of  $ \, \Utilde_\bq(\hskip0,8pt\lieg) $  at  $ \, q = \varepsilon \, $}  the quotient
  $$  \Utildeqepsgd  \; := \;  \Utilde_\bq(\hskip0,8pt\lieg) \Big/ p_\ell(q) \, \Utilde_\bq(\hskip0,8pt\lieg)  \,\; \cong \;\,
  \RbqBeps \otimes_{\raise-3pt\hbox{$ \scriptstyle \RbqB $}} \Utilde_\bq(\hskip0,8pt\lieg)  $$
endowed with its natural (quotient) structure of Hopf algebra over  $ \RbqBeps \; $;
 \vskip2pt
   \quad ---  {\it (b.2)}\,  we call  {\it specialization of  $ \, \Uhatdot_\bq(\hskip0,8pt\lieg) $  at  $ \, q = \varepsilon \, $}  the quotient
  $$  \Uhatdot_{\bq\hskip1pt,\hskip1pt{}\varepsilon}(\hskip0,8pt\lieg)  \; := \;  \Uhatdot_\bq(\hskip0,8pt\lieg) \Big/ p_\ell(q) \,
  \Uhatdot_\bq(\hskip0,8pt\lieg)  \,\; \cong \;\,  \RbqBeps \otimes_{\raise-3pt\hbox{$ \scriptstyle \RbqB $}} \Uhatdot_\bq(\hskip0,8pt\lieg)  $$
endowed with its natural (quotient) structure of Hopf algebra over  $ \RbqBeps \; $;
 \vskip2pt
   \quad ---  {\it (b.3)}\,  if  $ \bq $  is of  {\sl strongly integral type},  we call
   {\it specialization of  $ \, \Uhat_\bq(\hskip0,8pt\lieg) $  at  $ \, q = \varepsilon \, $}  the quotient
  $$  \Uhat_{\bq\hskip1pt,\hskip1pt{}\varepsilon}(\hskip0,8pt\lieg)  \; := \;
  \Uhat_\bq(\hskip0,8pt\lieg) \Big/ p_\ell(q) \, \Uhat_\bq(\hskip0,8pt\lieg)  \,\; \cong \;\,
  \RbqBeps \otimes_{\raise-3pt\hbox{$ \scriptstyle \RbqB $}} \Uhat_\bq(\hskip0,8pt\lieg)  $$
endowed with its natural (quotient) structure of Hopf algebra over  $ \RbqBeps \; $.
 \vskip4pt
   {\sl Note that},  using the isomorphism  $ \, \RbqBsqeps \cong \RbqBeps \, $  of  Lemma \ref{lemma: RbqBeps = RbqBsqeps},  {\sl all the above mentioned specializations of MpQG's at  $ \, q = \varepsilon \, $  can be also considered as Hopf algebras over the ring  $ \RbqBsqeps \, $},  by scalar extension: hereafter we shall freely do that.   \ \hfill  $ \diamondsuit $
\end{definition}

\vskip9pt

   The above definitions and our results in  \S \ref{int-forms_mpqgs}  yield the following:

\medskip

\begin{theorem}  \label{PBW-mpqgs_q=eps}
 The PBW bases (over  $ \, \Rbq $  or  $ \, \RbqB \, $)  of  $ \; \Utildeqepsgd \, $,  resp.\ of
 $ \; \Uhatdot_{\bq\hskip1pt,\hskip1pt{}\varepsilon}(\hskip0,8pt\lieg) \, $,  resp.\ of
 $ \; \Uhat_{\bq\hskip1pt,\hskip1pt{}\varepsilon}(\hskip0,8pt\lieg) \, $   ---  cf.\
 Theorems \ref{thm:PBW_hat-MpQG}  and  \ref{thm:PBW_tilde-MpQG}  ---   yield, through
 the specialization process, similar PBW-bases (over  $ \, \Rbqeps $  or  $ \, \RbqBeps \, $)  of
 $ \, \Utildeqepsgd \, $,  resp.\ of  $ \, \Uhatdot_{\bq\hskip1pt,\hskip1pt{}\varepsilon}(\hskip0,8pt\lieg) \, $,
 resp.\ of  $ \, \Uhat_{\bq\hskip1pt,\hskip1pt{}\varepsilon}(\hskip0,8pt\lieg) \, $.
\end{theorem}

\medskip

   Basing on the remark at the end of  Definition \ref{def_spec-eps},  consider now  both  $ \Utildeqepsgd $  and
   $ \, \Utilde_{\check{\bq}\hskip1pt,\hskip1pt{}\varepsilon}(\hskip0,8pt\lieg) $  as algebras over  $ \RbqBsqeps \, $.
   Let  $ \sigma_\varepsilon $  be the unique  $ 2 $--cocycle  of  $ \, \Utilde_{\check{\bq}\hskip1pt,\hskip1pt{}\varepsilon}(\hskip0,8pt\lieg) $
   naturally induced by the  $ 2 $--cocycle  $ \sigma $  of  $ \, \Utilde_{\check{\bq}}(\hskip0,8pt\lieg) $  as given in  Definition \ref{def-sigma}:
   that is,  $ \; \sigma_{\eps} : \Utilde_{\check{\bq}\hskip1pt,\hskip1pt{}\varepsilon}(\hskip0,8pt\lieg)\otimes
 \Utilde_{\check{\bq}\hskip1pt,\hskip1pt{}\varepsilon}(\hskip0,8pt\lieg) \relbar\joinrel\longrightarrow \Rbqeps \; $
 is the unique  $ \Rbqeps $--linear  map given by
%
%   $$  \sigma_{\eps}(x,y)  \; := \;  \begin{cases}
%          \  \eps^{\;1/2}_{\mu\,\nu}  &  \quad  \text{if}  \quad
%          x = K_\mu  \text{\;\ or \ }
% x = L_\mu \; ,  \!\quad  y = K_\nu  \text{\;\ or \ }  y = L_\nu  \\
%          \;\  0  &  \;\quad  \text{otherwise}
%                              \end{cases}  $$
%
  $$  \sigma_{\eps}(x,y) \, := \, \eps^{\;1/2}_{\mu\,\nu}  \;\;\;\text{\bf \ if \ }\;\;  x = K_\mu  \text{\;\ or \ }
x = L_\mu \; ,  \!\quad\text{and \ }  y = K_\nu  \text{\;\ or \ }  y = L_\nu  $$
and  $ \; \sigma_{\eps}(x,y) \, := \, 0 \; $  otherwise.  The results in  \S \ref{int-forms_mpqgs}  then lead us to the following

\medskip

\begin{theorem}  \label{Mpqg-roots-of-1_2-coc-def}
\
 Let  $ \bq $  be a multiparameter matrix of Cartan type.  Then
 \vskip2pt
   {\it (a)}\,  The Hopf  $ \, \Rbqeps $
   --algebra  $ \; \Utildeqepsgd $  is a  $ 2 $--cocycle  deformation of
   $ \, \Utilde_{\check{\bq}\hskip1pt,\hskip1pt{}\varepsilon}(\hskip0,8pt\lieg) \, $,  namely
   $ \; \Utilde_{\bq\hskip1pt,\hskip1pt{}\varepsilon}(\hskip0,8pt\lieg) \cong
   {\big(\, \Utilde_{\check{\bq}\hskip1pt,\hskip1pt{}\varepsilon}(\hskip0,8pt\lieg) \big)}_{\sigma_\varepsilon} \; $.
 \vskip2pt
   {\it (b)}\,  Assume that  $ \, \bq $  is of  {\sl integral type\/};  then the Hopf  $ \, \RbqBeps $--algebra  $ \; \Utildeqepsgd $  is a
   $ 2 $--cocycle  deformation of  $ \; \Utilde_{\check{\bq}\hskip1pt,\hskip1pt{}\varepsilon}(\hskip0,8pt\lieg) \, $,
namely $ \; \Utilde_{\bq\hskip1pt,\hskip1pt{}\varepsilon}(\hskip0,8pt\lieg) \cong
 {\big(\, \Utilde_{\check{\bq}\hskip1pt,\hskip1pt{}\varepsilon}(\hskip0,8pt\lieg) \big)}_{\sigma_\varepsilon} \; $.
%
%%%
%   where  $ \sigma_\varepsilon $  is
% determined like in  {\it (a)}  above.
%%%
%
\end{theorem}

\pf
 Directly from definitions along with  Proposition \ref{prop:Utilde-cocy-def},  we get claim  {\it (a)\/}  from
%
%%%%%
%   $$  \displaylines{
%    \quad   \Utilde_{\bq\hskip1pt,\hskip1pt{}\varepsilon}(\hskip0,8pt\lieg)  \; =
%    \; \Rbqeps \otimes_{\raise-3pt\hbox{$ \scriptstyle \Rbq $}} \! \Utilde_\bq(\hskip0,8pt\lieg)  \; =
%    \; \Rbqeps \otimes_{\raise-3pt\hbox{$ \scriptstyle \Rbq $}} \!
% {\big(\, \Utilde_{\check{\bq}}(\hskip0,8pt\lieg) \big)}_\sigma  \; =   \hfill  \cr
%    \hfill   = \; {\Big(\, \Rbqeps \otimes_{\raise-3pt\hbox{$ \scriptstyle \Rbq $}}
%    \! \Utilde_{\check{\bq}}(\hskip0,8pt\lieg) \Big)}_{\sigma_\varepsilon}  \, =
%    \;  {\big(\, \Utilde_{\check{\bq}\hskip1pt,\hskip1pt{}\varepsilon}
% (\hskip0,8pt\lieg) \big)}_{\sigma_\varepsilon}   \quad  }  $$
%%%%%
%
  $$  \Utilde_{\bq\hskip1pt,\hskip1pt{}\varepsilon}(\hskip0,8pt\lieg)  =  \Rbqeps \otimes_{\raise-3pt\hbox{$ \scriptstyle \Rbq $}} \!\! \Utilde_\bq(\hskip0,8pt\lieg)  =  \Rbqeps \otimes_{\raise-3pt\hbox{$ \scriptstyle \Rbq $}} \!\! {\big(\, \Utilde_{\check{\bq}}(\hskip0,8pt\lieg) \big)}_\sigma  =  {\Big(\, \Rbqeps \otimes_{\raise-3pt\hbox{$ \scriptstyle \Rbq $}} \!\! \Utilde_{\check{\bq}}(\hskip0,8pt\lieg) \!\Big)}_{\sigma_\varepsilon}  \! =  {\big(\, \Utilde_{\check{\bq}\hskip1pt,\hskip1pt{}\varepsilon}(\hskip0,8pt\lieg) \big)}_{\sigma_\varepsilon}  $$
 and likewise we prove claim {\it (b)\/}  as well.
\epf

\medskip

\subsection{Quantum Frobenius morphisms for MpQG's}  \label{q-Frob}  \
 \vskip7pt
 When dealing with uniparameter quantum groups, the so-called ``quantum Frobenius morphisms'' set a strong link between specializations of these quantum groups (either restricted or unrestricted) at 1 and specializations at roots of unity.
%%%%%
% When dealing with uniparameter quantum groups, the so-called ``quantum Frobenius morphisms''
% set a strong link between specializations (of these quantum groups) at 1 and specializations
% at roots of unity.  Namely, for every root of unity (of odd order) there exists a corresponding
% quantum Frobenius morphism giving such a link.  In fact, even for a single root of unity, say
% $ \varepsilon \, $,  the situation is even richer: indeed, taking specializations   --- at 1
% or at  $ \varepsilon \, $  ---   can be done  {\it a priori\/}  in several different ways,
% which depend on the choice of the integral form we select for specializing.
%%%%%
                                                         \par
   When one chooses the  {\sl restricted\/}  and the  {\sl unrestricted\/}  integral forms, these quantum Frobenius morphisms (for uniparameter quantum groups) look as
  $$  \widehat{\hbox{\it F{\hskip0,5pt}r\,}}_{\!\raise-1pt\hbox{$ \scriptstyle \ell $}} \, : \; \Uhat_{\varepsilon}(\lieg) \,
  \relbar\joinrel\relbar\joinrel\relbar\joinrel\relbar\joinrel\relbar\joinrel\twoheadrightarrow \, \ZZ[\varepsilon] \otimes_{\raise-3pt\hbox{$ \scriptstyle \ZZ $}} \! \Uhat_1(\lieg)
 \eqno \text{({\sl restricted\/}  case)}  \quad  $$
and
  $$  \widetilde{\hbox{\it F{\hskip0,5pt}r\,}}_{\!\raise-1pt\hbox{$ \scriptstyle \ell $}} \, : \;
  \ZZ[\varepsilon] \otimes_{\raise-3pt\hbox{$ \scriptstyle \ZZ $}} \! \Utilde_1(\lieg) \,
  \lhook\joinrel\relbar\joinrel\relbar\joinrel\relbar\joinrel\relbar\joinrel\longrightarrow \, \Utilde_{\varepsilon}(\lieg)
  \eqno \text{({\sl unrestricted\/}  case)}  \quad  $$
where  $ \; \ZZ[\varepsilon] := \Zqqm \Big/ p_\ell(q) \, \Zqqm \; $,  $ \,\; \Uhat_s(\lieg) \, := \, \Uhat_q(\lieg) \Big/ \big(\, p_\ell(q) \, \Uhat_q(\lieg) \big)
%
%%%
% \; \cong \; \ZZ[\varepsilon]
% \otimes_{\raise-3pt\hbox{$ \scriptstyle \ZZ[\,q\,,\,q^{-1}] $}} \Uhat_q(\lieg)
%%%
%
 \;\, $  and similarly also  $ \,\; \Utilde_s(\lieg) \, := \, \Utilde_q(\lieg) \Big/ \big(\, p_\ell(q) \, \Utilde_q(\lieg) \big)
%
%%%
% \; \cong \; \ZZ[\varepsilon]
% \otimes_{\raise-3pt\hbox{$ \scriptstyle \ZZ[\,q\,,\,q^{-1}] $}} \Utilde_q(\lieg)
%%%
%
 \; $,  \, for  $ \, s \in \{ 1 \, , \varepsilon \,\} \, $.  Roughly speaking,  $ \widehat{\hbox{\it F{\hskip0,5pt}r\,}}_{\!\raise-1pt\hbox{$ \scriptstyle \ell $}} $  is given by taking  ``$ \ell $-th  roots'' of algebra generators of  $ \Uhat_{\eps}(\lieg) \, $,  namely quantum divided powers and quantum binomial coefficients, while (dually, in a sense)  $ \widetilde{\hbox{\it F{\hskip0,5pt}r\,}}_{\!\raise-1pt\hbox{$ \scriptstyle \ell $}} $  is given by raising to the  ``$ \ell $-th  power'' the algebra generators of  $ \Utilde_{\eps}(\lieg) \, $,  i.e.\ quantum root vectors and toral generators.
 \vskip5pt
   In the present subsection we shall show that similar quantum Frobenius morphisms do exist for MpQG's as well,
%%%%%
%    Even more, they admit essentially the same description as in the uniparameter case.
%%%%%
 with a similar description.

\medskip

\begin{free text}  \label{q-Frob-morph_Uhat}
 {\bf Quantum Frobenius morphisms in the restricted case.}  We start by considering quantum Frobenius
 morphisms in the restricted case, i.e.\ for the specializations at roots of unity of  $ \Uhatdotqgd $  and
 $ \Uhatqgd \, $.  Like in the uniparameter case, they will map any specialization at a root of unity
 $ \varepsilon $  onto a specialization at 1.
 \vskip7pt
   The following provides our  {\it quantum Frobenius morphisms\/}  for  {\sl restricted\/}  MpQG's:
\end{free text}

\vskip11pt

\begin{theorem}  \label{thm:qFrob-Uhat}
 Let  $ \, \bq := {\big(\,q_{ij}\big)}_{i,j \in I} \, $  be a multiparameter matrix  {\sl of integral type}.  Then there exists a Hopf algebra epimorphism (over  $ \; \RbqBeps \cong \, \ZZ[\varepsilon] \, $)
\begin{equation}  \label{eq:Frob-Uhatdoteps}
  {\dot{\widehat{\hbox{\it F{\hskip0,5pt}r\,}}}}_{\!\raise-1pt\hbox{$ \scriptstyle \ell $}} \, :
  \; \Uhatdot_{\bq\hskip1pt,\hskip1pt{}\varepsilon}(\hskip0,8pt\lieg) \, \relbar\joinrel\relbar\joinrel\relbar\joinrel\relbar\joinrel\twoheadrightarrow \,
  \RbqBeps \otimes_{\raise-3pt\hbox{$ \scriptstyle \RbqBuno $}} \!\! \Uhatdotqunogd \, \cong \, \ZZ[\varepsilon] \otimes_{\raise-3pt\hbox{$ \scriptstyle \ZZ $}} U_\ZZ\big(\hskip0,8pt\liegdotb\big)
\end{equation}
(cf.\ Theorem \ref{thm:specUqhat_q=1}{\it (a)\/})  given on generators by
\begin{align}
    E_i^{\,(n)}  \; \mapsto \;\,
       \begin{cases}
          \; \erm_i^{\,(n/\ell)}  &  \text{if}  \quad  \ell \,\Big|\, n  \\
          \hskip15pt  0  &  \text{if}  \quad  \ell \!\not\Big|\, n
       \end{cases}
 \quad  ,  \qquad  &  \quad
    F_i^{\,(n)}  \; \mapsto \;\,
       \begin{cases}
          \; \frm_i^{\,(n/\ell)}  &  \text{if}  \quad  \ell \,\Big|\, n  \\
          \hskip15pt  0  &  \text{if}  \quad  \ell \!\not\Big|\, n
       \end{cases}
 \\
   {\bigg( {{K_i \, ; c} \atop n} \bigg)}_{\!\!\varepsilon}  \mapsto
       \begin{cases}
          {\displaystyle \bigg(\, {{\kdotrm_i + c} \atop {n \big/ \ell}} \bigg)}  &  \!\!\text{if}\!  \quad  \ell \,\Big|\, n  \\
          \hskip17pt  0  &  \!\!\text{if}\!  \quad  \ell \!\not\Big|\, n
       \end{cases}
 \,  ,  \quad  &
    {\bigg( {{L_i \, ; c} \atop n} \bigg)}_{\!\!\varepsilon}  \mapsto
       \begin{cases}
          {\displaystyle \bigg(\, {{\ldotrm_i + c} \atop {n \big/ \ell}} \bigg)}  &  \!\!\text{if}\!  \quad  \ell \,\Big|\, n  \\
          \hskip17pt  0  &  \!\!\text{if}\!  \quad  \ell \!\not\Big|\, n
       \end{cases}
 \\
    {\bigg( {{G_i \, ; c} \atop n} \bigg)}_{\!\!\varepsilon_{i{}i}}  \mapsto \;
       \begin{cases}
          {\displaystyle \bigg(\, {{\hrm_i + c} \atop {n \big/ \ell}} \bigg)}  &  \text{if}  \quad  \ell \,\Big|\, n  \\
          \hskip17pt  0  &  \text{if}  \quad  \ell \!\not\Big|\, n
       \end{cases}
 \,  ,  &  \quad \qquad
    K_i^{\pm 1}  \, \mapsto \,  1  \;\;\; ,  \;\;\quad   L_i^{\pm 1}  \, \mapsto \,  1
\end{align}
   \indent
  Moreover, the image  $ \, \text{\sl Im}\Big( {\dot{\widehat{\hbox{\it F{\hskip0,5pt}r\,}}}}_{\!\raise-1pt\hbox{$ \scriptstyle \ell $}} \Big) \, $
   is co-central in  $ \; \RbqBeps \otimes_{\raise-3pt\hbox{$ \scriptstyle \RbqBuno $}} \!\! \Uhatdotqunogd \, $,  that is
\begin{equation}  \label{Frhatdot_cocentral}
  \big( \Delta - \Delta^{\text{\rm op}} \big)(u) \, \in \,
  \text{\sl Ker}\Big( {\dot{\widehat{\hbox{\it F{\hskip0,5pt}r\,}}}}_{\!\raise-1pt\hbox{$ \scriptstyle \ell $}} \Big)
  \otimes \text{\sl Ker}\Big( {\dot{\widehat{\hbox{\it F{\hskip0,5pt}r\,}}}}_{\!\raise-1pt\hbox{$ \scriptstyle \ell $}} \Big)  \qquad
  \text{for all}  \quad  u \in \Uhatdot_{\bq\hskip1pt,\hskip1pt{}\varepsilon}(\hskip0,8pt\lieg)
\end{equation}

 \vskip-5pt
   In addition,  when  $ \bq $  is of  {\sl strongly integral}  type, there exists yet another Hopf algebra epimorphism  (over  $ \; \RbqBeps \cong \, \ZZ[\varepsilon] \, $)
\begin{equation}  \label{eq:Frob-Uhateps}
  {\widehat{\hbox{\it F{\hskip0,5pt}r\,}}}_{\!\raise-1pt\hbox{$ \scriptstyle \ell $}} \, : \;
\Uhat_{\bq\hskip1pt,\hskip1pt{}\varepsilon}(\hskip0,8pt\lieg) \, \relbar\joinrel\relbar\joinrel\relbar\joinrel\relbar\joinrel\twoheadrightarrow \,
\RbqBeps \otimes_{\raise-3pt\hbox{$ \scriptstyle \RbqBuno $}} \!\! \Uhatqunogd \, \cong \, \ZZ[\varepsilon] \otimes_{\raise-3pt\hbox{$ \scriptstyle \ZZ $}} U_\ZZ\big(\hskip0,8pt\lieghatb\big)
\end{equation}
(cf.\ Theorem \ref{thm:specUqhat_q=1}{\it (b)\/})  for which
similar properties and a similar description hold true
with each  $ \, {\bigg(\! \displaystyle{{L_j \, ; c} \atop l_j} \bigg)}_{\!\!\varepsilon} \, $,  resp.\
$ \, {\bigg( \displaystyle{{K_j \, ; c} \atop k_j} \bigg)}_{\!\!\varepsilon} \, $,  replaced by
$ \, {\bigg( \displaystyle{{L_j \, ; c} \atop l_j} \bigg)}_{\!\!\varepsilon_j} \, $,  resp.\  $ \, {\bigg( \displaystyle{{K_j \, ; c} \atop k_j} \bigg)}_{\!\!\varepsilon_j} \, $.
\end{theorem}

\begin{proof}
 We present the proof for  $ \, {\dot{\widehat{\hbox{\it F{\hskip0,5pt}r\,}}}}_{\!\raise-1pt\hbox{$ \scriptstyle \ell $}} \, $  and
 $ \, \Uhatdot_{\bq\hskip1pt,\hskip1pt{}\varepsilon}(\hskip0,8pt\lieg) \, $,
 the rest being similar.
%%%%%
% the case for
%  $ \, {\widehat{\hbox{\it F{\hskip0,5pt}r\,}}}_{\!\raise-1pt\hbox{$ \scriptstyle \ell $}} \, $  and
%  $ \, \Uhat_{\bq\hskip1pt,\hskip1pt{}\varepsilon}(\hskip0,8pt\lieg) \, $,  being entirely similar.
%%%%%
%
                                                  \par
   By  Theorem \ref{thm:pres_Uhatgdq_gens-rels}{\it (a)},  we have a presentation of  $ \, \Uhatdot_{\bq\hskip1pt,\hskip1pt{}\varepsilon} := \Uhatdot_{\bq\hskip1pt,\hskip1pt{}\varepsilon}(\hskip0,8pt\lieg) \, $  by generators and relations. Then this also yields a similar presentation for  $ \, \RbqBeps \otimes_{\raise-3pt\hbox{$ \scriptstyle \RbqBuno $}} \!\!\! \Uhatdotqunogd \, $,  which is isomorphic to  $ \, \ZZ[\varepsilon] \otimes_{\raise-3pt\hbox{$ \scriptstyle \ZZ $}} U_\ZZ\big(\hskip0,8pt\liegdotb\big) \, $  as a Hopf algebra, by  Theorem \ref{thm:specUqhat_q=1}{\it (a)}.  Now, a moment's check shows that under the prescriptions given in the claim each relation in the presentation of  $ \Uhatdotqeps $  is mapped by  $ \, {\dot{\widehat{\hbox{\it F{\hskip0,5pt}r\,}}}}_{\!\raise-1pt\hbox{$ \scriptstyle \ell $}} \, $  onto either a similar relation in  $ \Uhatdotquno $  or zero, hence they do provide a well-defined algebra morphism as required.
                                                       \par
   To show what happens in a specific example, let us consider the relations
  $$  E_i^{\,(m)} \, F_i^{\,(n)}  \, = \;  \sum_{s=0}^{m \wedge n} \, F_i^{\,(n-s)} \, q_{ii}^{\,s} \,
  {\bigg( {{G_i \, ; \, 2\,s - m - n} \atop s} \bigg)}_{\!\!q_{ii}} \! L_i^{\,s} \, E_i^{\,(m-s)}   \eqno  \quad \forall \;\;  m, n \in \NN  \quad  $$
for every index  $ i \, $,  holding true in  $ \Uhatdot_\bq $  (cf.\ Theorem \ref{thm:pres_Uhatgdq_gens-rels}).
By specialization, these yield in  $ \Uhatdotqeps $  the relations
\begin{equation}  \label{comm-rel_E_vs_F_Ueps}
  E_i^{\,(m)} \, F_i^{\,(n)}  = \,  \sum_{s=0}^{m \wedge n} \, F_i^{\,(n-s)} \, \varepsilon_{ii}^{\,s} \,
  {\bigg( {{G_i \, ; \, 2\,s - m - n} \atop s} \bigg)}_{\!\!\varepsilon_{ii}} \! L_i^{\,s} \, E_i^{\,(m-s)}   \;\quad  \big(\, m, n \in \NN \,\big)
\end{equation}
and likewise in  $ \, \Uhatdotquno \cong U_\ZZ\big(\hskip0,8pt\liegdotb\big) \, $  the relations (cf.\
Definition \ref{def_spec-1} and Theorem \ref{thm:specUqhat_q=1})
\begin{equation}  \label{comm-rel_E_vs_F_Uuno}
  \erm_i^{\,(m)} \, \frm_i^{\,(n)}  \, = \;  \sum_{s=0}^{m \wedge n} \, \frm_i^{\,(n-s)} \, \bigg( {{\hrm_i \, + (2\,s - m - n)} \atop s} \bigg) \, \erm_i^{\,(m-s)}
  \qquad  \big(\, m, n \in \NN \,\big)
\end{equation}
where one uses a bit of arithmetic of  $ p $--binomial  coefficients (namely, the sixth line identity in the list of
Lemma \ref{commut_q-bin-coeff})  and of (classical) binomial coefficients to realize that specializing
$ \, {\Big(\! {{G_i \, ; \, 2\,s - m - n} \atop s} \!\Big)}_{\!\!q_{ii}} $  at  $ \, q = 1 \, $  eventually yields
$ \, \Big( {{\hrm_i \, + (2\,s - m - n)} \atop s} \!\Big) \, $.
                                                        \par
   Now, a moment's thought shows that if in left-hand side of  \eqref{comm-rel_E_vs_F_Ueps}  either  $ m $  or  $ n $
   is  {\sl not\/}  divisible by  $ \ell \, $,  then for each summand in right-hand side all of  $ (n\!-\!s) \, $,  $ s $  and
   $ (m\!-\!s) $  are not divisible either; hence our prescriptions for
   $ \, {\dot{\widehat{\hbox{\it F{\hskip0,5pt}r\,}}}}_{\!\raise-1pt\hbox{$ \scriptstyle \ell $}} \, $
   actually do map both sides of  \eqref{comm-rel_E_vs_F_Ueps}  to zero.  If instead both  $ m $  and
   $ n $  {\sl are\/}  divisible by  $ \ell \, $,  then there are also summands in right-hand side for which all
   of  $ (n\!-\!s) \, $,  $ s $  and  $ (m\!-\!s) $  are divisible as well; more explicitly, if  $ \, m = h\,\ell \, $  and
   $ \, n = k\,\ell \, $,  say, then the ``relevant'' summands on right-hand side are exactly those with index  $ \, s = r\,\ell \, $
   for all  $ \, r \in \big\{ 0 \, , 1 \, , \dots , h \wedge k \big\} \, $.  In this case, our prescriptions for
   $ \, {\dot{\widehat{\hbox{\it F{\hskip0,5pt}r\,}}}}_{\!\raise-1pt\hbox{$ \scriptstyle \ell $}} \, $
   map the left-hand side of  \eqref{comm-rel_E_vs_F_Ueps}  to  $ \; \erm_i^{\,(m/\ell)} \frm_i^{\,(n/\ell)} \! = \erm_i^{\,(h)} \frm_i^{\,(k)} \; $
   and the right-hand side to
  $$  {\textstyle \sum\limits_{r=0}^{h \wedge k} \, \frm_i^{\,((k\,\ell-r\,\ell)/\ell)} \, \Big( {{\hrm_i \, + (2\,r\,\ell - h\,\ell - k\,\ell)/\ell} \atop {r\,\ell/\ell}} \Big) \,
  \erm_i^{\,((h\,\ell-r\,\ell)/\ell)}  \; = \;  \sum\limits_{r=0}^{h \wedge k} \, \frm_i^{\,(k-r)} \, \Big( {{\hrm_i \, + (2\,r - h - k)} \atop r} \Big) \, \erm_i^{\,(h-r)} }  $$
where the right-hand side is equal to  $ \; \erm_i^{\,(h)} \frm_i^{\,(k)} \, $,  \, by  \eqref{comm-rel_E_vs_F_Uuno}  for  $ \, m := h \, $  and  $ \, n := k \, $.
 \vskip5pt
   Therefore the given formulas do provide a well-defined morphism of algebras
   $ \, {\dot{\widehat{\hbox{\it F{\hskip0,5pt}r\,}}}}_{\!\raise-1pt\hbox{$ \scriptstyle \ell $}} \, $  as required.
 By construction  $ {\dot{\widehat{\hbox{\it F{\hskip0,5pt}r\,}}}}_{\!\raise-1pt\hbox{$ \scriptstyle \ell $}} $  is clearly onto, as the generators of
 $ \, \RbqBeps \otimes_{\raise-3pt\hbox{$ \scriptstyle \RbqBuno $}} \!\!\! \Uhatdotqunogd \, \cong \, \ZZ[\varepsilon] \otimes_{\raise-3pt\hbox{$ \scriptstyle \ZZ $}} U_\ZZ\big(\hskip0,8pt\liegdotb\big) \, $  are the images via  $ {\dot{\widehat{\hbox{\it F{\hskip0,5pt}r\,}}}}_{\!\raise-1pt\hbox{$ \scriptstyle \ell $}} $  of the corresponding generators of  $ \, \Uhatdot_\bq \, $.
 \vskip3pt
   Finally, we must prove that  $ {\dot{\widehat{\hbox{\it F{\hskip0,5pt}r\,}}}}_{\!\raise-1pt\hbox{$ \scriptstyle \ell $}} $  is also a  {\sl Hopf algebra\/}  morphism and that its image is co-central.  This follows from the uniparameter case, as the coalgebra structure of the integral form of these MpQG's (cf.\ Theorem \ref{thm:pres_Uhatgdq_gens-rels}{\it (a)\/}  and \cite[Proposition 6.4]{DL})  is the same as in the canonical case (the cocycle deformation process does not change the coalgebra structure), and our quantum Frobenius morphism is described by the same formulas.
\end{proof}

\vskip9pt

\begin{free text}  \label{q-Frob-morph_Utilde}
 {\bf The unrestricted case: quantum Frobenius morphisms for  $ \Utildeqgd \, $.}
%
% Looking for quantum Frobenius morphisms, the
%
 In the unrestricted case, i.e.\ that of  $ \Utildeqgd \, $,  quantum Frobenius morphisms, in comparison with the restricted case,
 ``go the other way round''.  Indeed, like in the uniparameter case, we shall find them mapping the specialization at 1
 (of the given unrestricted integral form of a MpQG) into any specialization at a root of unity  $ \varepsilon \, $.
                                                                    \par
   The very construction of such quantum Frobenius morphisms requires some preparation.  Mimicking what was found
   in  \cite{DP}  for the canonical case, the first ingredient is the subalgebra of  $ \Utildeqepsgd $ generated
   by the  $ \ell $--th  powers of its generators.
\end{free text}

\smallskip

\begin{definition}  \label{def:Z_0}
 We define  $ Z_0 $  to be the  $ \Rbqeps $--subalgebra
  $$  Z_0  \,\; := \;\,  {\Big\langle\, \fbar_\alpha^{\;\ell} \, , \; l_i^{\,\pm \ell} \, , \; k_i^{\,\pm \ell} \, ,
  \, \ebar_\alpha^{\;\ell} \,\Big\rangle}_{\! \alpha \in Q \, , \, i \in I}  $$
of  $ \, \Utildeqepsgd \, $  generated by the  $ \ell $--th  powers of the generators of  $ \Utildeqepsgd \, $.   \hfill  $ \diamondsuit $
\end{definition}

\smallskip

   {\sl N.B.:}\,  the original definition of  $ Z_0 $  given in  \cite[Chapter 5, \S 19.1]{DP}  reads different, but it is also proved
   --- still in  [{\it loc. cit.}]  ---   to be equivalent to the one given above.

\smallskip

   The main properties of  $ Z_0 $  were investigated in  \cite[ \S 4]{An4}, with a slightly more general approach.
   The main outcome reads as follows:

\smallskip

\begin{prop}  \label{prop: struct-Z_0}
 {\sl (cf.\  \cite[\S 4]{An4})}
 \vskip2pt
   (a)\,  $ Z_0 $  is  $ \boldsymbol{\varepsilon} $--central  in  $ \Utildeqepsgd \, $,  i.e., for each monomial  $ b $  in a PBW basis of
   $ \, \Utildeqepsgd $  as in  Theorem \ref{PBW-mpqgs_q=eps}  and each generator
   $ \, z \in {\big\{\, \fbar_\alpha^{\;\ell} \, , \; l_i^{\,\pm \ell} \, , \; k_i^{\,\pm \ell} \, , \, \ebar_\alpha^{\;\ell} \,\big\}}_{\! \alpha \in Q \, , \, i \in I} \, $
   of  $ \, Z_0 $  there exists a (Laurent) monomial  $ \, m_{z,b}\big( \boldsymbol{\varepsilon}^{\pm \ell} \big) \, $  in the  $ \varepsilon_{ij}^{\pm \ell} $'s
 such that
  $$  z \, b \; = \; m_{z,b}\big( \boldsymbol{\varepsilon}^{\pm \ell} \big) \, b \, z  $$
%
%%%
 In particular, when\/  $ \bq $  is of  {\sl integral}  type  $ Z_0 $  is  {\sl central},
hence  {\sl normal},  in  $ \Utildeqepsgd \, $.
 \vskip2pt
   (b)\,  $ Z_0 $  is a Hopf subalgebra of  $ \, \Utildeqepsgd \, $,  which is isomorphic as an algebra over  $ \Rbqeps $  to a partially Laurent
   $ \boldsymbol{\varepsilon} $--polynomial  algebra, namely
  $$  Z_{0}  \,\; \cong \;\,  \Rbqeps \Big[{\big\{ f_\alpha^{\;\ell} , l_i^{\pm \ell} , k_i^{\pm \ell} , e_\alpha^{\;\ell} \big\}}_{\alpha \in \Phi^+}^{i \in I} \Big]  $$
where the indeterminates  {\boldmath{$ \varepsilon $}}--commute  (notation as in  \S \ref{spec-1})  among them
%%%%%
%  in the following sense:
%%%%%
 i.e.
  $$  \displaylines{
   f_{\alpha'}^{\;\ell} \, f_{\alpha''}^{\;\ell} \, = \, \varepsilon_{\raise-1pt\hbox{$ \scriptstyle \alpha'' \alpha' $}}^{\,\ell^{\,\raise1pt\hbox{$ \scriptscriptstyle 2 $}}}
   \, f_{\alpha''}^{\;\ell} \, f_{\alpha'}^{\;\ell}  \;\; ,  \qquad
e_{\alpha'}^{\;\ell} \, f_{\alpha''}^{\;\ell} \, = \, f_{\alpha''}^{\;\ell} \, e_{\alpha'}^{\;\ell}  \;\; ,  \qquad
e_{\alpha'}^{\;\ell} \, e_{\alpha''}^{\;\ell} \, = \, \varepsilon_{\raise-1pt\hbox{$ \scriptstyle \alpha' \alpha'' $}}^{\,\ell^{\,\raise1pt\hbox{$ \scriptscriptstyle 2 $}}}
\, e_{\alpha''}^{\;\ell} \, e_{\alpha'}^{\;\ell}  \cr
   k_i^{\pm \ell} \, e_\alpha^{\;\ell} \, = \, \varepsilon_{\raise-1pt\hbox{$ \scriptstyle \alpha_i \, \alpha $}}^{\,\pm \ell^{\,\raise1pt\hbox{$ \scriptscriptstyle 2 $}}}
   \, e_\alpha^{\;\ell} \, k_i^{\pm \ell}  \quad ,  \quad \qquad  l_i^{\pm \ell} \, e_\alpha^{\;\ell} \, =
   \, \varepsilon_{\raise-1pt\hbox{$ \scriptstyle \alpha \, \alpha_i $}}^{\,\mp \ell^{\,\raise1pt\hbox{$ \scriptscriptstyle 2 $}}} \, e_\alpha^{\;\ell} \, l_i^{\pm \ell}  \cr
   k_i^{\pm \ell} \, f_\alpha^{\;\ell} \, = \, \varepsilon_{\raise-1pt\hbox{$ \scriptstyle \alpha_i \, \alpha $}}^{\,\mp \ell^{\,\raise1pt\hbox{$ \scriptscriptstyle 2 $}}}
   \, f_\alpha^{\;\ell} \, k_i^{\pm \ell}  \quad ,  \quad \qquad  l_i^{\pm \ell} \, f_\alpha^{\;\ell} \, =
   \, \varepsilon_{\raise-1pt\hbox{$ \scriptstyle \alpha \, \alpha_i $}}^{\,\pm \ell^{\,\raise1pt\hbox{$ \scriptscriptstyle 2 $}}} \, f_\alpha^{\;\ell} \, l_i^{\pm \ell}  \cr
   k_i^{\pm \ell} \, k_j^{\pm \ell} \, = \, k_j^{\pm \ell} \, k_i^{\pm \ell}  \;\; ,  \qquad  k_i^{\pm \ell} \, l_j^{\pm \ell} \, = \, l_j^{\pm \ell} \, k_i^{\pm \ell}  \;\; ,
   \qquad  k_i^{\pm \ell} \, k_j^{\pm \ell} \, = \, k_j^{\pm \ell} \, k_i^{\pm \ell}  }  $$
 In particular, if  $ \, \bq $  is of  {\sl integral}  type   --- hence  $ \, \Rbqeps = \ZZ[\varepsilon]] \, $  ---   then  $ \, \Utildeqepsgd $  is a  {\sl commutative}  Hopf algebra of partially Laurent polynomials.
 \vskip2pt
   (c)\;  $ \Utildeqepsgd \, $  is a free (left or right)  $ Z_0 $--module  of rank  $ \ell^{\,\raise2pt\hbox{$ \scriptstyle \text{\sl dim}(\hskip0,8pt\lieg) $}} \, $.
\end{prop}

\pf
 Almost everything is proved in   \cite[\S 4]{An4},  so we just stress a single detail concerning claim  {\it (c)}.
 Indeed,  in  \cite[Proposition 4.1]{An4}  yields claim  {\it (b)\/}  as well as  {\it (c)},  but for the latter the involved
 coefficients read differently, for instance one has
 $ \; f_{\alpha'}^{\;\ell} \, f_{\alpha''}^{\;\ell} \, = \, \varepsilon_{\raise-1pt\hbox{$ \scriptstyle \ell\,\alpha'' ,
 \, \ell\,\alpha' $}} \, f_{\alpha''}^{\;\ell} \, f_{\alpha'}^{\;\ell} \; $.  But the symbol  $ \, \varepsilon_{\raise-1pt\hbox{$ \scriptstyle \alpha , \, \beta $}} \, $
 is bimultiplicative in  $ \alpha $  and  $ \beta $   --- i.e., it is a bicharacter on  $ Q \times Q \, $
 ---   hence  $ \; \varepsilon_{\raise-1pt\hbox{$ \scriptstyle \ell\,\alpha'' , \, \ell\,\alpha' $}} =
 \varepsilon_{\raise-1pt\hbox{$ \scriptstyle \alpha'' \alpha' $}}^{\,\ell^{\,\raise1pt\hbox{$ \scriptscriptstyle 2 $}}} \; $,  and we are done.
\epf

\vskip11pt

   We shall now compare the subalgebra  $ Z_0 \, $,  a sub-object inside  $ \Utildeqepsgd \, $,  which is the specialisation of  $ \Utildeqgd $  at  $ \, q = \varepsilon \, $,  with the specialization at  $ \, q = 1 \, $,  that is  $ \, \Utildequnogd \, $.  This leads to find a special morphism, which we call  {\sl quantum Frobenius morphism for  $ \Utildeqgd \, $},  which links  $ \Utildequnogd $  with  $ \Utildeqepsgd $ --- once again generalising what occurs in the uniparameter case.  In order to formalise this, we need to make  $ \Utildequnogd $  into a Hopf algebra  {\sl over  $ \Rbqeps \, $},  so that we can compare it with  $ \Utildeqepsgd \, $.
 \vskip3pt
   Let us consider the unique ring embedding
\begin{equation}  \label{Frobenius-for-Rings}
   {\ } \hskip-57pt
  \Rbquno \!\lhook\joinrel\relbar\joinrel\longrightarrow \Rbqsqeps \, \big( \cong \Rbqeps \,\big)  \quad ,
  \qquad  y_{ij}^{\pm 1} \mapsto \varepsilon_{\raise-1pt\hbox{$ \scriptstyle i j $}}^{\,\ell^{\,\raise1pt\hbox{$ \scriptscriptstyle 2 $}}} \;
  \Big( \cong {\big(\, \varepsilon_{ij}^{\pm 1/2} \,\big)}^{\!2\,\ell^2} \,\Big)
\end{equation}
where in right-hand side we take into account the isomorphism  $ \, \Rbqsqeps \cong \Rbqeps \, $  given by  Lemma \ref{lemma: RbqBeps = RbqBsqeps};
we use this embedding to perform scalar extension from  $ \Rbquno $  to  $ \Rbqeps $  for  $ \Utildequnogd \, $,  so to make
$ \, \Rbqeps \otimes_{\raise-3pt\hbox{$ \scriptstyle \Rbquno $}} \!\!\! \Utildequnogd \, $  into a (Hopf) algebra over  $ \Rbqeps \, $.
                                                                   \par
   Besides, recall   --- from  Proposition \ref{prop: proport_root-vects}  ---   that for any  $ \, \alpha \in \Phi^+ \, $  there exists suitable (Laurent) monomials
   $ \, m^+_\alpha\big( \bq^{\pm 1/2} \big) \, $  and  $ \, m^-_\alpha\big( \bq^{\pm 1/2} \big) \, $  in the  $ q_{ij}^{\pm 1/2} $'s  such that
  $$  E_\alpha  \, = \,  m^+_\alpha\big( \bq^{\pm 1/2} \big) \, \check{E}_\alpha  \quad ,  \qquad  F_\alpha = m^-_\alpha\big( \bq^{\pm 1/2} \big) \, \check{F}_\alpha  $$
where  $ \check{E}_\alpha \, $,  resp.\  $ \check{F}_\alpha \, $,  is the quantum root vector associated with  $ \, \alpha \in \Phi^+ \, $,  resp.\
$ \, -\alpha \in \Phi^- \,$,  in  $ \QEqcheck $  and  $ E_\alpha \, $,  resp.\  $ F_\alpha \, $,  is the similar vector in  $ \, \QEq = {\big( \QEqcheck \big)}_\sigma \, $.
                                                                   \par
   As a direct consequence, we have similar relations among quantum root vectors in
   $ \, \Utildequnogd = {\big(\, \Rbquno \otimes_{\raise-2pt\hbox{$ \scriptstyle \ZZ[q,\hskip1pt{}q^{-1}] $}}
   \Utilde_{\check{\bq}\hskip1pt,1}(\hskip0,8pt\lieg) \big)}_\sigma \, $  and  $ \, \Utildeqepsgd
   = {\big(\, \Rbqeps \otimes_{\raise-2pt\hbox{$ \scriptstyle \ZZ[\varepsilon,\hskip1pt{}\varepsilon^{-1}] $}}
   \Utilde_{\check{\bq}\hskip1pt,\hskip1pt\varepsilon}(\hskip0,8pt\lieg) \big)}_\sigma \, $, namely
\begin{equation}
  \ebar_\alpha  \, = \,  m^+_\alpha\big(\, \by^{\pm 1/2} \big) \, \check{\ebar}_\alpha  \quad ,  \qquad
  \fbar_\alpha = m^-_\alpha\big( \by^{\pm 1/2} \big) \, \check{\fbar}_\alpha  \qquad\;  \text{in}  \quad
  \Utildequnogd   \label{prop_root-vects_Utildequno}
\end{equation}
and
\begin{equation}
  \ebar_\alpha  \, = \,  m^+_\alpha\big(\, \boldsymbol{\varepsilon}^{\pm 1/2} \big) \,
  \check{\ebar}_\alpha  \quad ,  \qquad  \fbar_\alpha = m^-_\alpha\big( \boldsymbol{\varepsilon}^{\pm 1/2} \big) \, \check{\fbar}_\alpha
  \qquad\;  \text{in}  \quad  \Utildeqepsgd   \label{prop_root-vects_Utildeqeps}
\end{equation}
 \vskip7pt
   Our main result in this subsection is the existence of  {\it quantum Frobenius morphisms\/}  for {\sl unrestricted MpQG's},  that are the monomorphisms mentioned below:

\vskip11pt

\begin{theorem}  \label{thm: unrestr_q-Frob-morph}
 There exists a Hopf algebra monomorphism
  $$  \widetilde{\hbox{\it F{\hskip0,5pt}r\,}}_{\!\raise-1pt\hbox{$ \scriptstyle \ell $}} \, : \;
   \Rbqeps \otimes_{\raise-3pt\hbox{$ \scriptstyle \Rbquno $}} \!\! \Utildequnogd \,
   \lhook\joinrel\relbar\joinrel\relbar\joinrel\relbar\joinrel\relbar\joinrel\longrightarrow \,
   \Utildeqepsgd  $$
uniquely determined (still identifying  $ \, \Rbqsqeps \cong \Rbqeps \, $)  for all  $ \, \alpha \in \Phi^+ \, , \, i \in I \, $,  by
  $$  {\ } \hskip-17pt   \fbar_\alpha \mapsto
  {m^-_\alpha\big( \text{\boldmath{$ \varepsilon $}}^{\pm 1/2} \big)}^{\! \ell^{\,\raise1pt\hbox{$ \scriptscriptstyle 2 $}} - \ell} \,
  \fbar_\alpha^{\;\ell}  \; ,  \;\;  l_i^{\,\pm 1} \mapsto l_i^{\,\pm \ell}  \; ,  \;\;  k_i^{\,\pm 1} \mapsto k_i^{\,\pm \ell}  \; ,
  \;\;  \ebar_\alpha \mapsto {m^+_\alpha\big( \text{\boldmath{$ \varepsilon $}}^{\pm 1/2} \big)}^{\! \ell^{\,\raise1pt\hbox{$ \scriptscriptstyle 2 $}} - \ell}
  \, \ebar_\alpha^{\;\ell}   \eqno (\star)  $$
whose image is the  {\boldmath{$ \varepsilon $}}--central  Hopf subalgebra  $ Z_0 $  of  $ \, \Utilde_{\bq\hskip1pt,\varepsilon}(\hskip0,8pt\lieg) \, $;
as a consequence, the Hopf algebra  $ Z_0 $  itself is isomorphic to  $ \, \Rbqeps \otimes_{\raise-3pt\hbox{$ \scriptstyle \Rbquno $}} \!\! \Utildequnogd \, $.
                                                                  \par
   In particular,  {\sl when  $ \bq $  is integral}  the morphism  $ \widetilde{\hbox{\it F{\hskip0,5pt}r\,}}_{\!\raise-1pt\hbox{$ \scriptstyle \ell $}} $
   is described by the simpler formulas  (for all  $ \, \alpha \in \Phi^+ \, , \, i \in I \, $)
  $$  \fbar_\alpha \mapsto \fbar_\alpha^{\;\ell}  \quad ,  \qquad  l_i^{\,\pm 1} \mapsto l_i^{\,\pm \ell}  \quad ,  \qquad  k_i^{\,\pm 1} \mapsto k_i^{\,\pm \ell}  \quad ,
  \qquad  \ebar_\alpha \mapsto \ebar_\alpha^{\;\ell}  \quad .  $$
\end{theorem}

\pf
 To begin with, the morphism in  (\ref{Frobenius-for-Rings})  maps every
 $ \, y_{\alpha , \beta}^{\pm 1/2} \in \Rbqsquno \cong \Rbquno \, $  into the corresponding
 $ \, \varepsilon_{\alpha , \beta}^{\pm \ell^{\,\raise1pt\hbox{$ \scriptscriptstyle 2 $}} \! / \, 2} \in \Rbqsqeps \cong \Rbqeps \, $.  Moreover, for  $ \, \bq = \check{\bq} \, $  the analysis in  \cite{DP}  yields a (``quantum Frobenius'') morphism
 $ \; \widetilde{\hbox{\it F{\hskip0,5pt}r\,}}^\vee_{\!\raise-1pt\hbox{$ \scriptstyle \ell $}} \, : \,
 \Utilde_{\check{\bq}\hskip1pt,1}(\hskip0,8pt\lieg) \, \lhook\joinrel\relbar\joinrel\relbar\joinrel\relbar\joinrel\longrightarrow \,
 \Utilde_{\check{\bq}\hskip1pt,\hskip1pt\varepsilon}(\hskip0,8pt\lieg) \; $  {\sl of Hopf algebras\/}
 which is determined by the formulas given for the integral case in the above statement when
 $ \, \fbar_\alpha = \check{\fbar}_\alpha \, $,  etc., that is
 $ \, \widetilde{\hbox{\it F{\hskip0,5pt}r\,}}^\vee_{\!\raise-1pt\hbox{$ \scriptstyle \ell $}} \big( \check{\fbar}_{\!\alpha} \big)
 = \check{\fbar}_{\!\alpha}^{\;\raise-3pt\hbox{$ \scriptstyle \ell $}} \, $  and so on; in particular, this
 $ \widetilde{\hbox{\it F{\hskip0,5pt}r\,}}^\vee_{\!\raise-1pt\hbox{$ \scriptstyle \ell $}} $  preserves the coproduct.
%%%
 Now, extending scalars, we obtain yet another Hopf algebra monomorphism
  \,\;  $ \widetilde{\hbox{\it F{\hskip0,5pt}r\,}}_{\!\raise-1pt\hbox{$ \scriptstyle \ell $}} \, : \,
\Rbqeps \otimes_{\raise-2pt\hbox{$ \scriptstyle \ZZ[q,\hskip1pt{}q^{-1}] $}}
\Utilde_{\check{\bq}\hskip1pt,1}(\hskip0,8pt\lieg) \,
\lhook\joinrel\relbar\joinrel\relbar\joinrel\relbar\joinrel\relbar\joinrel\longrightarrow \,
\Rbqeps \otimes_{\raise-2pt\hbox{$ \scriptstyle \ZZ[\varepsilon,\hskip1pt{}\varepsilon^{-1}] $}}
\Utilde_{\check{\bq}\hskip1pt,\hskip1pt{}\varepsilon}(\hskip0,8pt\lieg) \;\, $
 that fits into the commutative diagram
 \vskip-15pt
  $$  \xymatrix{
   \Utilde_{\check{\bq}\hskip1pt,1}(\hskip0,8pt\lieg)_{\raise-9pt\hbox{\phantom{|}}} \hskip-5pt
   \ar@{^{(}->}[rr]^{\widetilde{\hbox{\it F{\hskip0,5pt}r\,}}^\vee_{\!\raise-1pt\hbox{$ \scriptstyle \ell $}}}
   \ar@{_{(}->}[d]_{}
   &  &  \hskip5pt \Utilde_{\check{\bq}\hskip1pt,\hskip1pt\varepsilon}(\hskip0,8pt\lieg)_{\raise-9pt\hbox{\phantom{|}}}  \ar@{^{(}->}[dd]^{}  \\
   \Rbquno \otimes_{\raise-2pt\hbox{$ \scriptstyle \ZZ[q,\hskip1pt{}q^{-1}] $}}
   \Utilde_{\check{\bq}\hskip1pt,1}(\hskip0,8pt\lieg)_{\raise-9pt\hbox{\phantom{|}}}  \ar@{_{(}->}[d]_{}  &  &  \\
  \Rbqeps \otimes_{\raise-2pt\hbox{$ \scriptstyle \Rbquno $}} \!
  \Big(\, \Rbquno \otimes_{\raise-2pt\hbox{$ \scriptstyle \ZZ[q,\hskip1pt{}q^{-1}] $}} \Utilde_{\check{\bq}\hskip1pt,1}(\hskip0,8pt\lieg) \Big)  \hskip7pt
  \ar@{_{(}->}[rr]_{\hskip37pt \widetilde{\hbox{\it F{\hskip0,5pt}r\,}}_{\!\raise-1pt\hbox{$ \scriptstyle \ell $}}}  &  &  \hskip5pt  \Rbqeps
  \otimes_{\raise-2pt\hbox{$ \scriptstyle \ZZ[\varepsilon,\hskip1pt{}\varepsilon^{-1}] $}}
  \Utilde_{\check{\bq}\hskip1pt,\hskip1pt{}\varepsilon}(\hskip0,8pt\lieg)  }  $$
 \vskip5pt
   Let us check that  $ \widetilde{\hbox{\it F{\hskip0,5pt}r\,}}_{\!\raise-1pt\hbox{$ \scriptstyle \ell $}} $
   satisfies the equalities in  $ (\star) \, $.  First, as by  (\ref{prop_root-vects_Utildequno})
   we have  $ \; \ebar_\alpha \; = {m^+_\alpha\big(\, \by^{\pm 1/\,2} \big)} \, \check{\ebar}_\alpha $
we find that
  $$  \displaylines{
   \quad   \widetilde{\hbox{\it F{\hskip0,5pt}r\,}}_{\!\raise-1pt\hbox{$ \scriptstyle \ell $}}\big(\ebar_\alpha \big)
   \, = \,  \widetilde{\hbox{\it F{\hskip0,5pt}r\,}}_{\!\raise-1pt\hbox{$ \scriptstyle \ell $}}\big( {m^+_\alpha\big(\, \by^{\pm 1/\,2} \big)}\, \check{\ebar}_\alpha \big)
   \, = \,  {m^+_\alpha\big(\, \text{\boldmath{$ \varepsilon $}}^{\pm 1/\,2} \big)}^{\ell^{\,\raise1pt\hbox{$ \scriptscriptstyle 2 $}}} \,
   \widetilde{\hbox{\it F{\hskip0,5pt}r\,}}_{\!\raise-2pt\hbox{$ \scriptstyle \ell $}}\big(  \check{\ebar}_\alpha \big)  \, =   \hfill  \cr
   \hfill   = \,  {m^+_\alpha\big(\, \text{\boldmath{$ \varepsilon $}}^{\pm 1/\,2} \big)}^{\ell^{\,\raise1pt\hbox{$ \scriptscriptstyle 2 $}}} \,
   \widetilde{\hbox{\it F{\hskip0,5pt}r\,}}^{\vee}_{\!\raise-2pt\hbox{$ \scriptstyle \ell $}}\big(  \check{\ebar}_\alpha \big)
   \, = \,  {m^+_\alpha\big(\, \text{\boldmath{$ \varepsilon $}}^{\pm 1/\,2} \big)}^{\ell^{\,\raise1pt\hbox{$ \scriptscriptstyle 2 $}}} \, \check{\ebar}_\alpha^{\ell}  \,
   =   \hfill  \cr
   \hfill   = \,  {m^+_\alpha\big(\, \text{\boldmath{$ \varepsilon $}}^{\pm 1/\,2} \big)}^{\ell^{\,\raise1pt\hbox{$ \scriptscriptstyle 2 $}}}
   \, {m^+_\alpha\big(\, \text{\boldmath{$ \varepsilon $}}^{\pm 1/\,2} \big)}^{ - \ell} \, \ebar_\alpha^{\,\ell}
   \, = \,  {m^+_\alpha\big(\, \text{\boldmath{$ \varepsilon $}}^{\pm 1/\,2} \big)}^{\ell^{\,\raise1pt\hbox{$ \scriptscriptstyle 2 $}} - \ell} \, \ebar_\alpha^{\,\ell}   \quad }  $$
--- thanks to  (\ref{prop_root-vects_Utildeqeps})  ---   i.e.\
  $ \; \widetilde{\hbox{\it F{\hskip0,5pt}r\,}}_{\!\raise-1pt\hbox{$ \scriptstyle \ell $}}\big(\ebar_\alpha \big) =
  {m^+_\alpha\big(\, \text{\boldmath{$ \varepsilon $}}^{\pm 1/\,2} \big)}^{\ell^{\,\raise1pt\hbox{$ \scriptscriptstyle 2 $}} - \ell} \, \ebar_\alpha^{\,\ell} \; $
 for every root vector  $ \, {\ebar}_\alpha \, $  in  $ \Utildequnogd \, $;  similarly one finds
  $ \widetilde{\hbox{\it F{\hskip0,5pt}r\,}}_{\!\raise-1pt\hbox{$ \scriptstyle \ell $}}\big(\fbar_\alpha \big) =
{m^-_\alpha\big(\, \text{\boldmath{$ \varepsilon $}}^{\pm 1/\,2} \big)}^{\ell^{\,\raise1pt\hbox{$ \scriptscriptstyle 2 $}} - \ell} \, \fbar_\alpha^{\,\ell} \; $
 when dealing with the  $ \, {\fbar}_\alpha $'s,  and
  $ \widetilde{\hbox{\it F{\hskip0,5pt}r\,}}_{\!\raise-1pt\hbox{$ \scriptstyle \ell $}}\big(l_i^{\pm 1} \big) = \,  l_i^{\,\pm \ell} \; $,
  $\widetilde{\hbox{\it F{\hskip0,5pt}r\,}}_{\!\raise-1pt\hbox{$ \scriptstyle \ell $}}\big(k_i^{\pm 1} \big) = \, k_i^{\,\pm\ell} \; $
  for all  $ \, i \in I \, $.
                                                                     \par
   Now recall that we have identifications of coalgebras
  $$  \Utildequnogd  \, =  {\Big(\, \Rbquno \otimes_{\raise-2pt\hbox{$ \scriptstyle \ZZ[q,\hskip1pt{}q^{-1}] $}}
   \Utilde_{\check{\bq}\hskip1pt,1}(\hskip0,8pt\lieg) \Big)}_{\!\sigma_{\raise-2pt\hbox{$ \scriptscriptstyle 1 $}}}  \; ,  \!\quad
   \Utildeqepsgd  \, =  {\Big(\, \Rbqeps \otimes_{\raise-2pt\hbox{$ \scriptstyle \ZZ[\varepsilon,\hskip1pt{}\varepsilon^{-1}] $}}
   \Utilde_{\check{\bq}\hskip1pt,\hskip1pt{}\varepsilon}(\hskip0,8pt\lieg) \Big)}_{\!\sigma_\varepsilon}  $$
hence the monomorphism $ \, \widetilde{\hbox{\it F{\hskip0,5pt}r\,}}_{\!\raise-1pt\hbox{$ \scriptstyle \ell $}} \, $
defines also a monomorphism of  {\sl coalgebras\/}   --- over  $ \Rbqeps $
---   from  $ \; \Rbqeps \otimes_{\raise-2pt\hbox{$ \scriptstyle \Rbquno $}} \! \Utildequnogd \; $ to  $ \, \Utildeqepsgd \, $.
To prove that  $ \, \widetilde{\hbox{\it F{\hskip0,5pt}r\,}}_{\!\raise-1pt\hbox{$ \scriptstyle \ell $}} \, $  is also a  {\sl Hopf algebra\/}
morphism, it is enough to prove that
$ \, \widetilde{\hbox{\it F{\hskip0,5pt}r\,}}_{\!\raise-1pt\hbox{$ \scriptstyle \ell $}} (\, x\cdot_{\sigma_{1}} y) =
\widetilde{\hbox{\it F{\hskip0,5pt}r\,}}_{\!\raise-1pt\hbox{$ \scriptstyle \ell $}}(x)\cdot_{\sigma_{\eps}} \!
\widetilde{\hbox{\it F{\hskip0,5pt}r\,}}_{\!\raise-1pt\hbox{$ \scriptstyle \ell $}}(y) \, $  for all  $ x $,  $ y $  in
$ \, \Rbqeps \otimes_{\raise-2pt\hbox{$ \scriptstyle \ZZ[q,\hskip1pt{}q^{-1}] $}} \Utilde_{\check{\bq}\hskip1pt,1}(\hskip0,8pt\lieg) \, $.
This in turn follows from the fact that
\begin{equation}  \label{s_eps Fr_ell = s_1}
  \sigma_{\eps}\Big(\widetilde{\hbox{\it F{\hskip0,5pt}r\,}}_{\!\raise-1pt\hbox{$ \scriptstyle \ell $}}(x) \, ,
  \widetilde{\hbox{\it F{\hskip0,5pt}r\,}}_{\!\raise-1pt\hbox{$ \scriptstyle \ell $}}(y) \Big)  \, = \,  \sigma_{1}(x,y)
  \qquad \text{for all} \quad  x, y \, \in \, \Rbqeps \otimes_{\raise-2pt\hbox{$ \scriptstyle \ZZ[q,\hskip1pt{}q^{-1}] $}}
  \Utilde_{\check{\bq}\hskip1pt,1}(\hskip0,8pt\lieg)
\end{equation}
which can be checked by direct computation with  $ x $  and  $ y $  being generators of
$ \Utilde_{\check{\bq}\hskip1pt,1}(\hskip0,8pt\lieg) $
and using the ring embedding \eqref{Frobenius-for-Rings}; for example, one has
  $$  \sigma_{\eps}\Big(\widetilde{\hbox{\it F{\hskip0,5pt}r\,}}_{\!\raise-1pt\hbox{$ \scriptstyle \ell $}}(k_{i}) \, ,
  \widetilde{\hbox{\it F{\hskip0,5pt}r\,}}_{\!\raise-1pt\hbox{$ \scriptstyle \ell $}}(k_{j}) \Big)  \; =
  \;  \sigma_{\eps}\Big(k_{i}^{\ell} \, ,k_{j}^{\ell} \Big)  \, = \;  \eps_{ij}^{\ell^{\,2}/2}  \, = \;  y_{ij}^{1/2}  \, =
  \;  \sigma_1(k_i,k_j)   \quad  \text{for all}  \;\; i, j \in I \; .  $$
 In fact, from  \eqref{s_eps Fr_ell = s_1}  we get
\begin{align*}
  \widetilde{\hbox{\it F{\hskip0,5pt}r\,}}_{\!\raise-1pt\hbox{$ \scriptstyle \ell $}}  &  \big( x \cdot_{\sigma_1} y \big)  \; = \;
  \widetilde{\hbox{\it F{\hskip0,5pt}r\,}}_{\!\raise-1pt\hbox{$ \scriptstyle \ell $}}
  \big( \sigma_1\big( x_{(1)} , y_{(1)} \big) \, x_{(2)} \, y_{(2)} \, \sigma_1^{-1}\big( x_{(3)} , y_{(3)} \big) \big)  \; =  \\
%
%%%%%
%   &  \; = \;  \sigma_1\big( x_{(1)} , y_{(1)} \big) \,
%   \widetilde{\hbox{\it F{\hskip0,5pt}r\,}}_{\!\raise-1pt\hbox{$ \scriptstyle \ell $}}\big( x_{(2)}
% \, y_{(2)} \big) \, \sigma_1^{-1}\big( x_{(3)} , y_{(3)} \big)  \\
%%%%%
%
  &  \; = \;  \sigma_1\big( x_{(1)} , y_{(1)} \big) \,
  \widetilde{\hbox{\it F{\hskip0,5pt}r\,}}_{\!\raise-1pt\hbox{$ \scriptstyle \ell $}}\big( x_{(2)} \big) \,
\widetilde{\hbox{\it F{\hskip0,5pt}r\,}}_{\!\raise-1pt\hbox{$ \scriptstyle \ell $}}\big( y_{(2)} \big) \,
\sigma_1^{-1}\big( x_{(3)} , y_{(3)} \big)  \; =  \\
%
%%%%%
%   &  \; = \;  \sigma_\eps\Big( \widetilde{\hbox{\it
% F{\hskip0,5pt}r\,}}_{\!\raise-1pt\hbox{$ \scriptstyle \ell $}}\big( x_{(1)} \big) \, ,
%   \widetilde{\hbox{\it F{\hskip0,5pt}r\,}}_{\!\raise-1pt\hbox{$ \scriptstyle \ell $}}
% \big( y_{(1)} \big) \!\Big) \,
%   \widetilde{\hbox{\it F{\hskip0,5pt}r\,}}_{\!
% \raise-1pt\hbox{$ \scriptstyle \ell $}}\big( x_{(2)} \big) \,
%   \widetilde{\hbox{\it F{\hskip0,5pt}r\,}}_{\!\raise-1pt\hbox{$ \scriptstyle \ell $}}
% \big( y_{(2)} \big) \, \sigma_\eps^{-1}\Big( \widetilde{\hbox{\it
% F{\hskip0,5pt}r\,}}_{\!\raise-1pt\hbox{$ \scriptstyle \ell $}}\big( x_{(3)} \big) \, ,
% \widetilde{\hbox{\it F{\hskip0,5pt}r\,}}_{\!\raise-1pt\hbox{$ \scriptstyle \ell $}}
% \big( y_{(3)} \big) \!\Big)  \\
%%%%%
%
  &  \; = \;  \sigma_\eps\Big( {\widetilde{\hbox{\it F{\hskip0,5pt}r\,}}_{\!\raise-1pt\hbox{$ \scriptstyle \ell $}}(x)}_{(1)} \, ,
  {\widetilde{\hbox{\it F{\hskip0,5pt}r\,}}_{\!\raise-1pt\hbox{$ \scriptstyle \ell $}}(y)}_{(1)} \Big) \,
  {\widetilde{\hbox{\it F{\hskip0,5pt}r\,}}_{\!\raise-1pt\hbox{$ \scriptstyle \ell $}}(x)}_{(2)} \,
{\widetilde{\hbox{\it F{\hskip0,5pt}r\,}}_{\!\raise-1pt\hbox{$ \scriptstyle \ell $}}(y)}_{(2)} \,
\sigma_\eps^{-1}\Big( {\widetilde{\hbox{\it F{\hskip0,5pt}r\,}}_{\!\raise-1pt\hbox{$ \scriptstyle \ell $}}(x)}_{(3)} \,
,{\widetilde{\hbox{\it F{\hskip0,5pt}r\,}}_{\!\raise-1pt\hbox{$ \scriptstyle \ell $}}(y)}_{(3)} \Big)  \; =  \\
  &  \; = \;  \widetilde{\hbox{\it F{\hskip0,5pt}r\,}}_{\!\raise-1pt\hbox{$ \scriptstyle \ell $}}(x) \cdot_{\sigma_\eps} \,
  \widetilde{\hbox{\it F{\hskip0,5pt}r\,}}_{\!\raise-1pt\hbox{$ \scriptstyle \ell $}}(y)
\end{align*}
thus the proof is completed.
\epf

\medskip

\subsection{Small multiparameter quantum groups}  \label{small-MpQG's}  \
 \vskip7pt
   In the study of uniparameter quantum groups, a relevant role is played by the so called ``small quantum groups''.  These are usually introduced as  {\sl Hopf subalgebras of\/}  the  {\sl restricted\/}  quantum groups at roots of unity; nonetheless, they can also be realized as  {\sl Hopf algebra quotients of\/}  the  {\sl unrestricted\/}  quantum groups at roots of unity.  In this subsection we extend their construction to the multiparameter context.

\vskip9pt

\begin{free text}  \label{small-MpQG's_HAT-version}
 {\bf Small MpQG's: the ``restricted realization''.}  Let  $ \bq $  be a multiparameter of integral type, hence possibly of strongly integral type.  Correspondingly, we consider the restricted MpQG's  $ \Uhatdotqeps $  and  $ \Uhatqeps $  at a root of unity  $ \varepsilon \, $,  like in  \S \ref{spec-eps}.  Inside them, we consider the following subalgebras, defined by generating sets:
 \vskip-5pt
\begin{equation}  \label{def-small_MpQGs-dot}
  \uhatdotqeps  \, = \,  \uhatdotqepsgd  \; := \;
\bigg\langle\, F_i^{\,(n)} \, , \, L_i^{\pm 1} \, , \,
{\bigg( {L_i \atop n} \bigg)}_{\!\!\varepsilon} \, , \, K_i^{\pm 1} \, , \,
{\bigg( {K_i \atop n} \bigg)}_{\!\!\varepsilon} \, , \, E_i^{\,(n)}
\,\bigg\rangle_{\! i \in I}^{\! 0 \leq n \lneqq \ell}
\end{equation}
as a  $ \Rbqeps $--subalgebra  of  $ \, \Uhatdotqeps = \Uhatdotqepsgd \, $,  when  $ \bq $  is integral, and
\begin{equation}  \label{def-small_MpQGs}
  \uhatqeps  \, = \,  \uhatqepsgd  \; := \;
\bigg\langle\, F_i^{\,(n)} \, , \, L_i^{\pm 1} \, , \,
{\bigg( {L_i \atop n} \bigg)}_{\!\!\varepsilon_i} \, , \, K_i^{\pm 1} \, , \,
{\bigg( {K_i \atop n} \bigg)}_{\!\!\varepsilon_i} \, , \, E_i^{\,(n)}
\,\bigg\rangle_{\! i \in I}^{\! 0 \leq n \lneqq \ell}
\end{equation}
 as a  $ \Rbqeps $--subalgebra  of  $ \, \Uhatqeps = \Uhatqepsgd \, $,  when  $ \bq $  is  {\sl strongly\/} integral.  Similarly, one defines  $ \uhatdotqeps^{\,\raise-3pt\hbox{$ \scriptstyle \,\pm $}} \, $,
$ \uhatdotqeps^{\,\raise-3pt\hbox{$ \scriptstyle \pm,0 $}} \, $,  $ \uhatdotqeps^{\,\raise-3pt\hbox{$ \scriptstyle \,0 $}} \, $,  $ \uhatdotqeps^{\,\raise-3pt\hbox{$ \scriptstyle \,\leq $}} $  and  $ \uhatdotqeps^{\,\raise-3pt\hbox{$ \scriptstyle \,\geq $}} $  inside  $ \uhatdotqeps \, $,  and similarly
$ \uhatqeps^{\,\raise-3pt\hbox{$ \scriptstyle \,\pm $}} \, $,  $ \uhatqeps^{\,\raise-3pt\hbox{$ \scriptstyle \pm,0 $}} \, $,  etc., inside  $ \uhatqeps \, $,  just mimicking  Definition \ref{def:int-form_hat-MpQG}  {\sl but\/}  working inside  $ \Uhatdot_\bq $  or  $ \Uhat_\bq \, $,  respectively,  {\sl and\/}  imposing the restriction  ``$ \, n \lneqq \ell \, $''  everywhere.  All these objects will be called  {\sl ``restricted small multiparameter quantum (sub)groups\/''}.
                                                                 \par
   Note that for  $ \, \bq = \check{\bq} \, $  the canonical multiparameter, the small quantum group  $ \uhatqcheckeps $  is a quantum double version of the one-parameter small quantum group by Lusztig.

\vskip9pt

   Our first result is a structural one:

\vskip11pt

\begin{theorem}  \label{thm: uhat(dot)qeps_Hopf-subalg_&_presentation}
 For any\/  $ \bq $  of integral type,  $ \uhatdotqeps $  is a Hopf\,
$ \RbqBeps $--subalgebra
  of  $ \, \Uhatdotqeps = \Uhatdotqepsgd \, $.  In addition, if\/  $ \bq $  of strongly integral type, then  $ \, \uhatqeps $  is a Hopf\,
$ \RbqBeps $--subalgebra
  of  $ \, \Uhatqeps \! = \Uhatqepsgd \, $.  Moreover,  $ \uhatdotqeps $  admits a presentation by generators
  and relations that is the same as in  Theorem \ref{thm:pres_Uhatgdq_gens-rels}  (with  $ q $ specialized to  $ \varepsilon $)  but for the bound on generators   --- i.e., they must have  $ \, 0 \leq n \lneqq \ell \, $  as in  \eqref{def-small_MpQGs-dot}
 ---   and for the additional relations
\begin{equation}  \label{nilpotency-in-small-MpQGs}
   \quad   X_i^{(n)} X_i^{(m)} = \, 0   \;\; ,  \!\quad   {\bigg(\! {{M \,; \, c} \atop n} \bigg)}_{\!\!\varepsilon}
   {\bigg(\! {{M \,; \, c-n} \atop m} \!\bigg)}_{\!\!\varepsilon} = \, 0   \qquad \hfill  \forall \;\; n, m \lneqq \ell \, : \, n+m \geq \ell
\end{equation}
for all  $ \, X \in \{F,E\} \, $,  $ \, M \in \{K,L\} \, $,  $ \, c \in \ZZ \, $.
Similar statements hold true for all the other restricted small MpQG's, namely
 $ \; \uhatdotqeps^{\,\raise-3pt\hbox{$ \scriptstyle \,\pm $}} \, $,
 $ \uhatdotqeps^{\,\raise-3pt\hbox{$ \scriptstyle \pm,0 $}} \, $,
 $ \uhatdotqeps^{\,\raise-3pt\hbox{$ \scriptstyle \,0 $}} \, $,
 $ \uhatdotqeps^{\,\raise-3pt\hbox{$ \scriptstyle \,\leq $}} \, $  and
 $ \; \uhatdotqeps^{\,\raise-3pt\hbox{$ \scriptstyle \,\geq $}} \, $,
and   --- in the  {\sl strongly}  integral case ---
 $ \, \uhatqeps \, $,
 $ \uhatqeps^{\,\raise-3pt\hbox{$ \scriptstyle \,\pm $}} \, $,
 $ \uhatqeps^{\,\raise-3pt\hbox{$ \scriptstyle \pm,0 $}} \, $,
 $ \uhatqeps^{\,\raise-3pt\hbox{$ \scriptstyle \,0 $}} \, $,
 $ \uhatqeps^{\,\raise-3pt\hbox{$ \scriptstyle \,\leq $}} \, $  and
 $ \; \uhatqeps^{\,\raise-3pt\hbox{$ \scriptstyle \,\geq $}} \, $.
\end{theorem}

\pf
 The claim follows from the very definitions together with  Theorem \ref{thm:pres_Uhatgdq_gens-rels}
 --- noting in particular that all relations between generators given there do ``fit properly'' with the bound  $ \, n \lneqq \ell \, $
 on generators of the small MpQG.  In particular, the additional relations in  \eqref{nilpotency-in-small-MpQGs}
 are direct consequence of the relations
  $$  X_i^{(n)} X_i^{(m)} = \, {\bigg( {{n+m} \atop n} \bigg)}_{\!\!q_{ii}} \, X_i^{(n+m)}   \;\; ,
   \quad   {\bigg(\! {{M \,; \, c} \atop n} \bigg)}_{\!\!q} {\bigg(\! {{M \,; \, c-n} \atop m} \!\bigg)}_{\!\!q} = {{n+m} \choose n}_{\!q} \,
   {\bigg( {{M \,; \, c} \atop {n+m}} \bigg)}_{\!\!q}  $$
 (for all  $ \, X \in \{F,E\} \, $,  $ \, M \in \{K,L\} \, $  and  $ \, c \in \ZZ \, $)  holding true in our restricted MpQG's for every  $ \, n, m \in \NN \, $,  that for  $ \, n, m \lneqq \ell \, $  such that  $ \, n+m \geq \ell \, $  yield  \eqref{nilpotency-in-small-MpQGs}  because then  $ \, \displaystyle{{n+m} \choose n}_{\!\!q_{ii}} \! = 0 \, $  and  $ \, \displaystyle{{{n+m} \choose n}_{\!q} \! = 0} \, $  for $ \, q = \varepsilon \, $.
 \vskip3pt
   Similarly, the result about the Hopf structure follows from the explicit formulas for the coalgebra structure of  $ \, \Uhatdotqgd \, $  or  $ \, \Uhatqgd \, $
%%%%%
% --- hence in turn of  $ \, \Uhatdotqeps \, $  or  $ \, \Uhatqeps \, $  ---
%%%%%
 coming from  Lemma \ref{lemma:commut_q-div-pows}.
\epf

\vskip7pt

   Our second result yields triangular decompositions for restricted small MpQG's:
%%%%%
 % \eject
%%%%%
\vskip13pt
%%%%%
%%%%%

\begin{prop}  \label{prop:triang-decomps_uhat}
 {\sl (triangular decompositions for restricted small MpQG's)}
 \vskip3pt
 The multiplication in  $ \, \uhatdotqeps $  provides
$ \, \RbqBeps $--module
  isomorphisms
  $$  \displaylines{
   \uhatdotqeps^{\,\raise-3pt\hbox{$ \scriptstyle - $}} \mathop{\otimes}_\Rbqeps \!\! \uhatdotqeps^{\,\raise-3pt\hbox{$ \scriptstyle \,0 $}} \,
   \cong \; \uhatdotqeps^{\,\raise-3pt\hbox{$ \scriptstyle \,\leq $}} \,
  \cong \; \uhatdotqeps^{\,\raise-3pt\hbox{$ \scriptstyle \,0 $}} \mathop{\otimes}_\Rbqeps \!\! \uhatdotqeps^{\,\raise-3pt\hbox{$ \scriptstyle - $}}  \;\; ,
 \quad  \uhatdotqeps^{\,\raise-3pt\hbox{$ \scriptstyle + $}} \mathop{\otimes}_\Rbqeps \!\! \uhatdotqeps^{\,\raise-3pt\hbox{$ \scriptstyle \,0 $}} \,
 \cong \; \uhatdotqeps^{\,\raise-3pt\hbox{$ \scriptstyle \,\geq $}} \,
  \cong \; \uhatdotqeps^{\,\raise-3pt\hbox{$ \scriptstyle \,0 $}} \mathop{\otimes}_\Rbqeps \!\! \uhatdotqeps^{\,\raise-3pt\hbox{$ \scriptstyle + $}}  \cr
   \uhatdotqeps^{\,\raise-3pt\hbox{$ \scriptstyle +,0 $}} \mathop{\otimes}_\Rbqeps \!\! \uhatdotqeps^{\,\raise-3pt\hbox{$ \scriptstyle -,0 $}}  \,
   \cong \;  \uhatdotqeps^{\,\raise-3pt\hbox{$ \scriptstyle \,0 $}}  \,
  \cong \;  \uhatdotqeps^{\,\raise-3pt\hbox{$ \scriptstyle -,0 $}} \mathop{\otimes}_\Rbqeps \!\! \uhatdotqeps^{\,\raise-3pt\hbox{$ \scriptstyle +,0 $}}  \;\; ,
 \quad  \uhatdotqeps^{\,\raise-3pt\hbox{$ \scriptstyle \,\leq $}} \mathop{\otimes}_\Rbqeps \! \uhatdotqeps^{\,\raise-3pt\hbox{$ \scriptstyle \,\geq $}}  \,
  \cong \;  \uhatdotqeps  \, \cong \;  \uhatdotqeps^{\,\raise-3pt\hbox{$ \scriptstyle \,\geq $}} \mathop{\otimes}_\Rbqeps \!
  \uhatdotqeps^{\,\raise-3pt\hbox{$ \scriptstyle \,\leq $}}  \cr
   \uhatdotqeps^{\,\raise-3pt\hbox{$ \scriptstyle \,+ $}} \mathop{\otimes}_\Rbqeps \!
   \uhatdotqeps^{\,\raise-3pt\hbox{$ \scriptstyle \,0 $}} \mathop{\otimes}_\Rbqeps \! \uhatdotqeps^{\,\raise-3pt\hbox{$ \scriptstyle \,- $}}  \;
  \cong \;\,  \uhatdotqeps  \; \cong \;\,  \uhatdotqeps^{\,\raise-3pt\hbox{$ \scriptstyle \,- $}} \mathop{\otimes}_\Rbqeps \!
  \uhatdotqeps^{\,\raise-3pt\hbox{$ \scriptstyle \,0 $}} \mathop{\otimes}_\Rbqeps \! \uhatdotqeps^{\,\raise-3pt\hbox{$ \scriptstyle \,+ $}}  }  $$
 and similarly with  ``$ \; \dot{\widehat{\mathfrak{u}}} $''  replaced by  ``$ \; \widehat{\mathfrak{u}} $''  if  $ \, \bq $  is  {\sl strongly}  integral.
\end{prop}

\pf
 This is proved like for restricted MpQG's: one observes that the presentation of (restricted) small MpQG's given in  Theorem \ref{thm: uhat(dot)qeps_Hopf-subalg_&_presentation}  above presents the same special features that were exploited for the proof of  Proposition \ref{prop:triang-decomps_Uhat},  so the same arguments apply again.  A quicker argument is the following: the isomorphisms of  Proposition \ref{prop:triang-decomps_Uhat}  restricts to maps for small quantum groups which are linear isomorphisms by  Theorem \ref{thm: uhat(dot)qeps_Hopf-subalg_&_presentation}.
\epf

\vskip7pt

   The third result is a PBW-like theorem for these restricted small MpQG's:

\vskip11pt

\begin{theorem}  \label{thm: PBW_small-uhat}
 {\sl (PBW theorem for restricted small MpQG's)}
 \vskip3pt
 Every restricted small MpQG is a free
$ \RbqBeps $--module
  with
$ \RbqBeps $--basis
  the subset of a PBW basis   --- as given in Theorem \ref{thm:PBW_hat-MpQG}  ---
  of the corresponding specialized restricted MpQG made by those PBW-like monomials in
  which the degree of each factor is less than $ \ell \, $.  For instance,  $ \uhatdotqeps \, $  has
$ \RbqBeps $--basis
  $$  \bigg\{\; {\textstyle \prod\limits_{k=N}^1} F_{\beta^k}^{\,(f_k)} \, {\textstyle \prod\limits_{j \in I}}
  \, {\bigg( {L_j \atop l_j} \bigg)}_{\!\!q} {L_j}^{- \lfloor l_j/2 \rfloor} \, {\textstyle \prod\limits_{i \in I}} \,
  {\bigg( {G_i \atop g_i} \bigg)}_{\!\!q_{ii}} \! {G_i}^{- \lfloor g_i/2 \rfloor} \, {\textstyle \prod\limits_{h=1}^N}
  E_{\beta^h}^{\,(e_h)} \;\bigg|\; 0 \leq f_k, l_j, g_i, e_h < \ell \;\bigg\}  $$
and similarly holds for  $ \; \uhatdotqeps^{\,\raise-3pt\hbox{$ \scriptstyle \pm $}} \, $,
 $ \uhatdotqeps^{\,\raise-3pt\hbox{$ \scriptstyle \pm,0 $}} \, $,
 $ \uhatdotqeps^{\,\raise-3pt\hbox{$ \scriptstyle 0 $}} \, $,
 $ \uhatdotqeps^{\,\raise-3pt\hbox{$ \scriptstyle \leq $}} \, $  and
 $ \; \uhatdotqeps^{\,\raise-3pt\hbox{$ \scriptstyle \geq $}} \, $,
as well as   --- in the  {\sl strongly}  integral case ---   for
 $ \; \uhatqeps \, $,
 $ \uhatqeps^{\,\raise-3pt\hbox{$ \scriptstyle \pm $}} \, $,
 $ \uhatqeps^{\,\raise-3pt\hbox{$ \scriptstyle \pm,0 $}} \, $,
 $ \uhatqeps^{\,\raise-3pt\hbox{$ \scriptstyle 0 $}} \, $,
 $ \uhatqeps^{\,\raise-3pt\hbox{$ \scriptstyle \leq $}} \, $  and
 $ \; \uhatqeps^{\,\raise-3pt\hbox{$ \scriptstyle \geq $}} \, $.
\end{theorem}

\pf
 First we discuss the case of  $ \, \uhatdotqeps^{\,\raise-3pt\hbox{$ \scriptstyle + $}} \, $,  whose ``candidate''
$ \RbqBeps $--basis
  is the set of ``truncated'' (ordered) PBW monomials
  $ \,\; B_\bq^+ := \Big\{ \prod_{h=1}^N \! E_{\beta^h}^{\,(e_h)} \,\Big|\; 0 \leq e_1 , \dots , e_N < \ell \,\Big\} \; $.
                                                  \par
   In the canonical case  $ \, \bq = \check{\bq} \, $  the required property (i.e.,  $ B_\bq^+ $  is an
$ \RbqBeps $--basis
  of  $ \uhatdotqeps^{\,\raise-3pt\hbox{$ \scriptstyle + $}} \, $)  is proved by Lusztig (cf.\ \cite{Lu}  and references therein).
  For general  $ \bq \, $,  we deduce the claim from the canonical case, arguing like in the proof of  Theorem \ref{thm:PBW_hat-MpQG}{\it (a)}.
                                                  \par
   Let us consider a quantum root vector  $ E_\beta $  for  $ \, \beta \in \Phi^+ \, $  in  $ \uhatdotqeps^{\,\raise-3pt\hbox{$ \scriptstyle + $}} \, $,
   coming (through specialization) from the same name quantum root vector in  $ \Uhatdot_\bq^{\,\raise-5pt\hbox{$ \scriptstyle + $}} \, $.
   We want to prove that  $ \, E_\beta^{(n)} \! \in \uhatdotqeps^{\,\raise-3pt\hbox{$ \scriptstyle + $}} \, $  for all  $ \, 0 \leq n < \ell \, $:
   indeed, once we have  $ \, E_\beta^{(n)} \! \in \uhatdotqeps^{\,\raise-3pt\hbox{$ \scriptstyle + $}} \, $  for all  $ \, \beta \in \Phi^+ $
   and  $ \, 0 \leq n < \ell \, $  we argue that all of  $ \, \mathcal{S} := \text{\sl Span}_{\,\Rbqeps}\!\big(B_\bq^+\big) \, $  is included in
   $ \uhatdotqeps^{\,\raise-3pt\hbox{$ \scriptstyle + $}} \, $.
                                                  \par
   Let us resume notation as in  \S  \ref{subsubsec:comp-formulas}.  As we said, the claim being true in the canonical case implies
   $ \, E_\beta^{\,\check{\cdot}\,(n)} \! \in {\dot{\widehat{\mathfrak{u}}}}_{\,\check{\bq}\hskip1pt,\hskip1pt\varepsilon}^{\,\raise-3pt\hbox{$ \scriptstyle + $}} \, $,
   so that $ E_\beta^{\,\check{\cdot}\,(n)} $  can be written as a non-commutative polynomial   --- with coefficients in
$ \RbqBeps $
  ---   in the  $ E_i^{\,\check{\cdot}\,(c)} $'s  with  $ \, i \in I \, $  and  $ \, c < \ell \, $,  say
\begin{equation}  \label{eq:E_beta-divpow=polyn}
   E_\beta^{\,\check{\cdot}\,(n)}  \, = \;  P\big( \big\{ E_i^{\,\check{\cdot}\,(c)} \big\}_{i \in I , c < \ell} \,\big)
\end{equation}
 Now, the formulas in  \S  \ref{subsubsec:comp-formulas}  tell us that  $ \; E_\beta^{\,*(n)} \! = {\big( \varepsilon^{1/2} \big)}^{z_\beta} \,
 E_\beta^{\,\check{\cdot}\,(n)} \; $  and
\begin{equation}  \label{eq: PBW-mon-E-q vs PBW-mon-E-qcheck}
   E_{i_1}^{\,*(s_1)} * E_{i_2}^{\,*(s_2)} * \cdots * E_{i_k}^{\,*(s_k)}  \, = \;  {\big( \varepsilon^{1/2} \big)}^{z_{\underline{i},\underline{s}}} \,
   E_{i_1}^{\,\check{\cdot}\,(s_1)} \, \check{\cdot} \, E_{i_2}^{\,\check{\cdot}\,(s_2)} \, \check{\cdot} \, \cdots \, \check{\cdot} \, E_{i_k}^{\,\check{\cdot}\,(s_k)}
\end{equation}
 for some  $ \, z_\beta , z_{\underline{i},\underline{s}} \in \ZZ \, $   --- with  $ \, \beta \in \Phi^+ \, $,  $ \, k \in \NN \, $,  $ \, \underline{i} := \big( i_1 , i_2 , \dots , i_k \big) \in I^k \, $  and  $ \, \underline{s} := \big( s_1 , s_2 , \dots , s_k \big) \in \NN^k \, $  ---   where  $ \, \varepsilon^{1/2} $  arises as specialization of  $ q^{1/2} $  but identifies with  $ \, \varepsilon^{(\ell + 1)/2} \in
\RbqBeps \, $.  These identities together with  \eqref{eq:E_beta-divpow=polyn}  lead us in turn
to write
  $$  E_\beta^{\,*(n)}  \, = \;  P_\bullet\big( \big\{ E_i^{\,*(c)} \big\}_{i \in I , c < \ell} \,\big)  $$
where  $ P_\bullet $  is again a non-commutative polynomial in the  $ E_i^{\,*(c)} $'s  with coefficients in
$ \RbqBeps \, $,  so that  $ \, E_\beta^{\,*(n)} \in \uhatdotqeps^{\,\raise-3pt\hbox{$ \scriptstyle + $}} \, $,  q.e.d.
 \vskip4pt
   We have seen above that  $ \; \mathcal{S}_\bq := \text{\sl Span}_{\,\Rbqeps}\!\big(B_\bq^+\big) \, \subseteq \, \uhatdotqeps^{\,\raise-3pt\hbox{$ \scriptstyle + $}} \; $,  \,now we prove the converse.  First of all, by construction  $ \uhatdotqeps^{\,\raise-3pt\hbox{$ \scriptstyle + $}} $  is spanned over  $ \Rbqeps $  by monomials in the  $ E_i^{(n)} $'s  of the form  $ \; E_{\underline{i}}^{\,*(\underline{s})} := \, E_{i_1}^{\,*(s_1)} * E_{i_2}^{\,*(s_2)} * \cdots * E_{i_k}^{\,*(s_k)} \; $  that can also be re-written as
   $ \; E_{\underline{i}}^{\,*(\underline{s})} = \; {\big( \varepsilon^{1/2} \big)}^{z_{\underline{i},\underline{s}}} \, E_{i_1}^{\,\check{\cdot}\,(s_1)} \,
   \check{\cdot} \, E_{i_2}^{\,\check{\cdot}\,(s_2)} \, \check{\cdot} \, \cdots \, \check{\cdot} \, E_{i_k}^{\,\check{\cdot}\,(s_k)} =:
   \, {\big( \varepsilon^{1/2} \big)}^{z_{\underline{i},\underline{s}}} \, E_{\underline{i}}^{\,\check{\cdot}\,(\underline{s})} \; $
   --- with notation as above; we aim to prove that each such  $ E_{\underline{i}}^{\,*(\underline{s})} $  belongs to  $ \mathcal{S}_\bq \, $,  as this will then entail at once that  $ \, \uhatdotqeps^{\,\raise-3pt\hbox{$ \scriptstyle + $}} \subseteq \mathcal{S}_\bq \, $.  The claim is true in the canonical case, so  $ \, E_{\underline{i}}^{\,\check{\cdot}\,(\underline{s})} \in \uhatdotqeps^{\,\raise-3pt\hbox{$ \scriptstyle + $}} = \mathcal{S}_{\check{\bq}} \, $  hence  $ E_{\underline{i}}^{\,\check{\cdot}\,(\underline{s})} $  expands as  $ \,\; E_{\underline{i}}^{\,\check{\cdot}\,(\underline{s})} = \, \sum_{E_{\underline{\alpha}}^{\,\check{\cdot}\,(\underline{c})} \in B_{\check{\bq}}^+} \kappa_{\underline{c}} \, E_{\underline{\alpha}}^{\,\check{\cdot}\,(\underline{c})} \;\, $  for suitable  $ \, \kappa_{\underline{c}} \in \mathcal{R}_{\,\check{\bq}\hskip1pt,\hskip1pt{}\varepsilon} = \RbqBeps \, $.
Using  \eqref{eq: PBW-mon-E-q vs PBW-mon-E-qcheck}  again we get  $ \,\; E_{\underline{i}}^{\,*(\underline{s})} = \, \sum_{E_{\underline{\alpha}}^{\,*(\underline{c})} \in B_\bq^+} {\big( \varepsilon^{1/2} \big)}^{z_{\underline{c}}} \,
\kappa_{\underline{c}} \, E_{\underline{\alpha}}^{\,*(\underline{c})} \;\, $  for suitable  $ \, z_{\underline{c}} \in \ZZ \, $,  \,so that  $ \, E_{\underline{i}}^{\,*(\underline{s})} \in \mathcal{S}_{\check{\bq}} \, $,  \,q.e.d.
 \vskip4pt
   Just like for  $ \uhatdotqeps^{\,\raise-3pt\hbox{$ \scriptstyle + $}} \, $,  the same arguments prove the claim is true for  $ \uhatdotqeps^{\,\raise-3pt\hbox{$ \scriptstyle - $}} $  as well.
 \vskip4pt
   As to  $ \uhatdotqeps^{\,\raise-3pt\hbox{$ \scriptstyle -,0 $}} \, $,  $ \uhatdotqeps^{\,\raise-3pt\hbox{$ \scriptstyle +,0 $}} $  and  $ \uhatdotqeps^{\,\raise-3pt\hbox{$ \scriptstyle 0 $}} \, $,  the claim follows at once from the analogous PBW theorem for  $ \Uhatdotqeps^{\,\raise-5pt\hbox{$ \scriptstyle 0 $}} \, $,  together with the relations  $ \; \displaystyle{{\bigg(\! {{M \,; \, c} \atop n} \bigg)}_{\!\!\varepsilon} {\bigg(\! {{M \,; \, c-n} \atop m} \!\bigg)}_{\!\!\varepsilon} = \, 0} \; $  for all  $ \, n, m \lneqq \ell \, $  such that  $ \, n+m \geq \ell \, $  when  $ \, M \in \big\{ K_i \, , L_i \big\}_{i \in I} \, $  (cf.\  Theorem \ref{thm: uhat(dot)qeps_Hopf-subalg_&_presentation}).
 \vskip4pt
   Finally, the claim for  $ \uhatdotqeps^{\,\raise-3pt\hbox{$ \scriptstyle \geq $}} \, $,  for  $ \uhatdotqeps^{\,\raise-3pt\hbox{$ \scriptstyle \leq $}} $
   and for  $ \uhatdotqeps $  follows from the previous results together with triangular decompositions
   (cf.\  Proposition \ref{prop:triang-decomps_uhat}).
 \vskip4pt
   The cases where  ``$ \, \dot{\widehat{\mathfrak{u}}} \, $''  is replaced by  ``$ \, \widehat{\mathfrak{u}} \, $''  are treated similarly.
\epf

\vskip5pt

   {\sl \dbend \  $ \underline{\text{Note}} $:  from now on, for the rest of the present discussion of restricted small MpQG's, we extend our ground ring from  $ \, \RbqBeps \! = \ZZ[\varepsilon] \, $  to  $ \, \QQ_\varepsilon \, $},  the latter being the  $ \ell $--th  cyclotomic field over  $ \QQ $   --- i.e., the field extension of  $ \QQ $  generated by a primitive  $ \ell $--th  root of unity  $ \varepsilon \, $.  Thus,  {\sl all our MpQG's at a root of unity will be considered}   --- via scalar extension from  $ \RbqBeps $  to  $ \, \QQ_\varepsilon \, $  ---   as Hopf algebras  {\sl defined over\/  $ \QQ_\varepsilon \, $}.

\vskip13pt

   A first, elementary result follows easily from definitions:

\vskip13pt

\begin{prop}  \label{prop: uhatdotqeps = uhatqeps}
 Assume that\/  $ \bq $  is of  {\sl strongly}  integral type.  Then we have
 \vskip3pt
   (a)\,  $ \; \Uhatdotqeps = \Uhatqeps \;\, $,  $ \,\; \Uhatdotqeps^{\,\raise-7pt\hbox{$ \scriptstyle \pm $}} =
   \Uhatqeps^{\,\pm} \;\, $,  $ \,\; \Uhatdotqeps^{\,\raise-7pt\hbox{$ \scriptstyle \geq $}} = \Uhatqeps^{\,\geq} \;\, $,
   $ \,\; \Uhatdotqeps^{\,\raise-7pt\hbox{$ \scriptstyle \leq $}} = \Uhatqeps^{\,\leq} \;\, $,
   $ \,\; \Uhatdotqeps^{\,\raise-7pt\hbox{$ \scriptstyle 0 $}} = \Uhatqeps^{\,0} \;\, $  and
   $ \,\; \Uhatdotqeps^{\,\raise-7pt\hbox{$ \scriptstyle \pm , 0 $}} = \Uhatqeps^{\,\pm , 0} \;\, $  \,via natural identifications;
 \vskip3pt
   (b)\,  $ \; \uhatdotqeps = \uhatqeps \; $,  \,and both algebras are generated by  $ \; {\big\{ E_i \, , L_i^{\pm 1} , K_i^{\pm 1} , F_i \big\}}_{i \in I} \; $;
 \vskip3pt
   (c)\,  $ \; \uhatdotqeps^{\,\raise-3pt\hbox{$ \scriptstyle \pm $}} = \uhatqeps^{\,\raise-3pt\hbox{$ \scriptstyle \pm $}} \; $,  \,
   and both algebras are generated respectively by  $ \; {\big\{ E_i \big\}}_{i \in I} \; $   --- for the ``$ + $''  case ---  and by
   $ \; {\big\{ F_i \big\}}_{i \in I} \; $   --- for the ``$ - $''  case;
 \vskip3pt
   (d)\,  $ \; \uhatdotqeps^{\,\raise-3pt\hbox{$ \scriptstyle \pm,0 $}} = \uhatqeps^{\,\raise-3pt\hbox{$ \scriptstyle \pm,0 $}} \; $,  \,
   and both algebras are generated respectively by  $ \; {\big\{ K_i^{\pm 1} \big\}}_{i \in I} \; $   --- for the ``$ + $''  case ---
   and by  $ \; {\big\{ L_i^{\pm 1} \big\}}_{i \in I} \; $   --- for the ``$ - $''  case;
 \vskip3pt
   (e)\,  $ \; \uhatdotqeps^{\,\raise-3pt\hbox{$ \scriptstyle 0 $}} = \uhatqeps^{\,\raise-3pt\hbox{$ \scriptstyle 0 $}} \; $,  \,
   resp.\  $ \; \uhatdotqeps^{\,\raise-3pt\hbox{$ \scriptstyle \leq $}} = \uhatqeps^{\,\raise-3pt\hbox{$ \scriptstyle \leq $}} \; $,  \,
   resp.\  $ \; \uhatdotqeps^{\,\raise-3pt\hbox{$ \scriptstyle \geq $}} = \uhatqeps^{\,\raise-3pt\hbox{$ \scriptstyle \geq $}} \; $,  \,
   and both algebras are generated by  $ \; {\big\{ K_i^{\pm 1} , L_i^{\pm 1} \big\}}_{i \in I} \; $,  \, resp.\ by
   $ \; {\big\{ L_i^{\pm 1} , F_i \big\}}_{i \in I} \; $,  \, resp.\ by  $ \; {\big\{ E_i \, , K_i^{\pm 1} \big\}}_{i \in I} \; $.
\end{prop}

\pf
 As to claim  {\it (a)},  by construction it is enough to show that  $ \; \Uhatdotqeps^{\,\raise-7pt\hbox{$ \scriptstyle 0 $}} = \Uhatqeps^{\,0} \; $
 or more precisely  $ \; \Uhatdotqeps^{\,\raise-7pt\hbox{$ \scriptstyle \pm , 0 $}} = \Uhatqeps^{\,\pm , 0} \; $.  In turn, the latter identity follows
 from definitions together with the following formal identity among quantum binomial coefficients
  $$  {X \choose n}_{\!\varepsilon}  \; = \;\,  {\textstyle \prod\limits_{s=1}^n} {\big( d_i \big)}_{\varepsilon^s} \, {X \choose n}_{\!\varepsilon_i}  $$
which proves that the  $ \varepsilon $--binomial  coefficients and the  $ \varepsilon_i $--binomial
coefficients generate over  $ \QQ_\varepsilon $  the same algebra, since
$ \, {\textstyle \prod\limits_{s=1}^n} {\big( d_i \big)}_{\varepsilon^s} \, $  is invertible in the field  $ \QQ_{\,\varepsilon} \, $.
 \vskip5pt
   As to the remaining claims, everything follows again from a simple remark.  Namely, definitions give
  $$  {\textstyle \prod\limits_{r=1}^n} \big( X_i \, \varepsilon_i^{1-r} \! - 1 \big) \, =
  \, {\textstyle \prod\limits_{r=1}^n} \big( \varepsilon_i^r - 1 \big) \, {\bigg( {X_i \atop n} \bigg)}_{\!\!\varepsilon_i} \;\; ,  \quad
      {\textstyle \prod\limits_{r=1}^n} \big( X_i \, \varepsilon^{1-r} - 1 \big) \, =
      \, {\textstyle \prod\limits_{r=1}^n} \big( \varepsilon^r \! - 1 \big) \, {\bigg( {X_i \atop n} \bigg)}_{\!\!\varepsilon}  $$
for all  $ \, X \in \{ K , L \} \, $,  $ \, i \in I \, $  and  $ \, 0 \leq n \leq \ell - 1 \, $,  and similarly
$ \; Z_i^{\,n} \, = \, {[n]}_{\varepsilon_i\!}! \, Z_i^{(n)} \; $  for all  $ \, Z \in \{ F , E \} \, $,  $ \, i \in I \, $
and  $ \, 0 \leq n \leq \ell - 1 \, $.  Now, the condition  $ \, n \leq \ell - 1 \, $  implies that all the coefficients
$ \; \prod\limits_{r=1}^n \big( \varepsilon_i^r - 1 \big) \; $,  $ \; \prod\limits_{r=1}^n \big( \varepsilon^r - 1 \big) \; $
and  $ \; {[n]}_{\varepsilon_i\!}! \; $  that occur above are non-zero elements in  $ \QQ_\varepsilon \, $,
whence we deduce at once our claim.
\epf
\end{free text}

\vskip7pt

\begin{rmk}
 From the PBW  Theorem \ref{thm: PBW_small-uhat}  and  Proposition \ref{prop: uhatdotqeps = uhatqeps}
 it follows that  $ \; \uhatdotqeps = \uhatqeps \; $ is a finite-dimensional $\QQ_\varepsilon$-Hopf algebra of dimension
$ \ell^{\,\dim (\liegd)} \, $.
\end{rmk}

\vskip7pt

   Next result yields a strict link (a multiparameter version of a well-known result) between small MpQG's and quantum Frobenius morphisms for restricted MpQG's; indeed, one could take it as an alternative way to introduce small MpQG's.

\vskip11pt

\begin{theorem}  \label{thm: restr-small_MpQG = Hopf-cleft}
% {\ }
%
% \vskip-1pt
%
   Let  $ \, \bq := {\big(\hskip1pt{}q_{ij}\big)}_{i,j \in I} \, $  be of integral type, let
 $ \; \Uhatdot_{\bq\hskip1pt,\hskip1pt{}\varepsilon}(\hskip0,8pt\lieg) \,{\buildrel
 {{\dot{\widehat{\text{\it F{\hskip0,5pt}r\,}}}}_{\!\raise-1pt\hbox{$ \scriptscriptstyle \ell $}}} \over
 {\relbar\joinrel\relbar\joinrel\rightarrow}}\,
U_{\QQ_\varepsilon}\!\big(\hskip0,8pt\liegdotb\hskip0,5pt\big) \; $
 be the scalar extension of the quantum Frobenius morphism of Theorem \ref{thm:qFrob-Uhat}  and finally let
 $ \; \uhatdotqeps \,{\buildrel \iota \over {\relbar\joinrel\relbar\joinrel\rightarrow}}\, \Uhatdotqeps \, $
 the natural embedding of  $ \, \uhatdotqeps $  into  $ \, \Uhatdotqeps \, $.
%%%
 Then
 \vskip-7pt
  $$  1 \relbar\joinrel\relbar\joinrel\relbar\joinrel\rightarrow\;
 \uhatdotqeps \;{\buildrel \iota \over {\relbar\joinrel\relbar\joinrel\relbar\joinrel\rightarrow}}\;
 \Uhatdot_{\bq\hskip1pt,\hskip1pt{}\varepsilon}(\hskip0,8pt\lieg) \;{\buildrel
 {{\dot{\widehat{\text{\it F{\hskip0,5pt}r\,}}}}_{\!\raise-1pt\hbox{$ \scriptscriptstyle \ell $}}}
 \over {\relbar\joinrel\relbar\joinrel\relbar\joinrel\relbar\joinrel\rightarrow}}\;
 U_{\QQ_\epsilon}\big(\hskip0,8pt\liegdotb\big)
 \,\relbar\joinrel\relbar\joinrel\relbar\joinrel\rightarrow\, 1  $$
 \vskip3pt
\noindent
 is an exact sequence of Hopf\/  $ \QQ_\epsilon$--algebras  which is cleft.
                                                         \par
   A similar statement holds true for  $ \, \uhatqeps \, $  and the scalar extension of the quantum Frobenius morphism
 $ \,\; \Uhat_{\bq\hskip1pt,\hskip1pt{}\varepsilon}(\hskip0,8pt\lieg)
 \;{\buildrel {{\widehat{\text{\it F{\hskip0,5pt}r\,}}}_{\!\raise-1pt\hbox{$ \scriptscriptstyle \ell $}}}
 \over {\relbar\joinrel\relbar\joinrel\relbar\joinrel\relbar\joinrel\relbar\joinrel\rightarrow}}\;
 U_{\QQ_\epsilon}\big(\hskip0,8pt\lieghatb\big) \; $
 when  $ \bq $  is strongly integral.
\end{theorem}

\pf
 By Theorem \ref{thm: PBW_small-uhat},  $ \Uhatdot_{\bq\hskip1pt,\hskip1pt{}\varepsilon}(\hskip0,8pt\lieg) $  is free over  $ \uhatdotqeps \, $.  So, to show that the sequence is exact it is enough to prove that  $ \, \Ker\Big( {\dot{\widehat{\text{\it F{\hskip0,5pt}r\,}}}}_{\!\raise-1pt\hbox{$ \scriptscriptstyle \ell $}} \Big) = \Uhatdot_{\bq\hskip1pt,\hskip1pt{}\varepsilon}(\hskip0,8pt\lieg) \, \uhatdotqeps^{\;\raise-3pt\hbox{$ \scriptstyle + $}} \; $.  This follows along the same lines as for the canonical case (proved in
\cite[Lemma 3.4.2]{A}),  so we skip it.
                                                           \par
   To prove that the extension is cleft, we use the well-known fact that an extension of algebras is cleft if and only if it is Galois and has a normal basis (see, e.g.,  \cite{DT}).  Since the extension is a Hopf algebra extension, it follows that it is Galois, see \cite[Remark 1.6]{Sch2}.  The normal basis property follows from  \cite[4.3]{Sch1}  since  $ \Uhatdot_{\bq\hskip1pt,\hskip1pt{}\varepsilon}(\hskip0,8pt\lieg) $  is pointed.  Indeed, by the PBW  Theorem \ref{PBW-mpqgs_q=eps}  one may define an algebra filtration  $ U_n $  of  $ \Uhatdot_{\bq\hskip1pt,\hskip1pt{}\varepsilon}(\hskip0,8pt\lieg) $  such that  $ U_0 $  is the subalgebra generated by  $ K_i^{\pm 1} $,  $ L_i^{\pm 1} $  ($ \, i \in I \, $),  and  $ E_i^{(n)} $,  $ F_i^{(n)} $,  $ \displaystyle{\bigg(\! {{M \,; \, c \,} \atop n} \!\bigg)}_{\!\!\varepsilon} \in U_n \, $  ($ \, i \in I \, $,  $ \, n \in \NN \, $).  By  Theorem \ref{thm:pres_Uhatgdq_gens-rels} and  Lemma \ref{commut_q-bin-coeff},  this is a coalgebra filtration, so the coradical of  $ \Uhatdot_{\bq\hskip1pt,\hskip1pt{}\varepsilon}(\hskip0,8pt\lieg) $  is contained in  $ U_0 \, $.  As the latter is the linear span of group-like elements, it follows that  $ \Uhatdot_{\bq\hskip1pt,\hskip1pt{}\varepsilon}(\hskip0,8pt\lieg) $  is pointed.
\epf

\vskip9pt

\begin{rmks}  \label{rmks: comments on Uhat-cleft}  {\ }
 \vskip3pt
    {\it (a)}\, The proof that the Hopf algebra extension above is cleft also follows by the proof of the canonical case given in  \cite[Lemma 3.4.3]{A}.  On the other hand, let us point out that the normal basis property means that  $ \Uhatdot_{\bq\hskip1pt,\hskip1pt{}\varepsilon}(\hskip0,8pt\lieg) $  is isomorphic to  $ \, \uhatdotqeps \otimes U_{\QQ_\epsilon}\big(\hskip0,8pt\liegdotb\big) \, $  as left  $ \uhatdotqeps $--module  and right  $ U_{\QQ_\epsilon}\big(\hskip0,8pt\liegdotb\big) $--comodule.  Hence, the MpQG at a root of unity  $ \, \Uhatdotqepsgd = \Uhatqepsgd \, $  can be seen as a ``blend'' of a restricted small MpQG, namely  $ \, \uhatdotqepsgd = \uhatqepsgd \, $,  \,and a ``classical'' geometrical object, namely  $ \; U_{\QQ_\varepsilon}\big(\hskip0,8pt\liegdotb\big) = U_{\QQ_\varepsilon}\big(\hskip0,8pt\lieghatb\big) = U_{\QQ_\varepsilon}\big(\hskip0,8pt\liegb\big) \; $.
 \vskip3pt
    {\it (b)}\,  Beside the canonical case, some variations of the quantum Frobenius homorphism are treated in the literature.  For example, Lentner  \cite{Le}  studies the quantum Frobenius map for the positive Borel algebras at small roots of unity, which are in fact Nichols algebras.  Another is in  \cite{Mc},  which provides the construction of the quantum Frobenius homomorphism for the positive part using Hall algebras.  As in the Hopf algebra case, the quantum Frobenius map is used to study exact sequences of Nichols algebras.  In  \cite{AAR2}  it is shown how Nichols algebras give rise to positive parts of semisimple Lie algebras as images of the quantum Frobenius morphism.
\end{rmks}

\vskip5pt

\begin{free text}  \label{small-MpQG's_TILDE-version}
 {\bf Small MpQG's: the ``unrestricted realization''.}  We introduce now a second type of small MpQG's, defined in terms of unrestricted MpQG's.  As in the restricted case, these are defined for a multiparameter of integral type  $ \bq \, $.  We shall eventually see that these ``unrestricted'' small MpQG's actually do coincide with the ``restricted'' ones.
 \vskip7pt
   Let  $ \bq $  be a multiparameter of integral type, hence possibly of strongly integral type.
   Let  $ \,\; \widetilde{\hbox{\it F{\hskip0,5pt}r\,}}_{\!\raise-1pt\hbox{$ \scriptstyle \ell $}} : \ZZ[\varepsilon] \otimes_{\raise-3pt\hbox{$ \scriptscriptstyle \ZZ $}} \Utildequnogd \, \cong \, \RbqBeps \otimes_{\raise-3pt\hbox{$ \scriptstyle \Rbquno $}} \!\! \Utildequnogd \, \lhook\joinrel\relbar\joinrel\relbar\joinrel\relbar\joinrel\relbar\joinrel\longrightarrow \, \Utildeqepsgd \;\; $
 be the  {\sl unrestricted quantum Frobenius morphism\/}  introduced in  Theorem \ref{thm: unrestr_q-Frob-morph}, a Hopf algebra monomorphism whose image is the central Hopf subalgebra  $ Z_0 $  of $ \Utildeqepsgd $  given in Definition \ref{def:Z_0}.  We consider the Hopf cokernel of  $ \, \widetilde{\hbox{\it F{\hskip0,5pt}r\,}}_{\!\raise-1pt\hbox{$ \scriptstyle \ell $}} \, $,  i.e.\  the quotient Hopf algebra
  $$  \utildeqeps  \; := \;  \utildeqepsgd  \; := \;  \Utildeqepsgd \Big/\, \Utildeqepsgd \, {Z_0}^+  $$
--- where  $ {Z_0}^+ $  denotes the augmentation ideal of  $ Z_0 $  --- and similarly the
cokernels of the restrictions of  $ \, \widetilde{\hbox{\it F{\hskip0,5pt}r\,}}_{\!\raise-1pt\hbox{$ \scriptstyle \ell $}} \, $
to all various relevant multiparameter quantum subgroups of  $ \, \utildeqepsgd \, $:  for instance,
$ \; {\widetilde{\mathfrak{u}}}_{\bq,\varepsilon}^{\,\geq} \, := \, \Utilde_{\bq,\varepsilon}^{\,\geq}
\Big/\, \Utilde_{\bq,\varepsilon}^{\,\geq} \, {\big( Z_0^\geq \big)}^+ \; $,  \, and so on and so forth.
We call all these objects  {\sl ``unrestricted small multiparameter quantum (sub)groups''}.
When  $ \, \bq = \check{\bq} \, $  is the canonical multiparameter, this definition coincides with the
one for the one-parameter small quantum group associated with  $ \lieg $  given in  \cite[III.6.4]{BG}.
                                                                 \par
   Since, by  Proposition \ref{prop: struct-Z_0},  $ \Utildeqepsgd$ is a free  $ \Utildequnogd $--module  of rank
   $ \ell^{\,\dim (\liegd)} $,  it follows that  $ \utildeqeps $  is a finite-dimensional Hopf algebra of dimension
   $ \ell^{\,\dim (\liegd)} \, $;  indeed, we shall show that it actually coincides with $ \, \uhatdotqeps = \uhatqeps \, $.

\vskip9pt

   As direct consequence of definitions and previous results, we find  structure results for unrestricted small MpQG's.
   The first one is about triangular decompositions:
\end{free text}
%%%

\vskip13pt

\begin{prop}  \label{prop:triang-decomps_utilde}
 {\sl (triangular decompositions for unrestricted small MpQG's)}
 \vskip3pt
 The multiplication in  $ \, \utildeqeps $  provides  $ \, \Rbqeps $--module  isomorphisms
  $$  \displaylines{
   \utildeqeps^{\,\raise-3pt\hbox{$ \scriptstyle - $}} \mathop{\otimes}_\Rbqeps \!\! \utildeqeps^{\,\raise-3pt\hbox{$ \scriptstyle \,0 $}} \, \cong \;
   \utildeqeps^{\,\raise-3pt\hbox{$ \scriptstyle \,\leq $}} \,
  \cong \; \utildeqeps^{\,\raise-3pt\hbox{$ \scriptstyle \,0 $}} \mathop{\otimes}_\Rbqeps \!\! \utildeqeps^{\,\raise-3pt\hbox{$ \scriptstyle - $}}  \;\; ,
 \quad  \utildeqeps^{\,\raise-3pt\hbox{$ \scriptstyle + $}} \mathop{\otimes}_\Rbqeps \!\! \utildeqeps^{\,\raise-3pt\hbox{$ \scriptstyle \,0 $}} \, \cong \;
 \utildeqeps^{\,\raise-3pt\hbox{$ \scriptstyle \,\geq $}} \,
  \cong \; \utildeqeps^{\,\raise-3pt\hbox{$ \scriptstyle \,0 $}} \mathop{\otimes}_\Rbqeps \!\! \utildeqeps^{\,\raise-3pt\hbox{$ \scriptstyle + $}}  \cr
   \utildeqeps^{\,\raise-3pt\hbox{$ \scriptstyle +,0 $}} \mathop{\otimes}_\Rbqeps \!\! \utildeqeps^{\,\raise-3pt\hbox{$ \scriptstyle -,0 $}}  \, \cong \;
   \utildeqeps^{\,\raise-3pt\hbox{$ \scriptstyle \,0 $}}  \,
  \cong \;  \utildeqeps^{\,\raise-3pt\hbox{$ \scriptstyle -,0 $}} \mathop{\otimes}_\Rbqeps \!\! \utildeqeps^{\,\raise-3pt\hbox{$ \scriptstyle +,0 $}}  \;\; ,
 \quad  \utildeqeps^{\,\raise-3pt\hbox{$ \scriptstyle \,\leq $}} \mathop{\otimes}_\Rbqeps \! \utildeqeps^{\,\raise-3pt\hbox{$ \scriptstyle \,\geq $}}  \,
  \cong \;  \utildeqeps  \, \cong \;  \utildeqeps^{\,\raise-3pt\hbox{$ \scriptstyle \,\geq $}}
  \mathop{\otimes}_\Rbqeps \! \utildeqeps^{\,\raise-3pt\hbox{$ \scriptstyle \,\leq $}}  \cr
   \utildeqeps^{\,\raise-3pt\hbox{$ \scriptstyle \,+ $}} \mathop{\otimes}_\Rbqeps \!
   \utildeqeps^{\,\raise-3pt\hbox{$ \scriptstyle \,0 $}} \mathop{\otimes}_\Rbqeps \! \utildeqeps^{\,\raise-3pt\hbox{$ \scriptstyle \,- $}}  \;
  \cong \;\,  \utildeqeps  \; \cong \;\,  \utildeqeps^{\,\raise-3pt\hbox{$ \scriptstyle \,- $}}
  \mathop{\otimes}_\Rbqeps \! \utildeqeps^{\,\raise-3pt\hbox{$ \scriptstyle \,0 $}} \mathop{\otimes}_\Rbqeps
  \! \utildeqeps^{\,\raise-3pt\hbox{$ \scriptstyle \,+ $}}  }  $$
\end{prop}

\pf
 This can be proved like the similar result for unrestricted MpQG's, or can be deduced from the latter: details are left to the reader.
\epf

\vskip7pt

   The second result is a PBW-like theorem for unrestricted small MpQG's:

\vskip11pt

\begin{theorem}  \label{thm: PBW_small-utilde}
 {\sl (PBW theorem for unrestricted small MpQG's)}
 \vskip3pt
 Every unrestricted small MpQG is a free
$ \Rbqeps $--module
  with
$ \Rbqeps $--basis
  made by the cosets of all PBW monomials   --- in the subset of a PBW basis (as given in
  Theorem \ref{thm:PBW_tilde-MpQG})
  of the corresponding specialized unrestricted MpQG ---   in which the degree of each factor is less than $ \ell \, $.
  For instance,  $ \utildeqeps \, $  has
$ \Rbqeps $--basis
  $$  \bigg\{\; {\textstyle \prod\limits_{k=N}^1} \Fbar_{\beta^k}^{\,f_k}
  \, {\textstyle \prod\limits_{j \in I}} \, L_j^{l_j}
  \, {\textstyle \prod\limits_{i \in I}} \, K_i^{c_i}
  \, {\textstyle \prod\limits_{h=1}^N} \Ebar_{\beta^h}^{\,e_h}
    \;\bigg|\; 0 \leq f_k, l_j, c_i, e_h < \ell \;\bigg\}  $$
and similarly holds for
 $ \; \utildeqeps^{\,\raise-3pt\hbox{$ \scriptstyle \pm $}} \, $,
 $ \utildeqeps^{\,\raise-3pt\hbox{$ \scriptstyle \pm,0 $}} \, $,
 $ \utildeqeps^{\,\raise-3pt\hbox{$ \scriptstyle 0 $}} \, $,
 $ \utildeqeps^{\,\raise-3pt\hbox{$ \scriptstyle \leq $}} \, $  and
 $ \; \utildeqeps^{\,\raise-3pt\hbox{$ \scriptstyle \geq $}} \, $.
\end{theorem}

\pf
 This follows at once from definitions and from  Proposition \ref{prop: struct-Z_0}.
\epf

\vskip11pt

   The results in  \S \ref{int-forms_mpqgs}  and  Theorem \ref{Mpqg-roots-of-1_2-coc-def}  lead us to the following theorem.

\vskip13pt

\begin{theorem}  \label{thm: utildeqeps as 2-coc deform}
% {\ }
%
% \vskip3pt
%
 The Hopf  $ \, \Rbqeps $--algebra  $ \, \utildeqeps $  is a  $ 2 $--cocycle  deformation of  $ \; \utildeqcheckeps \, $.
\end{theorem}

\pf
 Denote by  $ \check{Z}_0 $  the subalgebra of  $ \, \Utilde_{\check{\bq}\hskip1pt,\hskip1pt{}\varepsilon}(\hskip0,8pt\lieg) \, $  that defines  $ \, \utildeqcheckeps \, $.  Since $ \bq $  is of integral type,  $ Z_0 $  and  $ \check{Z}_0 $  are both central Hopf subalgebras of  $ \Utildeqepsgd $  and  $ \Utilde_{\check{\bq}\hskip1pt,\hskip1pt{}\varepsilon}(\hskip0,8pt\lieg) \, $,  respectively.  By Theorem \ref{Mpqg-roots-of-1_2-coc-def}{\it (a)},  we know that the Hopf  $ \, \Rbqeps $--algebra  $ \; \Utildeqepsgd $  is a  $ 2 $--cocycle  deformation of  $ \, \Utilde_{\check{\bq}\hskip1pt,\hskip1pt{}\varepsilon}(\hskip0,8pt\lieg) \, $.  As the  $ 2 $--cocycle giving the deformation is
   $$  \displaylines{
    \sigma_{\eps}(x,y)  \; := \;  \eps^{\;1/2}_{\mu\,\nu}   \qquad  \text{if}  \quad
         x = K_\mu  \text{\;\ or \ }  x = L_\mu \; ,  \!\quad  y = K_\nu  \text{\;\ or \ }  y = L_\nu  \cr
    \sigma_{\eps}\big( \QEqcheck \, , {\QEqcheck}^\oplus \big) \; := \; 0 \; =: \;  \sigma_{\eps}\big( {\QEqcheck}^\oplus \, , \QEqcheck \big)  }  $$
 it follows that $ \; \sigma_{\eps}{\big|}_{\Utilde_{\check{\bq},\varepsilon}(\hskip0,8pt\lieg) \,\otimes\, \check{Z}_0 + \check{Z}_0 \otimes\, \Utilde_{\check{\bq},\varepsilon}(\hskip0,8pt\lieg)} = \epsilon \otimes \epsilon \; $,  \,the trivial  $ 2 $--cocycle,  with  $ \epsilon $  the counit of  $ \Utilde_{\check{\bq},\varepsilon}(\hskip0,8pt\lieg) \, $;  \,in particular,  $ \, Z_0 = {(\check{Z}_0)}_{\sigma_\varepsilon} = \check{Z}_0 \, $  as Hopf algebras.  Finally, if we define  $ \; \bar{\sigma}_\varepsilon : \utildeqcheckeps \otimes \utildeqcheckeps \!\relbar\joinrel\longrightarrow \RbqBeps \; $  by  $ \, \bar{\sigma}_\varepsilon\big( \bar{x} , \bar{y} \big) := \sigma_\varepsilon(x,y) \, $  for  $ x $,  $ y \in \Utilde_{\check{\bq}\hskip1pt,\hskip1pt{}\varepsilon}(\hskip0,8pt\lieg) \, $,  a straightforward calculation shows that  $ \bar{\sigma}_\varepsilon $  is a  $ 2 $--cocycle  for  $ \utildeqcheckeps $ and  $ \; {\big( \utildeqcheckeps \big)}_{\sigma_{\varepsilon}} \cong \, \utildeqeps \; $.
\epf

\vskip9pt

   \dbend \  {\sl Now we extend the ground ring from  $ \, \RbqBeps = \ZZ[\varepsilon] \, $  to the cyclotomic field  $ \QQ_\varepsilon $  generated over  $ \QQ $  by an  $ \ell $--th  root of unity\/}:  all algebras then will be taken as defined over  $ \QQ_\varepsilon $  (via scalar extension), even though we keep the same notation.  In this case, we have the following structural result:

\vskip11pt

\begin{prop}  \label{prop: struct_Utildeqeps-utildeqeps}
 Let us consider  $ \Utildeqeps $  and  $ \utildeqeps \, $,  as well as their quantum subgroups, as defined over  $ \QQ_\varepsilon $  (via scalar extension).  Then we have:
 \vskip3pt
   (a)\,  $ \; \Utildeqeps $  is generated by  $ \; {\big\{\, \Ebar_i \, , L_i^{\pm 1} , K_i^{\pm 1} , \Fbar_i \big\}}_{i \in I} \; $,  and  $ \, \utildeqeps $
   is generated by the corresponding set of cosets;
 \vskip3pt
   (b)\,  $ \; \Utildeqeps^{\,\raise-3pt\hbox{$ \scriptstyle + $}} \, $  and  $ \, \Utildeqeps^{\,\raise-3pt\hbox{$ \scriptstyle - $}} \, $  are
   generated respectively by  $ \; {\big\{\, \Ebar_i \big\}}_{i \in I} \; $  and by  $ \; {\big\{\, \Fbar_i \big\}}_{i \in I} \; $,  and similarly
   $ \, \utildeqeps^{\,\raise-3pt\hbox{$ \scriptstyle + $}} $  and  $ \, \utildeqeps^{\,\raise-3pt\hbox{$ \scriptstyle - $}} \, $  are generated
   by the corresponding sets of cosets;
 \vskip3pt
   (c)\,  $ \; \Utildeqeps^{\,\raise-3pt\hbox{$ \scriptstyle +,0 $}} \, $,  $ \, \Utildeqeps^{\,\raise-3pt\hbox{$ \scriptstyle -,0 $}} \, $  and
   $ \; \Utildeqeps^{\,\raise-3pt\hbox{$ \scriptstyle 0 $}} \, $  are generated respectively by  $ \; {\big\{ K_i^{\pm 1} \big\}}_{i \in I} \; $,
   by  $ \; {\big\{ L_i^{\pm 1} \big\}}_{i \in I} \; $  and by  $ \; {\big\{ K_i^{\pm 1} , L_i^{\pm 1} \big\}}_{i \in I} \; $,  and similarly
   $ \, \utildeqeps^{\,\raise-3pt\hbox{$ \scriptstyle +,0 $}} \, $,  $ \, \utildeqeps^{\,\raise-3pt\hbox{$ \scriptstyle -,0 $}} \, $  and
   $ \; \utildeqeps^{\,\raise-3pt\hbox{$ \scriptstyle 0 $}} \, $  are generated by the corresponding sets of cosets;
 \vskip3pt
   (d)\,  $ \; \Utildeqeps^{\,\raise-3pt\hbox{$ \scriptstyle \leq $}} \, $,  \,resp.\  $ \, \Utildeqeps^{\,\raise-3pt\hbox{$ \scriptstyle \geq $}} \, $,  \,is generated by  $ \, {\big\{ L_i^{\pm 1} , \Fbar_i \big\}}_{i \in I} \; $,  \,resp.\ by  $ \, {\big\{\,  \Ebar_i \, , K_i^{\pm 1} \big\}}_{i \in I} \; $;  \,similarly  $ \, \utildeqeps^{\,\raise-3pt\hbox{$ \scriptstyle \leq $}} \, $,  \,resp.\  $ \, \utildeqeps^{\,\raise-3pt\hbox{$ \scriptstyle \geq $}} \, $,  \,is generated by the corresponding set of cosets;
 \vskip3pt
   (e)\,  in all claims (a) through (d) above, one can freely replace any  $ \Ebar_i $  or  $ \Fbar_j $  with  $ E_i $  or  $ F_j $  respectively, and still
   have a set of generators.
\end{prop}

\pf
 It is enough to prove claim  {\it (a)},  as the other are similar.  By construction,  $ \, \Utildeqeps \, $  is generated by (the specialization of) all the  $ K_i^{\pm 1} $'s,  all the  $ L_i^{\pm 1} $'s  and all the quantum root vectors  $ \Ebar_\alpha $  and  $ \Fbar_\alpha $.  Now,  $ \; \Ebar_\alpha = \big( \varepsilon_{\alpha,\alpha} - 1 \big) \, E_\alpha \; $  so the  $ \Ebar_\alpha $'s  can be replaced with the  $ E_\alpha $'s,  because  $ \, \big( \varepsilon_{\alpha,\alpha} - 1 \big) \, $  is invertible in  $ \QQ_\varepsilon \, $.  Moreover, each quantum root vector  $ E_\alpha $  can be expressed, by construction (cf.\  \S \ref{rvec-MpQG}),  as a suitable  $ q $--iterated  quantum bracket of some of the  $ E_i $'s;  as  $ \, E_i = {\big( \varepsilon_i^{\,2} - 1 \big)}^{-1} \, \Ebar_i \, $,  \,the  $ \Ebar_i $'s  alone are enough to generate all the  $ \Ebar_\alpha $'s  over  $ \QQ_\varepsilon \, $.  A similar argument works for the  $ \Fbar_\alpha $'s,  hence the claim for  $ \Utildeqeps $  follows, and that for  $ \utildeqeps $  is an obvious consequence.  Claim  {\it (e)\/}  is clear as well from the above analysis.
\epf

\vskip11pt

   By construction, the projection  $ \, \pi \, $  from  $ \Utildeqeps $  to  $ \utildeqeps $  and the scalar extension of the quantum Frobenius morphism  $ \, \widetilde{\hbox{\it F{\hskip0,5pt}r\,}}_{\!\raise-1pt\hbox{$ \scriptstyle \ell $}} \, $  match together to yield a short exact sequence of Hopf   $ \QQ_\varepsilon $--algebras.  As before, this sequence allows to reconstruct the unrestricted MpQG  $ \, \Utildeqeps \, $
%%%%%
% in the middle of that sequence from the left and right sides, namely
%%%%%
 as a  {\it cleft extension},  as the following shows:

\vskip13pt

\begin{theorem}  \label{thm: unrestr-small_MpQG = Hopf-cleft}
 \ Let  $ \, \bq := {\big(\hskip1pt{}q_{ij}\big)}_{i,j \in I} \, $  be a multiparameter of integral type.
                                                                            \par
  Let
 $ \; \Utildequnogd
  \;{\buildrel {\widetilde{F{\hskip0,7pt}r\,}_{\!\!\raise-2pt\hbox{$ \scriptscriptstyle \ell $}}} \over
  {\lhook\joinrel\relbar\joinrel\relbar\joinrel\relbar\joinrel\relbar\joinrel\longrightarrow}}\;
   \Utildeqepsgd \; $
 be the scalar extension
to  $ \QQ_\varepsilon $
  of the  unrestricted quantum Frobenius morphism
of  Theorem \ref{thm: unrestr_q-Frob-morph}  and let
 \vskip-10pt
  $$  \utildeqepsgd  \,\; := \;\,  \Utildeqepsgd \Big/\, \Utildeqepsgd \, { \Utildequnogd}^+  $$
be the quotient Hopf algebra.
%%%
 Then
\begin{equation}  \label{eq: cleft-ext_Utildeqeps}
  1 \relbar\joinrel\relbar\joinrel\relbar\joinrel\relbar\joinrel\rightarrow\;
%
%    \QQ_\varepsilon \otimes_{\raise-3pt\hbox{$ \scriptstyle \Rbquno $}} \!\!
%
   \Utildequnogd
\;{\buildrel {\widetilde{F{\hskip0,7pt}r\,}_{\!\!\raise-2pt\hbox{$ \scriptscriptstyle \ell $}}} \over
{\relbar\joinrel\relbar\joinrel\relbar\joinrel\relbar\joinrel\relbar\joinrel\rightarrow}}\;
   \Utildeqepsgd
\;{\buildrel \pi \over {\relbar\joinrel\relbar\joinrel\relbar\joinrel\relbar\joinrel\rightarrow}}\;
   \utildeqepsgd
\,\relbar\joinrel\relbar\joinrel\relbar\joinrel\relbar\joinrel\rightarrow\,
   1
\end{equation}
 \vskip3pt
\noindent
 is a central exact sequence of Hopf\/  $ \QQ_\epsilon$--algebras  which is cleft.
\end{theorem}

\pf
 By  Proposition \ref{prop: struct-Z_0}  we know that  $ \Utildeqepsgd $  is a free  $ \Utildequnogd $--module  of rank  $ \ell^{\,\dim (\liegd)} \, $.
 Since  $ \, \Utildequnogd \, $  is central and  $ \; \utildeqepsgd := \, \Utildeqepsgd \Big/\, \Utildeqepsgd \, { \Utildequnogd}^+ \, $,
 by \cite[Proposition 3.4.3]{Mo}  we have that  $ \, \Utildequnogd = \Utildeqepsgd^{\co \pi} \, $  and the sequence is exact.
 As we did before for the restricted case, to prove that the extension is cleft we show
 that it is Galois and has a normal basis.  Since the extension is a Hopf algebra extension, it follows that it is Galois, see  \cite[Remark 1.6]{Sch2}.
 The normal basis property follows from  \cite[4.3]{Sch1}  as  $ \Utildeqepsgd $  is a pointed Hopf algebra, since it is generated
 by group-like and skew-primitive elements.
\epf

\vskip11pt

\begin{rmks}  \label{rmks: comments on Utilde-cleft}
 {\it (a)}\,  By the normal basis property,  $ \Utildeqepsgd $  is isomorphic to  $ \, \Utildequnogd \otimes \utildeqepsgd \, $  as left $ \Utildequnogd $--module  and right  $ \utildeqepsgd $--comodule.  Hence, the MpQG at a root of unity  $ \, \Utildeqepsgd \, $  can be understood as a ``blend'' of a  {\sl classical\/}  geometrical object   --- namely  $ \, \Utildequnogd \, $,  which is  $ \, \Oc\big( \Gtildebstar \big) \, $  since  $ \bq $  is of integral type, see  Theorem \ref{thm: Utildequno_q-int-type}  ---   and a  {\sl quantum\/}  one   --- the unrestricted small MpQG  $ \, \utildeqepsgd \; $.
 \vskip3pt
   {\it (b)}\,  Borrowing language from geometry   ---  {\sl without claiming to be precise, by no means}  ---   the exact sequence  \eqref{eq: cleft-ext_Utildeqeps}  can be interpreted as follows:  $ \Utildeqepsgd $  defines a principal bundle of Hopf  $ \QQ_\varepsilon $--algebras  over the Poisson group  $ \; \text{\sl Spec}(Z_0) = \text{\sl Spec}\big(\, \Utildequnogd \big) = \text{\sl Spec}\big( \Oc\big( \Gtildebstar \big) \big) = \Gtildebstar \, $,  \,and, as the extension is  {\sl cleft},  that bundle is  {\sl globally trivializable}.
\end{rmks}

\vskip13pt

\begin{free text}  \label{small-MpQG's_identification}
 {\bf Small MpQG's: identifying the two realizations.}  So far we considered small MpQG's of two kinds,
 namely restricted and unrestricted ones.
%%%%%
% , defined by means of restricted or unrestricted integral forms of MpQG's
% and their associated quantum Frobenius morphisms.
%%%%%
 We will show now that these two types  {\sl over}  $ \, \QQ_\varepsilon $  actually coincide, up to isomorphism:
\end{free text}

\vskip7pt

\begin{theorem}  \label{thm: small-MpQG_hat=tilde}
 Consider the associated small MpQG's of either type over the ground ring\/  $ \QQ_\varepsilon $  (via scalar extension from  $ \, \RbqBeps = \ZZ[\varepsilon] \, $).
                                                                 \par
   Then  $ \,\; \utildeqepsgd \, \cong \, \uhatdotqepsgd \;\, \big( = \uhatqepsgd \big) \;\, $  as Hopf algebras over\/  $ \QQ_\varepsilon \, $.
                                                                 \par
   A similar statement holds true for the various (small) quantum subgroups, namely
 $ \,\; \utildeqeps^{\,\raise-3pt\hbox{$ \scriptstyle \geq $}} \cong \, \uhatdotqeps^{\,\raise-3pt\hbox{$ \scriptstyle \geq $}} \, \big( = \uhatqeps^{\,\raise-3pt\hbox{$ \scriptstyle \geq $}} \big) \;\, $,  \,etc.
\end{theorem}

\pf
 We prove that  $ \,\; \utildeqepsgd \, \cong \, \uhatdotqepsgd \, \big( = \uhatqepsgd \big) \;\, $,  \, the rest being similar.
 \vskip5pt
   To begin with, from  Proposition \ref{prop: uhatdotqeps = uhatqeps}  we know that  $ \, \uhatdotqeps := \uhatdotqepsgd \, $
   --- when defined over the extended ground ring  $ \QQ_\varepsilon $  ---   is generated by
   $ \; {\big\{ E_i \, , L_i^{\pm 1} , K_i^{\pm 1} , F_i \big\}}_{i \in I} \; $.  Moreover, from
   Theorem \ref{thm: uhat(dot)qeps_Hopf-subalg_&_presentation}  we can deduce a complete set of relations
   for this generating set: indeed, these relations can be also described as being of two types:
 \vskip3pt
   {\it (a)}\,  {\sl the relations arising (through specialization) from those respected by the same-name elements
   --- i.e.,  $ \, E_i \, , L_i^{\pm 1} , K_i^{\pm 1} , F_i \, $  ($ \, i \in I \, $)  ---   inside the restricted MpQG
   $ \, \Uhatdotqeps $  (before specialization) just by formally writing  ``$ \, \varepsilon \, $''  instead of  ``$ \, q \, $''},
 \vskip3pt
   {\it (b)}\,  {\sl the ``singular'' relations
  $ \; E_i^{\,\ell} \, = \, 0 \, $,  $ \; L_i^{\,\ell} - 1 \, = \, 0 \, $,  $ \; K_i^{\,\ell} - 1 \, = \, 0 \, $,  $ \; F_i^{\,\ell} \, = \, 0 \; $  ($ \, i \in I \, $)
  that are induced
from the relations in  $ \Uhatdotqeps $
  $$  X_i^{\,\ell} \; = \; {[\ell]}_{q_i\!}! \, X_i^{(\ell)}  \quad ,  \qquad  {\textstyle \prod\limits_{s=0}^{\ell-1}} {\bigg(\, {{Y_i \,; -s} \atop 1} \,\bigg)}_{\!q} \;
  = \; {\textstyle \prod\limits_{c=1}^{\ell-1}} {\bigg(\, {{c+1} \atop c} \,\bigg)}_{\!q} \cdot {\bigg(\, {{Y_i \,; 1-\ell} \atop \ell} \,\bigg)}_{\!q}  $$
--- for all  $ \, i \in I \, $,  $ \, X \in \{ E \, , F \} \, $  and  $ \, Y \in \{ L \, , K \} \, $  ---
when specializing  $ q $  to  $ \varepsilon \, $.
 }
 \vskip3pt
   Overall, this provides another concrete, explicit presentation of  $ \uhatdotqeps $  {\sl over}  $ \QQ_\varepsilon $  by generators and relations (with less generators than that arising from  Theorem \ref{thm: uhat(dot)qeps_Hopf-subalg_&_presentation}).  In addition, as a byproduct we find   --- comparing with  Theorem \ref{thm: PBW_small-uhat}  ---   another PBW theorem for  $ \uhatdotqeps $  (over  $ \QQ_\varepsilon \, $),  stating that  $ \, \uhatdotqeps \, $  admits the following  $ \QQ_\varepsilon $--basis
\begin{equation}  \label{2nd_PBW-basis_uhatdotqeps}
  \Big\{\, {\textstyle \prod_{k=N}^1} F_{\beta^k}^{\,f_k} \, {\textstyle \prod_{j \in I}} L_j^{l_j} \, {\textstyle \prod_{i \in I}} K_i^{c_i} \,
  {\textstyle \prod_{h=1}^N} E_{\beta^h}^{\,e_h} \,\Big\}_{0 \leq f_k , \, l_j , \, c_i , \, e_h < \, \ell}
\end{equation}
 \vskip1pt
   On the other hand, we know by  Proposition \ref{prop: struct_Utildeqeps-utildeqeps}  that  $ \, \utildeqeps := \utildeqepsgd \, $
   is generated over  $ \QQ_\varepsilon $  by
 $ \; {\big\{ E_i \, , L_i \, , K_i \, , F_i \big\}}_{i \in I} \; $,
because  $ \, L_i^{\,\ell} = 1 = K_i^{\,\ell} \, $  in  $ \, \utildeqeps \, $,  by definition (so that we can get rid of  $ L_i^{-1} $  and  $ K_i^{-1} \, $);
in particular, from  Theorem \ref{thm: PBW_small-utilde}  we can deduce that another possible PBW
$ \QQ_\varepsilon $--basis  for  $ \utildeqeps $  is
\begin{equation}  \label{2nd_PBW-basis_utildeqeps}
  {\Big\{\, {\textstyle \prod_{k=N}^1} F_{\beta^k}^{\,f_k} \, {\textstyle \prod_{j \in I}} L_j^{l_j} \, {\textstyle
  \prod_{i \in I}} K_i^{c_i} \, {\textstyle \prod_{h=1}^N} E_{\beta^h}^{\,e_h} \,\Big\}}_{0 \leq f_k , \, l_j , \, c_i , \, e_h < \, \ell}
\end{equation}
   \indent   Now, the generators  $ \, E_i \, , L_i , K_i , F_i \, $  ($ \, i \in I \, $)  of  $ \utildeqeps $  do respect all
   relations that come by straightforward rescaling from the relations respected by the generators
   $ \, \Ebar_i \, , L_i \, , K_i \, , \Fbar_i \, $  ($ \, i \in I \, $).  In turn, the latter are of two types:
 \vskip3pt
   {\it (a)}\,  {\sl the relations arising (through specialization) from those respected by the same-name elements
   --- i.e.,  $ \, \Ebar_i \, , L_i^{\pm 1} , K_i^{\pm 1} , \Fbar_i \, $  ($ \, i \in I \, $)  ---
   inside the unrestricted MpQG  $ \, \Utildeqeps \, $  (before specialization) by formally writing  ``$ \, \varepsilon \, $''
   instead of  ``$ \, q \, $''},
 \vskip3pt
   {\it (b)}\,  {\sl the ``singular'' relations
  $ \; \Ebar_i^{\,\ell} \, = \, 0 \, $,  $ \; L_i^{\,\ell} - 1 \, = \, 0 \, $,  $ \; K_i^{\,\ell} - 1 \, = \, 0 \, $,  $ \; \Fbar_i^{\,\ell} \, = \, 0 \; $  ($ \, i \in I \, $)
 induced from the ``relations'' in  $ \Utildeqeps $
  $$  \overline{X}_i^{\,\ell} \; \equiv \; 0  \mod {(Z_0)}^+  \quad ,  \qquad  \overline{Y}_i^{\,\ell} \; \equiv \; 1  \mod {(Z_0)}^+  $$
--- for all  $ \, X \in \{ E \, , F \} \, $,  $ \, Y \in \{ L \, , K \} \, $,  $ \, i \in I \, $  ---
when one specializes  $ q $  to  $ \varepsilon \, $}.
 \vskip3pt
   The outcome is that all this yields an explicit presentation of  $ \utildeqeps $  over  $ \QQ_\varepsilon $  by generators
   --- namely  $ \, E_i \, , L_i \, , K_i \, , F_i \, $  ($ \, i \in I \, $)  ---   and relations.
 \vskip5pt
   Comparing the previous analyses, we find that  $ \uhatdotqeps $  and  $ \utildeqeps $  share
   identical presentation: more precisely,  {\sl mapping
%
%%%%%
%   $$  E_i \mapsto E_i  \quad ,  \qquad  L_i \mapsto L_i  \quad ,  \qquad  K_i \mapsto K_i  \quad ,
%   \qquad  F_i \mapsto F_i   \eqno \qquad  \big(\, i \in I \,\big)   \quad  $$
% %
%%%%%
%
  $ \; E_i \mapsto E_i \, $,  $ \; L_i \mapsto L_i \, $,  $ \; K_i \mapsto K_i \, $,  $ \; F_i \mapsto F_i \; \big(\, i \in I \,\big) \; $
 yields a well-defined isomorphism of\/  $ \QQ_\varepsilon $--algebras\/};  in addition,
tracking the whole construction one sees at once that  {\sl this is also a morphism of Hopf algebras}.
Finally, comparing  \eqref{2nd_PBW-basis_uhatdotqeps}  and  \eqref{2nd_PBW-basis_utildeqeps}  shows that
{\sl this is indeed an isomorphism},  q.e.d.
\epf

\vskip13pt

\begin{rmk}
 As an application of the previous result, even for the classical (uniparameter) small quantum groups one
 can always make use of either realization of them: the
(most widely used) restricted one, or the unrestricted one.
\end{rmk}

\bigskip
 \bigskip

%%%%%%%%%

\vskip25pt

\end{document}